\documentclass[10pt,reqno]{amsart}
\usepackage{amssymb,mathrsfs,color}

\usepackage{enumerate}
\usepackage{graphicx}
\usepackage{xcolor} % A package to add color.
\usepackage{tensor}
\usepackage[letterpaper, margin=1in]{geometry}
\usepackage{amsfonts,amssymb,amsmath}
\usepackage{slashed}
\usepackage{ mathrsfs }
\usepackage[titletoc,title]{appendix}
\usepackage{wrapfig}
\usepackage{cancel}

\usepackage{cite}

\usepackage{nth}

\usepackage{marginnote}

\usepackage{pgfplots}

\usepackage{verbatim}

\usepackage{mathtools}
\numberwithin{equation}{section}
\mathtoolsset{showonlyrefs}
\usepackage{empheq}
\usepackage[shortlabels]{enumitem}

\definecolor{green}{rgb}{0,0.8,0} % Redefines the color green.

\newcommand{\brk}[1]{\langle#1\rangle}

\newcommand{\br}[1]{\overline{#1}}

\definecolor{deepgreen}{cmyk}{1,0,1,0.5}

\newcommand{\EQ}[1]{\begin{equation}\begin{split} #1 \end{split}\end{equation}}

\newcommand{\Del}[1]{}

\numberwithin{equation}{section}

\newtheorem{theorem}{Theorem}[section]
\newtheorem{corollary}[theorem]{Corollary}%[section]
\newtheorem{lemma}[theorem]{Lemma}%[section]
\newtheorem{proposition}[theorem]{Proposition}%[section]
\newtheorem{remark}[theorem]{Remark}%[section]
\newtheorem{definition}[theorem]{Definition}%[section]
\newcommand{\mas}{{\ \ \text{as} \ \ }}

\newcommand{\norm}[1]{\Vert#1\Vert}
\newcommand{\abs}[1]{\vert#1\vert}

\newcommand{\set}[1]{\{#1\}}

\newcommand{\jap}[1]{\langle #1\rangle}
\renewcommand{\Re}{\mathrm{Re}}
\renewcommand{\Im}{\mathrm{Im}}
\newcommand{\dvol}{\mathrm{dvol}}

\newcommand{\bfa}{{\bf a}}
\newcommand{\bfb}{{\bf b}}

\newcommand{\bfd}{{\bf d}}
\newcommand{\bfe}{{\bf e}}

\newcommand{\bfh}{{\bf h}}

\newcommand{\bfm}{{\bf m}}
\newcommand{\bfn}{{\bf n}}

\newcommand{\bfA}{{ A}}
\newcommand{\bfB}{{ B}}
\newcommand{\bfC}{{ C}}
\newcommand{\bfD}{{\bf D}}

\newcommand{\bfF}{{\bf F}}

\newcommand{\bfH}{{\bf H}}

\newcommand{\bfR}{{\bf R}}

\newcommand{\bfX}{{\bf X}}
\newcommand{\bfY}{{\bf Y}}
\newcommand{\bfZ}{{\bf Z}}

\newcommand{\bfmu}{\boldsymbol{\mu}}
\newcommand{\bfnu}{\boldsymbol{\nu}}
\newcommand{\bfxi}{\boldsymbol{\xi}}

\newcommand{\barg}{{\overline g}}

\renewcommand{\hbar}{{\underline h}}

\newcommand{\bbC}{\mathbb C}

\newcommand{\bbH}{\mathbb H}

\newcommand{\bbR}{\mathbb R}
\newcommand{\bbS}{\mathbb S}

\newcommand{\bbZ}{\mathbb Z}

\newcommand{\calD}{\mathcal D}
\newcommand{\calE}{\mathcal E}

\newcommand{\calH}{\mathcal H}

\newcommand{\calL}{\mathcal L}
\newcommand{\calM}{\mathcal M}
\newcommand{\calN}{\mathcal N}

\newcommand{\calS}{\mathcal S}

\newcommand{\tilM}{{\tilde{M}}}

\newcommand{\ringA}{{\mathring A}}

\newcommand{\ringC}{{\mathring C}}

\newcommand{\ringpsi}{\mathring{\psi}}
\newcommand{\ringPsi}{\mathring{\Psi}}

\newcommand{\supp}{\mathrm{supp}\,}
\newcommand{\ds}{\frac{\ud s}{s}}
\newcommand{\dsp}{\frac{\ud s'}{s'}}

\newcommand{\dspp}{\frac{\mathrm{d} s''}{s''}}

\newcommand{\dsigmap}{{\frac{\ud \sigma'}{\sigma'}}}
\newcommand{\dsigma}{{\frac{\ud \sigma}{\sigma}}}
\newcommand{\bR}{{\mathbb R}}

\newcommand{\bC}{{\mathbb C}}

\newcommand{\low}{{\mathrm{low}}}

\newcommand{\Db}{\mathbf{D}}

\newcommand{\ud}{\mathrm{d}}

\newcommand{\Ls}{L_{\frac{\ud s}{s}}}

\newcommand{\LLs}{L_{\frac{\ud s}{s}}^2\cap L_\ds^\infty}

\newcommand{\rest}{\!\!\restriction}

\newcommand{\bsomega}{{\boldsymbol{\omega}}}
\newcommand{\bsX}{{\boldsymbol{X}}}

\newcommand{\AQ}{\ringA_Q}
\newcommand{\AL}{\ringA_L}

\newcommand{\lot}{{\mathrm{l.o.t.}}}

\newcommand{\Ann}{A}

\newcommand{\nrm}[1]{\Vert#1\Vert}

\newcommand{\lap}{\Dlt}
\newcommand{\nb}{\nabla}
\newcommand{\rd}{\partial}

\newcommand{\aleq}{\lesssim}
\newcommand{\ageq}{\gtrsim}

\newcommand{\alp}{\alpha}
\newcommand{\bt}{\beta}

\newcommand{\Gmm}{\Gamma}
\newcommand{\dlt}{\delta}
\newcommand{\Dlt}{\Delta}
\newcommand{\eps}{\epsilon}

\newcommand{\kpp}{\kappa}
\newcommand{\lmb}{\lambda}

\newcommand{\sgm}{\sigma}
\newcommand{\Sgm}{\Sigma}
\newcommand{\tht}{\theta}

\newcommand{\vtht}{\vartheta}
\newcommand{\omg}{\omega}
\newcommand{\Omg}{\Omega}

\newcommand{\bsxi}{{\boldsymbol{\xi}}}
\newcommand{\bszeta}{{\boldsymbol{\zeta}}}

\newcommand{\bfkpp}{\boldsymbol{\kappa}}
\newcommand{\bflmb}{\boldsymbol{\lambda}}

\newcommand{\epsN}{\eps_{1}}							% small parameter for finding homotopy
\newcommand{\epshf}{\eps_{\ast}}							% small parameter for well-posedness near $Q$
\newcommand{\tgmfd}{\calN}							% Target manifold
\newcommand{\tg}{\tilde{\calN}}							% Target
\newcommand{\tcv}{\bfkpp}								% the Gauss curvature of the target manifold
\newcommand{\pfstep}[1]{\smallskip \noindent {\bf #1.}}

\usepackage{wrapfig}
\usepackage{tikz}
\usetikzlibrary{arrows,calc,decorations.pathreplacing}
\definecolor{light-gray1}{gray}{0.90}
\definecolor{light-gray2}{gray}{0.80}

\newcommand{\LL}{\mathcal{L}}
\newcommand{\NN}{\mathcal{N}}

\newcommand{\C}{\mathbb{C}}
\newcommand{\Hp}{\mathbb{H}}
\newcommand{\N}{\mathbb{N}}
\newcommand{\R}{\mathbb{R}}
\newcommand{\Sp}{\mathbb{S}}
\newcommand{\Z}{\mathbb{Z}}

\newcommand{\bs}[1]{\boldsymbol{#1}}

\newcommand{\al}{\alpha}
\newcommand{\be}{\beta}

\newcommand{\de}{\delta}

\newcommand{\fy}{\varphi}

\newcommand{\te}{\theta}

\newcommand{\s}{\sigma}

\newcommand{\De}{\Delta}
\newcommand{\Om}{\Omega}

\newcommand{\p}{\partial}
\newcommand{\na}{\nabla}

\makeatletter

\newcommand{\Rmnum}[1]{\expandafter\@slowromancap\romannumeral #1@}
\makeatother

\newcommand{\ti}{\widetilde}
\newcommand{\ba}{\overline}

\newcommand{\ang}[1]{\left\langle{#1}\right\rangle}
\renewcommand\Re{\mathrm{Re}\,}
\renewcommand\Im{\mathrm{Im}\,}

\newcommand{\bb}{\Big}

\newcommand{\Str}{\mathrm{Str}}

\newcommand{\ringbfA}{{\mathring{\bfA}}}
\newcommand{\ringbfB}{{\mathring{\bfB}}}
\newcommand{\ringbfC}{{\mathring{\bfC}}}

\newcommand{\bfAQ}{\ringbfA_Q}
\newcommand{\bfAL}{\ringbfA_L}

\title[Asymptotic stability of harmonic maps on $\bbH^2$ under the Schr\"odinger maps evolution]{Asymptotic stability of harmonic maps on the hyperbolic plane under the Schr\"odinger maps evolution}

\author{A. Lawrie}
\author{J. L\"uhrmann}
\author{S.-J. Oh}
\author{S. Shahshahani}

\begin{document}

\begin{abstract} 
We consider the Cauchy problem for the Schr\"odinger maps evolution when the domain is the hyperbolic plane.  An interesting feature of this problem compared to the more widely studied case on the Euclidean plane is the existence of a rich new family of finite energy harmonic maps. 
These are stationary solutions, and thus play an important role in the dynamics of Schr\"odinger maps. 
The main result of this article is the asymptotic stability of (some of) such harmonic maps under the Schr\"odinger maps evolution. More precisely, we prove the nonlinear asymptotic stability of a finite energy equivariant harmonic map $Q$ under the Schr\"odinger maps evolution with respect to non-equivariant perturbations, provided $Q$ obeys a suitable linearized stability condition. This condition is known to hold for all equivariant harmonic maps with values in the hyperbolic plane and for a subset of those maps taking values in the sphere. 
One of the main technical ingredients in the paper is a global-in-time local smoothing and Strichartz estimate for the operator obtained by linearization around a harmonic map, proved in the companion paper~\cite{LLOS1}. 
\end{abstract} 

\thanks{A. Lawrie was supported by NSF grant DMS-1700127 and a Sloan Research Fellowship. S.-J. Oh was supported by Samsung Science and Technology Foundation under Project Number SSTF-BA1702-02. 
A.L., J.L., and S.S. thank the Korea Institute for Advanced Study for its hospitality where part of this work was conducted.}

\maketitle

\section{Introduction}
Schr\"odinger maps are a geometric generalization of wave functions (i.e., complex-valued solutions $\psi$ to the Schr\"odinger equation $\rd_{t} \psi = i \lap \psi$) to Riemann-surface-valued maps. More precisely, given a Riemannian manifold $\Sgm$ and a Riemann surface\footnote{In this paper, by a \emph{Riemann surface}, we mean an orientable $2$-dimensional Riemannian manifold $(\tgmfd, g)$. By orientability, there exists a Riemannian volume form $\omg$, which in turn induces a compatible parallel complex structure $J$ by the relation $g(X, Y) = \omg(X, J Y)$.} $\tgmfd$, a map $u \colon \bR \times \Sgm \to \tgmfd$ is called a Schr\"odinger map if it solves the equation
\begin{equation} \label{equ:schroedinger_maps_equ}
 \partial_t u = \bfh^{jk} J(u) D_{j} \partial_{k} u,
\end{equation}
where $\bfh$ is the metric on $\Sgm$, $D$ is the pull-back covariant derivative on $u^{\ast} T \tgmfd$ (extended to $T^{\ast} \Sgm \otimes u^{\ast} T \tgmfd$ in the natural fashion) and $J(u)$ denotes the complex structure on~$u^\ast T\tgmfd$. 

In this article, we consider the Cauchy problem for the Schr\"odinger maps evolution when the domain is the hyperbolic plane, i.e., $\Sgm = \bbH^{2}$.
An interesting feature of this problem compared to the more widely studied case $\Sgm = \bbR^{2}$ is the existence of a rich new family of finite energy (i.e., $\rd u \in L^{2}(\Sgm)$) harmonic maps, which are stationary solutions to \eqref{equ:schroedinger_maps_equ} and thus play an important role in the dynamics of Schr\"odinger maps. For instance, while it is well-known that no finite energy harmonic maps from $\bbR^{2}$ into $\bbH^{2}$ exist other than the constant maps, there exist infinitely many nonconstant finite energy harmonic maps from $\bbH^{2}$ into $\bbH^{2}$, parametrized by the ``boundary data at infinity''; for more discussion on this point, we refer to \cite{LOS1}.

The main result of the present article is the asymptotic stability of (some of) such harmonic maps under the Schr\"odinger maps evolution. More precisely, we prove the asymptotic stability of a finite energy equivariant harmonic map $Q$ into $\bbS^{2}$ or $\bbH^{2}$ under the Schr\"odinger maps evolution with respect to possibly non-equivariant perturbations, provided that $Q$ obeys a suitable linearized stability condition. In particular, it is applicable to any finite energy equivariant harmonic map $Q$ when $\tgmfd = \bbH^{2}$, and to maps whose image is contained in a small enough geodesic ball when $\tgmfd = \bbS^{2}$. We refer to Subsection~\ref{subsec:result} for a more definite formulation of the main result (Theorem~\ref{thm:main}).

The present work may be thought of as the Schr\"odinger maps analogue of the work \cite{LOS1}, which concerns the asymptotic stability of finite energy equivariant harmonic maps from $\bbH^{2}$ to $\tgmfd = \bbS^{2}$ or $\bbH^{2}$ under the equivariant wave maps evolution, and of the works \cite{Li1, Li2, LMZ-DPDE}, which are on the same problem without any symmetry assumptions when $\tgmfd = \bbH^{2}$. However, analysis of the Schr\"odinger maps equation without symmetry around a nonconstant harmonic map brings on new challenges in comparison with the previous cases, which are discussed (along with their resolutions) in Subsection~\ref{subsec:ideas} below. A more detailed discussion of related prior works is given in Subsection~\ref{subsec:history}.

\subsection{Main result} \label{subsec:result}
The aim of this subsection is to provide a first statement of the main result (Theorem~\ref{thm:main}). We begin with some necessary preliminaries; we will often leave the details to Section~\ref{sec:prelim} below. In all of what follows, $\Sgm = \bbH^{2}$ and $\tgmfd = \bbS^{2}$ or $\bbH^{2}$. 

\subsubsection*{Finite energy equivariant harmonic maps into $\tgmfd = \bbS^{2}$ or $\bbH^{2}$}
Fix a point (the origin) in $\Sgm$ and $\tgmfd$, and consider the action of the rotation group $SO(2) = \set{e^{i \tht} : \tht \in \bbR}$ that fixes this point (there are two possible actions depending on the orientation, but the precise choice is inconsequential). We denote the infinitesimal generator of the rotation (in the direction of increasing $\tht$)  on $\Sgm$ and $\tgmfd$ by $\Omg$ and $\tilde{\Omg}$, respectively. Given a map $u : \Sgm \to \tgmfd$, we define its \emph{infinitesimal equivariant rotation} by
\begin{equation} \label{eq:equiv-rot}
	\calD_{\Omg} u = \ud u (\Omg) - \tilde{\Omg}(u).
\end{equation}
We say that $u$ is \emph{equivariant} (or co-rotational) if $\calD_{\Omg} u = 0$. Graphically, equivariance means that $u$ is invariant under the simultaneous rotations of the domain and the target manifolds by the same angle. In the polar coordinates\footnote{That is, given a point $P$ in $\Sgm$ (resp.~$\tgmfd$), $r$ (resp.~$\rho$) is the geodesic distance from the origin $O$, and $\tht$ (resp.~$\vtht$) is the angle measured at the origin between the geodesic segment $\overline{OP}$ and a reference ray emanating from $O$.} $(r, \tht)$ and $(\rho, \vtht)$ on $\Sgm$ and $\tgmfd$ defined with respect to the origin chosen above, an equivariant map $u$ takes the form
\begin{equation*}
	u(r, \tht) = (\rho(r), \tht).
\end{equation*}

In \cite{LOS1}, all finite energy equivariant harmonic maps from $\Sgm = \bbH^{2}$ into $\tgmfd = \bbS^{2}$ or $\bbH^{2}$ were classified.
\begin{proposition}[{\cite[Propositions~2.1 and 2.2]{LOS1}}] \label{prop:hm-classify}
The following statements hold.
\begin{enumerate}
\item Consider the case $\tgmfd = \bbS^{2}$. For every $0 \leq \alp < \pi$, there exists a unique finite energy equivariant harmonic map $Q_{\lmb} : \bbH^{2} \to \bbS^{2}$ such that $\lim_{r \to \infty} Q_{\lmb}(r, \tht) = (\alp, \tht)$. It takes the form
\begin{equation*}
	Q_{\lmb}(r, \tht) = (2 \mathrm{arctan}(\lmb \tanh(r/2)), \tht),
\end{equation*}
where $\lmb$ is determined by the relation $\alp = 2 \mathrm{arctan}(\lmb)$.
\item Consider the case $\tgmfd = \bbH^{2}$. For every $0 \leq \bt < \infty$, there exists a unique finite energy equivariant harmonic map $P_{\lmb} : \bbH^{2} \to \bbH^{2}$ such that $\lim_{r \to \infty} P_{\lmb}(r, \tht) = (\bt, \tht)$. It takes the form
\begin{equation*}
	P_{\lmb}(r, \tht) = (2 \mathrm{arctanh}(\lmb \tanh(r/2)), \tht),
\end{equation*}
where $\lmb$ is determined by the relation $\bt = 2 \mathrm{arctanh}(\lmb)$.
\end{enumerate}
\end{proposition}

\subsubsection*{Linearized stability conditions}
We are interested in the asymptotic stability of the harmonic maps in Proposition~\ref{prop:hm-classify} under the nonlinear Schr\"odinger maps evolution. The starting point for this investigation is the formulation of suitable stability conditions for the linearized equation about such a harmonic map. 

Given a harmonic map $Q : \Sgm \to \tgmfd$ and a section $\phi$ of $Q^{\ast} T \tgmfd$, we introduce the operator 
\begin{equation} \label{eq:lin-hm}
	\calH_{Q} \phi = - D^{\ell} D_{\ell} \phi - R(Q)(\phi, \rd_{\ell} Q) \rd^{\ell} Q,
\end{equation}
where $D$ is the pull-back covariant derivative on $Q^\ast T \tgmfd$ and $R$ the pull-back curvature operator. Here and below, the standard convention of raising and lowering tensorial indices on the domain $\Sgm = \bbH^{2}$ via the metric $\bfh$ is in effect. This operator is linear and symmetric for smooth compactly supported $\varphi$ with respect to the natural $L^{2}$-pairing for sections of $Q^{\ast} T \tgmfd$. It arises as the linearization of the harmonic maps equation $D^{\ell} \rd_{\ell} u = 0$ about $Q$. Following a common convention, we call $\calH_{Q}$ the \emph{linearized operator} about $Q$. 

To motivate the first linearized stability condition, we start with an alternative interpretation of $\calH_{Q}$ in the context of calculus of variations. A harmonic map is a critical point (at least formally, with respect to the $L^{2}$ inner product for sections of $Q^{\ast} T \tgmfd$) of the energy functional
\begin{equation*}
\calE(u) = \frac{1}{2} \int_{\Sgm} (\rd^{\ell} u, \rd_{\ell} u).
\end{equation*}
The linearized operator $\calH_{Q}$ is the Hessian of the energy functional at the critical point $u = Q$. Thus, if $Q$ is a local minimum of the energy functional (in a suitable class of maps), then $\int_{\Sgm} (\calH_{Q} \phi, \phi)$ needs to be nonnegative. Our first linearized stability condition is a strengthening of this property:

\begin{definition} [Weak linearized stability] \label{def:weak-stab}
We say that $Q$ obeys the \emph{weak linearized stability condition}, or that it is \emph{weakly linearly stable}, if there exists $\rho_{Q} > 0$ such that, for any smooth compactly supported section $\phi$ of $Q^{\ast} T \calN$, 
\begin{equation} \label{eq:weak-stab}
	\int_{\Sgm} (\calH_{Q} \phi, \phi)  \geq \rho_{Q}^{2} \int_{\Sgm} (\phi, \phi).
\end{equation}
\end{definition}
By the Poincar\'e inequality (or equivalently, the spectral gap $\frac{1}{4}$ for the Laplacian) on $\bbH^{2}$, we note that \eqref{eq:weak-stab} holds with $\rho_{Q}^{2} =\frac{1}{4}$ when $Q$ is the constant map \cite{Bray}. In fact, all harmonic maps in Proposition~\ref{prop:hm-classify} are weakly linearly stable; see Proposition~\ref{prop:weak-stab} below.

The weak linearized stability condition suffices for the local rigidity of the harmonic map $Q$ under suitably regular and localized perturbations, as well as for the asymptotic stability of $Q$ under the harmonic maps \emph{heat flow} evolution (see Theorem~\ref{thm:hmhf-gwp-simple}; note that the latter implies the former). However, the implication of the weak linearized stability condition for the Schr\"odinger maps evolution is less clear. For our main result, we require the following stronger condition, whose role is to ensure that the associated linearized Schr\"odinger maps equation has good dispersive behaviors (see Proposition~\ref{prop:strong-stab-led} and the related discussions):

\begin{definition} [Strong linearized stability] \label{def:strong-stab}
We say that $Q$ satisfies the \emph{strong linearized stability condition}, or that it is \emph{strongly linearly stable}, if the following conditions hold:
\begin{itemize}
\item {\bf No eigenvalues below $\frac{1}{4}$.} For any smooth compactly supported section $\phi$ of $Q^{\ast} T \calN$, 
\begin{equation} \label{eq:no-eval-cov}
	\int_{\Sgm} (\calH_{Q} \phi, \phi)  \geq \frac{1}{4} \int_{\Sgm} (\phi, \phi).
\end{equation}
\item {\bf No threshold resonance.} There exists $C_{0} > 0$ such that, for any smooth compactly supported section $\phi$ of $Q^{\ast} T \calN$, 
\begin{equation} \label{eq:no-thr-cov}
	\nrm{\phi}_{\bfH^{1}_{thr}} \leq C_{0} \nrm{(\tfrac{1}{4} - \calH_{Q}) \phi}_{\bfH^{-1}_{thr}}
\end{equation}
where 
\begin{equation} \label{eq:H1-thr-cov}
	\nrm{\phi}_{\bfH^{1}_{thr}} := \nrm{(D_{r} + \tfrac{1}{2}) \phi}_{L^{2}} + \nrm{\frac{1}{\sinh r} D_{\tht} \phi}_{L^{2}} + \nrm{\frac{1}{(1 + r^{2})^{\frac{1}{2}}} \phi}_{L^{2}},
\end{equation}
and $\bfH^{-1}_{thr}$ is the dual of $\bfH^{1}_{thr}$.
\end{itemize}
\end{definition}

For all harmonic maps in Proposition~\ref{prop:hm-classify} in the case $\tgmfd = \bbH^{2}$, and for those with sufficiently small $\lmb$ in the case $\tgmfd = \bbS^{2}$, the strong linearized stability condition holds; see Proposition~\ref{prop:strong-stab} below. However, not all harmonic maps in Proposition~\ref{prop:hm-classify} satisfy this condition. Indeed, in \cite{LOS1}, it was shown that $\calH_{Q}$ has an eigenvalue in the gap $(0, \frac{1}{4})$ if $\tgmfd = \bbS^{2}$ and $\lmb$ is large enough; see also \cite{LOS4} for analogous results for harmonic maps in higher equivariance classes. 

\begin{remark}
Analogous to \cite[\S~1.3.1, Remark~3]{LOS1}, the presence of a gap eigenvalue implies the existence of a non-decaying spatially localized solution to the linearized Schr\"odinger maps equation. Hence, the asymptotic stability of $Q$ cannot be proved by simply relying on the dispersive properties of the underlying linear equation. Nevertheless, on the basis of conservation of energy, local rigidity of $Q$ and the heuristic expectation that the Schr\"odinger map would asymptotically relax to a stationary solution, we still conjecture that all finite energy equivariant harmonic maps into $\bbS^{2}$ must be asymptotically stable, possibly via a nonlinear mechanism as in the work of Soffer--Weinstein \cite{SW99}.
\end{remark}

\subsubsection*{Fractional Sobolev spaces for maps} 
We now turn to the task of defining the analogue of fractional Sobolev spaces for maps $\Sgm \to \tgmfd$. While it is desirable to find an intrinsic construction, it is rather cumbersome to carry out. Instead, we follow a commonly taken shortcut and base our definition on an auxiliary device, namely, an isometric embedding of the relevant part of $\tgmfd$ into a Euclidean space.

In the case when $\tgmfd$ is a closed manifold (so, under our convention, $\tgmfd = \bbS^{2}$), the definitions are straightforward. By Nash's isometric embedding theorem, there exists an isometric embedding $\iota: \tgmfd \hookrightarrow \bbR^{N}$ for some $N$. Abusing the notation a bit, we identify $\tgmfd$ with the submanifold $\iota(\tgmfd) \subseteq \bbR^{N}$. Then for any $\sgm \in \bbR$, we define the space of $H^{\sgm}$ maps into $\tgmfd$ with the same boundary data at infinity as $Q$ to be
\begin{equation} \label{eq:HQs} 
	H^{\sgm}_{Q}(\Sgm ; \tgmfd)= \set{ u : \Sgm \to \bbR^{N}: u - Q \in H^{\sgm} \hbox{ and } u(x) \in \tgmfd \hbox{ for a.e. } x \in \Sgm},
\end{equation}
and equip it with the distance $\nrm{u - v}_{H^{\sgm}}$ (we refer to Subsection~\ref{subsec:fs} below for the definition of the space $H^{\sgm}$).
Similarly, we define the space of $H^{1} \cap C^{0}$ maps into $\tgmfd$ with the same boundary data at infinity as $Q$ to be
\begin{equation*}
	(H^{1} \cap C^{0})_{Q}(\Sgm ; \tgmfd)= \set{ u : \Sgm \to \bbR^{N}: u - Q \in H^{1} \cap C^{0} \hbox{ and } u(x) \in \tgmfd \hbox{ for a.e. } x \in \Sgm},
\end{equation*}
and equip it with the distance $\nrm{u - v}_{H^{1} \cap L^{\infty}}$. For $u \in H^{\sgm}_{Q}(\Sgm; \tgmfd)$ with $\sgm > 1$ or $(H^{1} \cap C^{0})_{Q}(\Sgm; \tgmfd)$, we always choose the continuous representative, so that $u(x) \in \tgmfd$ for all $x \in \Sgm$. 

The isometric embedding $\iota: \tgmfd \to \bbR^{N}$ also allows us to identify sections $\phi$ of $u^{\ast} T \tgmfd$ with $\bbR^{N}$-valued functions. We define the $L^{p}$ and $H^{\sgm}$ norms of $\phi$ by viewing it as an $\bbR^{N}$-valued function; note that the $L^{p}$ norms defined as such are equivalent to the intrinsic $L^{p}$ norms defined using the induced metric on $u^{\ast} T \tgmfd$.

Next, we consider the case when $\tgmfd$ is noncompact (so, under our convention, $\tgmfd = \bbH^{2}$). Note that there cannot exist a uniform isometric embedding of $\tgmfd = \bbH^{2}$ into any Euclidean space, simply by consideration of the growth of the length of circles. On the other hand, note that the image of all harmonic maps $Q$ in Proposition~\ref{prop:hm-classify} are bounded. As we are interested in small perturbations of $Q$, it suffices for our purposes to define distances for maps whose image is close to that of $Q$. 

More precisely, given a bounded neighborhood $\tg$ of $Q(\Sgm)$ in $\tgmfd$, consider a modification $\tgmfd'$ of $\tgmfd$ outside $\tg$ that is a \emph{closed} $2$-dimensional Riemann surface\footnote{In our context, $\tg$ is contained in a disk inside $\bbH^{2}$, so the desired closed Riemann surface (which is nothing but a closed orientable $2$-dimensional Riemannian manifold) $\tgmfd'$ is easily constructed.}. Fix an isometric embedding $\iota$ of $\tgmfd'$ into a Euclidean space. Then we define the spaces
$H^{\sgm}_{Q}(\Sgm; \tg)$ and $(H^{1} \cap C^{0})_{Q}(\Sgm; \tg)$ of maps taking values in $\tg$, as well as the $L^{p}$ and $H^{\sgm}$ norms of sections of the pull-back tangent bundles, via this isometric embedding in the same fashion as above.

To unify the two cases, in what follows, we always {\bf fix a bounded open set $\tg$ in $\tgmfd$ that contains the image of the harmonic map $Q$ we are interested in}, and then {\bf fix an isometric embedding (as Riemannian manifolds) of the modified closed Riemann surface $\tgmfd' \hookrightarrow \bbR^{N}$}, which agrees with $\tgmfd$ on $\tg \subset \tgmfd'$. All spaces of maps into $\tg$ and norms for sections of the pull-back tangent bundles are defined with respect to this fixed isometric embedding.

\subsubsection*{Asymptotic stability of harmonic maps under the heat flow and caloric gauge}
In order to discuss the asymptotic behavior of Schr\"odinger maps about $Q$, we need an appropriate choice of gauge, which, at the naive level, refers to a suitable way of describing the difference between $u$ and $Q$. A geometrically natural and highly useful choice, called a \emph{caloric gauge}, may be made with the help of the associated parabolic flow, namely the \emph{harmonic map heat flow}.
The caloric gauge was first introduced by Tao \cite{Tao04, Tao37} in the context of the wave maps equation on $\bbR^{2}$. We refer to Subsections~\ref{subsec:ideas} and \ref{subsec:history} for further discussion on the caloric gauge and its history.
A one-parameter family of maps $U(s) : \bbH^{2}_{x} \times \bbR^{+}_{s} \to \tgmfd$ is called a \emph{harmonic map heat flow} if
\begin{equation} \label{eq:hmhf}
	\rd_{s} U = D^{\ell} \rd_{\ell} U.
\end{equation}
As alluded to earlier, all harmonic maps $Q$ in Proposition~\ref{prop:hm-classify} are asymptotically stable with respect to \eqref{eq:hmhf}, essentially as a consequence of the weak linearized stability condition.
\begin{theorem} \label{thm:hmhf-gwp-simple}
Let $Q : \bbH^{2} \to \tgmfd$ be a finite energy equivariant harmonic map in Proposition~\ref{prop:hm-classify}. Fix a bounded neighborhood $\tg$ of $Q(\bbH^{2})$ in $\tgmfd$, and consider a smooth map $u$ from $\bbH^{2}$ into $\tg$ such that
\begin{equation*}
 	\nrm{u - Q}_{H^{1} \cap L^{\infty}} \leq \eps_{0}.
\end{equation*}
If $\eps_{0}$ is sufficiently small depending on $Q$ and $\tg$, then the harmonic map heat flow $U(x, s)$ with $U(x, 0) = u(x)$ exists globally. For all $s \in [0, \infty)$, we have $U(x, s) \in \tg$ and  
\begin{equation*}
 	\nrm{U(s) - Q}_{H^{1} \cap L^{\infty}} \aleq e^{-\frac{1}{2} \rho_{Q}^{2} s} \nrm{u - Q}_{H^{1} \cap L^{\infty}}.
\end{equation*}
\end{theorem}
Therefore, given a small perturbation $u$ of $Q$, the harmonic map heat flow $U$ with $U(x, 0) = u(x)$ gives a natural splitting of $u$ into $Q$ and the remainder, namely,
\begin{equation*}
	u^{A}(x) = Q^{A}(x) - \int_{0}^{\infty} \rd_{s} U^{A}(x, s) \, \ud s,
\end{equation*}
where we used the isometric embedding of $\tg$ into $\bbR^{N}$ to write the above relation. To analyze a Schr\"odinger map $u(t)$ near $Q$, our strategy is to work with the linearized object $\rd_{s} U(t)$ defined as above, which captures the difference between $u(t)$ and $Q$. The gauge freedom in this approach is the choice of an orthonormal frame $e$ of $U^{\ast} T \tgmfd$, with which $\rd_{s} U$ may be described intrinsically (i.e., without reference to any isometric embedding $\tg \hookrightarrow \bbR^{N}$) as a complex-valued function $\psi_{s}$ on $\bbH^{2} \times (0, \infty)$ via
\begin{equation*}
	\psi_{s} = (\rd_{s} U, e_{1}) + i (\rd_{s} U, e_{2}).
\end{equation*}
We refer to Subsection~\ref{subsec:moving-frame} for a more detailed discussion of this formalism. 

The gauge (or the orthonormal frame) we use is specified by the following definition. 
\begin{definition} \label{def:caloric-simple}
Let $U: \bbH^{2} \times [0, \infty) \to \bbH^{2}$ be a global harmonic map heat flow converging to $Q$ as in Theorem~\ref{thm:hmhf-gwp-simple}. We say that an orthonormal frame $e$ of $U^{\ast} T \tgmfd$ is the \emph{caloric gauge with the Coulomb gauge at infinity} if the following conditions hold:
\begin{itemize} 
\item {\bf Positive orientation.} $e_{2} = J(U) e_{1}$.
\item {\bf Caloric gauge condition.} $D_{s} e = 0$ on $\bbH^{2} \times (0, \infty)$.
\item {\bf Coulomb gauge condition at infinity.} The limit $e^{\infty}(x) = \lim_{s \to \infty} e(x, s)$ is a well-defined orthonormal frame on $Q^{\ast} T \tgmfd$ that obeys $D^{\ell} D_{\ell} e^{\infty} = 0$.
\end{itemize}
\end{definition}
We refer to Subsections~\ref{subsec:hmhf-lwp} and \ref{subsec:hmhf-stab} for more details on Theorem~\ref{thm:hmhf-gwp-simple}, and to Subsections~\ref{subsec:hm} and \ref{subsec:caloric} for further properties of the gauge given in Definition~\ref{def:caloric-simple} (including, of course, its existence and uniqueness up to a global rotation).

\subsubsection*{Asymptotic stability of harmonic maps under the Schr\"odinger maps evolution}
We are now ready to state the main result. 
\begin{theorem} \label{thm:main}
Let $Q \colon \bbH^{2} \to \tgmfd$ be a finite energy equivariant harmonic map as in Proposition~\ref{prop:hm-classify}. Assume, in addition, that $Q$ is strongly linearly stable. Fix any small number $\dlt > 0$ and a bounded neighborhood $\tg$ of $Q(\bbH^{2})$ in $\tgmfd$. Consider a smooth map $u_{0}$ from $\bbH^{2}$ into $\tg$ such that
\begin{equation} \label{eq:main-hyp}
 	\nrm{u_{0} - Q}_{H^{1+2\dlt}} + \nrm{\calD_{\Omg} u_{0}}_{H^{1+ 2\dlt}} \leq \eps_{0}.
\end{equation}
If $\eps_{0}$ is sufficiently small depending on $Q$ and $\tg$, then there exists a unique global Schr\"odinger map $u(t)$ with $u(0) = u_{0}$. For all $t \in \bbR$, we have $u(t, \bbH^{2}) \subset \tg$ and
\begin{equation*}
 	\nrm{u(t) - Q}_{H^{1+2\dlt}} + \nrm{\calD_{\Omg} u(t)}_{H^{1+ 2\dlt}} \aleq \nrm{u_{0} - Q}_{H^{1+2\dlt}} + \nrm{\calD_{\Omg} u_{0}}_{H^{1+ 2\dlt}}.
\end{equation*}
Moreover, $u$ asymptotes to $Q$ as $t \to \pm \infty$ in the sense that the following uniform point-wise convergence statement holds:
\begin{equation} \label{eq:main-ptwise}
	\nrm{u(t) - Q}_{L^{\infty}} \to 0 \quad \hbox{ as } t \to \pm \infty.
\end{equation}
In fact, $\psi_{s}$ belongs to the ``dispersive space'' $\calS(\bbR)$, to be defined in Subsection~\ref{subsec:disp} below, in the caloric gauge with the Coulomb gauge at infinity.
\end{theorem}

For a more precise statement of Theorem~\ref{thm:main}, see Theorem~\ref{thm:main1} below. We remark that \eqref{eq:main-ptwise} is \emph{qualitative}, in the sense that it requires the qualitative smoothness of $u_{0}$. On the other hand, we will be able to obtain a \emph{quantitative} control of the $\calS(\bbR)$ norm of $\psi_{s}$ in terms of the norm of $u_{0}$ on the LHS of \eqref{eq:main-hyp}. The finiteness of $\nrm{\psi_{s}}_{\calS(\bbR)}$ implies dispersive decay properties, namely the finiteness of appropriate global-in-time local smoothing and Strichartz norms; see Subsection~\ref{subsec:disp} for details. 

As a consequence of Theorem~\ref{thm:main} and Proposition~\ref{prop:strong-stab}, all harmonic maps in Proposition~\ref{prop:hm-classify} in the case $\tgmfd = \bbH^{2}$, as well as those with sufficiently small $\lmb$ in the case $\tgmfd = \bbS^{2}$, are asymptotically stable under the Schr\"odinger maps evolution in the sense of Theorem~\ref{thm:main}.

\begin{remark}
While the present paper only concerns the cases $\tgmfd = \bbS^{2}$ and $\bbH^{2}$ for the sake of simplicity, our approach seems applicable with minor modifications to any finite energy equivariant harmonic map $Q$ into any rotationally symmetric Riemann surface $\tgmfd$ (i.e., there exists an action of $SO(2)$ by isometries) obeying the strong linearized spectral stability condition and a suitable boundary condition at infinity. On the other hand, removing the equivariance assumption on $Q$ seems to require new ideas.
\end{remark}

\subsection{Main ideas for the proof of Theorem~\ref{thm:main}} \label{subsec:ideas}
\subsubsection*{The equations of motion in the caloric gauge}
As alluded to in Subsection~\ref{subsec:result}, our overall strategy for proving Theorem~\ref{thm:main} is to analyze the the linearized object $\rd_{s} U(t, x, s)$ in the caloric gauge (Definition~\ref{def:caloric-simple}). The main equations of motion in this formalism may be summarized as follows (see Subsection~\ref{subsec:equ-mo} for details):
\begin{itemize}
\item The complex-valued function $\psi_{s} = (\rd_{s} U, e_{1}) + i (\rd_{s} U, e_{2})$ obeys the inhomogeneous linearized Schr\"odinger maps equation about $U(t, x, s)$,
\begin{equation} \label{eq:caloric-sch-psi-s}
(i \Db_{t} + \Db^{k} \Db_{k}) \psi_{s} - i \tcv \Im(\psi^{k} \br{\psi_{s}}) \psi_{k} = i \rd_{s} w.
\end{equation}
This is the main equation for analyzing the time evolution of the Schr\"odinger map. Here, $\Db_{\bfmu} = \rd_{\bfmu} + i A_{\bfmu}$ is the induced covariant derivative with the connection $1$-form $A_{\bfmu} = (D_{\bfmu} e_{1}, e_{2})$, $\psi_{\bfmu} = (\rd_{\bfmu} U, e_{1}) + i (\rd_{\bfmu} U, e_{2})$ represents the differential of $U$ expressed with respect to the frame $(e_{1}, e_{2})$ and $\tcv$ is the Gauss curvature of $\tgmfd$; see Subsection~\ref{subsec:moving-frame} for more details on the notation. 

\item The complex-valued function $w = w(t, x, s)$ on the RHS of \eqref{eq:caloric-sch-psi-s} is the \emph{Schr\"odinger tension field}, 
\begin{equation*}
w = (\rd_{t} U - \bfh^{jk} J(U) D_{j} \rd_{k} U, e_{1}) + i (\rd_{t} U - \bfh^{jk} J(U) D_{j} \rd_{k} U, e_{2}), 
\end{equation*}
which measures the failure of the Schr\"odinger map equation for $U(t, x, s)$ with $s > 0$. It is estimated in terms of $\psi_{s}$ through the inhomogeneous linearized harmonic map heat flow about $U(t, x, s)$ in the $s$-direction,
\begin{equation} \label{eq:caloric-heat-w}
(\Db_{s} - \Db^{k} \Db_{k}) w + i \tcv \Im(\psi^{k} \br{w}) \psi_{k} = i \tcv \psi^{k} \psi_{k} \overline{\psi_{s}},
\end{equation}
along with the condition $w(s=0) = 0$, which is equivalent to the Schr\"odinger maps equation for $u(t, x) = U(t, x, 0)$. 

\item In the $s$-direction, $\psi_{s}$ obeys the linearized harmonic map heat flow about $U(t, x, s)$,
\begin{equation} \label{eq:caloric-heat-psi-s}
(\Db_{s} - \Db^{k} \Db_{k}) \psi_{s} + i \tcv \Im(\psi^{k} \br{\psi_{s}}) \psi_{k} = 0.
\end{equation}
\end{itemize}

To close this system of equations, we need to relate $A_{\bfmu}$ and $\psi_{\bfmu}$ with $\psi_{s}$; this is where the caloric gauge condition enters. First, computing $[\Db_{s}, \Db_{\mu}]$ in two different ways, it follows that $\rd_{s} A_{\mu} - \rd_{\mu} A_{s} = \tcv \Im(\psi_{s} \overline{\psi_{\ell}})$.
Moreover, the caloric gauge condition in Definition~\ref{def:caloric-simple} is equivalent to $A_{s} = (D_{s} e_{1}, e_{2}) = 0$. Hence, 
\begin{equation*}
	A_{\mu}(s) = -\int_{s}^{\infty} \tcv \Im(\psi_{s} \overline{\psi_{\ell}})(s') \, \ud s' + A^{\infty}_{\mu}(s),
\end{equation*}
where $A^{\infty}_{\mu} = (D_{\mu} e^{\infty}_{1}, e^{\infty}_{2})$. Next, by the compatibility condition $\Db_{s} \psi_{\mu} = \Db_{\mu} \psi_{s}$ and $A_{s} = 0$,
\begin{equation*}
	\psi_{\mu}(s) = - \int_{s}^{\infty} \Db_{\mu} \psi_{s}(s') \, \ud s' + \psi^{\infty}_{\mu}(s),
\end{equation*}
where $\psi^{\infty}_{\mu} = (\rd_{\mu} Q, e^{\infty}_{1}) + i (\rd_{\mu} Q, e^{\infty}_{2})$.
 
From these equations, we may observe several advantages of the caloric gauge formulation: 
\begin{enumerate}
\item While \eqref{equ:schroedinger_maps_equ} appears to be \emph{quasilinear} due to the presence of the coefficient $J(u)$ in front of $\bfh^{j k} D_{j} \rd_{k} u$, the main $t$-evolution equation \eqref{eq:caloric-sch-psi-s} is a \emph{semilinear Schr\"odinger equation} in $\psi_{s}$ (albeit with a derivative nonlinearity, which is still challenging; see the discussion below). 

\item The formula relating $A_{\mu}$ with the differential of $U$ is favorable (compared to, say, the Coulomb gauge) in that it does not involve inversion of an elliptic operator.

\item The key variables $\psi_{s}$ and $w$ are \emph{scalar}, as opposed to tensorial.

\item Even though $U(t, x, \infty) = Q(x)$ is a nontrivial map, the parts of the system \eqref{eq:caloric-sch-psi-s}, \eqref{eq:caloric-heat-w} and \eqref{eq:caloric-heat-psi-s} that are linear in $\psi_{s}$ and $w$ (over $\bbR$) are diagonal. The nontrivial observation is that the potentially non-diagonal linear terms, which are all of the form $i \tcv (\psi^{\infty})^{k} \psi^{\infty}_{k} \overline{\psi_{s}}$ or $i \tcv (\psi^{\infty})^{k} \psi^{\infty}_{k} \overline{w}$, vanish thanks to the Cauchy--Riemann equation satisfied by the harmonic map $Q$ between surfaces; see \eqref{eq:CR}.
\end{enumerate}
The first and second points were key in the proof of global regularity and scattering of small energy Schr\"odinger maps on $\bbR^{2}$ \cite{BIKT-annals}. The importance of avoiding the use of tensorial (as opposed to scalar) linear propagators, which may behave badly on a curved background such as $\Sgm = \bbH^{2}$, was emphasized in \cite{LOS5} in the context of wave maps on $\bbH^{d}$ $(d \geq 4)$. The final observation is one of the key structural underpinnings of our proof of Theorem~\ref{thm:main} in the caloric gauge.

\subsubsection*{Analysis of the Schr\"odinger equation for $\psi_{s}$}
The harmonic map heat flow equation \eqref{eq:hmhf} for $U$ allows us to relate $u(t) - Q$ with $\psi_{s}(t, s)$ in the caloric gauge under the bootstrap assumption that $u(t) - Q$ is not too large. More precisely, provided that $\nrm{u(t) - Q}_{H^{1} \cap L^{\infty}} \ll_{Q} 1$, for any $\dlt > 0$ we have
\begin{equation*}
\nrm{m(s) s^{\frac{1}{2}} \psi_{s}(t)}_{L^{2}_{\ds}((0, \infty); L^{2})} \aleq \nrm{u(t) - Q}_{H^{1+2\dlt}} \aleq \nrm{m(s) s^{\frac{1}{2}} \psi_{s}(t)}_{L^{2}_{\ds}((0, \infty); L^{2})}
\end{equation*}
where $m(s) = \min\set{s^{\dlt}, 1}$; see Subsections~\ref{subsec:forward-caloric} and \ref{subsec:backward-caloric} for the first and second inequalities, respectively. 

Accordingly, the key step of the proof of Theorem~\ref{thm:main} is setting up a bootstrap argument in the caloric gauge to
control $m(s) s^{\frac{1}{2}} \psi_{s}(t)$ in $L^{2}_{\ds}((0, \infty); L^{2}_{x})$ for all $t$ through the Schr\"odinger equation \eqref{eq:caloric-sch-psi-s}. The greatest difficulty in this step comes from the contribution of the nonlinearity in \eqref{eq:caloric-sch-psi-s} with a first-order derivative, referred to as the \emph{magnetic interaction term}. More specifically, \eqref{eq:caloric-sch-psi-s} may be rewritten as
\begin{equation*}
	(i \rd_{t} - H) \psi_{s}(s) = - 2i \ringA^{k}(s) \nb_{k} \psi_{s}(s) + (\hbox{easier terms}),
\end{equation*}
where $H$ is the expression of the linearized operator $\calH_{Q}$ with respect to the frame $(e_{1}, e_{2})$, and
\begin{equation*}
\ringA_{k}(s)=A_{k}(s)-A^{\infty}_{k}= - \int_{s}^{\infty} \tcv \Im(\psi_{s} \overline{\psi_{k}})(s') \, \ud s'.
\end{equation*}
To handle the presence of a derivative on the RHS, we rely on the local smoothing effect for the Schr\"odinger operator $i \rd_{t} - H$. The strong linearized stability condition (Definition~\ref{def:strong-stab}) allows us to apply the general theorem in the companion paper \cite{LLOS1} to conclude that $i \rd_{t} - H$ enjoys a global-in-time local smoothing estimate with spatial weights depending on $r$ (i.e., the distance to some fixed point); see Proposition~\ref{prop:strong-stab-led}. The global-in-time local smoothing estimate allows us to gain one derivative as needed; the price we pay is, among other things, that we need to uniformly bound $r^{4} \ringA^{k}(s)$ in the region $\set{r > 1}$. We obtain such a pointwise spatial decay of $\ringA^{k}$ by controlling its angular regularity (in the polar coordinates $(r, \tht)$ on $\Sgm$) and applying the radial Sobolev inequality (Lemma~\ref{lem:radial_sobolev}). The control of the extra angular regularity, in turn, is attained by commuting the Schr\"odinger maps equation with the equivariant rotation operator $\calD_{\Omg}$, which annihilates $Q$ thanks to its equivariance.

To make the above argument rigorous and sharp, we need to employ Littlewood--Paley projections $P_{\sgm}$ on $\bbH^{2}$, which we formulate using the linear heat flow (see Subsection~\ref{subsec:ineq}). The most dangerous paradifferential term (i.e., ``low-high'' interaction in $\ringA^{k}(s) \nb_{k} \psi_{s}(s)$) is treated using the global-in-time local smoothing estimate as sketched above; the sharp form of the spatial weights near $r = 0$ in the norms $LE$ and $LE^{\ast}$ (see Subsection~\ref{subsec:disp}) allows us to close the argument with an arbitrarily small $\dlt > 0$, where $\dlt = 0$ is optimal\footnote{Of course, Theorem~\ref{thm:main} is far from being optimal in terms of regularity and decay required for the initial data due to the presence of $\calD_{\Omg}$.}. Handling the remaining interactions in $\ringA^{k}(s) \nb_{k} \psi_{s}(s)$ requires additional tools, such as global-in-time Strichartz estimates for $i \rd_{t} - H$ on $\bbH^{2}$ (Lemma~\ref{lem:main_linear_estimate}, also proved in~\cite{LLOS1}), Bernstein-type inequalities adapted to the spatial weights in $LE$ (Lemmas~\ref{l:bern1} and \ref{l:bern2}) and the decay of the expression $P_{\sgm} \psi_{s}(s)$ off the diagonal $\set{\sgm \approx s}$ (Lemma~\ref{l:offdiag}). We refer to Section~\ref{s:schrod} for the rigorous treatment of the magnetic interaction term.

\subsection{Motivation and related works}  \label{subsec:history}

The Schr\"odinger maps equation arises in physics as a model equation in ferromagnetism, and is referred to there as the Heisenberg model or the Landau-Lifshitz equation; see e.g.,~\cite{KIK}.  From a purely mathematical perspective, Schr\"odinger maps are a natural generalization of the free Schr\"odinger equation for complex-valued fields to maps taking values in a K\"ahler manifold. The study of the Schr\"odinger maps equation~\eqref{equ:schroedinger_maps_equ} on Euclidean space as the domain $\Sigma = \bbR^d$, $d \geq 1$, has a long and rich history, see~\cite{DingWang, CSU, McGahagan, NSVZ, RRS09, GusKoo, GriStef, KLPST, GKT, GNT, BT-mams, MRR, Sm13, Smith14, Perel14, DS15, Bejenaru1, Bejenaru2, IK1, IK2, BIK, BIKT-annals, BIKT-duke, BIKT-kyoto, Li18, Li19, SSB1 } and references therein.

The case $\Sigma = \R^2$, $\NN = \Sp^2$ in~\eqref{equ:schroedinger_maps_equ} is of particular interest since it admits  explicit \emph{magnetic solitons} in the form of finite energy harmonic maps. In this setting, every smooth finite energy Schr\"odinger map comes with a topological invariant, its degree, and every harmonic map minimizes the energy amongst maps that share its degree.  The underlying symmetries of the equation turn each harmonic map of degree $k$ into a family of solutions.  The fact that the scaling symmetry preserves the energy makes this problem energy-critical.  The coercivity of the conserved energy (and the fact that harmonic maps are minimizers in their degree class) ensures that a solution that starts near the family generated by the ground state, stays near it. But the asymptotic dynamics of solutions near the ground state family of harmonic maps can be quite complicated. In fact, it was shown by Merle-Rapha\"{e}l-Rodnianski~\cite{MRR} and by Perelman~\cite{Perel14} that $1$-equivariant solutions can blow up in finite time by a concentration of energy in a dynamically rescaled harmonic map. Prior to this, Gustafson-Kang-Tsai~\cite{GKT} and later Gustafson-Nakanishi-Tsai~\cite{GNT} proved that harmonic maps with equivariance class $k \geq 3$ %$m \ge 3$ 
are asymptotically stable under equivariant perturbations and exhibited instability properties in the case $k=2$, %$m=2$
and  Bejenaru-Tataru~\cite{BT-mams} had proved instability in the case $k=1$. %$m=1$. 
Note that the stability results \cite{GKT} and \cite{GNT} are not expected to hold outside of equivariance, as the higher degree harmonic maps are not global minimizers of the energy. In fact, it is quite possible that no harmonic maps are stable under general energy class perturbations, although little is known in this fascinating direction. In this context, the change of domain geometry considered here ($\Sigma  = \Hp^2$) and Theorem~\ref{thm:main} produce a stark contrast. One appealing feature of~\eqref{equ:schroedinger_maps_equ} with $\Sigma = \Hp^2$ is that it is a geometrically natural, semilinear, dispersive PDE that exhibits asymptotically stable finite energy static solutions. Moreover, even though scaling is no longer a symmetry on $\bbH^2$, for highly localized solutions (or at high frequencies) the model exhibits many features of an energy-critical problem.

The study of nonlinear geometric PDEs on the hyperbolic plane was initiated in a sequence of works~\cite{LOS1, LOS3, LOS4, LOS5} on the related wave maps equation. There, one of the main goals was to understand the interplay between the geometry of the domain manifold and the nonlinear structure of the wave maps equation, the latter arising naturally from the constraints imposed by the geometry of the target. A main takeaway from these works is that there is a dramatic difference in nonlinear dynamics as compared to the Euclidean picture, mainly due to the presence of asymptotically stable harmonic maps.  In this vein, the present work is a natural continuation of~\cite{LOS1} which proved asymptotic stability of the same class of harmonic maps considered here, but in the much simpler setting of the \emph{equivariant symmetry reduced} wave maps equation. In view of the sophisticated linear machinery that is often needed for the global-in-time analysis of solutions to geometric wave equations for small initial data, stability results for non-scattering solutions outside of symmetry require overcoming many technical challenges.

For general, non-equivariant data the investigation of the wave maps equation on hyperbolic space $\Hp^d$ was commenced in~\cite{LOS5}, where optimal regularity small data global well-posedness in the case $d \geq 4$ was proved. There it was observed that the caloric gauge surmounts difficulties associated with other classical gauges (e.g., Coulomb gauge). The caloric gauge was first introduced by Tao \cite{Tao04, Tao37} in the context of the wave maps equation on $\bbR^{2}$, and was crucially used in the study of the energy-critical Schr\"odinger maps equation on $\bbR^{2}$, see \cite{BIKT-annals, Sm, Sm13, DS15, Li18}. 
A related caloric gauge construction based on the Yang-Mills heat flow is a key ingredient of the recent proof of the threshold conjecture for the energy-critical hyperbolic Yang-Mills equation \cite{Oh14, Oh15, OhTat_YM1, OhTat_YM2, OhTat_YM3, OhTat_YM4}. 
Finally, the caloric gauge played a fundamental role in the proof of asymptotic stability (in a subcritical Sobolev space) of nonconstant harmonic maps $Q : \bbH^{2} \to \bbH^{2}$ under the wave maps evolution in~\cite{Li1, Li2, LMZ-DPDE} and for the Landau-Lifshitz flow (but not the limiting case of Schr\"odinger maps) in~\cite{LZ-DCDS}. %We refer to Subsection~\ref{subsec:ideas} for further discussion on the caloric gauge.

As we discussed in Subsection~\ref{subsec:ideas}, the main linear estimates in this paper, i.e., local smoothing estimates, or local energy decay estimates, 
and Strichartz estimates for the operator obtained by linearization about a harmonic map, were established in the companion paper~\cite{LLOS1}. This local smoothing estimate is the key technical ingredient for dealing with the loss of smoothing for Schr\"odinger equations mentioned in the previous subsection. %paragraph.
%\Blue{This LED, or local smoothing estimate, is the key technical ingredient for dealing with the loss of smoothing for Schr\"odinger equations mentioned in the previous paragraph.}
There is a vast literature on local smoothing and dispersive estimates on hyperbolic, and more generally, symmetric spaces, but here we mention only a few of the most relevant ones:
A global-in-time local smoothing estimate for the unperturbed hyperbolic Laplacian $H = -\lap$ was proved by Kaizuka in~\cite{Kaizuka1},  heavily using the Helgason Fourier transform. Dispersive estimates for $- i \rd_{t} + \De$ were proved by Anker-Pierfelice \cite{AP09}, Banica~\cite{B07}, Banica-Carles-Staffilani~\cite{BCS}, and Ionescu-Staffilani \cite{IS}, see also Metcalfe-Taylor~\cite{MT11, MTay12}; by standard machinery, Strichartz estimates for $- i \rd_{t} + \De$ then follow.
For perturbations of the hyperbolic Laplacian, Borthwick-Marzuola \cite{BorMar1} recently proved a dispersive estimate for exponentially decaying, zeroth order perturbations of $-\De$ (and also the matrix Hamiltonian, which arises when linearizing the NLS around a standing wave). In a series of works~\cite{ChenHassell1, ChenHassell3} Chen-Hassell established uniform $L^{p}$ bounds for resolvents of the Laplace-Beltrami operator on non-trapped, asymptotically hyperbolic (with exponential decay) manifolds, and in the work \cite{Chen18} Chen proved global-in-time Strichartz estimates for such operators. 

Finally, we mention again the work of Li-Ma-Zhao~\cite{LMZ-DPDE} and Li \cite{Li1, Li2} in the closely related subject of perturbations of the wave equation on hyperbolic space. In their proof of stability of harmonic maps from $\bbH^{2}$ into $\bbH^{2}$ under the wave maps evolution, they established a global-in-time local energy decay estimate (yet with exponentially decaying weights) and $L^{2}_{t} L^{p}_{x}$-Strichartz estimates for the first and zeroth order perturbations of the wave equation arising from linearization around such harmonic maps. 
As a technical remark we point out that the asymptotic stability analysis for the Schr\"odinger maps equation substantially differs from that for wave maps due to the delicate smoothing properties of Schr\"odinger equations.

\section{Geometric and analytic preliminaries} \label{sec:prelim}
\subsection{Notation and conventions}
Here, we collect the notation and conventions used throughout the paper.
\begin{itemize}
\item We adopt the usual asymptotic notation and write $A \aleq B$ or $A = O(B)$ if there exists a constant $C > 0$, which may differ from line to line, such that $A \leq C B$. The dependency of the implicit constant is specified by subscripts, e.g., $A \aleq_{E} B$ or $A = O_{E}(B)$ if $A \leq C(E) B$.
\item Given two normed vector spaces $X$ and $Y$, $X \cap Y$ is equipped with the norm $\nrm{\cdot}_{X \cap Y } = \nrm{\cdot}_{X} + \nrm{\cdot}_{Y}$, and $X+Y$ is the smallest vector space containing both $X$ and $Y$ equipped with the norm $\nrm{z}_{X+Y} = \inf_{(x, y) : z = x + y} \nrm{x}_{X} + \nrm{y}_{Y}$.
\item Our convention for the indices is as follows: 
\begin{itemize}
\item Latin lower case indices beginning from $i$, namely $i, j, k, \ell, m, \ldots \in \set{1, 2}$ refer to components on the domain $\bbH^{2}$;
\item Greek lower case indices $\mu, \nu, \lmb, \kpp, \ldots \in \set{0, 1, 2}$ refer to coordinates $\set{x^{0}, x^{1}, x^{2}}$ on the domain $\bbR_{t} \times \bbH^{2}_{x^{1}, x^{2}}$, with $x^{0} = t$;
\item Bold greek lower case indices $\bfmu, \bfnu, \bflmb, \bfkpp, \ldots \in \set{0, 1, 2, s}$ refer to coordinates on $\bbR_{t} \times \bbH^{2}_{x^{1}, x^{2}} \times [0, \infty)_{s}$;
\item Latin lower case indices beginning from $a$, namely $a, b, c, d \in \set{1, 2}$ refer to components in the tangent space of the target $\tgmfd$;
\item Latin upper case indices beginning from $A$, namely $A, B, C, D$ refer to components in the ambient space $\bbR^{N}$ of the isometric embedding used in the definitions of spaces of maps, as well as in Section~\ref{sec:hmhf} below.
\end{itemize}
\item We adopt the usual convention of lowering and raising indices $i, j, k, \ell, m$ (resp. $a, b, c, d$) using the Riemannian metric on the domain (resp. on the target manifold).
\item For a one-form $\Phi$ on $\bbR_{t} \times \bbH^{2}_{x^{1}, x^{2}} \times [0, \infty)_{s}$ we write $\Phi_x$ for $\Phi_{x^1}\ud x^1+\Phi_{x^2}\ud x^2$. If $Y$ is a vector field on $\bbH^2$, we write $\Phi_Y$ for the contraction of $\Phi_x$ with $Y$.
\end{itemize}

\subsection{Geometry of the domain} \label{subsec:domain}

Let $\bbR^{2+1}$  denote the $(2+1)$-dimensional Minkowski space with rectilinear coordinates  $\{y^1,y^2, y^0\}$ and metric $\bfm$ given in these coordinates by $\bfm = \textrm{diag}(1, 1,  -1)$.  The hyperbolic plane $\bbH^{2}$ is defined as 
\begin{align*}
\bbH^2 := \{ y \in \bbR^{2+1} \colon (y^{0})^{2} - (y^{1})^{2} - (y^{2})^{2} = 1\, , \,y^0>0\}.
\end{align*}
We equip $\bbH^{2}$ with the pullback Riemannian metric $\bfh = \iota^{*} \bfm$, where $\iota : \bbH^{2} \hookrightarrow \bbR^{2+1}$ denotes the inclusion map. Furthermore, we endow $\bbH^{2}$ with the Riemannian volume form (which determines the orientation) $\dvol_{\bfh} = \iota^{\ast} \, i_{\bfn} (\ud y^{0} \wedge \ud y^{1} \wedge \ud y^{2})$, where $\bfn = \left( y^{0} \rd_{y^{0}} + y^{1} \rd_{y^{1}} + y^{2} \rd_{y^{2}} \right)$ is the unit normal to $\bbH^{2}$ pointing in the increasing $y^{0}$ direction.  For any coordinate system $\{x^{1}, x^{2}\}$ on $\bbH^{2}$, we denote by $\bfh_{jk}$ the components of $\bfh$ and by $\bfh^{jk}$ the components of the inverse matrix $\bfh^{-1}$, i.e.,
\begin{equation*}
\bfh_{jk} :=  \bfh\left( \frac{\rd}{\rd x^j}, \frac{\rd}{\rd x^k} \right), \quad \bfh^{jk}:= \bfh^{-1}(dx^j, dx^k).
\end{equation*}
We denote the Christoffel symbols on $\bbH^{2}$  by
\begin{equation*}
\Gmm_{jk}^{\ell}  = \frac{1}{2} \bfh^{\ell m}( \rd_{j} \bfh_{k m} + \rd_k \bfh_{m j} - \rd_m \bfh_{jk}).
\end{equation*}
Given a vector field $\bfX = \bfX^{j} \frac{\rd}{\rd x^{j}}$ in $\Gmm(T \bbH^2)$  or a $1$-form $\bsomega = \bsomega_j dx^j \in \Gmm( T^*\bbH^2)$, the metric covariant derivative $\nb$ is defined in coordinates by 
\begin{equation*}
\nb_{\rd_{j}} X = (\rd_{j} X^{k} + \Gmm^{k}_{j \ell} X^{\ell}) \rd_{k}, \quad
\nb_{\rd_{j}} \bsomega = (\rd_{j} \bsomega_{k} - \Gmm^{\ell}_{jk} \bsomega_{\ell}) d x^{k}
\end{equation*}
The covariant derivative extends to an arbitrary $(p,q)$ tensor field $\bsxi$ by requiring that $\nabla f=df$ for functions and that it satisfies the Leibniz rule with respect to tensor products. The $k$-th iteration of the covariant derivative, which is then a $(p, q+k)$ tensor field, is denoted by $\nb^{(k)} \bsxi$.

We denote the Riemann curvature tensor on $\bbH^2$ by $\bfR$ with components $\bfR_{ijk\ell}$. The Ricci tensor $\bfR_{ij}$ is defined by $\bfR_{ij}=\bfR^k_{~ikj}$. The curvature and Ricci tensors are explicitly given by
\[
 \bfR( \bfX, \bfY) \bfZ = - \bigl( \bfh( \bfY, \bfZ) \bfX - \bfh( \bfX, \bfZ) \bfY \bigr), \qquad \bfR_{jk}=-\bfh_{jk}.
\]
For any vector field $\bfX$
%%%%%%%
%%%%%%%
\begin{align*}
\begin{split}
[\nabla_i,\nabla_j]\bfX^k=\bfR^k_{~\ell ij}\bfX^\ell, \qquad
[\nabla_i,\nabla_j]\bfX^i=\bfR_{\ell j}\bfX^\ell=-\bfX_j.
\end{split}
\end{align*}
Analogous formulas hold for commutators of covariant derivatives applied to arbitrary $(p,q)$ tensor fields.

When working with tensor fields defined on $\bbR_{t} \times \bbH_{x}^{2} \times [0, \infty)_{s}$ we trivially extend the covariant derivative $\nabla$ to the $t$ and $s$ variables, so that $\nabla_s$ and $\nabla_t$ are the regular derivatives with respect to these parameters. 

Given a vector field $\bfY$ on $\bbH^{2}$, we denote the \emph{Lie derivative} with respect to $\bfY$ by $\calL_{\bfY}$. For a scalar function $f$ on $\bbH^{2}$, $\calL_{\bfY} f = \bfY^{j} \rd_{j} f$. For a vector field $\bfX = \bfX^{j} \frac{\rd}{\rd x^{j}} \in \Gmm(T \bbH^{2})$ or a $1$-form $\bsomega = \bsomega_{j} \ud x^{j} \in \Gmm(T^{\ast} \bbH^{2})$, $\calL_{Y}$ is related to $\nb_{Y} = Y^{j} \nb_{j}$ by
\begin{equation*}
	\calL_{\bfY} \bfX^{k} = \bfY^{j} \nb_{\rd_{j}} \bfX^{k} - \bfX^{j} \nb_{\rd_{j}} \bfY^{k}, \quad
	\calL_{\bfY} \bsomega_{k} = \bfY^{j} \nb_{\rd_{j}} \bsomega_{k} + \bsomega_{j} \nb_{\rd_{k}} \bfY^{j}.
\end{equation*}
%\begin{equation*}
%	(\bfY^{j} \nb_{j} \bsomga)(\rd_{k}) + \bsomega(\bfY^{j} \nb_{j} \rd_{k})
%	= \bfY(\bsomga(\rd_{k})) = (\calL_{\bfY} \bsomega) (\rd_{k}) - \bsomega (\calL_{\bfY}\rd_{k})
%	= (\calL_{\bfY} \bsomega) (\rd_{k}) + \bsomega_{j} (\bfY^{j} \nb_{\rd_{j}} \rd_{k}- \nb_{\rd_{k}} \bfY^{j})
%\end{equation*}
Analogous formulas for arbitrary $(p, q)$ tensor fields may be derived by the Leibniz rules for $\calL_{Y}$ and $\nb$.

One useful set of coordinates on $\bbH^{2}$ is the \emph{polar coordinates}. As in Subsection~\ref{subsec:result}, fix a point $o$ (the origin) in $\bbH^{2}$; without loss of generality, we may take $o = (0, 0, 1)$. Then the distance function $r = \bfd_{\bbH^{2}}(o, x)$ and the angle $\tht = \tan^{-1} (y^{2} / y^{1})$ defines the \emph{polar coordinates} with respect to $o$. The metric takes the form
\begin{equation*}
	\bfh = \ud r^{2} + \sinh^{2} r \, \ud \tht^{2}.
\end{equation*}
To describe tensor fields on $\bbH^{2}$, we will often use the following positively oriented frame defined using $(r, \tht)$:
\begin{equation} \label{eq:dom-frame}
\bfe_{1} = \rd_{r}, \quad \bfe_{2} = \frac{1}{\sinh r} \rd_{\tht}.
\end{equation}

Let $\Omg$ be the rotational vector field on $\bbH^{2}$ as in Subsection~\ref{subsec:result}, which is nothing but $\Omg = \rd_{\tht}$ in the polar coordinates. Since $\Omg$ is the infinitesimal generator of one parameter family of isometries (i.e., $\Omg$ is a Killing vector field and $\calL_{\Omg} \bfh = 0$), $\calL_{\Omg}$ commutes with $\nb$: 
\begin{lemma}  \label{l:commLOm} 
If $f$ is a scalar function then
\begin{align*}
\begin{split}
\calL_\Omega\nabla f=\nabla \Omega f.
\end{split}
\end{align*}
If $\bsomega$ is a $1$-form then \footnote{In our notation $\calL_\Omega \nabla_\mu\bsomega_{\nu}:=(\calL_\Omega \nabla\bsomega)_{\mu\nu}$ and $\nabla_\mu \calL_\Omega \bsomega_\nu:=\nabla_\mu (\calL_\Omega \bsomega)_\nu:=(\nabla\calL_\Omega \bsomega)_{\mu\nu}$.}
\begin{align*}
\begin{split}
\calL_\Omega \nabla^\mu\bsomega_{\mu}-\nabla^\mu \calL_\Omega \bsomega_\mu =0.
\end{split}
\end{align*}
Similarly if $\bsX$ is a vector field then
\begin{align*}
\begin{split}
\calL_\Omega \nabla_\mu\bsX^\mu-\nabla_\mu\calL_\Omega\bsX^\mu=0.
\end{split}
\end{align*}
\end{lemma}

\subsection{Geometry of maps into a Riemann surface $\tgmfd$} \label{subsec:target}

Let $\calM$ be a smooth manifold equipped with a torsion-free connection $\nb$ on the tensor bundle, and assume furthermore that $\calM$ is contractible. 
In this subsection, we develop the differential-geometric formalism for analysis on the bundle $u^{*} T \tgmfd$, where $u$ is a (smooth) map from $\calM$ into a Riemann surface $\tgmfd$.
Specific examples to have in mind are of course $\calM = \bbH^{2}_{x}$, $\calM = \bbR_{t} \times \bbH^{2}_{x}$ or $\calM = \bbR_{t} \times \bbH^{2}_{x} \times [0, \infty)_{s}$ with $\nb$ as in Subsection~\ref{subsec:domain}.

Let $u$ be a smooth map from $\calM$ to $\tgmfd$, where $\tgmfd$ is a Riemann surface with metric $g$ and complex structure~$J$. The pullback bundle $u^\ast T \tgmfd$ is the vector bundle over $\calM$ whose fiber at $p \in \calM$ is the tangent space $T_{u(p)} \tgmfd$. Its dual vector bundle is naturally given by $u^{\ast} T^{\ast} \tgmfd$, whose fiber at $p \in \calM$ is the cotangent space $T^{\ast}_{u(p)} \tgmfd$. A section of $u^* T\tgmfd$ is a map $\phi : \calM  \to  T \tgmfd$ so that $\phi(p) \in T_{u(p)} \tgmfd$ for each $p \in \calM$. More generally, we call a section of $(T \calM)^{\otimes p} \otimes (T^{\ast} \calM)^{\otimes q} \otimes (u^{\ast} T \tgmfd)^{\otimes r} \otimes (u^{\ast} T^{\ast} \tgmfd)^{\otimes s}$ a $u^{\ast} T \tgmfd$-valued $(p, q; r, s)$ tensor field.

Given any vector field $\bfX$ on $\calM$, the pushforward $\ud u(\bfX) = \bfX^{\alp} \rd_{\alp} u$  naturally defines a section of $u^{\ast} T \tgmfd$. Moreover, given any $(r, s)$ tensor field $W$ on $\tgmfd$, we may define its pullback $W(u)$, which is a $u^{\ast} T \tgmfd$-valued $(0, 0; r, s)$ tensor field.

The Levi-Civita connection on $T\tgmfd$ induces the pullback covariant derivative $D$ on $u^* T \tgmfd$. In local coordinates $\{x^{\alp}\}$ on $\calM$ and  $ \{u^a\}$ on $\tgmfd$ we have 
\begin{equation*}
	D_{\alp} \phi = \left(\rd_{\alp} \phi^{a} + {}^{(\tgmfd)} \Gmm^{a}_{bc}(u) \phi^{b} \rd_{\alp} u^{c}\right) \frac{\rd}{\rd u^{a}}
\end{equation*}
where ${}^{(\calN)}\Gmm^{a}_{bc}$ are the Christoffel symbols on $T \tgmfd$. The pullback covariant derivative $D_{\alp}$ is naturally extended to $u^{\ast} T \tgmfd$-valued $(p, q; r, s)$ tensor fields using the connection $\nb_{\alp}$ on the tensor bundle of $\calM$ and the Levi-Civita connection on $\tgmfd$. As before, we denote by $D^{(k)}$ the $k$-fold application of the pullback covariant derivative $D$. 

For a section $\phi$ of $u^{\ast} T \tgmfd$, we have
\begin{equation*}
	[D_{\alp}, D_{\bt}] \phi(p) = R(u)(\rd_{\alp} u, \rd_{\bt} u) \phi(p),
\end{equation*}
where $R$ is the curvature tensor on $\tgmfd$, which is a $(1, 3)$ tensor field on $\tgmfd$. Analogous formulas hold for commutators of the pullback covariant derivatives applied to arbitrary $u^{\ast} T \tgmfd$-valued $(p, q; r, s)$ tensor fields.

The complex structure $J$ on $\tgmfd$ is a $(1, 1)$ tensor field on $\tgmfd$ such that $J^{2} = - 1$, $g(J X, J Y) = g(X, Y)$ and $J$ is parallel with respect to the Levi-Civita connection. Equivalently, it is a $(1, 1)$ tensor field that relates the metric $g$ and the Riemannian volume form $\omg$ by $g(X, Y) = \omg(X, J Y)$. Denote by $\tcv(u)$ the holomorphic sectional curvature (i.e., the Gauss curvature) at $u \in \tgmfd$, i.e., for any unit vector $X \in T_{u} \tgmfd$,
\begin{equation*}
	\tcv(u) = \bigl(R(u)(X, JX) JX, X \bigr).
\end{equation*}
Note that $\tcv(u)$ is invariant under the choice of the unit vector $X \in T_{u} \tgmfd$; thus it defines a real-valued function on $\tgmfd$. Observe that
\begin{equation*}
	\tcv \equiv \begin{cases} - 1 & \hbox{when } \tgmfd = \bbH^{2}, \\ + 1 & \hbox{when } \tgmfd = \bbS^{2}. \end{cases}
\end{equation*}

\subsection{Moving frame formalism} \label{subsec:moving-frame}
Let $e = (e_1, e_2)$ be a global orthonormal frame on $u^\ast T \tgmfd$ such that $e_{2} = J e_{1}$. Then the frame $e$ defines a global trivialization of $u^\ast T\tgmfd$ over the trivial complex line bundle $E = \calM \times \bC$, and we have an isomorphism
\begin{equation*}
 e \colon \Gamma(\bC) \to \Gamma(u^\ast T\tgmfd), \quad \varphi = \varphi^1 + i \varphi^2  \mapsto e \varphi := e_1 \Re(\varphi) + e_2 \Im(\varphi).
\end{equation*}
Under this isomorphism, we denote the pushforward $\rd_{\alp} u$ of each coordinate vector field $\rd_{\alp}$ by $\psi_{\alp}$, i.e.,
\begin{equation*}
\partial_\alpha u = e  \psi_\alpha. 
\end{equation*}
The spatial components of $\psi$ are denoted by $\Psi$, that is, 
\begin{align*}
\begin{split}
 \Psi=\psi_{x^1}\ud x^1+\psi_{x^2}\ud x^2.
\end{split}
\end{align*}
Next, we introduce the covariant derivative $\Db_\ell$ on $E$ induced by $D_\ell$ via the identity
\begin{equation*}
 D_\ell e \varphi = e \Db_{\ell} \varphi.
\end{equation*}
With respect to the trivial connection $\nabla$ on $E$, we may decompose
\begin{equation*}
 \Db_\ell = \nabla_\ell + i A_\ell
\end{equation*}
where $A = A_\ell \, dx^\ell$ is the real-valued connection $1$-form given by
\begin{equation*}
	\begin{pmatrix}
	(D_{\ell} e_{1}, e_{1}) & (D_{\ell} e_{1}, e_{2}) \\
	(D_{\ell} e_{2}, e_{1}) & (D_{\ell} e_{2}, e_{2})
	\end{pmatrix}
	= \begin{pmatrix}
	0 & A_{\ell} \\
	-A_{\ell} & 0
	\end{pmatrix}.
\end{equation*}
 The curvature two-form associated with $\Db$, which is defined to be the commutator $[\Db_k, \Db_\ell] \varphi$ is related through the map $u$ to the Riemann curvature $R$ of the target manifold $\tgmfd$ as follows
\begin{equation*}
 e [\Db_k, \Db_\ell] \varphi = R(u)(\partial_k u, \partial_\ell u) e \varphi = R(u)(e \psi_k, e \psi_\ell) e \varphi.
\end{equation*}
A direct computation reveals that, for a section $\varphi \in \Gamma(u^{\ast} T \tgmfd)$,
\begin{equation} \label{eq:comm-K}
 [ \Db_k, \Db_\ell ] \varphi = i \tcv(u) \Im( \psi_k \overline{\psi_\ell} ) \varphi.
\end{equation}
On the other hand, we also note that
\[
 [ \Db_k, \Db_\ell ] \varphi = i (\nabla_k A_\ell - \nabla_\ell A_k) \varphi.
\]

With respect to such a frame $e = (e_1, e_2)$, the Schr\"odinger maps equation~\eqref{equ:schroedinger_maps_equ}, the harmonic maps equation $D^{\ell} \rd_{\ell} U = 0$ and the harmonic maps heat flow equation \eqref{eq:hmhf} read, respectively,
\begin{align} 
 \psi_t &= i \Db^\ell \psi_\ell, 	\label{equ:schroedinger_maps_equ_derivative_formulation} \\
\Db^\ell \psi_\ell &= 0, 		\label{eq:harmonic_maps_equ_derivative_formulation} \\
\psi_{s} &= \Db^\ell \psi_\ell. 	\label{eq:hmhf_derivative_formulation}
\end{align}

\subsection{Basic function spaces on $\bbH^{2}$} \label{subsec:fs}
Let $C^{\infty}_{0}(\bbH^{2})$ denote the space of all complex-valued smooth and compactly supported functions on $\bbH^{2}$. 
For $f \in C^{\infty}_{0}(\bbH^{2})$ and $1 \leq p < \infty$, its $L^{p}$ norm is defined by 
\begin{align*}
\nrm{f}_{L^{p}} = \nrm{f}_{L^{p}_{x}} =  \left( \int_{\bbH^{2}} \abs{f}^{p} \, \dvol_{\bfh} \right)^{\frac{1}{p}}.
\end{align*} 
The $L^{p}(\bbH^{2})$ space for $1 \leq p < \infty$ is defined by taking the completion of $C_{0}^{\infty}(\bbH^{2})$ with respect to the $L^{p}$ norm. 
We also introduce the space $C^{0}(\bbH^{2})$ of bounded continuous complex-valued functions on $\bbH^{2}$, equipped with the norm
\begin{equation*}
	\nrm{f}_{L^{\infty}} = \nrm{f}_{L^{\infty}_{x}} =  \sup_{x \in \bbH^{2}} \abs{f(x)}.
\end{equation*}
Mixed norms in the polar coordinates $(r, \tht)$ are sometimes used. For $f \in C^{\infty}_{0}(\bbH^{2})$, we define
\begin{equation*}
	\nrm{f}_{L^{p}_{r} L^{q}_{\tht}} = \nrm{\nrm{f(r, \tht)}_{L^{q}_{\tht}(0, 2\pi)}}_{L^{p}_{r}(0, \infty)}.
\end{equation*}
Clearly, $\nrm{f}_{L^{p}} = \nrm{f (\sinh r)^{\frac{1}{p}}}_{L^{p}_{r} L^{p}_{\tht}}$.

To extend the definition of $L^{p}$ norms to real-, complex- or $u^{\ast} T \tgmfd$-valued tensor fields on $\bbH^{2}$, we need to specify the norm on the fiber, which is a finite dimensional vector space in each case. Abusing the notation a bit, we denote by $(\cdot, \cdot)$ the inner product on these finite dimensional vector spaces, where the space being used should be understood from the context. For instance, 
\begin{align*}
	(f, g) & = f \overline{g} \qquad \hbox{ for } f, g \in \bbC; \\
	(u, v) & = \sum_{I=1}^{N} u^{I} v^{J} \qquad \hbox{ for } u, v \in \bbR^{N}; \\
	(\bsxi, \bszeta) &= \bfh_{i_{1} i'_{1}} \cdots \bfh_{i_{p} i'_{p}} \bfh^{j_{1} j_{1}'} \cdots \bfh^{j_{q} \cdots j'_{q}}\bsxi^{i_{1} \cdots i_{p}}_{j_{1} \cdots j_{q}}  \bszeta^{i_{1}' \cdots i_{p}'}_{j_{1}' \cdots j_{q}'}  \qquad \hbox{ for } \bsxi, \bszeta \in (T_{p} \bbH^{2})^{\otimes p} \otimes (T^{\ast}_{p} \bbH^{2})^{\otimes q}; \\
	(X, Y) &= \bfh_{i_{1} i'_{1}} \cdots \bfh_{i_{p} i'_{p}} \bfh^{j_{1} j_{1}'} \cdots \bfh^{j_{q} \cdots j'_{q}}
	g_{a_{1} a'_{1}} \cdots g_{a_{r} a'_{r}} g^{b_{1} b_{1}'} \cdots g^{b_{s} \cdots b'_{s}} \\
	& \phantom{=}
	\times \tensor{X}{^{i_{1} \cdots i_{p}}_{j_{1} \cdots j_{q}}^{a_{1} \cdots a_{r}}_{b_{1} \cdots b_{s}}} \tensor{Y}{^{i'_{1} \cdots i'_{p}}_{j'_{1} \cdots j'_{q}}^{a'_{1} \cdots a'_{r}}_{b'_{1} \cdots b'_{s}}}
	\qquad \hbox{ for } \bsxi, \bszeta \in (T_{p} \bbH^{2})^{\otimes p} \otimes (T^{\ast}_{p} \bbH^{2})^{\otimes q} \otimes T_{u(p)} \tgmfd. 
\end{align*}
In all cases, the norm is given by $\abs{\cdot}^{2} = (\cdot, \cdot)$.

For $f \in C^{\infty}_{0}(\bbH^{2})$, $1 < p < \infty$ and a positive integer $k$, we define the $W^{k, p}$ norm of $f$ by
\begin{equation*}
	\nrm{f}_{W^{k, p}} = \sum_{k' = 0}^{k} \nrm{\nb^{(k')} f}_{L^{p}}.
\end{equation*}
The space $W^{k, p}(\bbH^{2})$ is then defined by taking the completion of $C_{0}^{\infty}(\bbH^{2})$ with respect to the $W^{k, p}$ norm. We define the space $W^{-k, p}(\bbH^{2})$ to be the dual of $W^{k, q}(\bbH^{2})$ where $\frac{1}{p} + \frac{1}{q} = 1$. 

To extend $W^{\sgm, p}(\bbH^{2})$ to nonintegral values of $\sgm$, we note the following well-known equivalence of norms for positive integers $\sgm$ (see, for instance, \cite{Tat01hyp}; see also Lemma~\ref{l:pi} and Lemma~\ref{l:riesz} below):
\begin{equation*}
	\nrm{f}_{W^{\sgm, p}} \simeq \nrm{(-\lap)^{\sgm} f}_{L^{p}} \qquad \hbox{ for } \sgm = 1, 2, \ldots
\end{equation*}
In view of such an equivalence, when $\sgm \in \bbR \setminus \bbZ$, we \emph{define} the $W^{\sgm, p}$ norm to be $\nrm{(-\lap)^{\sgm} f}_{L^{p}}$ and the space $W^{\sgm, p}(\bbH^{2})$ to be the completion of $C^{\infty}_{0}(\bbH^{2})$ with respect to this norm. As usual, in the case $p = 2$, we write $H^{\sgm} = W^{\sgm, p}$.

We end this subsection with some notational conventions. Let $X$ be a norm for functions on $\bbH^{2}$. For a function $f$ that depends also on $t \in \bbR$ and/or $s \in [0, \infty)$, we use the following notation for mixed norms:
\begin{align*}
	\nrm{f}_{L^{p}_{t} (I; X)} &= \nrm{\nrm{f(t)}_{X}}_{L^{p}_{t}(I)} = \nrm{f}_{L^{p}_{t} X(I \times \bbH^{2})}, \\
	\nrm{f}_{L^{p}_{s} (J; X)} &= \nrm{\nrm{f(s)}_{X}}_{L^{p}_{s}(J)}, \qquad
	\nrm{f}_{L^{p}_{\ds} (J; X)} = \nrm{\nrm{f(s)}_{X}}_{L^{p}_{\ds}(J)}.
\end{align*}
When the interval $I$ (resp.~$J$) is not specified, it is meant to be the whole line, i.e, $I = \bbR$ (resp. $J = [0, \infty)$).
We also introduce the notation
\begin{equation*}
 \| \jap{\Omega} f \|_{X} = \|f\|_X + \|\Omega f\|_X.
\end{equation*}
Similarly, for any integer $k\geq0$ we will use the notation
\begin{equation*}
 \| \jap{s\Delta}^\frac{k}{2} f\|_X = \sum_{k'=0}^k \|  (s^{\frac{1}{2}}\nabla)^{(k')} f \|_X.
\end{equation*}
For a tensor field $\bsxi$, we define
\begin{equation*}
 \| \jap{\calL_\Omega} \bsxi \|_X = \|\bsxi\|_X + \|\calL_\Omega\bsxi\|_X.
\end{equation*}
We also use the notation
\begin{equation*}
 \| \nabla^{(k)} \jap{\Omega} f \|_{X} = \| \nabla^{(k)} f \|_X + \| \nabla^{(k)} \Omega f\|_X
\end{equation*}
with similar definitions for $\| \nabla^{(k)} \jap{s\Delta} f \|_X$ and $\| \nabla^{(k)} \jap{\calL_\Omega} \bsxi \|_X$. Finally, when it is clear from the context that $k$ does not denote a tensorial index, we will often write $\nabla^{k} = \nabla^{(k)}$.

\subsection{Parabolic theory for the linear heat operator $\rd_{s} - \lap$}
We collect various estimates for the linear heat operator $\rd_{s} - \lap$ that are used in this paper.
Unless otherwise stated, all functions in this subsection are complex-valued.

\subsubsection{$L^{2}$ theory for $\rd_{s} - \lap$}
We begin with the fundamental $L^{2}$ bound for the inhomogeneous heat equation 
$(\rd_{s} - \lap) \varphi = F$.
\begin{lemma} \label{lem:lin-heat-L2}
Consider the initial value problem
\begin{align*}
	(\rd_{s} - \lap) \varphi = F \quad \hbox{ in } \bbH^{2} \times J, \quad \varphi(s=0) = \varphi_{0} \in L^{2}(\bbH^{2}),
\end{align*}
where $J$ is an interval of the form $J = [0, s_{0}]$ and $F \in L^{1}_{s}(J; L^{2}) + L^{2}_{s}(J; H^{-1})$. This problem is well-posed, and the unique solution $\varphi$ obeys 
\begin{equation} \label{eq:lin-heat-L2}
\nrm{\varphi}_{L^{\infty}_{s} (J; L^{2})}
+ \nrm{\varphi}_{L^{2}_{s}(J; H^{1})}
+ \nrm{\rd_{s} \varphi}_{L^{2}_{s}(J; H^{-1})}
\aleq \nrm{\varphi_{0}}_{L^{2}} + \nrm{F}_{L^{1}_{s} (J; L^{2}) + L^{2}_{s}(J; H^{-1})}.
\end{equation}
Moreover, for any $k = 1, 2, \ldots$,
\begin{equation} \label{eq:lin-heat-L2-high}
\begin{aligned}
& \nrm{s^{k}(-\lap)^{k} \varphi}_{L^{\infty}_{s} (J; L^{2})}
+ \nrm{s^{k}(-\lap)^{k} \varphi}_{L^{2}_{s}(J; H^{1})}
+ \nrm{s^{k}(-\lap)^{k} \rd_{s} \varphi}_{L^{2}_{s}(J; H^{-1})} \\
& \aleq_{k} \nrm{\varphi_{0}}_{L^{2}} + \nrm{F}_{L^{1}_{s} (J; L^{2}) + L^{2}_{s}(J; H^{-1})}
+ \nrm{s^{k}(-\lap)^{k} F}_{L^{1}_{s} (J; L^{2}) + L^{2}_{s}(J; H^{-1})}.
\end{aligned}\end{equation}
\end{lemma}

In what follows, we denote by $e^{s \lap}$ the solution operator for $(\rd_{s} -\lap) \varphi = 0$.
\begin{proof}
We first prove the a-priori estimate \eqref{eq:lin-heat-L2}. With the exception of the term $\nrm{\rd_{s} \varphi}_{L^{2}(J; H^{-1})}$, \eqref{eq:lin-heat-L2} immediately follows by multiplying by $\overline{\varphi}$, taking the real part and integrating by parts on $\bbH^{2} \times J$. To add in $\rd_{s} \varphi$, use
\begin{equation*}
	\nrm{\rd_{s} \varphi}_{L^{2}_{s}(J; H^{-1})}
	\leq \nrm{\varphi}_{L^{2}_{s}(J; H^{1})}
	+ \nrm{(\rd_{s} - \lap) \varphi}_{L^{2}_{s}(J; H^{-1})},
\end{equation*}
which implies $\nrm{\rd_{s} \varphi}_{L^{2}_{s}(J; H^{-1})} \aleq \nrm{\varphi_{0}}_{L^{2}} + \nrm{F}_{L^{2}_{s}(J; H^{-1})}$. In particular, $\nrm{\rd_{s} e^{s \lap} \varphi_{0}}_{L^{2}_{s}(J; H^{-1})} \aleq \nrm{\varphi_{0}}_{L^{2}}$. By Duhamel's principle, it follows that $\nrm{\rd_{s} \varphi}_{L^{2}_{s}(J; H^{-1})} \aleq \nrm{\varphi_{0}}_{L^{2}} + \nrm{F}_{L^{1}_{s}(J; L^{2})}$ as well. By linearity, \eqref{eq:lin-heat-L2} follows. 

The well-posedness assertion is a standard consequence of \eqref{eq:lin-heat-L2}. To prove \eqref{eq:lin-heat-L2-high}, we begin by noting that
\begin{equation*}
(\rd_{s} - \lap) (s^{k} (-\lap)^{k} \varphi) = s^{k} (-\lap)^{k} F + k s^{k-1} (-\lap)^{k} \varphi.
\end{equation*}
We apply \eqref{eq:lin-heat-L2} and, for a small number $\eps > 0$ to be chosen later, estimate the last term as follows:
\begin{align*}
\nrm{k s^{k-1} (-\lap)^{k} \varphi}_{L^{2}(J; H^{-1})}
& \aleq k \nrm{s^{k-1} (-\lap)^{k-1} \varphi}_{L^{2}(J; H^{1})} \\
& \aleq k \nrm{\varphi}_{L^{2}(J; H^{1})}^{\frac{1}{k}} \nrm{s^{k} (-\lap)^{k} \varphi}_{L^{2}(J; H^{1})}^{\frac{k-1}{k}} \\
& \aleq \eps^{-(k-1)} \nrm{\varphi}_{L^{2}(J; H^{1})} + (k-1) \eps \nrm{s^{k} (-\lap)^{k} \varphi}_{L^{2}(J; H^{1})}.
\end{align*}
On the second line, we used a simple $L^{2}$ interpolation inequality \cite[Lemma~2.4]{LOS5}. Choosing $\eps > 0$ sufficiently small to absorb the last term into the LHS, then invoking \eqref{eq:lin-heat-L2} (applied to $(\rd_{s} - \lap) \varphi = F$) to control $\nrm{\varphi}_{L^{2}_{s}(J; H^{1})}$, \eqref{eq:lin-heat-L2-high} follows. \qedhere
\end{proof}

We also need the following ``Strichartz'' estimate for $\rd_{s} - \lap$:
\begin{lemma} \label{lem:lin-heat-L2Linfty}
Let $\varphi$ be the solution to $(\rd_{s} - \lap) \varphi = F$ in $\bbH^{2} \times J$ with $\varphi(s=0) = \varphi_{0} \in L^{2}$ and $F \in L^{1}_{s}(J; L^{2})$. Then
\begin{equation} \label{eq:lin-heat-L2Linfty}
	\nrm{\varphi}_{L^{2}_{s}(J; L^{\infty})}
	\aleq \nrm{\varphi_{0}}_{L^{2}} + \nrm{F}_{L^{1}_{s}(J; L^{2})}.
\end{equation}
\end{lemma}
This lemma follows, via Duhamel's principle, from the homogeneous estimate $\nrm{e^{s \lap} f}_{L^{2}_{s}(J; L^{\infty})} \aleq \nrm{f}_{L^{2}}$. As stated in \cite[Lemma~2.5]{LMZ-DPDE}, the latter estimate may be proved by following \cite[Paper IV, Proof of Lemma~2.5]{Tao37}, which only relies on the self-adjointness of $e^{s \lap}$ and the decay estimate $\nrm{e^{s \lap} f}_{L^{\infty}} \aleq s^{-\frac{1}{2}} \nrm{f}_{L^{2}}$ (see Lemma~\ref{l:lin-heat-Lp-Lq} below).

\subsubsection{$L^{p} \to L^{q}$ bounds for $\rd_{s} - \lap$.}
Next, we consider $L^{p} \to L^{q}$ bounds for $e^{s \lap}$. For small times, we rely on the following (fairly general) heat-kernel bounds:
\begin{lemma} [Short-time heat kernel bounds]  \label{l:hk} Let $p_\s$ denote the heat kernel on $\Hp^{2}$ (i.e., the integral kernel for $e^{\s \lap}$).  
For $0 < \s \leq 2$ and $k = 0, 1, 2$, there exists $N_{k} > 0$ so that the following \emph{short-time heat kernel estimate} holds:
\begin{equation} \label{eq:hk-short}
	\abs{\rd_{\s}^{k} p_{\s}(x, y)} \aleq_{k} \s^{-\frac{d}{2} - k} \bb( 1+ \frac{\bfd_{\bbH^{2}}(x, y)^{2}}{\s} \bb)^{N_{k}} \exp \left( -\frac{\bfd_{\bbH^{2}}(x, y)^{2}}{4\s}\right).
\end{equation}
\end{lemma} 

Lemma~\ref{l:hk} in the case $k =0$ follows from~\cite[Theorem 3.1]{DaviesMandouvalos}. Then the cases $k=1, 2$ follow by standard machinery for deducing estimates for time derivatives of the heat kernel from that of the heat kernel itself; see~\cite{Davies} and~\cite[Theorems 1.1 and 1.2]{Grigoryan} for details. 

Combining Lemma~\ref{l:hk} with the long-time $L^{2}$ bound and the maximum principle for the linear heat equation, we obtain the following long-time $L^{p} \to L^{q}$ bounds for $e^{s \lap}$.
\begin{lemma}  \label{l:lin-heat-Lp-Lq}
Let $1 < p < \infty$ and $p \leq q \leq \infty$. Let $\rho_{0}$ satisfy
\begin{equation*}
	0 < \rho_{0}^{2} < \frac{1}{2} \min \set{\frac{1}{p}, 1 - \frac{1}{p}}.
\end{equation*}
For $f \in L^p(\Hp^2)$ and $s > 0$, we have
\begin{equation}  \label{eq:lin-heat-Lp-Lq} 
	\nrm{e^{s \lap} f}_{L^{q}} + \nrm{s \lap e^{s \lap} f}_{L^{q}}
	\aleq s^{- (\frac{1}{p} - \frac{1}{q})} e^{-\rho_{0}^{2} s} \nrm{f}_{L^{p}}.
\end{equation}
\end{lemma} 
\begin{proof}
For $0 < s \leq 1$, the desired estimate follows from Lemma~\ref{l:hk} and Schur's test. 

To handle the case $s > 1$, we begin by establishing the following long-time $L^{p}$ bound:
\begin{equation} \label{eq:lin-heat-Lp-long}
	\nrm{e^{s \lap} f}_{L^{p}} \leq e^{-\rho^{2} s} \nrm{f}_{L^{p}} \quad \hbox{ for all } s >0,
\end{equation}
where $\rho^{2} = \frac{1}{2} \min \set{\frac{1}{p}, 1 -\frac{1}{p}}$. Indeed, the case $p = 2$ follows by integrating the differential inequality,
\begin{equation*}
	\frac{1}{2} \rd_{s} \int (e^{s \lap} f, e^{s \lap} f) = \int (\lap e^{s \lap} f, e^{s \lap} f) \leq - \frac{1}{4} \int (e^{s \lap} f, e^{s \lap} f),
\end{equation*}
which in turn follows by the well-known fact that $-\lap$ on $\bbH^{2}$ has a spectral gap of $\frac{1}{4}$ \cite{Bray}. The case $2 < p < \infty$ then follows by interpolation of the case $p = 2$ and the maximum principle, which implies $\nrm{e^{s \lap} f}_{L^{\infty}} \leq \nrm{f}_{L^{\infty}}$. Finally, the case $1 < p < 2$ follows by duality.

By \eqref{eq:lin-heat-Lp-long} and the short time bound,
\begin{equation*}
	\nrm{e^{s \lap} f}_{L^{q}} +	\nrm{s \lap e^{s \lap} f}_{L^{q}}
	\aleq \nrm{e^{(s-1) \lap} f}_{L^{p}} \aleq e^{-\rho^{2} s} \nrm{f}_{L^{p}}.
\end{equation*}
Decreasing $\rho^{2}$ a bit to $\rho_{0}^{2}$, the desired estimate \eqref{eq:lin-heat-Lp-Lq} follows. \qedhere.
\end{proof}

We conclude with two corollaries of Lemma~\ref{l:lin-heat-Lp-Lq}, which will be useful for establishing basic $L^{p}$ functional inequalities on $\bbH^{2}$ in Subsection~\ref{subsec:ineq}.
\begin{corollary} \label{c:spec-gap-Lp}
Let $1 < p < \infty$. For $f \in W^{2, p}$, we have
\begin{equation} \label{eq:spec-gap-Lp}
	\nrm{f}_{L^{p}} \aleq \nrm{\lap f}_{L^{p}}.
\end{equation}
\end{corollary}
\begin{proof}
We begin by writing
\begin{equation*}
f = - \int_{0}^{\infty} \rd_{s} e^{s \lap} f \, \ud s = \int_{0}^{\infty} e^{s \lap} (-\lap) f \, \ud s,
\end{equation*}
which is easily justified using Lemma~\ref{l:lin-heat-Lp-Lq}. By \eqref{eq:lin-heat-Lp-Lq} with $p = q$, we have
\begin{equation*}
\nrm{\int_{0}^{\infty} e^{s \lap} (-\lap) f \, \ud s}_{L^{p}}
\leq \int_{0}^{\infty} \nrm{e^{s \lap} (-\lap) f}_{L^{p}} \, \ud s
\aleq \int_{0}^{\infty} e^{-\rho_{0}^{2} s} \nrm{\lap f}_{L^{p}} \, \ud s
\aleq \nrm{\lap f}_{L^{p}},
\end{equation*}
which proves the corollary. \qedhere
\end{proof}

\begin{corollary} \label{c:frac-heat-Lp}
Let $0 < \alp < 1$ and $1 < p < \infty$. For $f \in L^{p}$, we have
\begin{equation*}
	\nrm{s^{\alp} (-\lap)^{\alp} e^{s \lap} f}_{L^{p}} \aleq \nrm{f}_{L^{p}},
\end{equation*}
where the operator $(-\lap)^{\alp}$ is defined by the spectral calculus of the self-adjoint operator $-\lap$ on (say) $H^{2}$ and then extended by continuity.
\end{corollary}
\begin{proof}
For $0 < \alp < 1$ and $g \in H^{2}$, by the spectral calculus for $-\lap$, we have
\begin{equation*}
(-\lap)^{\alp} g = c_{\alp} \int_{0}^{\infty} s^{-\alp} (-\lap) e^{s \lap} g \, \ud s,
\end{equation*}
where $c_{\alp} = \int_{0}^{\infty} s^{-\alp} e^{-s} \, \ud s > 0$. Thus, for $f \in C^{\infty}_{0}$ we may write
\begin{equation*}
s^{\alp} (-\lap)^{\alp} e^{s \lap} f = c_{\alp} \int_{0}^{\infty} s^{\alp} (s')^{-\alp} (-\lap) e^{(s+s') \lap} f \, \ud s',
\end{equation*}
Then by Lemma~\ref{l:lin-heat-Lp-Lq},
\begin{align*}
\nrm{s^{\alp} (-\lap)^{\alp} e^{s \lap} f}_{L^{p}}
& = \nrm{c_{\alp} \int_{0}^{\infty} s^{\alp} (s')^{-\alp} (-\lap) e^{(s+s') \lap} f \, \ud s' }_{L^{p}} \\
& \leq c_{\alp} \int_{0}^{\infty} s^{\alp} (s')^{-\alp} (s+s')^{-1} \nrm{(s+s') (-\lap) e^{(s+s') \lap}f}_{L^{p}}\, \ud s' \\
& \aleq \left(\int_{0}^{\infty} s^{\alp} (s')^{-\alp} (s+s')^{-1} \, \ud s'\right) \nrm{f}_{L^{p}} \aleq \nrm{f}_{L^{p}}.
\end{align*}
By the density of $C^{\infty}_{0}$ in $L^{p}$, the corollary follows. \qedhere
\end{proof}

\subsection{Heat-flow-based Littlewood--Paley projections and functional inequalities on $\bbH^{2}$}  \label{subsec:ineq}
Here, we state some basic functional inequalities on $\bbH^{2}$.

\subsubsection{Heat-flow-based Littlewood--Paley projections}
We introduce Littlewood--Paley projections on $\bbH^{2}$ based on the linear heat equation $e^{s \lap}$ as in \cite{IPS, LOS2}; however, the notation we follow is from \cite[\S~1.1.3]{LLOS1}. For any $s > 0$, we define
\begin{equation} \label{eq:LP-proj}
	P_{\geq s} f = e^{s \lap} f, \qquad
	P_{s} f = s (-\lap) e^{s \lap} f.
\end{equation}
Intuitively, $P_{s} f$ may be interpreted as a projection of $f$ to frequencies comparable to $s^{-\frac{1}{2}}$. Indeed, this interpretation can be justified on the Euclidean space $\bbR^{d}$ by observing that the symbol of $s (-\lap) e^{s \lap}$ is simply $s \abs{\xi}^{2} e^{-s \abs{\xi}^{2}}$, so $s (-\lap) e^{s \lap} f$ is the localization of the Fourier transform of $f$ to the annulus $\set{\abs{\xi} \simeq s^{-\frac{1}{2}}}$.

By the fundamental theorem of calculus, it is straightforward to verify that
\begin{equation} \label{eq:LP-id-s}
	P_{\geq s} f = \int_{s}^{\infty} P_{s'} f \frac{\ud s'}{s'} \quad \hbox{ for } s > 0,
\end{equation}
which explains our notation $P_{\geq s}$. In particular, we have
\begin{equation} \label{eq:LP-id}
f = \int_{0}^{\infty} P_{s'} f \frac{\ud s'}{s'},
\end{equation}
which is the basic identity that relates $f$ with its Littlewood--Paley resolution $\set{P_{s} f}_{s \in (0, \infty)}$.

\subsubsection{$L^{p}$ functional inequalities} 
Here, we collect basic $L^{p}$ functional inequalities on $\bbH^{2}$ used in the paper.
\begin{lemma}[Poincar\'e inequality] \label{l:pi} 
Let $f \in C^{\infty}_{0}(\bbH^{2})$. Then for any $1 < p < \infty$, we have
\begin{align} 
	\nrm{f}_{L^{p}} \aleq \nrm{\lap f}_{L^{p}}, \\ %\label{eq:spec-gap-Lp} \\
	\nrm{f}_{L^{p}} \aleq \nrm{\nb f}_{L^{p}}. \label{eq:Pip} 
\end{align}
\end{lemma} 

\begin{proof}
The first inequality was proved in Corollary~\ref{c:spec-gap-Lp}. To prove the second inequality, by interpolation and duality, it suffices to just consider the case $p = 2k$ for each positive integer $k$. When $k = 1$, \eqref{eq:Pip} again follows from the fact that $-\lap$ on $\bbH^{2}$ has a spectral gap of $\frac{1}{4}$. For $k > 2$, we argue as follows:
\begin{align*}
\begin{split}
\|f\|_{L^{2k}}^{2k}= \||f|^{k}\|_{L^2}^2\lesssim \|\nabla|f|^k\|_{L^2}^2\lesssim \||f|^{2k-2} |\nabla f|^2\|_{L^1} \leq \|f\|_{L^{2k}}^k\|\nabla f\|_{L^{2k}}^k.
\end{split}
\end{align*}
The desired estimate follows if we divide through by $\|f\|_{L^{2k}}^k$. \qedhere
\end{proof}

\begin{lemma}[Boundedness of Riesz transform] \label{l:riesz}
Let $f \in C^{\infty}_{0}(\bbH^{2})$. Then for $1 < p < \infty$ we have
\begin{equation*}
	\nrm{\nb f}_{L^{p}} \simeq \nrm{(-\Dlt)^{\frac{1}{2}} f}_{L^{p}}.
\end{equation*}
\end{lemma}
For a proof of Lemma~\ref{l:riesz}, we refer to \cite[Theorem~6.1]{strichartz}.

\begin{lemma}[$L^{p}$ interpolation inequalities] \label{l:Lp-int}
Let $f \in C^{\infty}_{0}(\bbH^{2})$.  Then for any $0 \leq \bt \leq \alp$ and $1 < p < \infty$ we have
\begin{equation} \label{eq:Lp-int}
	\nrm{(-\Dlt)^{\bt} f}_{L^{p}} \aleq \nrm{f}_{L^{p}}^{1-\frac{\bt}{\alp}} \nrm{(-\Dlt)^{\alp} f}_{L^{p}}^{\frac{\bt}{\alp}}.
\end{equation}
Moreover, for $1 < p < \infty$, $p \leq q \leq \infty$ and $0 < \tht = \frac{2}{\alp} (\frac{1}{p} - \frac{1}{q}) < 1$, we have
\begin{equation} \label{eq:Lp-GN}
	\nrm{f}_{L^{q}} \aleq \nrm{f}_{L^{p}}^{1-\tht} \nrm{(-\lap)^{\alp} f}_{L^{p}}^{\tht}.
\end{equation}
\end{lemma}
\begin{proof}
We begin with \eqref{eq:Lp-int}. By the formal property $(-\lap)^{\alp} (-\lap)^{\bt} = (-\lap)^{\alp+\bt}$, \eqref{eq:Lp-int} can be easily reduced to the case $0 < \bt < \alp \leq 1$. 
Since $0 < \bt < 1$, as in the proof of Corollary~\ref{c:frac-heat-Lp}, we have
\begin{equation*}
(-\lap)^{\bt} f = c_{\bt} \int_{0}^{\infty} s^{-\bt} (-\lap) e^{s \lap} f \, \ud s
= c_{\bt} \int_{0}^{\infty} s^{-\bt} P_{s} f \, \ds. 
\end{equation*}
By Lemma~\ref{l:lin-para-Lp} and Corollary~\ref{c:frac-heat-Lp}, as well as the formal property $(-\lap)^{\alp} (-\lap)^{\bt} = (-\lap)^{\alp+\bt}$, we obtain the following two estimates for $\nrm{P_{s} f}_{L^{p}}$:
\begin{align*}
	\nrm{P_{s} f}_{L^{p}} 
	& = \nrm{s \lap e^{s \lap} f}_{L^{p}} 
	 \aleq \nrm{f}_{L^{p}}, \\
	\nrm{P_{s} f}_{L^{p}} 
	& = \nrm{s \lap e^{s \lap} f}_{L^{p}} 
	= \nrm{s^{\alp} s^{1-\alp} (-\lap)^{1-\alp} e^{s \lap} (-\lap)^{\alp} f}_{L^{p}} 
	 \aleq s^{\alp} \nrm{(-\lap)^{\alp} f}_{L^{p}}.
\end{align*}
Introducing an auxiliary parameter $s_{0} > 0$, to be chosen soon, it follows that
\begin{equation*}
\nrm{(-\lap)^{\bt} f}_{L^{p}} \leq c_{\bt} \int_{0}^{s_{0}} s^{-\bt} \nrm{P_{s} f}_{L^{p}} \ds 
+ c_{\bt} \int_{s_{0}}^{\infty} s^{-\bt} \nrm{P_{s} f}_{L^{p}} \ds 
\aleq s_{0}^{\alp - \bt} \nrm{(-\lap)^{\alp} f}_{L^{p}} + s_{0}^{-\bt} \nrm{f}_{L^{p}}.
\end{equation*}
Optimizing the choice of $s_{0}$, we obtain \eqref{eq:Lp-int}.
The proof of \eqref{eq:Lp-GN} is very similar. By \eqref{eq:LP-id} and Lemma~\ref{l:lin-heat-Lp-Lq}, we have
\begin{equation*}
	\nrm{f}_{L^{q}} \aleq \int_{0}^{\infty} s^{\frac{1}{2} (\frac{2}{q} - \frac{2}{p})} \nrm{P_{s} f}_{L^{p}} \, \ds.
\end{equation*}
Using the preceding two estimates for $\nrm{P_{s} f}_{L^{p}}$, we obtain the desired estimate. \qedhere
\end{proof}

\begin{lemma}[Sobolev embedding]\label{l:se} 
Let $f \in C^{\infty}_{0}(\bbH^{2})$. Then for any $1 < p < 2$ and $p \leq q < \infty$ with $\frac{1}{q} = \frac{1}{p} - \frac{1}{2}$, we have
\begin{equation} \label{eq:se}
	\nrm{f}_{L^{q}} \aleq \nrm{\nb f}_{L^{p}}.
\end{equation}
\end{lemma} 
For a proof, we refer to \cite[Theorem~3.2]{Heb}.

\begin{lemma}[Gagliardo-Nirenberg inequality]\label{lem:gagliardo_nirenberg}
Let $f \in C^{\infty}_{0}(\bbH^{2})$. Then for any $1 < p < \infty$, $p \leq q \leq \infty$ and $0 < \tht < 1$ such that $\frac{1}{q} =  \frac{1}{p} - \frac{\te}{2}$, we have
\EQ{
\| f\|_{L^q} \lesssim \| f\|_{L^p}^{1-\te} \| \na f\|_{L^p}^{\te}
}
In particular, for any $s>0$
\begin{equation} \label{eq:GN-Linfty}
  \|f\|_{L^\infty}\lesssim \|f\|_{L^4}^{\frac{1}{2}} \|\nabla f\|_{L^4}^{\frac{1}{2}} \lesssim \|f\|_{L^2}^{\frac{1}{4}}\|\nabla f\|_{L^2}^{\frac{1}{2}}\|\Delta f\|_{L^2}^{\frac{1}{4}}\lesssim s^{-\frac{1}{2}}\|\jap{s\Delta}f\|_{L^2}.
\end{equation}
\end{lemma}
\begin{proof}
The first estimate follows from Lemma~\ref{l:riesz} and \eqref{eq:Lp-GN}. The last estimate follows from two applications of the first.
\end{proof}

We conclude with an improvement of \eqref{eq:GN-Linfty} under the assumption of extra angular regularity.
\begin{lemma}[Radial Sobolev] \label{lem:radial_sobolev} 
 For any function $f$
 \begin{equation}
  \| \sinh^{\frac{1}{2}}(r) f \|_{L^\infty_x(\bbH^2)} \lesssim \| \langle \Omega \rangle f \|_{L^2_x(\bbH^2)}^{\frac{1}{2}} \| \nabla \langle \Omega \rangle f \|_{L^2_x(\bbH^2)}^{\frac{1}{2}}\lesssim s^{-\frac{1}{4}}\|\jap{s\Delta}^{\frac{1}{2}} \jap{\Omega} f\|_{L_x^2(\bbH^2)}.
 \end{equation}
 Similarly, for any tensor $\bsxi$
 \begin{equation}
  \| \sinh^{\frac{1}{2}}(r) \bsxi \|_{L^\infty_x(\bbH^2)} \lesssim \| \langle \calL_\Omega \rangle \bsxi\|_{L^2_x(\bbH^2)}^{\frac{1}{2}} \| \nabla \langle \calL_\Omega \rangle \bsxi \|_{L^2_x(\bbH^2)}^{\frac{1}{2}}\lesssim s^{-\frac{1}{4}}\|\jap{s\Delta}^{\frac{1}{2}}\jap{\calL_\Omega}\bsxi\|_{L_x^2(\bbH^2)}.
 \end{equation}
 In particular for $\bsxi = \nabla f$ we get
 \begin{equation*}
  \|\sinh^{\frac{1}{2}}(r) \nabla f\|_{L_x^\infty}\lesssim s^{-\frac{1}{4}} \|\nabla \jap{s\Delta}^{\frac{1}{2}}\jap{\Omega}f\|_{L_x^2}.
 \end{equation*}
\end{lemma}
\begin{proof}
First for a function $f$ note that by Sobolev embedding on $\bbS^1$ we have
%%%%%%%
%%%%%%%
\begin{align*}
\begin{split}
\|f\|_{L^\infty_x}\lesssim \|\jap{\Omega}f\|_{L_r^\infty L_\theta^2},
\end{split}
\end{align*}
and with $g=f$ or $g=\Omega f$
%%%%%%%
%%%%%%%
\begin{align*}
\begin{split}
\int_{\bbS^1}|g(r,\theta)|^2\ud \theta &= -2\Re \, \int_r^\infty \int_{\bbS^1} \barg(r',\theta)\partial_rg(r',\theta) \, \ud\theta \, \ud r'\\
&\lesssim \sinh^{-1}(r) \Big(\int_r^\infty\int_{\bbS^2}|g(r',\theta)|^2 \sinh(r') \, \ud \theta \, \ud r' \Big)^{\frac{1}{2}} \Big( \int_r^\infty\int_{\bbS^1} |\partial_r g(r',\theta)|^2 \sinh(r') \, \ud \theta \, \ud r' \Big)^{\frac{1}{2}} \\
&\lesssim \sinh^{-1}(r) \|g\|_{L^2_x} \|\nabla g\|_{L_x^2}.
\end{split}
\end{align*}
The proof for tensors is similar. Using the fact that by the diamagnetic inequality (since $\Omega$ is an isometry)
%%%%%%%
%%%%%%%
\begin{align*}
\begin{split}
\Omega|\bsxi|\leq |\calL_\Omega \bsxi|
\end{split}
\end{align*}
we get from the Sobolev estimate on $\bbS^1$ that
%%%%%%%
%%%%%%%
\begin{align*}
\begin{split}
\|\bsxi\|_{L^\infty_x}=\||\bsxi|\|_{L_x^\infty}\lesssim \|\jap{\calL_\Omega}\bsxi\|_{L_r^\infty L_\theta^2}.
\end{split}
\end{align*}
The rest of the proof is as in the scalar case applied to $g=|\bszeta|$ with $\bszeta=\bsxi$ or $\bszeta=\calL_\Omega \xi$, and where we also use $\partial_r|\bszeta|^2=2\Re((\nabla_{\partial_r}\bszeta^\mu)\overline{\bszeta}_\mu)$ and $|\nabla_{\partial_r}\bszeta|\leq |\nabla\bszeta|$ since $\partial_r$ has length one.

Finally, the last statement follows from the identity $\calL_\Omega \nabla f = \nabla \Omega f$.
\end{proof}

\subsubsection{Fractional calculus in $L^{2}$-based Sobolev spaces} 
Here, we prove the following results for the fractional $L^{2}$-based Sobolev spaces on $\bbH^{2}$:
\begin{proposition} [Sobolev product rule] \label{prop:frac-leib}
For $\sgm \geq 0$, we have
\begin{equation*}
\nrm{f g}_{H^{\sgm}} \aleq \nrm{f}_{L^{\infty}} \nrm{g}_{H^{\sgm}} + \nrm{f}_{H^{\sgm}} \nrm{g}_{L^{\infty}}.
\end{equation*}
\end{proposition}

\begin{proposition} [Moser-type estimate in Sobolev spaces] \label{prop:frac-moser}
Let $\sgm \geq 0$. Consider a function $F : \bbR \times \bbH^{2} \to \bbR$ satisfying $F(0; x) = 0$ and $\nrm{F}_{C^{\lfloor \sgm \rfloor+2}} < \infty$. Then for any function $g \in H^{\sgm} \cap L^{\infty}$ such that $\nrm{g}_{L^{\infty}} \leq A$, we have
\begin{equation*}
	\nrm{F(g(x); x)}_{H^{\sgm}} \aleq_{\sgm, \nrm{F}_{C^{\lfloor \sgm \rfloor+2}}, A} \nrm{g}_{H^{\sgm}}.
\end{equation*}
\end{proposition}

One way to prove these results is to use the intrinsic heat-flow-based Littlewood--Paley projections $P_{\sgm}$ as in \cite{LLOS1}. Instead, we rely on the following lemma to reduce these propositions to the corresponding results on the Euclidean space:
\begin{lemma} [Localization lemma] \label{lem:sob-loc}
For any $\dlt > 0$, consider points $x_{j} \in \bbH^{2}$ such that the balls $B_{\dlt}(x_{j})$ form a locally finite covering of $\bbH^{2}$. Let $\chi_{j}$ be a smooth partition of unity subordinate to $\set{B_{\dlt}(x_{j})}$ and let $\exp_{x_{j}} : \bbR^{2} \to \bbH^{2}$ be the exponential map centered at $x_{j}$. For any $\sgm \in \bbR$ and $f \in H^{\sgm}(\bbH^{2})$, we have
\begin{equation*}
	\nrm{f}_{H^{\sgm}(\bbH^{2})}^{2} \simeq \sum_{j} \nrm{\chi_{j} f}_{H^{\sgm}(\bbH^{2})}^{2} \simeq \sum_{j} \nrm{(\chi_{j} f) \circ \exp_{x_{j}}}_{H^{\sgm}(\bbR^{2})}^{2}
\end{equation*}
where the implicit constant depends only on $\sgm$, $\dlt$, $\set{x_{j}}$ and $\set{\nrm{\chi_{j}}_{C^{\lceil \abs{\sgm} \rceil}}}$.
\end{lemma}
These equivalences are obvious when $\sgm$ is a nonnegative integer; the general case then follows by interpolation (for $\sgm \in (0, \infty)$) and duality (for $\sgm < 0$). We refer to \cite[Ch.~7]{TriebelFS2} for further details. Note that this result holds in any domain manifold with bounded geometry, in which case $\dlt > 0$ should now be chosen sufficiently small. In the applications of Lemma~\ref{lem:sob-loc} below, we use arbitrary but fixed choices of $\dlt$, $\set{x_{j}}$ and $\set{\chi_{j}}$.

By Lemma~\ref{lem:sob-loc}, Proposition~\ref{prop:frac-leib} readily reduces to the standard Sobolev product rule in $\bbR^{2}$, whose proof can be found in, e.g., \cite[Ch.~2, Prop.~1.1]{TayTools}. The Euclidean counterpart of Proposition~\ref{prop:frac-moser} is less standard, so we provide a proof below:

\begin{proof}[Proof of Proposition~\ref{prop:frac-moser}]
When $\sgm = 0$, the desired bound follows by the point-wise inequality
\begin{equation*}
	\abs{F(g; x)} \leq \abs{\nb_{g} F(g; x)} \abs{g} \aleq_{\nrm{F}_{C^{1}}} \abs{g}.
\end{equation*}
Hence we may assume that $\sgm > 0$. By the localization lemma (Lemma~\ref{lem:sob-loc}), Proposition~\ref{prop:frac-moser} readily reduces to the following Euclidean counterpart: 
\begin{center}
{\it For a function $F : \bbR \times \bbR^{2} \to \bbR$ satisfying $F(0; x) = 0$ and $\nrm{F}_{C^{\lceil \sgm \rceil + 1}} < \infty$ and $g \in H^{\sgm} \cap L^{\infty}(\bbR^{2})$, 
\begin{equation} \label{eq:frac-moser-euc}
	\nrm{F(g(x); x)}_{H^{\sgm}(\bbR^{2})} \aleq_{\sgm, \nrm{F}_{C^{\lfloor \sgm \rfloor+2}(\bbR \times \bbR^{2})}, \nrm{g}_{L^{\infty}(\bbR^{2})}} \nrm{g}_{H^{\sgm}(\bbR^{2})}.
\end{equation}}
\end{center}
The goal of the remainder of the proof is to prove \eqref{eq:frac-moser-euc}. To simplify the notation, in what follows we will suppress the domain $\bbR^{2}$ and simply write $H^{\sgm} = H^{\sgm}(\bbR^{2})$, etc. We fix a radial cutoff $m_{\leq 0}$ supported in $\set{\abs{\xi} < 2}$ that equals $1$ on $\set{\abs{\xi} < 1}$. For $j \in \bbR^{+}$, we denote by $P_{\leq j}$ the Fourier multiplier $m_{\leq 0}(2^{-j} D)$. For $j > 0$, we define $P_{j} = \frac{\ud}{\ud j} P_{\leq j}$; in the case $ j = 0$, we adopt the convention $P_{0} = P_{\leq 0}$. For any tempered distribution $g$, we have the identity $g(x) = \int_{0}^{\infty} P_{j} g(x) \, \ud j + P_{0} g$. Moreover, using the Fourier transform, it is not difficult to see that
\begin{equation*}
	\nrm{g}_{H^{\sgm}}^{2} \simeq \int_{0}^{\infty} (2^{\sgm j}\nrm{P_{j} g}_{L^{2}})^{2} \, \ud j + \nrm{P_{0} g}_{L^{2}}^{2}.
\end{equation*}

We now begin the proof in earnest by computing
\begin{align*}
	F(g(x); x) 
	&= \int_{0}^{\infty} \frac{\ud}{\ud j} F(P_{\leq j} g(x); x) \, \ud j + F(P_{0} g(x); x) \\
	&= \int_{0}^{\infty} \nb_{g} F(P_{\leq j} g(x); x) P_{j} g(x) \, \ud j + F(P_{0} g(x); x),
\end{align*}
which is justified since $F(g(x); x) \in L^{2}$. Thus, by Schur's test and the above characterization of the $H^{\sgm}$-norm, the proof of \eqref{eq:frac-moser-euc} reduces to the verification of the following estimate for $j, k \geq 0$ for some $c = c(\sgm) > 0$:
\begin{align}
	2^{\sgm k} \nrm{P_{k} (\nb_{g} F(P_{\leq j} g(x); x) P_{j} g(x))}_{L^{2}} & \aleq 2^{-c \abs{j-k}} 2^{\sgm j}\nrm{P_{j} g(x)}_{L^{2}}, \label{eq:frac-moser-euc-core}
\end{align}
where the implicit constants may depend on $\sgm$, $\nrm{F}_{C^{\lfloor \sgm \rfloor+2}}$ and $\nrm{g}_{L^{\infty}}$. 

For $k \leq j + 10$, we simply estimate
\begin{align*}
2^{\sgm k} \nrm{P_{k} (\nb_{g} F(P_{\leq j} g(x); x) P_{j} g(x))}_{L^{2}}
\aleq 2^{\sgm (k-j)} \nrm{F}_{C^{1}} 2^{\sgm j}\nrm{P_{j} g(x)}_{L^{2}},
\end{align*}
which is acceptable. In the case $k > j + 10$, note that the frequency support of $\nb_{g} F(P_{\leq j} g(x); x)$ may be localized to $\set{\abs{\xi} \simeq 2^{k}}$. Thus,
\begin{align*}
2^{\sgm k} \nrm{P_{k} (\nb_{g} F(P_{\leq j} g(x); x) P_{j} g(x)}_{L^{2}} 
& \aleq 2^{\sgm (k-j)} 2^{\sgm j} \nrm{P_{j} g(x) }_{L^{2}} \int_{\abs{k' - k} < 1} \nrm{P_{k'} \nb_{g} F(P_{\leq j} g(x); x)}_{L^{\infty}} \, \ud k' .
\end{align*}
For the last factor, we use the following bound with $n = \lfloor \sgm \rfloor+2$, which is obtained using the chain rule and the bound $\nrm{\nb^{(n)} P_{\leq j} g}_{L^{\infty}} \aleq 2^{n j} \nrm{g}_{L^{\infty}}$:
\begin{equation*}
	 \nrm{P_{k'} \nb_{g} F(P_{\leq j} g(x); x)}_{L^{\infty}}
	 \aleq 2^{-n k'} \nrm{\nb^{(n)} \nb_{g} F(P_{\leq j} g(x); x)}_{L^{\infty}} \aleq_{\sgm, \nrm{F}_{C^{\lfloor \sgm \rfloor+2}}, \nrm{g}_{L^{\infty}}} 2^{-n(k' - j)}.
\end{equation*}
In conclusion, for $k > j+ 10$,
\begin{equation*}
2^{\sgm k} \nrm{P_{k} (\nb_{g} F(P_{\leq j} g(x); x) P_{j} g(x))}_{L^{2}} \aleq_{\sgm, \nrm{F}_{C^{\lfloor \sgm \rfloor+2}}, \nrm{g}_{L^{\infty}}} 2^{- (n - \sgm) (k - j)} 2^{\sgm j} \nrm{P_{j} g}_{L^{2}},
\end{equation*}
which is good since $n > \sgm$. \qedhere
\end{proof}

Finally, we point out a simple corollary of Propositions~\ref{prop:frac-leib} and \ref{prop:frac-moser} that gives a difference bound for $F(g(x); x)$.
\begin{corollary} \label{cor:frac-moser-diff}
Let $\sgm \geq 0$. Consider a function $F: \bbR \times \bbH^{2} \to \bbR$ satisfying $F(0; x) = 0$ and $\nrm{F}_{C^{\lfloor \sgm \rfloor + 3}} < \infty$. Then for any functions $g, h \in H^{\sgm} \cap L^{\infty}$ such that $\nrm{g}_{L^{\infty}} + \nrm{h}_{L^{\infty}} \leq A$, we have
\begin{equation*}
	\nrm{F(g(x); x) - F(h(x); x)}_{H^{\sgm}}
	\aleq_{\sgm, \nrm{F}_{C^{\lfloor \sgm \rfloor + 3}}, A} \nrm{g-h}_{H^{\sgm}} + (\nrm{g}_{H^{\sgm}} + \nrm{h}_{H^{\sgm}}) \nrm{g-h}_{L^{\infty}}.
\end{equation*}
\end{corollary}
\begin{proof}
By the fundamental theorem of calculus, we may write
\begin{equation*}
	F(g(x); x) - F(h(x); x)
	= (g-h)(x) \int_{0}^{1} \rd_{g} F(\lmb g(x) + (1-\lmb) h(x); x) \, \ud \lmb.
\end{equation*}
Note that the $L^{\infty}$ norm of $\int_{0}^{1} \rd_{g} F(\lmb g(x) + (1-\lmb) h(x); x) \, \ud \lmb$ is $O_{\nrm{F}_{C^{2}}} (A)$ by another application of the fundamental theorem of calculus, whereas its $H^{\sgm}$ norm is bounded by $O_{\sgm, \nrm{F}_{C^{\lfloor \sgm \rfloor + 3}}, A} (\nrm{g}_{H^{\sgm}} + \nrm{h}_{H^{\sgm}})$ by Proposition~\ref{prop:frac-moser}. Thus, by Proposition~\ref{prop:frac-leib}, the desired estimate follows. \qedhere
\end{proof}

%%%%%%%%%%%%%%%%%%%%%%%%%%%%%%%%%%%%%%%%%%
%%%%%%%%%%%%%%%%%%%%%%%%%%%%%%%%%%%%%%%%%%
\subsection{Local well-posedness of the Schr\"odinger maps equation} 
%%%%%%%%%%%%%%%%%%%%%%%%%%%%%%%%%%%%%%%%%%
%%%%%%%%%%%%%%%%%%%%%%%%%%%%%%%%%%%%%%%%%%

A basic ingredient of the proof of Theorem~\ref{thm:main} is the existence of a unique local solution with regular data, and a continuation criterion. In the case of Schr\"odinger maps on $\bbR^d$, local well-posedness was proved by McGahagan in~\cite{McGahagan}, see also \cite{KLPST, DingWang}. McGahagan's proof consists of an approximation scheme by a wave map-like equation, and is carried out in a geometric framework using covariant derivatives. As such, it easily extends to the case of the hyperbolic domain $\bbH^2$. We omit the straightforward modifications and state the corresponding result without proof.

%%%%%%%%%
%%%%%%%%%
\begin{lemma}[McGahagan \emph{\cite{McGahagan}}]\label{l:lwp}
Let $\NN = \Hp^2$ or $\NN= \Sp^2$ and let $Q$ be an equivariant harmonic map as in Proposition~\ref{prop:hm-classify}. Given initial data $u_0:\bbH^2\to\NN$ with the same boundary data as $Q$, and lying in the space $H^5_Q(\Hp^2; \NN)$, there exists a time $T>0$ and a unique solution $u:[0,T)\times \bbH^2\to \NN$ to the Schr\"odinger map equation \eqref{equ:schroedinger_maps_equ} with $u\in C^0([0,T);H^5_Q(\bbH^2;\NN))$  and $u(0,\cdot)=u_0(\cdot)$. Moreover, if $T^\ast>0$ is the maximal such $T$, then either $T^\ast=\infty$ or $\lim_{t\to T^\ast}\|u(t,\cdot)- Q\|_{H^5}=\infty$. Finally, any additional regularity of the initial data is preserved by the flow. In particular, if $(u_0- Q) \in H^\infty$, then $\p_t^k u(t) \in H^\infty$ for every $k \in \N$. 
\end{lemma}
%%%%%%%%%
%%%%%%%%%

\begin{remark}
The local well-posedness result in \cite{McGahagan} is actually sharper, in that it only requires the data to be in $H^\ell$ for $\ell>2$. However, this relies on the almost optimal well-posedness result of Klainerman and Selberg from \cite{KlaSel97} for the approximate wave-map like system mentioned above. Although no such almost optimal well-posedness result for wave maps on hyperbolic space is available, this is not a problem in our setting as we may restrict to smoother data and rely on the standard well-posedness theory for wave maps given by energy methods in two space dimensions say for $H^4$ data. Following McGahagan's core argument in this more restrictive (and easier) setting then yields the $H^5$ well-posedness result given in Theorem~\ref{l:lwp}, and also avoids many technical difficulties confronted by McGahagan in~\cite{McGahagan}. Finally, the persistence of regularity statement  in the final two sentences of Lemma~\ref{l:lwp} follows from standard arguments. 
\end{remark}

\section{Linearized operators and the associated parabolic/dispersive theories} \label{sec:hm}
The purpose of this section is to describe more precisely the linearized operator $\calH_{Q}$ about a finite energy equivariant harmonic map $Q$ as in Proposition~\ref{prop:hm-classify}, as well as the associated parabolic and dispersive theories. In Subsection~\ref{subsec:hm}, we introduce a special frame $e^{\infty}$ on $Q^{\ast} T \tgmfd$ (Coulomb frame) and express $\calH_{Q}$ in the associated moving frame formalism. As a result, we arrive at a self-adjoint operator $H$ acting on complex-valued functions on $\bbH^{2}$. Using the expression for $H$, we also verify the weak and strong linearized stability properties of $Q$ that were claimed in Subsection~\ref{subsec:result} (Propositions~\ref{prop:weak-stab} and \ref{prop:strong-stab}). In the remaining subsections, we collect key analytic tools for analyzing the parabolic and dispersive equations arising from the Schr\"odinger maps evolution in the caloric gauge. More specifically, in Subsection~\ref{subsec:lin-para}, we prove $L^{p}$ estimates for the parabolic operator $\rd_{s} + H$, and in Subsection~\ref{subsec:disp}, we describe the global-in-time local smoothing and Strichartz estimates for the Schr\"odinger operator $- i \rd_{t} + H$ that were established in \cite{LLOS1}.

\subsection{Harmonic maps from $\bbH^{2}$ into $\tgmfd$ and the linearized operator in the Coulomb gauge} \label{subsec:hm} 
Let $Q$ be a finite energy harmonic map from $\bbH^{2}$ into either $\tgmfd = \bbS^{2}$ or $\bbH^{2}$ as in Proposition~\ref{prop:hm-classify}. Let $e = (e_{1}, e_{2})$ be a global orthonormal frame on $Q^{\ast} T \tgmfd$ that is positively oriented, i.e., $e_{2} = J e_{1}$. We denote the linearized operator $\calH_{Q}$, which was introduced in \eqref{eq:lin-hm}, in the moving frame formalism with respect to $e$ by $H$, i.e.,
\begin{equation*}
	\calH_{Q} (e \varphi) = e H \varphi.
\end{equation*}
By \eqref{eq:comm-K}, $H$ takes the form
\begin{equation*}
	H \varphi
	= - \Db^{\ell} \Db_{\ell} \varphi + i \tcv \Im(\psi_{\ell} \overline{\varphi}) \psi^{\ell},
\end{equation*}
where $e \psi_{\ell} = \rd_{\ell} Q$. To simplify the RHS, we note that $Q$ obeys the Cauchy--Riemann equation
\begin{equation} \label{eq:CR-int}
	\bfe_{1} Q  = J \bfe_{2} Q,
\end{equation}
where $(\bfe_{1}, \bfe_{2})$ is a positively oriented orthonormal frame on $\bbH^{2}$. Indeed, using the explicit form of $Q$ in Proposition~\ref{prop:hm-classify}, \eqref{eq:CR-int} may be easily verified for the specific orthonormal frame $(\bfe_{1}, \bfe_{2})$ given by \eqref{eq:dom-frame}; the general case then easily follows. In the moving frame formalism, \eqref{eq:CR-int} reads
\begin{equation} \label{eq:CR}
	\psi_{1} = i \psi_{2},
\end{equation}
where $\psi_{1} = \bfe_{1}^{\ell} \psi_{\ell}$, $\psi_{2} = \bfe_{2}^{\ell} \psi_{\ell}$. Thus,
\begin{equation} \label{eq:lin-Q}
	H \varphi
%	= - \Db^{\ell} \Db_{\ell} \varphi + \tcv \frac{i}{2i}(\psi_{1} \overline{\varphi} - \overline{\psi_{1}} \varphi) \psi_{1}
%	+ \tcv \frac{i}{2i}(\psi_{2} \overline{\varphi} - \overline{\psi_{2}} \varphi) \psi_{2} 
	= - \Db^{\ell} \Db_{\ell} \varphi - \frac{\tcv}{2}(\psi_{2} \overline{\varphi} + \overline{\psi_{2}} \varphi) \psi_{2}
	+ \frac{\tcv}{2}(\psi_{2} \overline{\varphi} - \overline{\psi_{2}} \varphi) \psi_{2} 
	= - \Db^{\ell} \Db_{\ell} \varphi - \tcv \abs{\psi_{2}}^{2} \varphi.
\end{equation}

The property that the second-order harmonic maps equation $D^{\ell} \rd_{\ell} Q = 0$ reduces to the first-order Cauchy--Riemann equation \eqref{eq:CR-int} is reflected by the so-called \emph{Bogomoln'yi structure} of the linearized operator. Specifically, consider the first-order \emph{Bogomoln'yi operator} (with respect to the frame \eqref{eq:dom-frame})
\begin{equation} \label{eq:bogomolnyi}
	L \varphi = \Db_{r} \varphi +  \frac{i}{\sinh r} \Db_{\tht} \varphi.
\end{equation}
Then it may be checked that $H$ admits the splitting
\begin{equation} \label{eq:L-ast-L=H}
	L^{\ast} L \varphi = H \varphi.
\end{equation}

We now specify a special frame $e^{\infty}$ on $Q^{\ast} T \tgmfd$, which will turn out to satisfy the Coulomb condition \eqref{eq:cf-coulomb} (see also Remark~\ref{rem:cf}), with respect to which the connection $1$-form $A$ and the $0$-order potential $\tcv \abs{\psi_{2}}^{2}$ may be explicitly computed in \eqref{eq:lin-Q}. As in Subsection~\ref{subsec:result}, fix a point (the origin) in $\tgmfd = \bbS^{2}$ or $\bbH^{2}$, and consider the associated polar coordinates $(\rho, \vtht)$. Define
\begin{equation*}
	S(\rho) = \begin{cases} \sin \rho & \hbox{ when } \tgmfd = \bbS^{2}, \\
					\sinh \rho & \hbox{ when } \tgmfd = \bbH^{2}, \end{cases} \qquad
	S'(\rho) = \begin{cases} \cos \rho & \hbox{ when } \tgmfd = \bbS^{2}, \\
					\cosh \rho & \hbox{ when } \tgmfd = \bbH^{2}. \end{cases} 
\end{equation*}
so that in both cases, the metric $g$ takes the form $g = \ud \rho^{2} + S^{2} \, \ud \vtht^{2}$.
We introduce
\begin{align}\label{eq:final_frame}
\begin{split}
e^\infty_1(r,\theta)=\cos\theta \partial_\rho- \frac{\sin\theta}{S(Q(r))} \partial_\vartheta
,\qquad e^\infty_2= \sin\theta\partial_\rho+\frac{\cos\theta}{S(Q(r))} \partial_\vartheta.
\end{split}
\end{align}

%%%%%%%%%%%%%%%%%
%%%%%%%%%%%%%%%%%
\begin{lemma}\label{lem:cf}  Under the choice \eqref{eq:final_frame} for $\set{e^{\infty}_{1}, e^{\infty}_{2}}$, the connection $1$-form $A_{j}^{\infty} = (D_{\bfe_{j}} e^{\infty}_{1}, e^{\infty}_{2})$ and $\psi_{j}^{\infty}$ (defined by $e \psi_{j}^{\infty} = \bfe_{j}^{\ell} \rd_{\ell} Q$) are given by
%%%%%%%
%%%%%%%
\begin{equation} \label{eq:cf-coeff}
\begin{split}
&A^{\infty}_{1} := A^{\infty}_{r}=0,\qquad A^{\infty}_{2} := \frac{1}{\sinh r} A^{\infty}_{\theta}= \frac{S'(Q)-1}{\sinh r},\\
&\psi_1^\infty := \psi_r^\infty=\frac{S(Q)}{\sinh r} e^{i\theta}  ,\qquad \psi_2^\infty := \frac{1}{\sinh r}\psi_\theta^\infty =\frac{S(Q)}{\sinh r}e^{i(\theta+\frac{\pi}{2})}.
\end{split}
\end{equation}
In particular, $A^{\infty}$ obeys the Coulomb gauge condition
\begin{equation} \label{eq:cf-coulomb}
	\rd^{\ell} A^{\infty}_{\ell} 
	= \rd_{r} A^{\infty}_{r} + r^{-1} A^{\infty}_{r} + \rd_{\tht} A^{\infty}_{\tht} = 0.
\end{equation}
\end{lemma}
%%%%%%%%%%%%%%%%%
%%%%%%%%%%%%%%%%%
\begin{proof}
Note that 
\begin{align*}
\begin{split}
D_{j} e_{\bfa}^\infty = (\partial_j e_{\bfa}^{\infty,a} + {}^{(\tgmfd)} \Gmm_{bc}^{a}(Q) \partial_j Q^b e_{\bfa}^{\infty,c}) \rd_{a}.
\end{split}
\end{align*} 
After a straightforward computation of ${}^{(\tgmfd)} \Gmm$, this yields
\begin{align*}
D_{r} e^{\infty}_{1} = D_{r} e^{\infty}_{2} = 0, \qquad
 D_{\tht} e^{\infty}_{1} = (S'(Q) - 1) e_{2}^{\infty}, \qquad
D_{\tht} e^{\infty}_{2} = - (S'(Q) - 1) e_{1}^{\infty}.
\end{align*}
Next, the expressions for $\psi_r^\infty$ and $\psi_\theta^\infty$ follow by observing that, since $\rd_{r} Q^{\rho} = \frac{S(Q)}{\sinh r}$ by \eqref{eq:CR-int},
\begin{align*}
	\rd_{r} Q 
	& = \rd_{r} Q^{\rho} \rd_{\rho} 
	= \frac{S(Q)}{\sinh r} (\cos \tht e^{\infty}_{1} + \sin \tht e^{\infty}_{2}), \\
	\frac{1}{\sinh r} \rd_{\tht} Q
	& = \frac{1}{\sinh r} \rd_{\vtht} 
	= \frac{S(Q)}{\sinh r} (- \sin \tht e^{\infty}_{1} + \cos \tht e^{\infty}_{1}). 
\end{align*}
Finally, the Coulomb condition \eqref{eq:cf-coulomb} trivially follows.
\end{proof}

\begin{remark} \label{rem:cf}
It is not difficult to show that, under a suitable decay condition on $A$, the Coulomb gauge condition \eqref{eq:cf-coulomb} determines the frame $e$ up to a global constant rotation. Moreover, \eqref{eq:cf-coulomb} is equivalent to the harmonic condition $D^{\ell} D_{\ell} e = 0$ in Definition~\ref{def:caloric-simple}. Indeed,
\begin{align*}
i \rd^{\ell} A_{\ell} = \rd^{\ell} \brk{D_{\ell} e_{1}, e_{2}}
=& ( D^{\ell} D_{\ell} e_{1}, e_{2}) + (D_{\ell} e_{1}, D_{\ell} e_{2}) \\
=& ( D^{\ell} D_{\ell} e_{1}, e_{2}) + \frac{1}{2} \left( \lap (e_{1}, e_{2}) - (D^{\ell} D_{\ell} e_{1}, e_{2}) - (e_{1}, D^{\ell} D_{\ell} e_{2}) \right) \\
=& \frac{1}{2} \left((D^{\ell} D_{\ell} e_{1}, e_{2}) - (e_{1}, D^{\ell} D_{\ell} e_{2}) \right).
\end{align*}
 On the other hand, $D^{\ell} D_{\ell} e_{1} = 0$ is equivalent to $D^{\ell} D_{\ell} e_{2}$, since $e_{2} = J e_{1}$ and $D J = 0$.
\end{remark}

Using the expressions for $A_{j}^{\infty}$ and $\psi_{j}^{\infty}$, we can compute $\calL_\Omega A_{j}^\infty$ and $\calL_\Omega \psi_{j}^\infty$. In the polar coordinates $(r, \tht)$ on the domain, we find that
\begin{align} \label{eq:LOmA} 
\begin{split}
\calL_\Omega A^\infty_j=\Omega A^\infty_j+\partial_j\Omega^\ell A^\infty_\ell =\partial_\theta A^\infty_j= 0.
%\quad \Rightarrow \quad\calL_\Omega \bfA^\infty = 0.
\end{split}
\end{align}
Similarly,
\begin{align} \label{eq:LOmpsi} 
\begin{split}
\calL_\Omega\psi^\infty_{j} = \Omega \psi^\infty_{j}+(\partial_{j}\Omega^{\ell})\psi^\infty_{\ell}=\partial_\theta \psi^\infty_{j}=i\psi^\infty_{j}.
%\quad \Rightarrow \quad \calL_\Omega\Psi^\infty = i \Psi^\infty.
\end{split}
\end{align}

With the help of Lemma~\ref{lem:cf}, we now verify that all $Q$ in Proposition~\ref{prop:hm-classify} satisfy the weak linearized stability condition.

\begin{proposition} \label{prop:weak-stab}
Let $Q$ be an equivariant finite energy harmonic map from $\bbH^{2}$ into $\tgmfd = \bbH^{2}$ or $\bbS^{2}$, which is given by Proposition~\ref{prop:hm-classify}. Then the associated linearized operator $\calH_{Q}$ obeys the weak linearized stability condition (Definition~\ref{def:weak-stab}).
\end{proposition}
\begin{proof}
Suppose, for the purpose of contradiction, that the weak linearized stability condition \eqref{eq:weak-stab} is violated. The first step of the proof is to show that there exists a nontrivial solution $\varphi_{0} \in L^{2}$ to $H \varphi_{0} = 0$, where $H$ is given by \eqref{eq:lin-Q} and \eqref{eq:final_frame}. 
We begin by noting that, by \eqref{eq:L-ast-L=H}, we have
\begin{equation*}
	\int (\calH_{Q} \phi, \phi) = \int (H \varphi, \varphi) = \int (L \varphi, L \varphi) \geq 0, \quad \hbox{ where } e \varphi = \phi.
\end{equation*}
Therefore, by the contradiction hypothesis, there exists a sequence $\varphi^{n} \in C^{\infty}_{0}$ such that $\int (H \varphi^{n}, \varphi^{n}) \to 0$, yet $\nrm{\varphi^{n}}_{L^{2}} = 1$. By the $L^{2}$ elliptic theory, it follows that $\nrm{\varphi^{n}}_{H^{1}} \aleq 1$. Hence, there exists $\varphi_{0} \in H^{1}$ such that, after passing to a subsequence, $\varphi^{n} \rightharpoonup \varphi_{0}$ in $H^{1}$ and $\varphi^{n} \to \varphi_{0}$ strongly in $L^{2}_{loc}$. Clearly, $H \varphi_{0} = 0$, so it only remains to verify that $\varphi_{0} \neq 0$. 

To prove that $\varphi_{0} \neq 0$, it suffices, by $\nrm{\varphi^{n}}_{L^{2}} = 1$ and the strong convergence in $L^{2}_{loc}$, to show that 
\begin{equation} \label{eq:weak-stab-rem}
	\inf_{n} \nrm{(1+r^{2})^{-1} \varphi^{n}}_{L^{2}} > 0,
\end{equation}
(where the exponent $1$ was arbitrarily chosen). Assume the contrary; then after passing to a subsequence, $\nrm{(1+r^{2})^{-1} \varphi^{n}}_{L^{2}} \to 0$. By the decay of $A^{\infty}_{x}$ and $\tcv \abs{\psi_{2}}^{2}$ in Lemma~\ref{lem:cf}, it follows that
\begin{align*}
	\int (-\lap \varphi^{n}, \varphi^{n})
	&= \int (H \varphi^{n}, \varphi^{n}) + \int (- 2i (A^{\infty})^{\ell} \rd_{\ell}\varphi^{n}, \varphi^{n}) + \int (((A^{\infty})^{\ell} (A^{\infty})_{\ell} - \tcv \abs{\psi_{2}}^{2}) \varphi^{n}, \varphi^{n}) \\
	&\leq \int (H \varphi^{n}, \varphi^{n}) + C  \nrm{\varphi^{n}}_{H^{1}} \nrm{(1+r^{2})^{-1} \varphi^{n}}_{L^{2}} \to 0 \quad \hbox{ as } n \to \infty.
\end{align*}
Then, by the Poincar\'e equality, $\nrm{\varphi^{n}}_{L^{2}} \to 0$ as well, which is impossible. Hence, \eqref{eq:weak-stab-rem} holds and $\varphi_{0} \neq 0$.

The remainder of the proof consists of disproving the existence of $\varphi_{0} \neq 0$. By \eqref{eq:L-ast-L=H}, $\varphi_{0}$ must also satisfy the Cauchy--Riemann-type equation 
\begin{equation*}
	L \varphi_{0} = \rd_{r} \varphi_0 + \frac{i}{\sinh r} \rd_{\tht} \varphi_0 + \left( i A_{r} - \frac{1}{\sinh r} A_{\tht}\right) \varphi_{0} = 0.
\end{equation*}
By Lemma~\ref{lem:cf}, note that $A_{r} = 0$ and $A_{\tht}$ is radial. Thus, we may separate $\varphi_{0}$ into angular Fourier modes $\varphi_{0; m} = \frac{1}{2 \pi} \int_{\bbS^{1}} \varphi_{0} e^{-i m \tht} \, \ud \tht e^{i m \tht}$. For each $\varphi_{0; m}$, the ODE is obviously decoupled and solvable; then $\varphi_{0} \in L^{2}$ implies that $\varphi_{0; m} = 0$ for each $m$, which is a contradiction. \qedhere
\end{proof}

Finally, we prove that the strong linearized stability condition is satisfied by all $Q$ in Proposition~\ref{prop:hm-classify} in the case $\tgmfd = \bbH^{2}$, and for small enough $\alp$ in the case $\tgmfd = \bbS^{2}$.

\begin{proposition} \label{prop:strong-stab}
Let $Q$ be an equivariant finite energy harmonic map from $\bbH^{2}$ into $\tgmfd = \bbH^{2}$ or $\bbS^{2}$, which is given by Proposition~\ref{prop:hm-classify}. Then the associated linearized operator $\calH_{Q}$ obeys the strong linearized stability condition when (1) $\tgmfd = \bbH^{2}$, or (2) $\tgmfd = \bbS^{2}$ and $\alp \geq 0$ is sufficiently small.
\end{proposition}
\begin{proof}
Let $H$ be given by \eqref{eq:lin-Q} and \eqref{eq:final_frame}. Observe that the $\bfH^{1}_{thr}$-norm for $\phi = e^{\infty} \varphi$ is equivalent to 
\begin{equation*}
	\nrm{\varphi}_{H^{1}_{thr}} = \nrm{(\rd_{r} + \frac{1}{2}) \varphi}_{L^{2}} + \nrm{\frac{1}{\sinh r} \rd_{\tht} \varphi}_{L^{2}} + \nrm{\frac{1}{(1+r^{2})^{\frac{1}{2}}} \varphi}_{L^{2}},
\end{equation*}
since $A_{r} = 0$ and $e^{r} A_{\tht} \in L^{\infty}$. Accordingly, the $\bfH^{-1}$-norm for $g = e^{\infty} G$ is equivalent to the dual norm $H^{-1}_{thr}$ for $G$. It follows that Definition~\ref{def:strong-stab} is equivalent to
\begin{align} 
	\int (H \varphi, \varphi) &\geq \frac{1}{4} \nrm{\varphi}_{L^{2}}^{2}, \label{eq:strong-stab-cf-1} \\
	\nrm{\varphi}_{H^{1}_{thr}} &\aleq \nrm{(\frac{1}{4} - H) \varphi}_{H^{-1}_{thr}}. \label{eq:strong-stab-cf-2}
\end{align}
When $\tgmfd = \bbH^{2}$, so that $\tcv = -1$, \eqref{eq:strong-stab-cf-1} and \eqref{eq:strong-stab-cf-2} follow from the diamagnetic inequality and the positivity property $-\tcv \abs{\psi^{\infty}_{2}}^{2} \geq 0$; see \cite[Corollary~1.14]{LLOS1} for details. When $\tgmfd = \bbS^{2}$, \eqref{eq:strong-stab-cf-1} and \eqref{eq:strong-stab-cf-2} obviously holds for $\alp = 0$, in which case $H = - \lap$. Clearly, \eqref{eq:strong-stab-cf-2} holds for $0 < \alp \ll 1$ by treating $H-(-\lap)$ as a perturbation (see Lemma~\ref{lem:cf}). Next, while \eqref{eq:strong-stab-cf-1} itself is not stable, note that 
\begin{equation*}
	\int ((-(\nb + i A^{\infty})^{\ell} (\nb + i A^{\infty})_{\ell} - \frac{1}{4} )\varphi,  \varphi)
	\geq  \int (\nb \abs{\varphi}, \nb \abs{\varphi}) - \frac{1}{4} (\varphi, \varphi)
	\ageq \nrm{(1+r^{2})^{-\frac{1}{2}} \varphi}_{L^{2}}^{2},
\end{equation*}
where we used the diamagnetic inequality $\abs{\nb \abs{\varphi}} \leq \abs{(\nb + i A^{\infty}) \varphi}$ (in the sense of distributions) and the Hardy--Poincar\'e inequality \cite[Lemma~5.2 with $w = 1$]{LLOS1}. Note that $H = - (\nb + i A^{\infty})^{\ell} (\nb + i A^{\infty})_{\ell} - \abs{\psi_{2}^{\infty}}^{2}$ where $\abs{\psi_{2}^{\infty}}^{2} = O(e^{-2r} \alp^{2})$. Thus, $\int ((H-\frac{1}{4}) \varphi, \varphi) \ageq \nrm{(1+r^{2})^{-\frac{1}{2}} \varphi}_{L^{2}}^{2}$ for $0 < \alp \ll 1$, which implies \eqref{eq:strong-stab-cf-1}. \qedhere
\end{proof}

\begin{remark} \label{rem:cf-nonimportance}
While we utilize the explicit Coulomb frame \eqref{eq:final_frame} to simplify the arguments in this paper, we believe that neither the precise formulas in \eqref{eq:cf-coeff} nor the Coulomb gauge condition \eqref{eq:cf-coulomb} are essential in our argument. What is important are the symmetry properties \eqref{eq:LOmA} and \eqref{eq:LOmpsi}, as well as the smoothness and decay properties of $A^{\infty}_{j}$ and $\psi^{\infty}_{j}$ (in particular, $e^{r} A^{\infty}_{j}, e^{r} \psi^{\infty}_{j} \in L^{\infty}$).
\end{remark}

\subsection{Parabolic theory for the linear operator $\rd_{s} + H$} \label{subsec:lin-para}
The purpose of this subsection is to develop the basic parabolic theory for the linear operator $\rd_{s} + H$, where $H$ was defined in \eqref{eq:lin-Q}. To clarify the dependence of constants, in this subsection we work under slightly more general assumptions on $H$. Let $H$ be a linear operator on scalar functions on $\bbH^{2}$ of the form
\begin{equation} \label{eq:H-form}
	H = -(\nb^{\ell} + i A^{\ell}) (\nb_{\ell} + i A_{\ell}) + V, \quad \hbox{ where $A_{j}$ and $V$ are real-valued},
\end{equation}
where we assume that $A_{x}, \nb^{\ell} A_{\ell}, V \in L^{2} \cap L^{\infty}$. Clearly, the linearized operator $H$ \eqref{eq:lin-Q} in the Coulomb gauge obeys this assumption (see Lemma~\ref{lem:cf}). We will often refer to $A_{j}$ and $V$ as magnetic and electric potentials, respectively.

Under the above conditions, it is not difficult to show that the linear parabolic equation
\begin{equation} \label{eq:lin-para}
	(\rd_{s} + H) \varphi = 0.
\end{equation}
is well-posed for $\varphi \in L^{p}$ for any $p \in [1, \infty]$. We denote by $e^{-s H}$ the solution operator for \eqref{eq:lin-para}. The main result of this subsection is the following long-time $L^{p}$ bounds for $e^{- s H}$:
\begin{lemma}  \label{l:lin-para-Lp} 
Let $H$ be of the form \eqref{eq:H-form} with $A_{x} \in L^{2} \cap L^{\infty}$ and $\rd^{\ell} A_{\ell}, V \in L^{2} \cap L^{\infty}$. Assume that
\begin{equation} \label{eq:H-spec-gap}
	\int_{\bbH^{2}} (H \varphi, \varphi) \geq \rho^{2} \int_{\bbH^{2}} (\varphi, \varphi)
\end{equation}
for some $\rho^{2} > 0$ and for all smooth and compactly supported $\varphi$. Then for any $p$ and $\rho_{0}$ such that
\begin{equation*}
1 < p < \infty, \quad 
0 \leq \rho_{0}^{2} < 2 \rho^{2} \min\set{\frac{1}{p}, 1 - \frac{1}{p}},
\end{equation*}
we have, for all $f \in L^{p}$ and $s \geq 0$,
\begin{equation} \label{eq:lin-para-Lp}
	\nrm{s H e^{-s H} f}_{L^{p}}
	+ \nrm{e^{-s H} f}_{L^{p}} \aleq e^{- \rho_{0}^{2} s} \nrm{f}_{L^{p}},
\end{equation}
where the implicit constant depends on $\nrm{A}_{L^{2} \cap L^{\infty}}$, $\nrm{V}_{L^{2} \cap L^{\infty}}$, $p$ and $\rho_{0}^{2}$.
\end{lemma} 
Lemma~\ref{l:lin-para-Lp} is applicable to the linearized operator $H$ in the Coulomb gauge introduced in Subsection~\ref{subsec:hm} (see Lemma~\ref{lem:cf} and Proposition~\ref{prop:weak-stab}).  Before we prove Lemma~\ref{l:lin-para-Lp}, we first give some corollaries.

\begin{corollary} \label{c:spec-gap-Lp-H}
Let $H$ satisfy the hypotheses of Lemma~\ref{l:lin-para-Lp}. For any $1 < p < \infty$, if $f \in W^{2, p}$, then
\begin{equation} \label{eq:spec-gap-Lp-H}
	\nrm{f}_{L^{p}} \aleq \nrm{H f}_{L^{p}}.
\end{equation}
where the implicit constant depends on $\nrm{A}_{L^{2} \cap L^{\infty}}$, $\nrm{V}_{L^{2} \cap L^{\infty}}$, $p$ and $\rho^{2}$.
\end{corollary}
The proof of Corollary~\ref{c:spec-gap-Lp-H} is analogous to that of Corollary~\ref{c:spec-gap-Lp}, where $H$ and Lemma~\ref{l:lin-para-Lp} are used in lieu of $-\lap$ and Lemma~\ref{l:lin-heat-Lp-Lq}, respectively.

\begin{corollary}  \label{c:lin-para-Lp} 
Let $H$ satisfy the hypotheses of Lemma~\ref{l:lin-para-Lp}. For any $1 < p < \infty$ we have 
\EQ{
 \nrm{e^{-s H} f}_{L^{p}} &\aleq \nrm{f}_{L^{p}} \\
 \nrm{s^{\frac{1}{2}} \nb e^{-s H} f}_{L^{p}} &\aleq \|f \|_{L^p}  \\
 \nrm{s \lap e^{-s H} f}_{L^{p}} &\aleq \| f \|_{L^p} 
}
where the implicit constants depend on $\nrm{A}_{L^{2} \cap L^{\infty}}$, $\nrm{V}_{L^{2} \cap L^{\infty}}$ and $p$.
\end{corollary} 
\begin{proof}[Proof of Corollary~\ref{c:lin-para-Lp}] 
This is a direct consequence of Lemma~\ref{l:lin-para-Lp} together with bound
\begin{equation*}
	\nrm{\lap f}_{L^{p}} \aleq \nrm{H f}_{L^{p}},
\end{equation*}
which is a consequence of \eqref{eq:spec-gap-Lp-H} and the $L^p$ elliptic regularity theory for  $H$. Note that for this purpose, it suffices to use the assumptions $A_{j} \in L^{\infty}$ and $\rd^{\ell} A_{\ell}, V \in L^{\infty}$.
\end{proof}

\begin{proof}[Proof of Lemma~\ref{l:lin-para-Lp}]
By density, it suffices to consider a smooth and compactly supported $f$, for which all the formal manipulations below are easily justified.

\pfstep{Step~1. $L^{p}$ bounds for $H = - \lap$}
For the free inhomogeneous equation
\begin{equation*}
	(\rd_{s} - \lap) \tilde{\varphi} = \tilde{F}^{0} + \rd_{\ell} \tilde{F}^{\ell}, \quad \tilde{\varphi}(s=0) = \tilde{f},
\end{equation*}
we claim that the following $L^{p}$ bounds hold for any $2 \leq p < \infty$\footnote{Alternatively, the RHS of \eqref{eq:lin-heat-Lp} may be written as $\nrm{\tilde{f}}_{L^{p}} + \nrm{F}_{L^{1}_{s} (J; L^{p}) + L^{2}_{s}(J; W^{-1, p})}$, which is analogous to \eqref{eq:lin-heat-L2}.}: On any interval $J$ of the form $J = (0, s_{1}]$, 
\begin{align}
	\nrm{\tilde{\varphi}}_{L^{\infty}_{s}(J; L^{p})} 
	& \aleq \nrm{\tilde{f}}_{L^{p}} 
	+ \nrm{\tilde{F}^{0}}_{L^{1}_{s} (J; L^{p})} + \nrm{\tilde{F}^{x}}_{L^{2}_{s} (J; L^{p})}, \label{eq:lin-heat-Lp} \\
	\nrm{\tilde{\varphi}}_{L^{\infty}_{s}(J; L^{p})} 
	+ \nrm{s \rd_{s} \tilde{\varphi}}_{L^{\infty}(J; L^{p})}
	&\aleq \nrm{\tilde{f}}_{L^{p}} 
	+ \nrm{(\tilde{F}^{0}, s \rd_{s} \tilde{F}^{0})}_{L^{1}_{s} (J; L^{p})} + \nrm{(\tilde{F}^{x}, s \rd_{s} \tilde{F}^{x})}_{L^{2}_{s} (J; L^{p})}. \label{eq:lin-heat-Lp-sds} 
\end{align}

Let us give a detailed proof of \eqref{eq:lin-heat-Lp-sds}, which is more involved. Our proof is a small variant of that of \cite[Lemma~2.11]{LOS5}. By interpolation, it suffices to verify \eqref{eq:lin-heat-Lp-sds} for $p = 2k$ for each positive integer $k$. Computing $(\rd_{s} - \lap) \abs{\tilde{\varphi}}^{2}$, we obtain
\begin{equation} \label{eq:lin-heat-varphi-sq}
	\frac{1}{2} \rd_{s} \abs{\tilde{\varphi}}^{2} - \frac{1}{2} \lap \abs{\tilde{\varphi}}^{2} + (\rd_{\ell} \tilde{\varphi}, \rd^{\ell} \tilde{\varphi})  = (\tilde{F}^{0}, \tilde{\varphi}) + \rd_{\ell} (\tilde{F}^{\ell}, \tilde{\varphi}) - (\tilde{F}^{\ell}, \rd_{\ell} \tilde{\varphi}).
\end{equation}
Similarly, computing $(\rd_{s} - \lap) s^{2} \abs{\rd_{s} \tilde{\varphi}}^{2}$, and using the fact that $\rd_{s}$ commutes with $(\rd_{s} - \lap)$, we obtain
\begin{align*}
&	\frac{1}{2} \rd_{s} \abs{s \rd_{s} \tilde{\varphi}}^{2} - \lap  \abs{s \rd_{s} \tilde{\varphi}}^{2} + (\rd_{\ell} s \rd_{s} \tilde{\varphi}, \rd^{\ell} s \rd_{s} \tilde{\varphi}) - (\rd_{s} \tilde{\varphi}, s \rd_{s} \tilde{\varphi}) \\
& = (s \rd_{s} \tilde{F}^{0}, s \rd_{s} \tilde{\varphi}) + \rd_{\ell} (s \rd_{s} \tilde{F}^{\ell}, s \rd_{s} \tilde{\varphi}) - (s \rd_{s} \tilde{F}^{\ell}, \rd_{\ell} s \rd_{s} \tilde{\varphi})  
\end{align*}
We introduce $\Phi = \sqrt{\abs{\tilde{\varphi}}^{2} + \dlt^{2} \abs{s \rd_{s} \tilde{\varphi}}^{2}}$, where $\dlt > 0$ is a small parameter to be fixed later. By the preceding two identities, we obtain
\begin{equation} \label{eq:lin-heat-Phi}
\begin{aligned}
	& \frac{1}{2} \rd_{s} \Phi^{2} - \frac{1}{2} \lap \Phi^{2} 
	+ (\rd_{\ell} \tilde{\varphi}, \rd^{\ell} \tilde{\varphi}) 
	+ \dlt^{2} (\rd_{\ell} s \rd_{s} \tilde{\varphi}, \rd^{\ell} s \rd_{s} \tilde{\varphi}) 
	- \dlt (\rd_{s} \tilde{\varphi}, \dlt s \rd_{s} \tilde{\varphi}) \\
	& = 
	(\tilde{F}^{0}, \tilde{\varphi}) + \rd_{\ell} (\tilde{F}^{\ell}, \tilde{\varphi}) - (\tilde{F}^{\ell}, \rd_{\ell} \tilde{\varphi})  
	+ \dlt (s \rd_{s} \tilde{F}^{0}, \dlt s \rd_{s} \tilde{\varphi}) + \dlt \rd_{\ell} (s \rd_{s} \tilde{F}^{\ell}, \dlt s \rd_{s} \tilde{\varphi}) - \dlt (s \rd_{s} \tilde{F}^{\ell}, \dlt \rd_{\ell} s \rd_{s} \tilde{\varphi})  
\end{aligned}
\end{equation}
We multiply \eqref{eq:lin-heat-Phi} by $\Phi^{2k-2}$ and integrate over $(s_{0}, s_{1}] \times \bbH^{2}$. For the first four terms on the LHS, we have
\begin{align*}
& 
\int_{s_{0}}^{s_{1}} \int \left( \frac{1}{2} \partial_s \Phi^{2} - \frac{1}{2} \lap \Phi^{2} 
	+ (\rd_{\ell} \tilde{\varphi}, \rd^{\ell} \tilde{\varphi}) 
	+ \dlt^{2} (\rd_{\ell} s \rd_{s} \tilde{\varphi}, \rd^{\ell} s \rd_{s} \tilde{\varphi}) \right) \Phi^{2k-2} \, \ud s \\
&	= 
	\frac{1}{2k} \left( \int \Phi^{2k}(s_{1}) - \int \Phi^{2k}(s_{0}) \right) 
	+ \frac{k-1}{2} \int_{s_{0}}^{s_{1}} \int \rd_{\ell} \Phi^{2} \rd^{\ell} \Phi^{2} \Phi^{2k-4} \\
& \phantom{=}
	+ \int_{s_{0}}^{s_{1}} \int \left((\rd_{\ell} \tilde{\varphi}, \rd^{\ell} \tilde{\varphi}) + \dlt^{2} (\rd_{\ell} s \rd_{s} \tilde{\varphi}, \rd^{\ell} s \rd_{s} \tilde{\varphi}) \right) \Phi^{2k-2} \, \ud s,
\end{align*}
where we note that all terms on the RHS are nonnegative except $- \frac{1}{2k} \int \Phi^{2k}(s_{0})$. For the contribution of the last term on the LHS of \eqref{eq:lin-heat-Phi}, by integration by parts and Cauchy--Schwarz, we have
\begin{align*}
& \int_{s_{0}}^{s_{1}} \int \dlt (\rd_{s} \tilde{\varphi}, \dlt s \rd_{s} \tilde{\varphi}) \Phi^{2k-2} \, \ud s -\int_{s_0}^{s_1}\int \dlt (\tilde{F}^{0} + \rd_{\ell} \tilde{F}^{\ell},\dlt s\rd_{s}\tilde{\varphi})\Phi^{2k-2}\,\ud s \\
& = - \int_{s_{0}}^{s_{1}} \int \dlt (\rd_{\ell} \tilde{\varphi}, \dlt \rd^{\ell} s \rd_{s} \tilde{\varphi}) \Phi^{2k-2} \, \ud s
- (k-1) \int_{s_{0}}^{s_{1}} \int \dlt (\rd_{\ell} \tilde{\varphi}, \dlt s \rd_{s} \tilde{\varphi}) \rd^{\ell} \Phi^{2} \Phi^{2k-4} \, \ud s \\
& \leq \frac{k}{2} \dlt  \int_{s_{0}}^{s_{1}} \int (\rd_{\ell} \tilde{\varphi}, \rd^{\ell} \tilde{\varphi}) \Phi^{2k-2} \, \ud s
+ \frac{1}{2} \dlt \int_{s_{0}}^{s_{1}} \int \dlt^{2} (\rd_{\ell} s \rd_{s} \tilde{\varphi}, \rd^{\ell} s \rd_{s} \tilde{\varphi})  \Phi^{2k-2} \, \ud s
+ \frac{k-1}{2} \dlt \int_{s_{0}}^{s_{1}} \int \rd_{\ell} \Phi^{2} \rd^{\ell} \Phi^{2} \Phi^{2k-4} \, \ud s,
\end{align*}
where we used $\abs{\dlt s \rd_{s} \tilde{\varphi}} \leq \Phi$ on the last line. Therefore, taking $\dlt < \min \set{ \frac{1}{2}, \frac{1}{k-1}}$, we arrive at the coercivity bound
\begin{equation} \label{eq:lin-heat-Phi-coercive}
\begin{aligned}
&	\sup_{s' \in J} \int_{0}^{s'} \int (\hbox{LHS of }\eqref{eq:lin-heat-Phi}) \Phi^{2k-2} \, \ud s + \frac{1}{2k} \int \Phi^{2k}(0) \\
&	\geq 
	\frac{1}{2k} \nrm{\Phi}_{L^{\infty}(J; L^{2k})}^{2k}
	+ \frac{k-1}{4} \int_{J} \int \rd_{\ell} \Phi^{2} \rd^{\ell} \Phi^{2} \Phi^{2k-4} \, \ud s\\
&	\quad 
	+ \frac{1}{2} \int_{J} \int \left((\rd_{\ell} \tilde{\varphi}, \rd^{\ell} \tilde{\varphi}) + \dlt^{2} (\rd_{\ell} s \rd_{s} \tilde{\varphi}, \rd^{\ell} s \rd_{s} \tilde{\varphi}) \right) \Phi^{2k-2} \, \ud s \\
&       \quad - \Big| \int_J\int \dlt (\tilde{F}^{0} + \rd_{\ell} \tilde{F}^{\ell},\dlt s\rd_{s}\tilde{\varphi})\Phi^{2k-2}\,\ud s \Big|.
\end{aligned}
\end{equation}
On the other hand, for any $s' \in J$, the contribution of the RHS of \eqref{eq:lin-heat-Phi} can be controlled by the first two terms on the RHS of \eqref{eq:lin-heat-Phi-coercive} via integration by parts, H\"older's inequality and the simple bounds $\abs{\tilde{\varphi}}, \abs{\dlt s \rd_{s} \tilde{\varphi}} \leq~\Phi$:
\begin{align*}
	\abs{\int_{0}^{s'} \int (\tilde{F}^{0}, \tilde{\varphi}) \Phi^{2k-2} \, \ud s}
	& \leq \nrm{\tilde{F}^{0}}_{L^{1}_{s}(J; L^{2k})} \nrm{\Phi}^{2k-1}_{L^{\infty}(J; L^{2k})}, \\
	\abs{\int_{0}^{s'} \int \rd_{\ell} (\tilde{F}^{\ell}, \tilde{\varphi}) \Phi^{2k-2} \, \ud s}
	& = (k-1) \abs{\int_{0}^{s'} \int (\tilde{F}^{\ell}, \tilde{\varphi}) \rd_{\ell} \Phi^{2} \Phi^{2k-4} \, \ud s} \\
	& \leq (k-1) \nrm{\tilde{F}^{x}}_{L^{2}_{s}(J; L^{2k})} \nrm{\Phi}^{k-1}_{L^{\infty}(J; L^{2k})} \left(\int_{J} \int \rd_{\ell} \Phi^{2} \rd^{\ell} \Phi^{2} \Phi^{2k-4} \, \ud s \right)^{\frac{1}{2}},
\end{align*}
\begin{align*}
	\abs{\int_{0}^{s'} \int (\tilde{F}^{\ell}, \rd_{\ell} \tilde{\varphi}) \Phi^{2k-2} \, \ud s}
	& \leq \nrm{\tilde{F}^{x}}_{L^{2}_{s}(J; L^{2k})} \nrm{\Phi}_{L^{\infty}(J; L^{2k})}^{k-1} \left( \int_{J} \int (\rd_{\ell} \tilde{\varphi}, \rd^{\ell} \tilde{\varphi}) \Phi^{2k-2} \, \ud s\right)^{\frac{1}{2}}, \\
	\abs{\int_{0}^{s'} \int \dlt (s \rd_{s} \tilde{F}^{0}, \dlt s \rd_{s} \tilde{\varphi}) \Phi^{2k-2} \, \ud s}
	& \leq \dlt \nrm{\tilde{F}^{0}}_{L^{1}_{s}(J; L^{2k})} \nrm{\Phi}^{2k-1}_{L^{\infty}(J; L^{2k})}, \\
	\abs{\int_{0}^{s'} \int \dlt \rd_{\ell} (s \rd_{s} \tilde{F}^{\ell}, \dlt s \rd_{s} \tilde{\varphi}) \Phi^{2k-2} \, \ud s}
	& = (k-1) \abs{\int_{0}^{s'} \int \dlt (s \rd_{s} \tilde{F}^{\ell}, \dlt s \rd_{s} \tilde{\varphi}) \rd_{\ell} \Phi^{2} \Phi^{2k-4} \, \ud s} \\
	& \leq \dlt (k-1) \nrm{s \rd_{s} \tilde{F}^{x}}_{L^{2}_{s}(J; L^{2k})} \nrm{\Phi}^{k-1}_{L^{\infty}(J; L^{2k})} \left(\int_{J} \int \rd_{\ell} \Phi^{2} \rd^{\ell} \Phi^{2} \Phi^{2k-4} \, \ud s \right)^{\frac{1}{2}}, \\
	\abs{\int_{0}^{s'} \int \dlt (s \rd_{s} \tilde{F}^{\ell}, \dlt \rd_{\ell} s \rd_{s} \tilde{\varphi}) \Phi^{2k-2} \, \ud s}
	& \leq \dlt \nrm{s \rd_{s} \tilde{F}^{x}}_{L^{2}_{s}(J; L^{2k})} \nrm{\Phi}_{L^{\infty}(J; L^{2k})}^{k-1} \left( \int_{J} \int \dlt^{2} (\rd_{\ell} s \rd_{s} \tilde{\varphi}, \rd^{\ell} \tilde{\varphi}) \Phi^{2k-2} \, \ud s\right)^{\frac{1}{2}}.
\end{align*}
Similarly, the last term on the RHS of \eqref{eq:lin-heat-Phi-coercive} can be bounded by the first two terms there. Putting these inequalities together, it follows that
\begin{equation*}
	\nrm{\Phi}_{L^{\infty}(J; L^{2k})} \leq \nrm{\Phi(0)}_{L^{2k}}
	+ C_{k} \left(\nrm{(\tilde{F}^{0}, s \rd_{s} \tilde{F}^{0})}_{L^{1}_{s}(J; L^{2k})}
	+ \nrm{(\tilde{F}^{x}, s \rd_{s} \tilde{F}^{x})}_{L^{2}_{s}(J; L^{2k})} \right).
\end{equation*}
Since $\nrm{\Phi(s)}_{L^{2k}} \geq \frac{1}{2}(\nrm{\varphi(s)}_{L^{2k}} + \dlt \nrm{s \rd_{s} \tilde{\varphi}(s)}_{L^{2k}})$ and $\Phi^{2}(0) = \tilde{f}^{2}$ by definition, \eqref{eq:lin-heat-Lp-sds} follows.

The case of \eqref{eq:lin-heat-Lp} is similar but only simpler, so we only sketch the proof. We multiply \eqref{eq:lin-heat-varphi-sq} by $\abs{\tilde{\varphi}}^{2k-2}$, integrate over $(0, s') \times \bbH^{2}$ and take the supremum over $s' \in J$. Arguing as above, we arrive at
\begin{align*}
	& \nrm{\tilde{\varphi}}^{2k}_{L^{\infty}(J; L^{2k})}
	+ 2k \int_{J} \int (\rd_{\ell} \tilde{\varphi}, \rd^{\ell} \tilde{\varphi}) \abs{\tilde{\varphi}}^{2k-2} \ud s + 2k (2k-2) \int_{J} \int (\rd_{\ell} \tilde{\varphi}, \tilde{\varphi}) (\rd^{\ell} \tilde{\varphi}, \tilde{\varphi}) \abs{\tilde{\varphi}}^{2k-4} \ud s \\
%	& \leq \int \abs{\tilde{f}}^{2k} + 2k \int_{J} \int (\tilde{F}^{0}, \tilde{\varphi}) \abs{\tilde{\varphi}}^{2k-2} \, \ud s
%	+ 2k \int_{J} \int ((\tilde{F}^{\ell}, \rd_{\ell} \tilde{\varphi}) \abs{\tilde{\varphi}}^{2} + (2k-2) (\tilde{F}^{\ell}, \tilde{\varphi})(\rd_{\ell} \tilde{\varphi}, \tilde{\varphi})) \abs{\tilde{\varphi}}^{2k-4} \, \ud s \\
	& \leq \int \abs{\tilde{f}}^{2k} + 2k \nrm{\tilde{F}^{0}}_{L^{1}(J; L^{2k})} \nrm{\tilde{\varphi}}_{L^\infty(J; L^{2k})}^{2k-1} 
		+ k (4k-2) \nrm{\tilde{F}^{x}}_{L^{2}_{s}(J; L^{2k})} \nrm{\tilde{\varphi}}_{L^{\infty}_{s}(J; L^{2k})}^{k-1} \left( \int_{J} \int (\rd_{\ell} \tilde{\varphi}, \rd^{\ell} \tilde{\varphi}) \abs{\tilde{\varphi}}^{2k-2} \right)^{\frac{1}{2}}
\end{align*}
It follows that
\begin{equation*}
	\nrm{\tilde{\varphi}}_{L^{\infty}_{s}(J; L^{2k})}
	\leq \nrm{\tilde{f}}_{L^{2k}} + C_{k} \left(\nrm{\tilde{F}^{0}}_{L^{1}_{s}(J; L^{2k})} + \nrm{\tilde{F}^{x}}_{L^{2}_{s}(J; L^{2k})} \right),
\end{equation*}
from which \eqref{eq:lin-heat-Lp} follows.

\pfstep{Step~2. $L^{2}$ bound for a general $H$}
In what follows, we abbreviate $\varphi (x, s) = e^{- s H} f(x)$. Moreover, we suppress the dependence of the implicit constants on $\nrm{A_{x}}_{L^{2} \cap L^{\infty}}$, $\nrm{\rd_{\ell} A^{\ell}}_{L^{2} \cap L^{\infty}}$ and $\nrm{V}_{L^{2} \cap L^{\infty}}$.

In this step, we establish the following bounds in the case $p = 2$:
\begin{align}
	\nrm{\varphi(s)}_{L^{2}} & \leq e^{- \rho^{2} s} \nrm{f}_{L^{2}}\quad \hbox{ for } s > 0, \label{eq:lin-para-L2-long} \\
	\nrm{s \rd_{s} \varphi(s)}_{L^{2}} & \aleq \nrm{f}_{L^{2}} \quad \hbox{ for } 0 < s \leq 1. \label{eq:lin-para-L2-sds}
\end{align}
The long time bound \eqref{eq:lin-para-L2-long} follows from solving the differential inequality
\begin{equation*}
	\frac{1}{2} \rd_{s} \int \abs{\varphi}^{2}(s) = - \int (H \varphi, \varphi) \leq - \rho^{2} \int \abs{\varphi}^{2},
\end{equation*}
where we used the equation $(\rd_{s} + H) \varphi = 0$ and \eqref{eq:H-spec-gap}. To prove \eqref{eq:lin-para-L2-sds}, we apply \eqref{eq:lin-heat-Lp-sds} with $p = 2$ to $(\rd_{s} - \lap ) \varphi = (-\lap - H)\varphi$. Then
\begin{equation*}
	\nrm{(\varphi, s \rd_{s} \varphi)}_{L^{\infty}(J; L^{2})}
	\aleq \nrm{f}_{L^{2}} + \nrm{A^{\ell} \varphi}_{L^{2}_{s}(J; L^{2})} + \nrm{(-i \rd^{\ell} A_{\ell} + A^{\ell} A_{\ell} + V) \varphi}_{L^{1}_{s}(J; L^{2})}.
\end{equation*}
Using the hypothesis $A_{x}, \rd^{\ell} A_{\ell}, V \in L^{\infty}$, the last two terms may be handled by Gronwall's inequality, so \eqref{eq:lin-para-L2-sds} follows.

\pfstep{Step~3. $L^{p}$ bound for a general $H$}
Next, we claim that
\begin{align}
	\nrm{\varphi(s)}_{L^{p}} & \aleq \nrm{f}_{L^{p}} \quad \hbox{ for } s > 0, \label{eq:lin-para-Lp-long} \\
	\nrm{s \rd_{s} \varphi(s)}_{L^{p}} & \aleq \nrm{f}_{L^{p}} \quad \hbox{ for } 0 < s \leq 1. \label{eq:lin-para-Lp-sds} 
\end{align}

First, consider the solution $\varphi$ to $(\rd_{s} + H) \varphi = F^{0} + \rd_{\ell} F^{\ell}$ with $\varphi(s=0) = f$. By applying \eqref{eq:lin-heat-Lp-sds} to $(\rd_{s} - \lap ) \varphi = (-\lap - H)\varphi + F^{0} + \rd_{\ell} F^{\ell}$ and then handling the contribution of $(-\lap - H) \varphi$ by Gronwall's inequality (as in the previous step), we obtain
\begin{align}
	\sup_{s \in (0, 1]} \nrm{\varphi(s)}_{L^{p}}
	& \aleq \nrm{f}_{L^{p}} 
	+ \nrm{F^{0}}_{L^{1}_{s}((0, 1]; L^{p})} + \nrm{F^{x}}_{L^{2}_{s}((0, 1]; L^{p})}, \label{eq:lin-para-Lp-short} \\
	\sup_{s \in (0, 1]} \nrm{s \rd_{s} \varphi(s)}_{L^{p}} 
	& \aleq \nrm{f}_{L^{p}} 
	+ \nrm{(F^{0}, s \rd_{s} F^{0})}_{L^{1}_{s} ((0, 1]; L^{p})} + \nrm{(F^{x}, s \rd_{s} F^{x})}_{L^{2}_{s} ((0, 1]; L^{p})}. \label{eq:lin-para-Lp-short-sds}
\end{align}
Thus, it only remains to establish \eqref{eq:lin-para-Lp-long} for $s > 1$.  Write $\varphi(s) = \tilde{\varphi}(s) + w^{near; s}(s) + w^{far; s}(s)$, where
\begin{align*}
	(\rd_{s'} + H) w^{near;s} (s', x) & = - 1_{(s-1, s]} (s')(-\lap - H) \tilde{\varphi}(s', x), \quad w^{near; s}(s'=0) = 0, \\
	(\rd_{s'} + H) w^{far;s} (s', x) & = - 1_{(0, s-1]} (s')(-\lap - H) \tilde{\varphi}(s', x), \quad w^{far; s}(s'=0) = 0,
\end{align*}
and $\tilde{\fy}$ is as in Step 1 but with zero RHS and initial data $f$. By \eqref{eq:lin-heat-Lp}, it follows that
\begin{equation} \label{eq:lin-para-Lp-free}
	\nrm{\tilde{\varphi}(s')}_{L^{p}} \leq \nrm{f}_{L^{p}} \quad \hbox{ for any } s' > 0.
\end{equation}
For $w^{near; s}$x, by \eqref{eq:lin-para-Lp-short} and \eqref{eq:lin-para-Lp-free}, we have
\begin{align*}
	\nrm{w^{near; s}(s)}_{L^{p}} 
	\aleq \nrm{ A \tilde{\varphi} }_{L^{2}_{s'}((s-1, s]; L^{p})}
	+ \nrm{(- i (\rd_{\ell} A^{\ell}) + A_{\ell} A^{\ell} + V)\tilde{\varphi})}_{L^{1}_{s'}((s-1, s]; L^{p})} 
	\aleq \nrm{f}_{L^{p}}.
\end{align*}
where we used $A^{x}, \rd_{\ell} A^{\ell}, V \in L^{\infty}$. Next, to estimate $w^{far; s}$, first note that $e^{\frac{1}{2} H}$, $\nb e^{\frac{1}{2} H}$ and $e^{\frac{1}{2} H} \nb$ are all bounded from $L^{2}$ into $L^{2}$. Indeed, the first case follows from \eqref{eq:lin-para-L2-long}, the second from \eqref{eq:lin-para-L2-sds}, $L^{2}$ elliptic regularity theory for $H$ and interpolation (note that only the assumptions $A_{j}, \rd_{\ell} A^{\ell}, V \in L^{\infty}$ are used), and third by duality from the second. In particular, by the Gagliardo--Nirenberg inequality (Lemma~\ref{lem:gagliardo_nirenberg}), it follows that $\nrm{e^{\frac{1}{2} H} g}_{L^{p}} \aleq \nrm{g}_{L^{2}}$ for any $2 \leq p < \infty$. Thus, we have
\begin{align*}
	\nrm{w^{far;s}(s)}_{L^{p}}
	&\aleq \int_{0}^{s-1} \nrm{e^{-\frac{1}{2} H} e^{- (s-1-s') H} e^{-\frac{1}{2} H} (-\lap - H)  \tilde{\varphi}(s')}_{L^{p}} \, \ud s' \\
	& \aleq \int_{0}^{s-1} \nrm{e^{- (s-1-s') H} e^{-\frac{1}{2} H} (-\lap - H)  \tilde{\varphi}(s')}_{L^{2}} \, \ud s' \\
	& \aleq \int_{0}^{s-1} e^{-\rho^{2} (s-1-s')}  \nrm{e^{- \frac{1}{2}H} (-\lap - H) \tilde{\varphi}(s')}_{L^{2}}  \, \ud s' \\
	& \aleq \sup_{s' \in (0, s-1]} \left( \nrm{- 2i e^{- \frac{1}{2}H} \rd_{\ell} (A^{\ell} \tilde{\varphi}(s'))}_{L^{2}} + \nrm{e^{-\frac{1}{2} H}(- i \rd_{\ell} A^{\ell} + A_{\ell} A^{\ell} + V) \tilde{\varphi}(s')}_{L^{2}}\right) \\
	& \aleq \sup_{s' \in (0, s-1]} \left( \nrm{A^{x} \tilde{\varphi}(s')}_{L^{2}} + \nrm{(- i \rd_{\ell} A^{\ell} + A_{\ell} A^{\ell} + V) \tilde{\varphi}(s')}_{L^{2}} \right).
\end{align*}
Combining the assumptions $A_{j}, \rd^{\ell} A_{\ell}, V \in L^{2} \cap L^{\infty}$ with \eqref{eq:lin-para-Lp-free}, the last line is bounded by $\aleq \nrm{f}_{L^{p}}$ as desired.

\pfstep{Step~4. Completion of proof}
Let $p > 2$.  By \eqref{eq:lin-para-Lp-long} and \eqref{eq:lin-para-Lp-sds}, it only remains to establish \eqref{eq:lin-para-Lp} for $s \geq 1$. For $\dlt > 0$, define $q = \frac{2 (1+\dlt-\frac{2}{p})}{\dlt}$. Interpolating the $L^{2}$ bound and the $L^{q}$ bound from Step~3,
\begin{equation*}
	\nrm{\varphi(s)}_{L^{p}} \aleq e^{- (\frac{2}{p} - \dlt) \rho^{2} s} \nrm{f}_{L^{p}}
\end{equation*}
Moreover, 
\begin{equation*}
	\nrm{s \rd_{s} \varphi(s)}_{L^{p}}
	\aleq s e^{- (\frac{2}{p} - \dlt) \rho^{2} (s-1)}\nrm{\rd_{s} \varphi(1)}_{L^{p}} \aleq s e^{- (\frac{2}{p} - \dlt) \rho^{2} s} \nrm{f}_{L^{p}}.
\end{equation*}
Choosing $\dlt$ so that $\rho_{0}^{2} < (\frac{2}{p} - \dlt) \rho^{2}$, \eqref{eq:lin-para-Lp} follows for $p > 2$. Finally, the remaining case $1 < p < 2$ follows by duality. \qedhere
\end{proof}

%\subsection{Dispersive theory for the linear Schr\"odinger equation on~$\Hp^2$}  
\subsection{Dispersive theory for the linear Schr\"odinger operator $- i \rd_{t} + H$}  \label{subsec:disp}
In this section we record a collection of linear dispersive estimates established in the companion paper~\cite{LLOS1} for solutions to the inhomogeneous Schr\"odinger equation 
\EQ{\label{eq:linS} 
(-i \partial_t  + H )u &=  F, \\
u(0) &= u_0 \in L^2(\Hp^2).
}
where $H$ is the linearized operator defined in~\eqref{eq:lin-Q} in the Coulomb gauge. We assume furthermore that $H$ is obtained by linearization about a harmonic map $Q$ satisfying the strong linearized stability condition, Definition~\ref{def:strong-stab}; see Proposition~\ref{prop:strong-stab}. 

For each $\ell \in \bbZ$, the dyadic spatial annulus $A_\ell$ is defined by
\begin{align}\label{eq:rAldef}
 A_\ell:=\{2^{\ell}\leq r\leq 2^{\ell+1}\},
\end{align} 
where $o$ is a fixed origin in $\bbH^{2}$ and $r(x) = \bfd_{\bbH^{2}}(o, x)$. Similarly the ball $A_{\leq\ell}$ and $A_{\geq\ell}$ are defined by
\begin{align*}
\begin{split}
A_{\leq\ell} := \{r\leq 2^{\ell}\},\qquad A_{\geq \ell}:=\{r\geq 2^{\ell}\}.
\end{split}
\end{align*}
Following~\cite{LLOS1} we then define the local smoothing space $LE$ via the norm, 
\EQ{
\|u\|_{LE(I)}^2 &:= \int_0^{\frac{1}{2}} \sigma^{-\frac{1}{2}} \| P_\sigma u \|_{LE_\sigma(I)}^2 \, \frac{\ud \sigma}{\sigma} + \int_{\frac{1}{8}}^4 \| P_{\geq \sigma} u \|_{LE_\low(I)}^2 \, \frac{\ud \sigma}{\sigma},
}
where $LE_\s$ and $LE_{\low}$ are given by 
\EQ{
 \|v\|_{LE_\sigma(I)} &:= \sigma^{-\frac{1}{4}} \|v\|_{L^2_t L^2_x(I \times \Ann_{\leq -k_\sigma})} + \sup_{-k_\sigma \leq \ell < 0} 2^{-\frac{1}{2} \ell} \|v\|_{L^2_t L^2_x(I \times \Ann_\ell)} + \|r^{-2} v\|_{L^2_t L^2_x(I\times \Ann_{\geq 0})}, \\
 \|v\|_{LE_\low} &:= \| \langle r \rangle^{-2} v \|_{L^2_t L^2_x(I \times \bbH^2)}.
} 
The Littlewood-Paley projections $P_{\geq \sgm}$ and $P_{\sgm}$ were defined in \eqref{eq:LP-proj}. For any heat time $\s > 0$ we use the notation $k_{\s}$ to denote the corresponding dyadic frequency
\begin{align*}
\begin{split}
 k_{\s} := \lfloor \log_2 ( \s^{-\frac{1}{2}} ) \rfloor .
\end{split}
\end{align*}
The norm $\|\cdot\|_{LE^*}$ is given by 
\EQ{
 \|F\|_{LE^\ast}^2 &:= \int_{\frac{1}{8}}^4\|P_{\geq \s} F\|_{LE_{\low}^\ast}^2\,\frac{\ud \sigma}{\sigma}+\int_{0}^{\frac{1}{2}}\s^{\frac{1}{2}}\|P_{\s} F\|_{LE_{\s}^\ast}^2\,\frac{\ud \sigma}{\sigma}.
}
where 
\begin{align*}
 \|G\|_{LE_\low^\ast} &:= \|\jap{r}^{2}G\|_{L^2_{t} L^{2}_{x} (\bbR\times \bbH^2)},\\
 \|G\|_{LE_{\s}^\ast} &:= \s^{\frac{1}{4}}\|G\|_{L^2_{t} L^{2}_{x}(\bbR\times A_{\leq-k_{\s}})}+\sum_{-k_{\s}\leq\ell < 0}\|r^{\frac{1}{2}}G\|_{L^2_{t} L^{2}_{x}(\bbR\times A_\ell)}+\|r^{2}G\|_{L^2(\bbR\times A_{\geq0})}.
\end{align*} 
Note that the norms $LE, LE^*$ correspond to the notation $\mathbb{LE}, \mathbb{LE}^*$ in~\cite{LLOS1}. 

We can now state the global-in-time local smoothing estimate. 
\begin{proposition}[{Local smoothing estimate \cite[Corollary 1.18]{LLOS1}}] \label{prop:strong-stab-led}   Let $u(t)$ denote the solution to~\eqref{eq:linS} for initial data $u_0 \in L^2(\Hp^2)$. Then, for any time interval $I \subset \R$, 
\EQ{ \label{eq:smoothing} 
 \| u \|_{LE(I) } \lesssim  \| u_0 \|_{L^2} + \|F \|_{LE^*(I)}
}

\end{proposition} 

\begin{proof} This is one of the main results from~\cite{LLOS1}, which can be applied directly after verifying that the linearized operator~$H$ defined in~\eqref{eq:lin-Q} satisfies the hypothesis  of~\cite[Corollary 1.18]{LLOS1}. From Proposition~\ref{prop:strong-stab} we know that the harmonic maps in question satisfy  the strong linearized stability condition, Definition~\ref{def:strong-stab}. Then a straightforward computation reveals that the Coulomb frame conditions in Lemma~\ref{lem:cf} ensure that the norms $\bfH^{1}_{thr}$, $\bfH^{-1}_{thr}$ are equivalent to $H^{1}_{thr}, H^{-1}_{thr}$ defined in \cite[Proposition~1.13]{LLOS1} (see also the proof of Proposition~\ref{prop:strong-stab}), which are used to characterize the operators $H$ considered in~\cite[Corollary 1.18]{LLOS1}. Lastly, one must check that the magnetic and electric potentials in $H$ have sufficient decay so that conditions~\cite[equations (1.8) and (1.12)]{LLOS1} are satisfied. However, this is manifestly clear since the results in~\cite{LLOS1} require polynomial decay, whereas the coefficients in $H$ as defined in~\eqref{eq:lin-Q} exhibit exponential decay. We thus refer the reader to the companion paper~\cite{LLOS1} for further details. 
\end{proof}

It is also proved in~\cite{LLOS1} that the local smoothing estimate~\eqref{eq:smoothing} can be combined with Strichartz estimates for the free Schr\"odinger equation on~$\Hp^d$ proved in~\cite{AP09, B07, IS} to deduce Strichartz estimates for~\eqref{eq:linS}. 

\begin{definition} Let $d \ge 2$. We say that a pair $(p, q)$ of real numbers is $S$-admissible if 
\EQ{
\Big\{(\frac{1}{p},\frac{1}{q})\in(0,\frac{1}{2})\times(0,\frac{1}{2})\,\big\vert \,\frac{2}{p}\geq \frac{d}{2}-\frac{d}{q}\Big\}\cup\Big\{(\frac{1}{p},\frac{1}{q})=(0,\frac{1}{2})\Big\},
}
With this notation we define the Strichartz norm: 
\EQ{
 \|v \|_{\Str(I)} :=   \sup_{(p, q) \, \, \textrm{S-admissible}} \| v \|_{L^p_t(I; L^q_x(\Hp^d))}.
}
\end{definition}

\begin{lemma}[Strichartz estimates] \label{l:strich} 
Let $u(t)$ be the solution to~\eqref{eq:linS} with initial data $u(0) = u_0 \in L^2(\Hp^2)$. Then, for any pair $(p, q)$ that is $S$-admissible (with $d = 2$) and any time interval $I$, we have 
\EQ{
\|u \|_{\Str(I) } \lesssim \| v_0 \|_{L^2} + \|F \|_{L^{p'}_tL^{q'}_x(I \times \Hp^2)}. 
}
\end{lemma} 

The main linear ingredient in our analysis of the Schr\"odinger maps system is the following estimate from~\cite{LLOS1} that puts together the Strichartz estimate and the local smoothing estimate in a combined functional framework. This is a consequence of Proposition~\ref{prop:strong-stab-led}, Lemma~\ref{l:strich} and the Christ--Kiselev lemma~\cite{CK01}.

\begin{lemma}[{Main linear estimate \cite[Corollary 1.20]{LLOS1}}] \label{lem:main_linear_estimate}  
Let $u(t)$ be the solution to~\eqref{eq:linS} with initial data $u(0) = u_0 \in L^2(\Hp^2)$. Then 
 \begin{equation*}
  \|u\|_{LE(I) \cap \Str(I)} \lesssim \|u(0)\|_{L^2_x} + \|F\|_{LE^\ast(I) + L^{\frac{4}{3}}_t L^{\frac{4}{3}}_x(I)}.
 \end{equation*}
\end{lemma}

Lastly we introduce notation for the main dispersive norms used here. For some small fixed absolute constant $0 < \delta < 1$ and any $s > 0$ we set 
\[
 m(s) := \max \{ s^{-\delta}, 1 \}.
\]
Given an $ \R^+ \ni s$ dependent function $v(s) = v(t, x, s)$, and a $t-$time interval $I \subset \R$ we denote by $\calS_{s}(I)$ the norm 
\EQ{
  \| v(s) \|_{\calS_s(I)} := \bigl\| m(s) s^{\frac{1}{2}} v(s) \bigr\|_{LE(I) \cap \Str(I)} + \bigl\| m(s) s^{\frac{1}{2}} \Omega  v(s) \bigr\|_{LE(I) \cap \Str(I)}
}
where $\Om$ denotes the rotational vector field defined earlier. Then we set, 
\begin{equation*}
 \| v \|_{\calS(I)} := \| v(s) \|_{L^\infty_{\ds} \cap L^2_{\ds} (\bbR^+; \calS_s(I))}.
\end{equation*}

\subsection{Bernstein-type estimates} 

The aim of this subsection is to prove the following two ``Bernstein-type'' estimates for the spaces $LE^{\ast}_{\sgm}$ and $LE^{\ast}_{\low}$:

\begin{lemma}  \label{l:bern1} 
Let $\bs \xi$ be an arbitrary $(0, 1)$ or $(1, 0)$ tensor field on $I \times \Hp^2$.  Let $1 \le p \le 2$ and set $\al_p = \frac{2-p}{2p}$. Then, 
\EQ{
\|\sigma^{\frac{1}{4}} P_\s \na_\mu \bs \xi^\mu \|_{LE^*_\s(I)} \lesssim \s^{-\al_p}  \| \bs \xi \|_{L^2_t L^p_x(I \times A_{\le - k_\s})}  + \s^{-\frac{1}{4}} \sum_{-k_\s \le \ell < 0}  \| r^{\frac{1}{2}} \bs \xi \|_{L^2_t L^2_x(I \times A_\ell)} + \s^{-\frac{1}{4}} \| r^2 \bs \xi \|_{L^2_t L^2_x(I \times A_{\ge 0})} . 
}
\end{lemma} 

\begin{lemma}  \label{l:bern2} 
Let $\bs \xi$ be an arbitrary $(0, 1)$ or $(1, 0)$ tensor field on $I \times \Hp^2$.  Let $1 \le p \le 2$ and set $\al_p = \frac{2-p}{2p}$. Then, 
\EQ{
\| \s^{\frac{1}{2}} P_{\ge \s} \na_\mu \bs\xi^\mu  \|_{LE_{\low}^*(I)}  \lesssim  \s^{-\frac{1}{2} + \al} \| \bs\xi \|_{L^2_t L^p_x(I \times A_{\le 0})} + \| r^2 \bs \xi \|_{L^2_t L^2_x(I \times A_{\ge 0})}  . 
}
\end{lemma} 

The estimates in Lemma~\ref{l:bern1} and Lemma~\ref{l:bern2} are consequences of the definitions of the spaces $LE_\s^*$ and $LE^*_{\low}$ together with the following ``core" localized parabolic regularity estimate of the type proved in~\cite[Section 4]{LLOS1}.  Let $\phi \in C^\infty_0$ be a smooth bump function such that $\phi(r) = 1$ if $\frac{1}{2} \le r \le 2$ and $\supp \phi (r) \subset [1/4, 4]$.   For each integer $\ell \in\Z$ define 
\EQ{ \label{eq:phidef} 
\phi_\ell(r)  := \phi( r/ 2^\ell).
}
Next, given a $(r, q)$ tensor  field $\bs \xi$ on $\Hp^2$, for $k = r+q$ we denote by $\na^{(k)} \cdot \bs \xi$ any contraction
\EQ{
 \na_{\mu_{\s(1)}} \cdots \na_{\mu_{\s(r)}} \na^{\mu_{\tau(r+1)}}  \cdots \na^{\mu_{\tau(r+q)}}  \xi^{\mu_1 \cdots \mu_r}_{\mu_{r+1} \cdots \mu_{r+q}} 
}
where $\s, \tau$ are any perturbations of $\{1, \dots r \}$ and $\{r+1, \dots, r+q\}$. 

\begin{lemma}\label{l:core-l-preg}\emph{\cite[Corollary 4.8 and Corollary 4.9]{LLOS1} }Let $1 \le p \le 2$ and set $\al_p = \frac{2-p}{2p}$. Let $v$ be an arbitrary function on $\Hp^2$, and let $\bs \xi$ be an arbitrary $(r, q)$ tensor field on $\Hp^2$. Let $\s_0 >0$ be fixed. Then, for each $\s$ with $0 < \s < \s_0$, each $\ell, m \in \Z$ with $\abs{\ell - m} \ge 10$ and $\max\{ 2^\ell, 2^m\} \ge \s^{\frac{1}{2}}$, and each $k, N \in \N$ we have  
\begin{align} 
   \| \phi_\ell  \s^{\frac{k}{2}}\nabla^{(k)} e^{\s \De}(\phi_m v) \|_{L^2_x} \lesssim_{N, k} \sigma^{-\alpha_p} \bigl( \sigma^{\frac{1}{2}} 2^{- \max\{\ell, m\} } \bigr)^N \| \phi_m v \|_{L^p_x} \\
   \| \phi_\ell e^{\s \De}  \s^{\frac{k}{2}}\nabla^{(k)} \cdot (\phi_m \bs \xi) \|_{L^2_x} \lesssim_{N, k} \sigma^{-\alpha_p} \bigl( \sigma^{\frac{1}{2}} 2^{- \max\{\ell, m\} } \bigr)^N \| \phi_m \bs \xi  \|_{L^p_x} 
  \end{align} 
  where in the latter estimate we require that $k = r+ q$. 
\end{lemma} 

\begin{proof}[Proof of Lemma~\ref{l:core-l-preg}]
This follows from the exact same argument given in~\cite[Section 4.1, Corollary 4.8 and Corollary 4.9]{LLOS1}  with the only distinction being that here we apply Lemma~\ref{l:hk} in lieu of the $L^2$-based parabolic regularity estimates used in~\cite{LLOS1}, i.e.~\cite[Lemma 4.2]{LLOS1}. 
\end{proof} 

\begin{proof}[Proof of Lemma~\ref{l:bern1} and Lemma~\ref{l:bern2}]
These again are consequences of the definitions of the spaces $LE_\s^*$ and $LE^*_{\low}$ together with Lemma~\ref{l:core-l-preg}. While these precise estimates are not covered in this exact form in~\cite{LLOS1}, the argument follows the same outline as~\cite[Proof of Lemma 4.11]{LLOS1}, using now Lemma~\ref{l:core-l-preg} in lieu of~\cite[Corollary 4.8 and Corollary 4.9]{LLOS1}. We omit the proof and refer the reader to~\cite{LLOS1}. 
\end{proof} 

\section{The harmonic map heat flow and the caloric gauge} \label{sec:hmhf}
This section is devoted to the analysis of the harmonic map heat flow \eqref{eq:hmhf}, which sets the stage for the analysis of the Schr\"odinger maps evolution and thus the proof of Theorem~\ref{thm:main}. In particular, we prove the asymptotic stability of weakly stable harmonic maps under \eqref{eq:hmhf} and construct the caloric gauge with the Coulomb gauge at infinity (Definition~\ref{def:caloric-simple}).

In Subsection~\ref{subsec:hmhf-ext}, we lay out the basic conventions for the \emph{extrinsic formulation} for analyzing maps, which will be the setting in which we develop a well-posedness theory for the \eqref{eq:hmhf}. In Subsection~\ref{subsec:hmhf-lwp}, we establish local well-posedness results for the harmonic map heat flow in the extrinsic formulation \eqref{eq:hmhf-ex} and its linearization \eqref{eq:l-hmhf-ex}. In Subsection~\ref{subsec:hmhf-stab}, we prove the asymptotic stability of any weakly linearly stable harmonic map $Q$ under \eqref{eq:hmhf-ex} (Theorem~\ref{thm:hmhf-gwp}; see also Theorem~\ref{thm:hmhf-gwp-simple}). At the same time, we develop the corresponding global theory for the linearized equation \eqref{eq:l-hmhf-ex}. Based on these results, in Subsection~\ref{subsec:caloric}, we construct the caloric gauge with the Coulomb gauge at infinity (Definition~\ref{def:caloric-simple}) for the harmonic map heat flow development of maps close to $Q$. Next, in Subsection~\ref{subsec:forward-caloric}, we prove forward-in-$s$ bounds for the components $\psi_{s}, \psi_{j}$, as well as the connection $1$-form $A_{j}$, in the caloric gauge constructed in Subsection~\ref{subsec:caloric}. Finally, in Subsection~\ref{subsec:backward-caloric}, we give backward-in-$s$ bounds in the same gauge that relate $\psi_{s}$ for $s > 0$ with the map $u(s=0)$.

\subsection{Conventions on the extrinsic formulation} \label{subsec:hmhf-ext}
To formulate the local and global well-posedness theory for \eqref{eq:hmhf}, we rely on an isometric embedding of the target manifold as in our definitions of spaces for maps. This point of view is called the \emph{extrinsic formulation}.

We recall some of the basic notation and refer the reader to Section~\ref{sec:prelim} for more details. As before, we fix a bounded open set $\tg$ of $\tgmfd$ that contains the image of $Q$, and consider a modification $\tgmfd'$ of $\tgmfd$ outside $\tg$ that is a closed $2$-dimensional Riemann surface (i.e., an oriented $2$-dimensional Riemannian manifold). Then we fix an isometric embedding (as Riemannian manifolds) $\iota : \tgmfd' \hookrightarrow \bbR^{N}$ and denote its second fundamental form by $S$. We identify maps into $\tg$ and sections of the associated pull-back tangent bundles with $\bbR^{N}$-valued function via $\iota$. 

For the induced covariant derivative, we have the formula
\begin{equation} \label{eq:pullback-d}
	D_{\alp} \phi^{A} = \rd_{\alp} \phi^{A} + S^{A}_{BC}(u) \phi^{B} \rd_{\alp} u^{C}.
\end{equation}
We use the following index convention for the induced curvature tensor:
\begin{equation*}
	[D_{\alp}, D_{\bt}] \phi^{C} = \tensor{R}{_{AB}^{C}_{D}}(u) \phi^{D} \rd_{\alp} u^{A} \rd_{\bt} u^{B}.
\end{equation*}

For simplicity of exposition, {\bf we simply write $\tgmfd$ for $\tgmfd'$ in this section}; this abuse of notation is admissible since all maps under consideration would take values in $\tg$. Moreover, we generally {\bf suppress the dependence of constants on the harmonic map $Q$, the bounded open neighborhood $\tg$ and the isometric embedding $\tgmfd' \hookrightarrow \bbR^{N}$}.

\subsection{Local theory for \eqref{eq:hmhf} in the extrinsic formulation} \label{subsec:hmhf-lwp}

In the extrinsic formulation, \eqref{eq:hmhf} takes the form
\begin{equation} \label{eq:hmhf-ex}
	\rd_{s} u^{A} = \lap u^{A} + S^{A}_{BC}(u) \rd^{\ell} u^{B} \rd_{\ell} u^{C}.
\end{equation}
We remark that, to lighten the notation, we write $u$ and $u_{0}$ here instead of $U$ and $u$ as in Subsection~\ref{subsec:result}.

The main local well-posedness result for \eqref{eq:hmhf-ex} we use in this paper is as follows.
\begin{proposition} [Local well-posedness of the harmonic map heat flow] \label{prop:hmhf-lwp}
Consider the initial value problem for \eqref{eq:hmhf-ex} with $u(s=0) = u_{0} \in (H^{1} \cap C^{0})_{Q}(\bbH^{2}; \tgmfd)$ and $u_{0}(x) \in \tg$ for all $x \in \bbH^{2}$.
\begin{itemize} 
\item {\bf Local well-posedness in $H^{1} \cap C^{0}$.} There exists a unique solution $u(s) \in (H^{1} \cap C^{0})_{Q}(\bbH^{2}; \tgmfd)$ for $s \in J = [0, s_{0}]$, where $s_{0}$ is any positive number for which 
\begin{equation} \label{eq:hmhf-lwp-hyp}
\nrm{\nb e^{s \lap}(u_{0} - Q)}_{L^{2}_{s} ([0, s_{0}]; L^{\infty})} + \abs{s_{0}}^{\frac{1}{2}}
\end{equation}
is sufficiently small compared to an upper bound for $\nrm{u_{0} - Q}_{H^{1} \cap L^{\infty}}$. The solution $u$ obeys $u(x, s) \in \tg$ for all $(x, s) \in \bbH^{2} \times J$, as well as the a-priori estimate
\begin{align} 
	\nrm{u - Q}_{L^{\infty}_{s}(J; H^{1} \cap L^{\infty})} + \nrm{u - Q}_{L^{2}_{s}(J; H^{2})} + \nrm{\rd_{s}u}_{L^{2}_{s}(J; L^{2})} + \nrm{\nb (u - Q)}_{L^{2}_{s} (J; L^{\infty})} & \aleq \nrm{u_{0} - Q}_{H^{1} \cap L^{\infty}}. \label{eq:hmhf-lwp-1}
\end{align}
The solution is unique in the function space defined by the LHS of \eqref{eq:hmhf-lwp-1}.
The solution map is smooth from $H^{1} \cap L^{\infty}$ to this function space. 

In what follows, we exclusively work with solutions that are given by iteration of the above statement.

\item {\bf Continuation criterion.} If a solution $u : \bbH^{2} \times J \to \tgmfd$ satisfies
\begin{equation*}
	\nrm{\nb (u - Q)}_{L^{2}_{s} (J; L^{\infty})} + \abs{J}^{\frac{1}{2}} < \infty, \quad
	u(x, s) \in \tg \hbox{ for all } (x, s) \in \bbH^{2} \times J
\end{equation*}
then $u$ can be extended as a solution past the future endpoint of $J$.

\item {\bf Persistence of regularity.} For any $\sgm > 1$,
\begin{equation} \label{eq:hmhf-lwp-sgm}
	\nrm{u-Q}_{L^{\infty}_{s}(J; H^{\sgm})} + \nrm{u-Q}_{L^{2}_{s}(J; H^{\sgm+1})} + \nrm{\rd_{s} u}_{L^{2}_{s}(J; H^{\sgm-1})} \aleq \nrm{u_{0}-Q}_{H^{\sgm}}.
\end{equation}
where the implicit constant only depends on $\sgm$ and upper bounds for $\nrm{u_{0}-Q}_{H^{1} \cap L^{\infty}}$ and $\nrm{\nb (u-Q)}_{L^{2}_{s} (J; L^{\infty})} + \abs{J}^{\frac{1}{2}}$.
\item {\bf Parabolic smoothing.}
For any $\sgm \geq 1$ and $k = 1, 2, \ldots$,
\begin{equation} \label{eq:hmhf-lwp-k}
	\nrm{s^{\frac{k}{2}} (u-Q)}_{L^{\infty}_{\ds}(J; H^{\sgm+k})}
	+ \nrm{s^{\frac{k+1}{2}} (u-Q)}_{L^{2}_{\ds}(J; H^{\sgm+k+1})}  + \nrm{s^{\frac{k+1}{2}} \rd_{s} u}_{L^{2}_{\ds}(J; H^{\sgm+k-1})} 
	\aleq \nrm{u_{0} - Q}_{H^{\sgm}},
\end{equation}
where the implicit constant only depends on $\sgm$, $k$ and upper bounds for $\nrm{u_{0}-Q}_{H^{1} \cap L^{\infty}}$ and $\nrm{\nb (u-Q)}_{L^{2}_{s} (J; L^{\infty})} + \abs{J}^{\frac{1}{2}}$.
\end{itemize}
\end{proposition}

Next, given a harmonic map heat flow $u : \bbH^{2} \times J \to \tg$, we consider the associated (inhomogeneous) linearized harmonic map heat flow for sections $\phi$ and $f$ of $u^{\ast} T \tgmfd$ (i.e., $\phi(x, s), f(x, s) \in T_{u(x, s)} \tgmfd$ for every $(x, s) \in \bbH^{2} \times J$), which takes the form 
\begin{equation} \label{eq:l-hmhf-ex}
	D_{s} \phi^{A} - D^{\ell} D_{\ell} \phi^{A} - \tensor{R}{_{CD}^{A}_{B}}(u) \rd_{\ell} u^{B} \phi^{C} \rd^{\ell} u^{D} = f^{A}.
\end{equation}
We remind the reader that the covariant derivative $D_{\bfa}$ is given by
\begin{equation*}
	D_{\bfa} \phi^{A} = \rd_{\bfa} \phi^{A} + S^{A}_{BC}(u) \rd_{\bfa} u^{B} \phi^{C}.
\end{equation*}
Indeed, note that given an one-parameter family $\set{u(\lmb; x, s)}_{\lmb \in (-\eps, \eps)}$ of harmonic map heat flows, the differential $\phi(x, s) = \frac{\ud}{\ud \lmb} u (x, s) \vert_{\lmb = 0}$ obeys \eqref{eq:l-hmhf-ex} with $f = 0$.

\begin{proposition} [Local theory for the linearized equation in the extrinsic formulation] \label{prop:l-hmhf-ex}
Let $u$ be a solution to \eqref{eq:hmhf-ex} given by Proposition~\ref{prop:hmhf-lwp} with $u(s=0) = u_{0} \in (H^{1} \cap C^{0})_{Q}(\bbH^{2}; \tg)$. 
\begin{itemize}
\item {\bf Local theory in $H^{\sgm}$ for $0 \leq \sgm < 1$.} Let $J =[0, T]$ be a compact interval on which $u$ exists. Then for any $\phi_{0} \in u_{0}^{\ast} T \tgmfd, f \in u^{\ast} T \tgmfd$ satisfying
\begin{equation*}
\phi_{0} \in H^{\sgm}, \quad f \in L^{1}_{s} (J; H^{\sgm}) + L^{2}_{s}(J; H^{\sgm-1}), 
\end{equation*}
there exists a unique solution $\phi$ to \eqref{eq:l-hmhf-ex} with $\phi(s=0) = \phi_{0}$ such that $\phi \in C^{0}_{s}(J; H^{\sgm}) \cap L^{2}_{s}(J; H^{\sgm+1})$, which obeys
\begin{equation} \label{eq:l-hmhf-ex-weak}
	\nrm{\phi}_{L^{\infty}_{s}(J; H^{\sgm})} + \nrm{\phi}_{L^{2}_{s}(J; H^{\sgm+1})} + \nrm{\rd_{s} \phi}_{L^{2}_{s}(J; L^{\sgm-1})} 
		\aleq \nrm{\phi_{0}}_{H^{\sgm}} + \nrm{f}_{L^{1}_{s} (J; H^{\sgm}) + L^{2}_{s}(J; H^{\sgm-1})},
\end{equation}
where the constant depends on upper bounds for $\nrm{u_{0}-Q}_{H^{1} \cap L^{\infty}}$ and $\nrm{\nb (u-Q)}_{L^{2}_{s} (J; L^{\infty})} + \abs{J}^{\frac{1}{2}}$.

Moreover, for any $k \geq 1$, we have
\begin{equation} \label{eq:l-hmhf-ex-k-weak}
\begin{aligned}
&	\nrm{s^{\frac{k}{2}} \phi}_{L^{\infty}_{s}(J; H^{\sgm+k})}
	+ \nrm{s^{\frac{k}{2}} \phi}_{L^{2}_{s}(J; H^{\sgm+k+1})}  + \nrm{s^{\frac{k}{2}} \rd_{s} \phi}_{L^{2}_{s}(J; H^{\sgm+k-1})} \\
&		\aleq \nrm{\phi_{0}}_{H^{\sgm}} + \nrm{f}_{L^{1}_{s} (J; H^{\sgm}) + L^{2}_{s}(J; H^{\sgm-1})}
			+ \nrm{s^{\frac{k}{2}} f}_{L^{1}_{s}(J; H^{\sgm+k}) + L^{2}_{s}(J; H^{\sgm+k-1})},
\end{aligned}\end{equation}
where the constant depends on $k$ and upper bounds for $\nrm{u_{0}-Q}_{H^{1} \cap L^{\infty}}$ and $\nrm{\nb (u-Q)}_{L^{2}_{s} (J; L^{\infty})} + \abs{J}^{\frac{1}{2}}$.

\item {\bf Local theory in $H^{1} \cap C^{0}$.} Let $J =[0, T]$ be a compact interval on which $u$ exists. Then for any $\phi_{0} \in u_{0}^{\ast} T \tgmfd, f \in u^{\ast} T \tgmfd$ satisfying
\begin{equation*}
\phi_{0} \in H^{1} \cap C^{0}, \quad f \in L^{1}_{s} (J; H^{1} \cap C^{0}), 
\end{equation*}
there exists a unique solution $\phi$ to \eqref{eq:l-hmhf-ex} with $\phi(s=0) = \phi_{0}$ such that $\phi \in C^{0}_{s}(J; H^{1} \cap C^{0})$ and $\nb \phi \in L^{2}_{s}(J; H^{1} \cap C^{0})$, which obeys
\begin{equation} \label{eq:l-hmhf-ex-1}
	\nrm{\phi}_{L^{\infty}_{s}(J; H^{1} \cap L^{\infty})} + \nrm{\phi}_{L^{2}_{s}(J; H^{2})} + \nrm{\rd_{s} \phi}_{L^{2}_{s}(J; L^{2})} 
	+ \nrm{\nb \phi}_{L^{2}_{s}(J; L^{\infty})}
		\aleq \nrm{\phi_{0}}_{H^{1} \cap L^{\infty}} + \nrm{f}_{L^{1}_{s} (J; H^{1} \cap L^{\infty})},
\end{equation}
where the constant depends on upper bounds for $\nrm{u_{0}-Q}_{H^{1} \cap L^{\infty}}$ and $\nrm{\nb (u-Q)}_{L^{2}_{s} (J; L^{\infty})} + \abs{J}^{\frac{1}{2}}$.

Moreover, for any $k \geq 1$, we have
\begin{equation} \label{eq:l-hmhf-ex-k}
\begin{aligned}
	& \nrm{s^{\frac{k+1}{2}} \phi}_{L^{\infty}_{s}(J; H^{k+1})}
	+ \nrm{s^{\frac{k+1}{2}} \phi}_{L^{2}_{s}(J; H^{k+2})}  + \nrm{s^{\frac{k+1}{2}} \rd_{s} \phi}_{L^{2}_{s}(J; H^{k-1})} \\
	& 	\aleq \nrm{\phi_{0}}_{H^{1} \cap L^{\infty}} 
		+ \nrm{f}_{L^{1}_{s}(J; H^{1} \cap L^{\infty})} + \nrm{s^{\frac{k+1}{2}} f}_{L^{1}_{s}(J; H^{k+1}) + L^{2}_{s}(J; H^{k})},
\end{aligned}\end{equation}
where the constant depends on $k$ and upper bounds for $\nrm{u_{0}-Q}_{H^{1} \cap L^{\infty}}$ and $\nrm{\nb (u-Q)}_{L^{2}_{s} (J; L^{\infty})} + \abs{J}^{\frac{1}{2}}$.

\item {\bf Local theory in $H^{\sgm}$ for $\sgm > 1$.} If $u_{0} - Q \in H^{\sgm}$, $\phi_{0} \in H^{\sgm}$ and $f \in L^{1}_{s}(J; H^{\sgm}) + L^{2}_{s}(J; H^{\sgm-1})$, then there exists a unique solution $\phi$ to \eqref{eq:l-hmhf-ex} with $\phi(s=0) = \phi_{0}$ in $C^{0}_{s}(J; H^{\sgm}) \cap L^{2}_{s}(J; H^{\sgm+1})$, which obeys
\begin{equation} \label{eq:l-hmhf-ex-sgm}
	\nrm{\phi}_{L^{\infty}_{s}(J; H^{\sgm})} + \nrm{\phi}_{L^{2}_{s}(J; H^{\sgm+1})} + \nrm{\rd_{s} \phi}_{L^{2}_{s}(J; H^{\sgm-1})} 
		\aleq \nrm{\phi_{0}}_{H^{\sgm}} + \nrm{f}_{L^{1}_{s}(J; H^{\sgm}) + L^{2}_{s}(J; H^{\sgm-1})},
\end{equation}
where the constant depends on $\sgm$ and upper bounds for $\nrm{u_{0}-Q}_{H^{\sgm}}$ and $\nrm{\nb (u-Q)}_{L^{2}_{s} (J; L^{\infty})} + \abs{J}^{\frac{1}{2}}$.

Moreover, for any $k \geq 1$, we have
\begin{equation} \label{eq:l-hmhf-ex-sgm-k}
\begin{aligned}
&	\nrm{s^{\frac{k+1}{2}} \phi}_{L^{\infty}_{s}(J; H^{\sgm+k})}
	+ \nrm{s^{\frac{k+1}{2}} \phi}_{L^{2}_{s}(J; H^{\sgm+k+1})}  + \nrm{s^{\frac{k+1}{2}} \rd_{s} \phi}_{L^{2}_{s}(J; H^{\sgm+k-1})} \\
		& \aleq \nrm{\phi_{0}}_{H^{\sgm}} 
		+ \nrm{f}_{L^{1}_{s}(J; H^{\sgm}) + L^{2}_{s}(J; H^{\sgm-1})}
		+ \nrm{s^{\frac{k+1}{2}}f}_{L^{1}_{s}(J; H^{\sgm+k}) + L^{2}_{s}(J; H^{\sgm+k-1})},
\end{aligned}\end{equation}
where the constants depend on $\sgm$, $k$ and upper bounds for $\nrm{u_{0}-Q}_{H^{\sgm}}$ and $\nrm{\nb (u-Q)}_{L^{2}_{s} (J; L^{\infty})} + \abs{J}^{\frac{1}{2}}$.
\end{itemize}

\end{proposition}

Some remarks concerning the above results are in order.
\begin{remark}[Local well-posedness in $H^{1}$] \label{rem:hmhf-lwp-H1}
To simplify the local theory for \eqref{eq:hmhf-ex}, we have elected to work with the space $H^{1} \cap C^{0}$. It is possible to extend the local theory to $u_{0} - Q \in H^{1}$, as we sketch now.

A key difference from the $H^{1} \cap C^{0}$ case is that now the image of $u_{0}$ may lie outside of $\tg$ even if $u_{0} - Q$ is small in $H^{1}$. Hence, we cannot rely on the trick of modifying $\tgmfd \setminus \tg$ to work with a compact target $\tgmfd'$ as in Subsection~\ref{subsec:hmhf-ext}. Instead, we need to \emph{assume} that the whole target manifold $\tgmfd$ is uniformly isometrically embedded in a Euclidean space $\bbR^{N}$, in the sense that there exists a tubular neighborhood $\tgmfd_{\eps}$ of $\tgmfd$ in $\bbR^{N}$ on which the closest-point projection map $\pi_{\tgmfd}$ is well-defined and its derivatives are uniformly bounded. 

Another important difference is that while $H^{1} \cap C^{0}$ is an algebra, $H^{1}$ is not. As a consequence, the standard Picard iteration does not apply in the $H^{1}$ case. Instead, one begins by showing:
\begin{itemize}
\item an $H^{1}$ a-priori estimate
\begin{equation} \label{eq:hmhf-lwp-H1}
	\nrm{u - Q}_{L^{\infty}_{s}(J; H^{1})} + \nrm{u - Q}_{L^{2}_{s}(J; H^{2})} + \nrm{\rd_{s}u}_{L^{2}_{s}(J; L^{2})} + \nrm{\nb (u - Q)}_{L^{2}_{s} (J; L^{\infty})} \aleq \nrm{u_{0} - Q}_{H^{1}},
\end{equation}
where the implicit constant depends only on an upper bound for $\nrm{\nb (u-Q)}_{L^{2}_{s}(J; L^{\infty})} + \abs{J}^{\frac{1}{2}}$, and 
\item an $H^{\sgm}$ (with $\sgm < 1$) bound \eqref{eq:l-hmhf-ex-weak} for the linearized equation \eqref{eq:l-hmhf-ex}, whose implicit constant depends only on upper bounds for the LHS of \eqref{eq:hmhf-lwp-H1} and $\abs{J}^{\frac{1}{2}}$. 
\end{itemize}
Estimate \eqref{eq:hmhf-lwp-H1} is essentially proved in the proof of Proposition~\ref{prop:hmhf-lwp} below, and the preceding remark concerning the dependencies of the implicit constant in \eqref{eq:l-hmhf-ex-weak} may be checked by inspection; we emphasize that the uniformity of the embedding $\iota : \tgmfd \to \bbR^{N}$ needs to be used here. Using these, one may show the existence of a unique local solution $u$ in the function space defined by the LHS of \eqref{eq:hmhf-lwp-H1} on an interval $J = [0, s_{0}]$ on which \eqref{eq:hmhf-lwp-hyp} holds. The solution map from $u_{0} - Q \in H^{1}$ in this function space can be shown to be continuous, but it is unlikely to be any smoother (say $C^{1}$), in contrast to the case of $H^{1} \cap C^{0}$.
\end{remark}
\begin{remark} \label{rem:hmhf-small-data}
By the local well-posedness statement in Proposition~\ref{prop:hmhf-lwp}, note that for any fixed $s$-interval $J = [0, s_{0})$, there exists $\eps = \eps(J, \tg)$ such that $\nrm{u_{0}-Q}_{H^{1} \cap L^{\infty}} < \eps$ implies that $u$ exists on $J$,
\begin{equation*}
\nrm{\nb (u - Q)}_{L^{2}_{s} (J; L^{\infty})} + \abs{J}^{\frac{1}{2}} \leq 1 \hbox{ and } u(x, s) \in \tg \hbox{ for all } (x, s) \in \bbH^{2} \times J.
\end{equation*}
Thus, the implicit constants in all bounds in Propositions~\ref{prop:hmhf-lwp} and \ref{prop:l-hmhf-ex} are independent of $u$ on the interval $J$. This simple observation will be used often in what follows.
\end{remark}

\begin{remark}
In the bound \eqref{eq:l-hmhf-ex-1}, we recover only one additional derivative for $\phi$ compared to $f$, whereas the correct smoothing bound should gain two additional derivatives (see, e.g., \eqref{eq:l-hmhf-ex-k}). This issue can be remedied by working with a function space that is similar to $H^{1} \cap \C^{0}$ but behaves better under the heat equation (e.g., the Besov space $B^{1, 2}_{1} \subset H^{1} \cap C^{0}$). Here we have chosen to avoid technicalities and use the simpler space $H^{1} \cap C^{0}$, as the nonsharp bound \eqref{eq:l-hmhf-ex-1} suffices for our purposes.
\end{remark}

Before we prove these propositions, we need to take care of one technical issue. Both \eqref{eq:hmhf-ex} and \eqref{eq:l-hmhf-ex} only make sense for $u(x, s)$ and $\phi(x, s)$ obeying suitable constraints, namely, $u(x, s) \in \tgmfd$ and $\phi(x, s) \in T_{u(x, s)} \tgmfd$. However, to set up an iteration scheme, we need to work with extensions of these equations to more general $\bbR^{N}$-valued functions. Then in order to return to the original equations, we need to ensure that the constraints are preserved in the course of the evolution.

We start with \eqref{eq:hmhf-ex}. Consider a tubular neighborhood $\tgmfd_{\eps} = \set{x \in \bbR^{N} : \bfd_{\bbR^{N}}(x, \tgmfd) < \eps}$ of $\tgmfd$, on which the closest-point projection $\pi_{\tgmfd} : \tgmfd_{\eps} \to \tgmfd$ is well-defined as a uniformly smooth map (i.e., all its derivatives are uniformly bounded). We extend \eqref{eq:hmhf-ex} to $u : \bbH^{2} \times J \to \tgmfd_{\eps}$ by 
\begin{equation} \label{eq:hmhf-ex-gen}
	\rd_{s} u^{A} = \lap u^{A} - (\rd_{B} \rd_{C} \pi_{\tgmfd}^{A})(u) \rd^{\ell} u^{B} \rd^{\ell} u^{C}.
\end{equation}
\begin{lemma}[Propagation of constraint for \eqref{eq:hmhf-ex}] \label{lem:hmhf-tan}
Let $u : \bbH^{2} \times J \to \tgmfd_{\eps}$ be a $C_{s}(J; H^{\infty}_{Q}(\bbH^{2}; \tgmfd))$ solution to \eqref{eq:hmhf-ex-gen} on $J$. If initially $u_{0}(x) \in \tgmfd$, then $u(x, s) \in \tgmfd$ for all $(x, s) \in \bbH^{2} \times J$.
\end{lemma}
For a proof, see \cite[Proof of Lemma~3.2]{LOS5}.

Next, we turn to \eqref{eq:l-hmhf-ex}. We need to extend the definition of $D_{\bfa}$ and the curvature tensor $\tensor{R}{_{AB}^{C}_{D}}$. For an $\bbR^{N}$-valued function $\phi$, we define $(\phi^{\top})^{A} = (\rd_{B} \pi_{\tgmfd}^{A}(u)) \phi^{B}$ via the point-wise orthogonal projection onto $T_{u} \tgmfd$. We also define $\phi^{\perp} := \phi - \phi^{\top}$. 
For a general $\bbR^{N}$-valued function $\phi$, we define
\begin{equation} \label{eq:D-ext}
	D_{\bfa} \phi = \rd_{\bfa} \phi - \rd_{B} \rd_{C} \pi^{A}_{\tgmfd}(u)\rd_{\bfa} u^{B} (\phi^{\top})^{C} + \rd_{B} \rd_{C} \pi^{A}_{\tgmfd}(u) \rd_{\bfa} u^{B} (\phi^{\perp})^{C}.
\end{equation}
We claim that \eqref{eq:D-ext} coincides with the induced covariant derivative \eqref{eq:pullback-d} for $\phi = \phi^{\top}$. Indeed, observe that
\begin{equation*}
	[(\rd_{\bfa} \phi^{\top})^{\top}]^A = \rd_{\bfa} (\phi^{\top})^A - \rd_{B} \rd_{C} \pi^{A}_{\tgmfd}(u)\rd_{\bfa} u^{B} (\phi^{\top})^{C}, \quad
	[(\rd_{\bfa} \phi^{\perp})^{\perp}]^A = \rd_{\bfa} (\phi^{\perp})^A + \rd_{B} \rd_{C} \pi^{A}_{\tgmfd}(u) \rd_{\bfa} u^{B} (\phi^{\perp})^{C},
\end{equation*}
%\begin{equation*}
%	(\rd_{\bfa} \phi^{\top})^{\top} = \rd_{\bfa} \phi^{\top} - \rd_{B} \rd_{C} \pi^{A}_{\tgmfd}(u)\rd_{\bfa} u^{B} (\phi^{\top})^{C}, \quad
%	(\rd_{\bfa} \phi^{\perp})^{\perp} = \rd_{\bfa} \phi^{\perp} + \rd_{B} \rd_{C} \pi^{A}_{\tgmfd}(u) \rd_{\bfa} u^{B} (\phi^{\perp})^{C},
%\end{equation*}
and recall that $D_{\bfa} \phi^{\top} = (\rd_{\bfa} \phi^{\top})^{\top}$. 

To extend the curvature tensor, we start with the formula
\begin{equation} \label{eq:S-ext}
S^{A}_{BC}(u) = - \rd_{B} \rd_{C} \pi^{A}_{\tgmfd}(u) 
\end{equation}
on $T_{u} \tgmfd$, which follows from the preceding computation and \eqref{eq:pullback-d}. Moreover, by the Gauss and Codazzi formulas, we have
\begin{equation} \label{eq:R-ext}
\tensor{R}{_{AB}^{C}_{D}} (u)= \rd_{C'} \pi^{C}_{\tgmfd}(u) \left( \rd_{A} S^{C'}_{BD} - \rd_{B} S^{C'}_{AD} + S^{C'}_{AE} S^{E}_{BD} - S^{C'}_{BE} S^{E}_{AD} \right)(u)
\end{equation}
on $T_{u} \tgmfd$. We take the RHSs of \eqref{eq:S-ext} and \eqref{eq:R-ext} as the extension of $S^{A}_{BC}(u)$ and $\tensor{R}{_{AB}^{C}_{D}}(u)$ to the whole ambient tangent space $\bbR^{N}$. We emphasize, however, that the induced covariant derivative $D_{\bfa}$ is extended using \eqref{eq:D-ext}, not \eqref{eq:pullback-d} and \eqref{eq:S-ext}.

\begin{lemma}[Propagation of constraint for \eqref{eq:l-hmhf-ex}] \label{lem:l-hmhf-tan}
Let $u$ be a $C_{s}(J; H^{\infty}_{Q}(\bbH^{2}; \tgmfd))$ solution to \eqref{eq:hmhf-ex} on $J$, and let $\phi$ be a $C_{s}(J; H^{\infty})$ solution to \eqref{eq:l-hmhf-ex} extended in the above fashion.
If $\phi_{0}(x) \in T_{u_{0}(x)} \tgmfd$ for all $x \in \bbH^{2}$ and $f(x, s) \in T_{u(x, s)} \tgmfd$ for all $(x, s) \in \bbH^{2} \times J$, then $\phi(x, s) \in T_{u(x, s)} \tgmfd$ for all $(x, s) \in \bbH^{2} \times J$.
\end{lemma}
\begin{proof}
Note that
\begin{equation*}
	D_{s} \phi = D_{s} \phi^{\top} + (\rd_{s} \phi^{\perp})^{\perp}, \quad
	D^{\ell} D_{\ell} \phi = D^{\ell} D_{\ell} \phi^{\top} + (\rd^{\ell} (\rd_{\ell} \phi^{\perp})^{\perp})^{\perp}.
\end{equation*}
Moreover, $D_{s} \phi^{\top}$, $D^{\ell} D_{\ell} \phi^{\top}$, $\tensor{R}{_{CD}^{A}_{B}} \rd_{\ell} u^{B} \phi^{C} \rd^{\ell} u^{D}$ and $f$ are all tangent to $T_{u} \tgmfd$. Thus \eqref{eq:l-hmhf-ex} implies 
\begin{equation*}
	0 = (\phi^{\perp}, \rd_{s} \phi^{\perp} - \rd^{\ell} (\rd_{\ell} \phi^{\perp})^{\perp}).
\end{equation*}
Integrating over $\bbH^{2}$, it follows that
\begin{equation*}
	\rd_{s} \int \abs{\phi^{\perp}}^{2} = - 2 \int ((\rd^{\ell} \phi^{\perp})^{\perp}), (\rd_{\ell} \phi^{\perp})^{\perp}) \leq 0.
\end{equation*}
Hence if $\phi^{\perp} = 0$ initially, then it stays so for $s > 0$. \qedhere
\end{proof}

We are now ready to prove Propositions~\ref{prop:hmhf-lwp} and \ref{prop:l-hmhf-ex}.

\begin{proof}[Proof of Proposition~\ref{prop:hmhf-lwp}]
We proceed in several steps.

\pfstep{Step~1: Reformulation of \eqref{eq:hmhf-ex}}
Fix a small number $\eps > 0$ so that $\pi_{\tgmfd}$ is well-defined on $\tgmfd_{\eps}$. Using a smooth cutoff, extend $\pi_{\tgmfd}$ to $\bbR^{N}$ so that $\pi_{\tgmfd} = id$ outside $\tgmfd_{2\eps}$. 
By the uniformity of the embedding $\tgmfd \hookrightarrow \bbR^{N}$, we may ensure that all derivatives of $\pi_{\tgmfd}$ are uniformly bounded. In what follows, the dependence of constants on $\pi_{\tgmfd}$ and $Q$ are often suppressed (we remark that for $Q$, only the norms of the type $\nrm{\nb Q}_{H^{\sgm}}$ for $\sgm \geq 0$ will be used). 

In view of Lemma~\ref{lem:hmhf-tan}, we work with \eqref{eq:hmhf-ex-gen}. We begin by reformulating this equation in terms of $v^{A} = u^{A} - Q^{A}$. Under the a-priori assumption that the image of $v+Q$ lies in $\tgmfd_{\eps}$, we introduce the notation
\begin{equation*}
\tilde{S}^{A}_{BC}(v; x) = - (\rd_{B} \rd_{C} \pi_{\tgmfd}^{A})(v + Q(x)) + (\rd_{B} \rd_{C} \pi_{\tgmfd}^{A})(Q(x)),
\end{equation*}
where we note that $-(\rd_{B} \rd_{C} \pi_{\tgmfd}^{A})(Q) = S^{A}_{BC}(Q)$. Then we rewrite \eqref{eq:hmhf-ex-gen} as
\begin{equation} \label{eq:hmhf-ex-v}
\begin{aligned}
	\rd_{s} v^{A} - \lap v^{A} & = -2 S^{A}_{BC}(Q) \rd^{\ell} Q^{B} \rd_{\ell} v^{C} + \tilde{S}^{A}_{BC}(v; x) \rd^{\ell} Q^{B} \rd_{\ell} Q^{C} \\
	& \phantom{=} + 2 \tilde{S}^{A}_{BC}(v; x) \rd^{\ell} Q^{B} \rd_{\ell} v^{C} - S^{A}_{BC}(Q) \rd^{\ell} v^{B} \rd_{\ell} v^{C}  + \tilde{S}^{A}_{BC}(v; x) \rd^{\ell} v^{B} \rd_{\ell} v^{C} \\
	& := \bfF^{A}[v]. 
\end{aligned}
\end{equation}

We end this step with a couple of pointwise bounds for $\tilde{S}^{A}_{BC}(v(x), x)$. Under the a-priori assumption that the image of $v+Q$ lies in $\tgmfd_{\eps}$, we claim that
\begin{align}
	\abs{\tilde{S}^{A}_{BC}(v(x), x)} & \aleq \min\set{1, \abs{v(x)}} \label{eq:hmhf-lwp-pf-tS-1} \\
	\abs{\nb \tilde{S}^{A}_{BC}(v(x), x)} & \aleq \abs{\nb v(x)} + \abs{v(x)} \label{eq:hmhf-lwp-pf-tS-2}
\end{align}
Indeed, if $\abs{v(x)} \ageq 1$, then \eqref{eq:hmhf-lwp-pf-tS-1} is vacuously true. If $\abs{v(x)} \ll 1$, then by the fundamental theorem of calculus,
\begin{equation*}
	\abs{\tilde{S}^{A}_{BC}(v(x), x)}
	\leq \int_{0}^{1} \abs{v^{D}(x) (\rd_{B} \rd_{C} \rd_{D} \pi_{\tgmfd}^{A})(\lmb v + Q(x))} \, \ud \lmb
	\aleq \abs{v(x)},
\end{equation*}
which proves \eqref{eq:hmhf-lwp-pf-tS-1}. Bound \eqref{eq:hmhf-lwp-pf-tS-2} is proved similarly after an application of the chain rule.

\pfstep{Step~2: Bounds for $\bfF[v]$}
Let $J$ be an $s$-interval. Consider a function $v : \bbH^{2} \times J \to \bbR^{N}$ that satisfies
\begin{equation} \label{eq:hmhf-lwp-pf-aux}
(v + Q)(x, s) \in \tgmfd_{\eps} \hbox{ for } (x, s) \in \bbH^{2} \times J,
\end{equation}
as well as
\begin{equation} \label{eq:hmhf-lwp-pf-0} 
	\nrm{\nb v}_{L^{2}_{s}(J; L^{\infty})} + \abs{J}^{\frac{1}{2}} \leq \dlt.
\end{equation}
We claim that the following bounds hold:
\begin{align}
	\nrm{\bfF[v]}_{L^{1}_{s}(J; H^{1})}
	& = \nrm{\nb \bfF[v]}_{L^{1}_{s}(J; L^{2})} 
	\aleq \dlt \nrm{\nb v}_{L^{2}_{s}(J; H^{1} \cap L^{\infty})}
	+ \dlt^{2} \nrm{v}_{L^{\infty}_{s}(J; H^{1})}, \label{eq:hmhf-lwp-pf-1} \\
	\nrm{\bfF[v]}_{L^{1}_{s}(J; L^{\infty})}
	& \aleq \dlt \nrm{\nb v}_{L^{2}_{s}(J; L^{\infty})} 
	+ \dlt^{2} \nrm{v}_{L^{\infty}_{s}(J; L^{\infty})}. \label{eq:hmhf-lwp-pf-2}
\end{align}
We emphasize that the implicit constants may depend on $\pi_{\tgmfd}$ or $Q$, but not on $v$. Moreover, if $\nrm{v}_{L^{\infty}_{s}(J; H^{1} \cap L^{\infty})}+ \nrm{w}_{L^{\infty}_{s}(J; H^{1} \cap L^{\infty})} \leq A$, then we claim that
\begin{equation} \label{eq:hmhf-lwp-pf-3}
	\nrm{\bfF[v] - \bfF[w]}_{L^{1}_{s}(J; H^{1} \cap L^{\infty})}
	\aleq_{A}
	\dlt \nrm{\nb (v - w)}_{L^{2}_{s}(J; H^{1} \cap L^{\infty})}
	+ \dlt^{2} \nrm{v - w}_{L^{\infty}_{s}(J; H^{1} \cap L^{\infty})},
\end{equation}
and for any $\sigma\geq0$
\begin{equation} \label{eq:hmhf-lwp-pf-4}
	\nrm{\bfF[v]}_{L^{1}_{s}(J; H^{\sgm})}
	\aleq_{A} \dlt \nrm{v}_{L^{2}_{s}(J; H^{\sgm+1})} + \dlt^{2} \nrm{v}_{L^{\infty}_{s}(J; H^{\sgm})}.
\end{equation}

We first verify \eqref{eq:hmhf-lwp-pf-1}. For all terms in the definition of $\bfF^{A}[v]$ in \eqref{eq:hmhf-ex-v} except $\tilde{S}^{A}_{BC}(v;x) \rd^{\ell} Q^{B} \rd_{\ell} Q^{C}$, the desired bound follows by the chain rule and the bounds $\abs{\tilde{S}(v(x), x)} \aleq 1$, $\nrm{\nb \tilde{S}(v(x), x)}_{L^{2}} \aleq \nrm{v}_{H^{1}}$ and $\nrm{\nb v}_{L^{2}_{s} (J; L^{\infty})} + \nrm{\nb Q}_{L^{2}_{s}(J; L^{\infty})} \aleq \dlt$, which in turn follow from \eqref{eq:hmhf-lwp-pf-tS-1}, \eqref{eq:hmhf-lwp-pf-tS-2} (as well as the Poincar\'e inequality) and \eqref{eq:hmhf-lwp-pf-0}, respectively. For $\tilde{S}^{A}_{BC}(v;x) \rd^{\ell} Q^{B} \rd_{\ell} Q^{C}$, we directly estimate
\begin{align*}
\nrm{\nb (\tilde{S}^{A}_{BC}(v;x) \rd^{\ell} Q^{B} \rd_{\ell} Q^{C})}_{L^{2}}
& \aleq \nrm{\nb \tilde{S}^{A}_{BC}(v;x)}_{L^{2}} + \nrm{\tilde{S}^{A}_{BC}(v;x) \nb \rd^{\ell} Q^{B} \rd_{\ell} Q^{C})}_{L^{2}} \\
& \aleq \nrm{\nb v}_{L^{2}} + \nrm{v}_{L^{4}} \nrm{\nb^{(2)} Q}_{L^{2}} \nrm{\nb Q}_{L^{4}} 
 \aleq \nrm{\nb v}_{L^{2}},
\end{align*}
where we used the chain rule, the H\"older inequality, \eqref{eq:hmhf-lwp-pf-tS-1}, \eqref{eq:hmhf-lwp-pf-tS-2}, the Poincar\'e inequality and interpolation. Next, bound \eqref{eq:hmhf-lwp-pf-2} is immediate from the bound $\nrm{\nb v}_{L^{2}_{s} (J; L^{\infty})} + \nrm{\nb Q}_{L^{2}_{s}(J; L^{\infty})} \aleq \dlt$ and \eqref{eq:hmhf-lwp-pf-tS-1}. Finally, bounds \eqref{eq:hmhf-lwp-pf-3} and \eqref{eq:hmhf-lwp-pf-4} are straightforward consequences of Proposition~\ref{prop:frac-leib}, Proposition~\ref{prop:frac-moser} and Corollary~\ref{cor:frac-moser-diff}; we leave the details to the reader.

\pfstep{Step~3: Completion of proof}
Let $v_{0} = u_{0} - Q \in H^{1} \cap C^{0}$. Starting with $v^{(0)} = 0$, we construct a sequence $v^{(n)}$ of Picard iterates by:
\begin{equation*}
	(\rd_{s} - \lap) v^{(n)} = \bfF[v^{(n-1)}], \quad v^{(n)}(s=0) = v_{0}.
\end{equation*}
If it were not for the auxiliary condition \eqref{eq:hmhf-lwp-pf-aux}, the bounds \eqref{eq:hmhf-lwp-pf-1}, \eqref{eq:hmhf-lwp-pf-2} and \eqref{eq:hmhf-lwp-pf-3}, in combination with Lemmas~\ref{lem:lin-heat-L2} and \ref{lem:lin-heat-L2Linfty}, would immediately lead to $H^{1} \cap C^{0}$ local well-posedness on an interval $J = [0, s_{0}]$ for which $\nrm{\nb e^{s \lap} v_{0}}_{L^{2}_{s}(J; L^{\infty})} + \abs{J}^{\frac{1}{2}}$ is sufficiently small compared to $\max\set{1, \nrm{v_{0}}_{H^{1} \cap C^{0}}}$ by the Banach contraction principle. On the same interval, persistence of regularity would also follow from \eqref{eq:hmhf-lwp-pf-4} and Lemma~\ref{lem:lin-heat-L2}. 

In reality, $H^{1} \cap C^{0}$ local well-posedness may be only proved first on a shorter interval on which \eqref{eq:hmhf-lwp-pf-aux} holds as well. However, approximating general initial data by $H^{\infty}$ initial data and appealing to Lemma~\ref{lem:hmhf-tan} for the $H^{\infty}$ solution, we may ensure that any $H^{1} \cap C^{0}$ well-posed solution $v$ obeys $v+Q \in \calN$ wherever it exists. Hence, it is straightforward to set up a continuous induction to recover the $H^{1} \cap C^{0}$ local well-posedness as asserted in Proposition~\ref{prop:hmhf-lwp}.

Let $v$ be a $H^{1} \cap C^{0}$ well-posed solution on $\bbH^{2} \times J$. If $\nrm{\nb v}_{L^{2}_{s}(J; L^{\infty})} + \abs{J}^{\frac{1}{2}}$ is finite, then given any $\dlt > 0$, $J$ can be divided into $K$-many subintervals on which \eqref{eq:hmhf-lwp-pf-0} holds. Taking $\dlt$ sufficiently small, by \eqref{eq:hmhf-lwp-pf-1}, \eqref{eq:hmhf-lwp-pf-2}, Lemma~\ref{lem:lin-heat-L2} and Lemma~\ref{lem:lin-heat-L2Linfty}, we see that $v$ obeys an a-priori estimate of the form \eqref{eq:hmhf-lwp-1} with the constant depends exponentially on $K$; in particular, $v$ may be continued past $J$ (here, it is important that the implicit constants in \eqref{eq:hmhf-lwp-pf-1} and \eqref{eq:hmhf-lwp-pf-2} are independent of $v$). Persistence of regularity (resp. parabolic smoothing) on $J$ follows similarly from \eqref{eq:hmhf-lwp-pf-4} and \eqref{eq:lin-heat-L2} (resp. \eqref{eq:lin-heat-L2-high}) in Lemma~\ref{lem:lin-heat-L2}. \qedhere
\end{proof}

\begin{proof}[Proof of Proposition~\ref{prop:l-hmhf-ex}]
As in the proof of Proposition~\ref{prop:hmhf-lwp}, we introduce a globally defined $\pi_{\tgmfd}$ that agrees with the closest-point projection on a fixed $\tgmfd_{\eps}$, and equals the identity outside $\tgmfd_{2\eps}$. We extend $D_{\bfa}$ and $\tensor{R}{_{AB}^{C}_{D}}$ by the formulas \eqref{eq:D-ext} and \eqref{eq:R-ext}, respectively, using the globally defined $\pi_{\tgmfd}$. Moreover, we always write $\rd^{(k)} \pi(u) =  (\rd^{(k)} \pi(v + Q) - \rd^{(k)} \pi(Q)) + \rd^{(k)} \pi(Q)$, and use Proposition~\ref{prop:frac-moser}  to estimate the first term by $v$, which in turn is controlled by Proposition~\ref{prop:hmhf-lwp}. Then Proposition~\ref{prop:l-hmhf-ex} follows by standard Picard iteration arguments based on Lemmas~\ref{lem:lin-heat-L2} and \ref{lem:lin-heat-L2Linfty} as in the proof of Proposition~\ref{prop:hmhf-lwp}; we omit the details. \qedhere
\end{proof}

\subsection{Asymptotic stability of $Q$ under the harmonic map heat flow} \label{subsec:hmhf-stab}
With the local theory for the harmonic map heat flow and its linearization in our hands, we turn to the question of the asymptotic stability of a harmonic map $Q$. Our main result is that a sufficient condition for a positive answer is the \emph{weak linearized stability condition} (Definition~\ref{def:weak-stab}).
\begin{theorem} \label{thm:hmhf-gwp}
Let $Q$ be a weakly linearly stable harmonic map from $\bbH^{2}$ into $\tg$ and fix $0 < c_{0} < \rho_{Q}$. There exists $\epshf = \epshf(Q, c_{0})  > 0$ such that the following holds. If $u_{0} \in (H^{1} \cap C^{0})_{Q}(\bbH^{2}; \tg)$ obeys
\begin{equation} \label{eq:hmhf-wp-hyp}
\nrm{u_{0} - Q}_{H^{1} \cap L^{\infty}} < \epshf,
\end{equation}
then the solution $u$ to \eqref{eq:hmhf-ex} with $u(s=0) = u_{0}$ exists globally. Moreover, for all $s \geq 0$, it obeys the long time bound
\begin{equation} \label{eq:hmhf-gwp}
	\nrm{u(s)-Q}_{H^{1} \cap L^{\infty}} \aleq e^{-c_{0}^{2} s} \nrm{u_{0} - Q}_{H^{1} \cap L^{\infty}},
\end{equation}
where the implicit constant depends on $Q$ and $c_{0}$.
\end{theorem}

We also have the following corresponding result for the linearized harmonic map heat flow.
\begin{proposition} \label{prop:l-hmhf-gwp}
Let $Q$ be a weakly linearly stable harmonic map from $\bbH^{2}$ into $\tg$, let $0 < c_{0} < \rho_{Q}$ and let $u$ be a solution to \eqref{eq:hmhf-ex} on $[0, \infty)$ with the initial bound \eqref{eq:hmhf-wp-hyp} that is given by Theorem~\ref{thm:hmhf-gwp}. Then any solution $\phi$ to the corresponding linearized harmonic map heat flow \eqref{eq:l-hmhf-ex} with $\phi (s=0) = \phi_{0}$ (given by Proposition~\ref{prop:l-hmhf-ex}) obeys the following long time bound for $s \geq 0$:
\begin{align}
	\nrm{\phi(s)}_{L^{2}} 
	+ \nrm{\phi}_{L^{2}_{s}([s, s+1]; H^{1} \cap L^{\infty})} &\aleq e^{- c_{0}^{2} s} \nrm{\phi_{0}}_{L^{2}} + \int_{0}^{s} e^{-c_{0}^{2}(s - s')} \nrm{f(s')}_{L^{2}} \, \ud s', \label{eq:l-hmhf-gwp-0} \\
	\nrm{\phi(s)}_{H^{1} \cap L^{\infty}} + \nrm{\nb \phi}_{L^{2}_{s}([s, s+1]; H^{1} \cap L^{\infty})} &\aleq e^{- c_{0}^{2} s} \nrm{\phi_{0}}_{H^{1} \cap L^{\infty}} \label{eq:l-hmhf-gwp-1} \\
	&\phantom{\aleq} 
	+ \int_{0}^{s} e^{-c_{0}^{2}(s - s')} \nrm{f(s')}_{L^{2}} \, \ud s' + \int_{\max \set{0, s-1}}^{s} \nrm{f(s')}_{H^{1} \cap L^{\infty}} \, \ud s'.  \notag
\end{align}
The implicit constants depend on $Q$ and $c_{0}$.
\end{proposition}

\begin{remark}[Asymptotic stability of $Q$ in $H^{1}$]
With the $H^{1}$ theory sketched in Remark~\ref{rem:hmhf-lwp-H1}, Theorem~\ref{thm:hmhf-gwp} may be immediately upgraded to an analogous result in $H^{1}$ (i.e., with the less stringent assumption $\nrm{u_{0} - Q}_{H^{1}} \ll 1$). Indeed, if $u_{0} - Q$ is sufficiently small in $H^{1}$, then by \eqref{eq:hmhf-lwp-H1} and the pigeonhole principle we may find a good $s_{1} \simeq 1$ at which $u(s_{1})$ exists and $\nrm{u(s_{1}) - Q}_{H^{2}} \ll 1$. Thus, we may apply Theorem~\ref{thm:hmhf-gwp} with $u(s_{1})$ as the initial data.
\end{remark}
Before we proceed to the proofs of these results, we point out a simple corollary of Theorem~\ref{thm:hmhf-gwp} and the local theory in the previous subsection that allows us to easily derive decay of any higher derivatives for $s \geq 1$.
\begin{corollary} \label{cor:hmhf-gwp-higher}
Let $Q$ be a weakly linearly stable harmonic map from $\bbH^{2}$ into $\tgmfd$, let $0 < c_{0} < \rho_{Q}$ and let $u$ be a solution to \eqref{eq:hmhf-ex} on $[0, \infty)$ with the initial bound \eqref{eq:hmhf-wp-hyp} that is given by Theorem~\ref{thm:hmhf-gwp}. 
\begin{itemize}
\item For any $s \geq 1$ and $k = 1, 2, \ldots$, we have
\begin{equation} \label{eq:hmhf-gwp-k}
	\nrm{u(s)-Q}_{H^{k}}
	+ \nrm{\rd_{s} u(s)}_{H^{k}}
	\aleq \nrm{u(s-1)-Q}_{H^{1} \cap L^{\infty}} \aleq e^{-c_{0}^{2} s} \nrm{u_{0} - Q}_{H^{1} \cap L^{\infty}},
\end{equation}
where the implicit constants depend on $Q$, $c_{0}$ and $k$.
\item Let $\phi$ be a solution to the corresponding linearized harmonic map heat flow \eqref{eq:l-hmhf-ex} (given by Propositions~\ref{prop:l-hmhf-ex} and \ref{prop:l-hmhf-gwp}). For any $s \geq 1$ and $k = 1, 2, \ldots$, we have
\begin{equation} \label{eq:l-hmhf-gwp-k}
\begin{aligned}
	& \nrm{\phi(s)}_{H^{k+1}} + \nrm{\phi}_{L^{2}_{s}([s, s+1]; H^{k+2})}
	+ \nrm{\rd_{s} \phi}_{L^{2}_{s}([s, s+1]; H^{k})} \\
	& \aleq \nrm{\phi(s-1)}_{L^{2}} + \nrm{f}_{L^{1}_{s} ([s-1, s+1]; H^{k+1}) +L^{2}_{s} ([s-1; s+1]; H^{k})}
\end{aligned}
\end{equation}
where the implicit constant depends on $Q$, $c_{0}$ and $k$.
\end{itemize}
\end{corollary}
\begin{proof}
Estimate  \eqref{eq:l-hmhf-gwp-k} is a direct consequence of \eqref{eq:l-hmhf-ex-k}, \eqref{eq:hmhf-gwp} and Remark~\ref{rem:hmhf-small-data}, so we only need to consider \eqref{eq:hmhf-gwp-k}. By \eqref{eq:hmhf-gwp} and \eqref{eq:hmhf-lwp-k} applied to $J = [s-1, s+1]$, we have
\begin{equation*}
	\nrm{u(s) - Q}_{H^{1+k}} + \nrm{\rd_{s} u}_{L^{2}_{s}([s, s+1]; H^{k})}\aleq \nrm{u(s-1) - Q}_{H^{1} \cap L^{\infty}} \aleq e^{-c_{0}^{2} s} \nrm{u_{0} - Q}_{H^{1} \cap L^{\infty}}.
\end{equation*}
To obtain \eqref{eq:hmhf-gwp-k}, it remains to upgrade the $L^{2}_{s} ([s, s+1]; H^{k})$ control of $\rd_{s} u$ to a point-wise bound for, say, $\nrm{\rd_{s} u(s+1)}_{H^{k}}$. By the pigeonhole principle, there exists $s' \in [s, s+1]$ such that $\nrm{\rd_{s} u(s')}_{H^{k}} \leq \nrm{\rd_{s} u}_{L^{2}_{s}([s, s+1]; H^{k})}$. Then, since $\rd_{s} u$ solves the linearized harmonic map heat flow with $f = 0$, we may apply \eqref{eq:l-hmhf-gwp-k} to conclude that $\nrm{\rd_{s} u(s+1)}_{H^{k}} \aleq \nrm{\rd_{s} u(s')}_{H^{k}}$. Putting all these bounds together, \eqref{eq:hmhf-gwp-k} follows. \qedhere

\end{proof}

Our proofs of Theorem~\ref{thm:hmhf-gwp} and Proposition~\ref{prop:l-hmhf-gwp} are based on a long term analysis of the linearized harmonic map heat flow in the moving frame formalism, which allows us to utilize the weak linearized stability assumption. Given $e$ on $u^{\ast} T \tgmfd$ such that $e_{2} = J e_{1}$, note that the linearized harmonic map heat flow \eqref{eq:l-hmhf-ex} is equivalent to
\begin{equation} \label{eq:l-hmhf-cov-gen}
	\Db_{s} \varphi - \Db^{\ell} \Db_{\ell} \varphi + i \tcv \Im(\psi^{\ell} \overline{\varphi}) \psi_{\ell} = F,
\end{equation}
where $\varphi, F : \bbH^{2} \times J \to \bbC$ are related to $\phi, f : \bbH^{2} \times J \to T_{u} \tgmfd$ by $\phi = e \varphi$ and $f = e F$, see Subsection~\ref{subsec:hm}. 
To use the moving frame formalism, we need the following lemma.

\begin{lemma} \label{lem:diff-maps}
There exists $\epsN > 0$ such that the following holds. For any pair $u_{(0)}, u_{(1)}$ of smooth maps from $\bbH^{2}$ into $\tgmfd$ satisfying
\begin{equation*}
	\nrm{u_{(j)} - Q}_{H^{1} \cap L^{\infty}} < \epsN, \quad (j=0, 1),
\end{equation*}
there exists a $C^{1}$-homotopy $u_{(\lmb)}$ for $0 \leq \lmb \leq 1$ between the two maps such that
\begin{equation*}
\nrm{\rd_{\lmb} u_{(\lmb)}}_{H^{1} \cap L^{\infty}} \aleq \nrm{u_{(j)} - Q}_{H^{1} \cap L^{\infty}} \quad \hbox{ for all } \lmb \in [0, 1].
\end{equation*}
\end{lemma}

\begin{proof}
Let $\pi_{\tgmfd}$ be the orthogonal projection to $\tgmfd$, which is a well-defined smooth map on a tubular neighborhood $\tgmfd_{\eps}$ of $\tgmfd$. Without loss of generality, take $\eps_{1} < 1$. For $\lmb \in [0, 1]$, we introduce
\begin{equation*}
	\tilde{u}_{(\lmb)} (x)= \left( \lmb u_{(1)}(x) + (1-\lmb) u_{(0)}(x) \right).
\end{equation*}
We claim that for each $x \in \bbH^{2}$, $\tilde{u}_{(\lmb)}(x)$ lies in $\tgmfd_{\eps}$. Then we may define
\begin{equation*}
	u_{(\lmb)}(x) = \pi_{\tgmfd} \tilde{u}_{(\lmb)} (x), 
\end{equation*}
for which the desired conclusion can be easily checked. To verify the claim, note that
\begin{equation*}
\min_{j=0, 1} \abs{\tilde{u}_{(\lmb)}(x) - u_{(j)}(x)} \leq \frac{1}{2} \abs{u_{(1)}(x) - u_{(0)}(x)} < \abs{u_{(1)}(x) - Q(x)} + \abs{u_{(0)}(x) - Q(x)} \aleq \epsN
\end{equation*}
where we used the hypothesis for the last inequality. Since each $u_{(j)}(x)$ lies in $\tgmfd \cap \set{x : \abs{x - Q(x)} < \epsN}$, it follows that $\tilde{u}_{(\lmb)} \in \tgmfd_{\eps}$ if $\epsN$ is chosen sufficiently small. \qedhere
\end{proof}

We are now ready to establish the main results of this subsection.

\begin{proof}[Proof of Theorem~\ref{thm:hmhf-gwp} and Proposition~\ref{prop:l-hmhf-gwp}]
The proof consists of several steps.

\pfstep{Step~1: Setting up the main bootstrap argument}
Consider the homotopy $u_{0} : [0, 1] \times \bbH^{2} \to \tgmfd$, $(\lmb; x) \mapsto u_{0}(\lmb; x)$ such that $u_{0}(0; x) = Q(x)$ and $u_{0}(1; x) = u_{0}$ given by Lemma~\ref{lem:diff-maps}. For each fixed $\lmb \in [0, 1]$, we extend $u_{0}(\lmb; x)$ to $u(\lmb; x, s)$ for $s > 0$ by solving the harmonic map heat flow. Thanks to Proposition~\ref{prop:hmhf-lwp}, by taking $\epshf$ sufficiently small we may ensure that $u(\lmb; x, s)$ exists on $[0, s_{0}]$ for $s_{0} \geq 10$, is smooth in $(\lmb, x, s) \in [0, 1] \times \bbH^{2} \times [0, s_{0}]$ and obeys 
\begin{equation} \label{eq:hmhf-gwp-btstrp-0}
	\nrm{\nb (u(\lmb; s) - Q)}_{L^{2}} \leq \epshf^{\frac{1}{2}} \quad \hbox{ for } s \in [0, s_{0}].
\end{equation}
Note that each $\rd_{\lmb} u$ obeys the linearized harmonic map heat flow \eqref{eq:l-hmhf-ex}. Taking $\epshf$ even smaller if necessary, we may also ensure (via Proposition~\ref{prop:l-hmhf-ex}) that
\begin{equation} \label{eq:hmhf-gwp-btstrp}
	\nrm{\rd_{\lmb} u(\lmb; s)}_{L^{\infty}} + \nrm{D \rd_{\lmb} u(\lmb; s)}_{L^{\infty}} \leq \epshf^{\frac{1}{2}} \quad \hbox{ for } s \in [1, s_{0}].
\end{equation}

Estimates \eqref{eq:hmhf-gwp-btstrp-0} and \eqref{eq:hmhf-gwp-btstrp} are our main bootstrap assumptions. Our goal up to Step~4 is to show that for $\epshf$ sufficiently small, the RHSs of \eqref{eq:hmhf-gwp-btstrp-0} and \eqref{eq:hmhf-gwp-btstrp} can be improved to $O(e^{- c_{0}^{2} s}\eps_{\ast})$.

\pfstep{Step~2: Setting up the moving frame formalism}
To pass to the moving frame formalism, we begin with the frame for $Q^{\ast} T \tgmfd$ given in Subsection~\ref{subsec:hm} in the case $\lmb = 0$:
\begin{equation*}
e(0; x, s) = e^{\infty}(x) \hbox{ for } u(0; x, s) = Q(x).
\end{equation*}
The exact properties of $e^{\infty}$ is not important for the current proof, except that the corresponding differential and connection coefficients $\psi^{\infty}$ and $A^{\infty}$ are smooth and well-decaying. 

Next, we extend the frame $e(0; x, s)$ to $e(\lmb; x, s)$ for $\lmb \in [0, 1]$ by parallel transportation along each constant-$(x, s)$ curve, i.e., by requiring $D_{\lmb} e = 0$; by smoothness of the map $u(\lmb; x, s)$ on $[0,1] \times \bbH^{2} \times [0, s_{0}]$, such a frame $e$ is well-defined. We denote the corresponding differential and connection coefficients by $\psi(\lmb; x, s)$ and $A(\lmb; x, s)$, i.e., $e \psi_{\bfa} = \rd_{\bfa} u$ and $e (i A_{\bfa}) = D_{\bfa} e$, where $\bfa = \lmb$, $x^{1}$, $x^{2}$ or $s$. Note that $\psi_{\bfa} (0; x, s) = \psi^{\infty}_{\bfa} (x, s)$ and $A_{\bfa}(0; x, s) = A^{\infty}_{\bfa}(x, s)$. The parallel transport condition takes the form 
\begin{equation*}
A_{\lmb} = 0. 
\end{equation*}
Thus, writing
\begin{equation*}
	\psi_{\bfa} = \ringpsi_{\bfa} + \psi^{\infty}_{\bfa}, \quad
	A_{\bfa} = \ringA_{\bfa} + A^{\infty}_{\bfa},
\end{equation*}
it follows that
\begin{align*}
\ringpsi_{\ell}(\lmb) = - \int_{0}^{\lmb} \Db_{\ell} \psi_{\lmb}(\lmb') \, \ud \lmb', \quad
\ringA_{\ell}(\lmb) =    \int_{0}^{\lmb} \tcv \Im(\overline{\psi_{\lmb}} \psi_{\ell})(\lmb') \, \ud \lmb'.
\end{align*}
Since $\abs{\rd_{\bfa} u} = \abs{\psi_{\bfa}}$ and $\abs{D_{\bfa} \rd_{\bfb} u} = \abs{\Db_{\bfa} \psi_{\bfb}}$, our bootstrap assumption \eqref{eq:hmhf-gwp-btstrp} implies
\begin{equation} \label{eq:hmhf-gwp-diff}
	\nrm{\ringpsi_{x}(\lmb; s)}_{L^{\infty}} 
	+ \nrm{\ringA_{x}(\lmb; s)}_{L^{\infty}}
	\aleq \epshf^{\frac{1}{2}} \quad \hbox{ for } s \in [1, s_{0}], \ \lmb \in [0, 1].
\end{equation}

\pfstep{Step~3: Analysis of the linearized harmonic map heat flow}
Fix $\lmb \in [0, 1]$ and consider a solution $\varphi$ to the inhomogeneous linearized harmonic map heat flow equation for $s \in [0, s_{0}]$,
\begin{equation*}
\Db_{s} \varphi - \Db^{\ell} \Db_{\ell} \varphi + i \tcv \Im(\psi^{\ell} \overline{\varphi}) \psi_{\ell} = F,
\end{equation*}
where $A_{\bfa} = A_{\bfa}(\lmb; x, s)$ and $\psi_{\bfa} = \psi_{\bfa}(\lmb; x, s)$ are as in Step~2. We denote the initial values by $\varphi_{0} = \varphi(s=0)$. The goal of this step is to derive the following long time bounds for $\varphi$ under the assumption \eqref{eq:hmhf-gwp-diff}: For $s \in [0, s_{0}]$,
\begin{equation} \label{eq:l-hmhf-gwp-key}
	\nrm{\varphi(s)}_{L^{2}}  + \left( \int_{s}^{\min \set{s+1, s_{0}}} \nrm{\Db_{x} \varphi}_{L^{2}}^{2} + \nrm{\varphi}_{L^{\infty}}^{2} \, \ud s' \right)^{\frac{1}{2}} \aleq e^{- c_{0}^{2} s} \nrm{\varphi_{0}}_{L^{2}} + \int_{0}^{\min\set{s+1, s_{0}}} e^{-c_{0}^{2}(s-s')}\nrm{F(s')}_{L^{2}} \, \ud s'.
\end{equation}
 We will only need the special case $(\varphi, F) = (\psi_{\lmb}, 0)$ for the proof of Theorem~\ref{thm:hmhf-gwp} (see Step~4), but a general choices of $\varphi$ and $F$ will lead to a simultaneous proof of Proposition~\ref{prop:l-hmhf-gwp} (see Step~5).

We rewrite the above equation as
\begin{align*}
\Db_{s} \varphi - (\Db^{\infty})^{\ell} (\Db^{\infty})_{\ell} \varphi + i \tcv \Im((\psi^{\infty})^{\ell} \overline{\varphi}) (\psi^{\infty})_{\ell}
&= 2 i \ringA^{\ell} (\Db^{\infty})_{\ell} \varphi - \ringA^{\ell} \ringA_{\ell} \varphi \\
& \phantom{=}
	- i \tcv \left( \Im(\ringpsi^{\ell} \overline{\varphi}) \psi_{\ell} + \Im(\psi^{\infty} \overline{\varphi}) \ringpsi_{\ell} \right) 
	 + F.
%	- i (\tcv(u) - \tcv(Q)) \Im((\psi^{\infty})^{\ell} \overline{\varphi}) (\psi^{\infty})_{\ell}    %% IF $\tcv$ is allowed to depend on $u$
\end{align*}
Multiplying by $\overline{\varphi}$, taking the real part, integrating on $\bbH^{2}$ and then using \eqref{eq:hmhf-gwp-diff} to estimate the RHS, we obtain the one-sided bound
\begin{align*}
	\frac{1}{2} \frac{\ud}{\ud s} \nrm{\varphi(s)}_{L^{2}}^{2}
	+ \Re \int H \varphi \overline{\varphi} \, \ud x
	\leq C \eps_{\ast}^{\frac{1}{2}} \left( \nrm{\Db^{\infty}\varphi(s)}_{L^{2}}^{2} + \nrm{\varphi(s)}_{L^{2}}^{2} \right) + \nrm{F(s)}_{L^{2}} \nrm{\varphi(s)}_{L^{2}}.
\end{align*}
for some absolute constant $C > 0$. Using Definition~\ref{def:weak-stab} and recalling that $\Re \int H \varphi \overline{\varphi} \, \ud x - \nrm{\Db^{\infty} \varphi}_{L^{2}}^{2} \aleq \nrm{\varphi}_{L^{2}}^{2}$, for any $\rho'_{Q} < \rho_{Q}$ we obtain the inequality
\begin{equation*}
	c_{\rho'_{Q}} \nrm{\Db^{\infty} \varphi}_{L^{2}}^{2} + (\rho'_{Q})^{2} \nrm{\varphi}_{L^{2}}^{2} \leq \Re \int H \varphi \overline{\varphi} \, \ud x,
\end{equation*}
for some $c_{\rho'_{Q}} > 0$. At this point, we fix $\rho'_{Q} > c_{0}$ and write $c = c_{\rho'_{Q}}$. Taking $\eps_{\ast}$ sufficiently small, we arrive at the differential inequality (with $c^{-}$ denoting a constant slighlty less than $c$)
\begin{equation*}
	\frac{1}{2} \frac{\ud}{\ud s} \nrm{\varphi(s)}_{L^{2}}^{2}
	+ c_{0}^{2} \nrm{\varphi(s)}_{L^{2}}^{2} + c^{-} \nrm{\Db^{\infty} \varphi}_{L^{2}}^{2} \leq \nrm{F(s)}_{L^{2}} \nrm{\varphi(s)}_{L^{2}}.
\end{equation*}
First, ignoring the nonnegative term $c^{-} \nrm{\Db^{\infty} \varphi}_{L^{2}}^{2}$ and solving the resulting differential inequality, we obtain, for $s \in [0, s_{0}]$,
\begin{equation*}
	\nrm{\varphi(s)}_{L^{2}} \aleq e^{- c_{0}^{2} s}\nrm{\varphi_{0}}_{L^{2}} + \int_{0}^{s} e^{- c_{0}^{2} (s-s')} \nrm{F(s')}_{L^{2}} \, \ud s'.
\end{equation*}
On the other hand, ignoring instead $c_{0}^{2} \nrm{\varphi}_{L^{2}}^{2}$ and integrating the differential inequality on $[s, \min\set{s+1, s_{0}}]$, we obtain
\begin{equation*}
\left( \int_{s}^{\min\set{s+1, s_{0}}} \nrm{\bfD^{\infty} \varphi(\lmb; s')}_{L^{2}}^{2} \, \ud s'\right)^{\frac{1}{2}} \aleq \nrm{\varphi(s)}_{L^{2}} + \int_{s}^{\min\set{s+1, s_{0}}} \nrm{F(s')}_{L^{2}} \, \ud s'.
\end{equation*}
By \eqref{eq:hmhf-gwp-diff}, we also have
\begin{equation*}
\left( \int_{s}^{\min\set{s+1, s_{0}}} \nrm{\bfD \varphi(\lmb; s')}_{L^{2}}^{2} \, \ud s' \right)^{\frac{1}{2}}
\aleq \left( \int_{s}^{\min\set{s+1, s_{0}}} \nrm{\bfD^{\infty} \varphi(\lmb; s')}_{L^{2}}^{2} \right)^{\frac{1}{2}} \, \ud s' + \sup_{s' \in [s, \min\{s+1,s_0\}]} \eps_{\ast}^{\frac{1}{4}} \nrm{\varphi(\lmb; s')}_{L^{2}}.
\end{equation*}
Thus, we have bounded all the $L^{2}_{x}$ type norms on the LHS of \eqref{eq:l-hmhf-gwp-key}. 

It remains to bound the $L^{2}_{s} L^{\infty}_{x}$ norm of $\varphi$. We begin with the usual Bochner identity for the linearized harmonic map heat flow,
\begin{equation*}
	\rd_{s} \abs{\varphi}^{2} - \lap \abs{\varphi}^{2} + 2 \Db^{\ell} \varphi \overline{\Db_{\ell} \varphi} = 2 \tcv \Im(\psi^{\ell} \overline{\varphi}) \Im(\psi_{\ell} \overline{\varphi}) + 2 \Re(F \overline{\varphi}).
\end{equation*}
Thus, we obtain the relation
\begin{equation*}
	\rd_{s} \abs{\varphi} - \lap \abs{\varphi} \leq M \abs{\varphi} + \abs{F},
\end{equation*}
where the inequality is to be interpreted in the sense of distributions. Here, $M$ is a positive constant that only depends on $\nrm{\psi_{x}}_{L^{\infty}} = O(\nrm{\psi^{\infty}}_{L^{\infty}} + C \eps_{\ast}^{\frac{1}{2}})$. For any $s' \in [0, s_{0}]$ and $s \in [s', s_{0}]$, we have
\begin{equation*}
	(\rd_{s} - \lap) \left( e^{-M(s-s')} \abs{\varphi} \right) \leq e^{-M(s-s')} \abs{F},
\end{equation*}
so that by Duhamel's formula and the comparison principle,
\begin{equation*}
	\abs{\varphi(x, s)}
	\leq e^{M(s-s')} e^{(s-s') \lap} \abs{\varphi(x, s')} + \int_{s'}^{s} e^{M(s-s'')} e^{(s-s'') \lap} \abs{F(x, s'')} \abs{\varphi(x, s'')} \, \ud s''.
\end{equation*}
Then by choosing $s' = s-1$ and applying the $L^{2}_{s} L^{\infty}_{x}$ estimate \eqref{eq:lin-heat-L2Linfty}, we obtain
\begin{equation*}
	\left( \int_{s-1}^{s}\nrm{\varphi(s')}_{L^{\infty}}^{2}  \, \ud s' \right)^{\frac{1}{2}} \aleq_{M} \nrm{\varphi(s-1)}_{L^{2}} + \int_{s-1}^{s} \nrm{F(s')}_{L^{2}} \, \ud s'.
\end{equation*}
Combined with the previously established bound on $\nrm{\varphi(s)}_{L^{2}}$, the proof of \eqref{eq:l-hmhf-gwp-key} is complete.

\pfstep{Step~4: Closing the main bootstrap argument and completing the proof of Theorem~\ref{thm:hmhf-gwp}}
We apply the estimates in the previous step with $\varphi = \psi_{\lmb}$ and $F = 0$. Since $\abs{\rd_{\lmb} u} = \abs{\psi_{\lmb}}$ and $\abs{D_{x} \rd_{\lmb} u} = \abs{\Db_{x} \psi_{\lmb}}$, \eqref{eq:l-hmhf-gwp-key} translates to 
\begin{equation} \label{eq:hmhf-gwp-key}
	\nrm{\rd_{\lmb} u(\lmb; s)}_{L^{2}}^{2} + \int_{s}^{\min\set{s+1, s_{0}}} \bigl( \nrm{D \rd_{\lmb} u(\lmb; s')}_{L^{2}}^{2} + \nrm{\rd_{\lmb} u(\lmb; s')}_{L^{\infty}}^{2} \bigr) \, \ud s' \aleq e^{-2 c_{0}^{2} s} \nrm{\rd_{\lmb} u_{0}(\lmb)}_{L^{2}}^{2},
\end{equation}
where $s \in [1, s_{0}]$ and  $\lmb \in [0, 1]$. In particular, for every integer $n$ such that $1 \leq n \leq s_{0}-2$, there exists $s_{n; \lmb} \in [n, n+1)$ such that
\begin{equation*}
	\nrm{D \rd_{\lmb} u_{\lmb}(\lmb; s_{n; \lmb})}_{L^{2}} + \nrm{\rd_{\lmb} u_{\lmb}(\lmb; s_{n; \lmb})}_{L^{\infty}}  \aleq e^{- c_{0}^{2} n} \nrm{\rd_{\lmb} u_{0}(\lmb)}_{L^{2}}^{2}.
\end{equation*}
By the above estimates, the formula
\begin{equation*}
	\rd_{j} \rd_{\lmb} u^{A} = \rd_{\lmb} \rd_{j}  u^{A} = D_{j} \rd_{\lmb} u^{A} - S^{A}_{BC}(u) \rd_{\lmb} u^{B} \rd_{j} u^{C},
\end{equation*}
and the bootstrap assumption \eqref{eq:hmhf-gwp-btstrp-0}, it follows that
\begin{equation} \label{eq:hmhf-gwp-good-s}
	\nrm{\rd_{\lmb} u_{\lmb}(\lmb; s_{n; \lmb})}_{H^{1} \cap L^{\infty}} \aleq e^{- c_{0}^{2} n} \nrm{\rd_{\lmb} u_{0}(\lmb)}_{L^{2}}^{2}.
\end{equation}
Applying Proposition~\ref{prop:l-hmhf-ex} with $\rd_{\lmb} u(\lmb; s_{n; \lmb})$ as the initial data on $J = [s_{n; \lmb}, \min\set{s_{n; \lmb}+2, s_{0}}]$, we obtain 
\begin{align} 
	\nrm{\rd_{\lmb} u(\lmb; s)}_{H^{1} \cap L^{\infty}} \aleq e^{-c_{0}^{2} s} \nrm{\rd_{\lmb} u_{0}(\lmb)}_{L^{2}} \quad \hbox{ for } s \in [s_{n; \lmb}, \min\set{s_{n; \lmb}+2, s_{0}}], \label{eq:hmhf-gwp-good-s-1} \\
	\nrm{D \rd_{\lmb} u(\lmb; s)}_{L^{\infty}} \aleq e^{-c_{0}^{2} s} \nrm{\rd_{\lmb} u_{0}(\lmb)}_{L^{2}} \quad \hbox{ for } s \in [s_{n; \lmb}+1, \min\set{s_{n; \lmb}+2, s_{0}}], \label{eq:hmhf-gwp-good-s-2}
\end{align}
We remark that the implicit constant are made absolute here by taking $\epshf$ small enough, thanks to \eqref{eq:hmhf-gwp-btstrp-0}. Applying the above argument for all possible $n$, and using Proposition~\ref{prop:l-hmhf-ex} with sufficiently small $\epshf$ on $J = [0, 2]$, we improve the bootstrap assumption \eqref{eq:hmhf-gwp-btstrp}. Moreover, it follows that $u - Q = \int_{0}^{1} \rd_{\lmb} u(\lmb) \, \ud \lmb$ satisfies
\begin{equation*}
	\nrm{(u - Q)(\lmb; s)}_{H^{1} \cap L^{\infty}} \aleq e^{-c_{0}^{2} s} \int_{0}^{1} \nrm{\rd_{\lmb} u_{0}(\lmb)}_{L^{2}} \, \ud \lmb \aleq e^{-c_{0}^{2} s} \epshf,
\end{equation*}
which also improves \eqref{eq:hmhf-gwp-btstrp-0}. It is now a standard procedure to set up a continuous induction on $s_{0}$ to conclude the global existence of $u$ and prove \eqref{eq:hmhf-gwp}; we omit the obvious details.

\pfstep{Step~5: Completing the proof of Proposition~\ref{prop:l-hmhf-gwp}}
Finally, we prove Proposition~\ref{prop:l-hmhf-gwp}. Bound \eqref{eq:l-hmhf-gwp-0} follows from \eqref{eq:l-hmhf-gwp-key} with $\lmb = 1$, since the bootstrap assumptions \eqref{eq:hmhf-gwp-btstrp-0} and \eqref{eq:hmhf-gwp-btstrp} had been closed in Step~4. Bound \eqref{eq:l-hmhf-gwp-1} follows immediately from Proposition~\ref{prop:l-hmhf-ex} when $s \leq 1$. When $s > 1$, by \eqref{eq:l-hmhf-gwp-0} and the pigeonhole principle, there exists $s_{0} \in [s-1, s]$ such that
\begin{equation*}
	\nrm{\phi(s_{0})}_{H^{1} \cap L^{\infty}} \aleq e^{-c_{0}^{2} s} \nrm{\phi_{0}}_{L^{2}} + \int_{0}^{s} e^{-c_{0}^{2}(s - s')} \nrm{f(s')}_{L^{2}} \, \ud s'.
\end{equation*}
Now \eqref{eq:l-hmhf-gwp-1} follows from an application of Proposition~\ref{prop:l-hmhf-ex} with the initial data $\phi(s_{0})$. \qedhere
\end{proof}

\subsection{Construction of caloric gauge} \label{subsec:caloric}
Theorem~\ref{thm:hmhf-gwp} states that a map $u_{0}$ that is sufficiently close to a weakly linearly stable harmonic map $Q$ in $H^{1}$ leads to a harmonic map heat flow evolution $u(s)$ that asymptotically converge to $Q$. The goal of this subsection is to construct a caloric gauge \cite{Tao04} for $u$, which is simply the orthonormal frame for the pullback bundle $u^{\ast} T \tgmfd$ on $\bbH^{2}_{x} \times [0, \infty)_{s}$ constructed by parallel-transporting a suitable limiting frame $e^{\infty}$ at $s = \infty$. The precise definition is as follows.

\begin{definition} [Caloric gauge] \label{def:caloric-gauge}
Let $u(x, s) : \bbH^{2} \times [0, \infty) \to \tgmfd$ be a harmonic map heat flow given by Theorem~\ref{thm:hmhf-gwp}, which converges to a harmonic map $Q$ as $s \to \infty$. Given a smooth frame $e^{\infty}(x)$ for $Q^{\ast} T \tgmfd$, we say that a smooth frame $e(x, s)$ for $u^{\ast} T \tgmfd$ defines a \emph{caloric gauge} for $u$ with limiting frame $e^{\infty}$ if the following properties hold:
\begin{enumerate}
\item The frame $e(x, s)$ satisfies the \emph{heat-temporal condition}
\begin{equation*}
	D_{s} e(x, s) = 0 \quad \hbox{ for } (x, s) \in \bbH^{2} \times [0, \infty).
\end{equation*}
\item For all $x \in \bbH^{2}$, we have $e(x, s) \to e^{\infty}(x)$ as $s \to \infty$.
\end{enumerate}
We call such a frame $e(x, s)$ the \emph{caloric frame} with limiting frame $e^{\infty}$.
\end{definition}

The main result of this subsection is the following.
\begin{proposition} \label{prop:caloric-gauge}
Let $u_{0} : \bbH^{2} \to \tgmfd$ satisfy $\nrm{u_{0} - Q}_{H^{1}} < \epshf$ and $u_{0} - Q \in H^{\infty}$.
Let $u(x, s) : \bbH^{2} \times [0, \infty) \to \tgmfd$ be a harmonic map heat flow with $u(x, 0) = u_{0}(0)$ given by Theorem~\ref{thm:hmhf-gwp}. Let $e^{\infty}(x)$ be a Coulomb frame for $Q^{\ast} T \tgmfd$ as in Subsection~\ref{subsec:hm}. Then the following statements hold: 
\begin{enumerate}
\item There exists a unique caloric frame $e(x, s)$ with limiting frame $e^{\infty}$, which depends smoothly on $(x, s) \in \bbH^{2} \times [0, \infty)$.

\item Let $I \subset \bbR$ be an interval, and let $u_{0} : I \times \bbH^{2} \to \tgmfd$ satisfy the following properties:
\begin{itemize}
\item $\sup_{t \in I} \nrm{u_{0}(t) - Q}_{H^{1}} < \epshf$; 
\item $\rd_{t}^{k} (u_{0}(t) - Q) \in H^{\infty}$ for every $k = 0, 1, 2, \ldots$ and $t \in I$.
\end{itemize}
Then the unique caloric frames $e(t, x, s)$ defined for each $t \in I$ depends smoothly on $(t, x, s) \in I \times \bbH^{2} \times [0, \infty)$.
\end{enumerate}
\end{proposition}
\begin{proof}
In what follows, we suppress the dependence of constants on the initial map $u_{0}$.
We begin with part~(1), where $u_{0}$ has no $t$-dependence. We need to solve the following system of ODEs:
\begin{equation} \label{eq:caloric-ODE}
\left\{
\begin{aligned}
	\rd_{s} e^{A}_{B} + S^{A}_{C D}(u) \rd_{s} u^{C} e^{D}_{B} & = 0, \\
	\lim_{s \to \infty} e^{A}_{B}(x, s) & = (e^{\infty})^{A}_{B}(x).
\end{aligned}
\right.
\end{equation}
As a consequence of Proposition~\ref{prop:hmhf-lwp} for $s \leq 1$ and Corollary~\ref{cor:hmhf-gwp-higher} for $s \geq 1$, we see that, for $k = 0, 1, 2, \ldots$ and $s \geq 0$,
\begin{align*}
\nrm{\nb^{(k)} u(s)}_{L^{\infty}} 
& \aleq \nrm{u(s) - Q}_{H^{k+2}} + \nrm{\nb^{(k)} Q}_{L^{\infty}} 
	\aleq 1, \\
\nrm{\nb^{(k)} \rd_{s} u(s)}_{L^{\infty}} 
& \aleq \nrm{\rd_{s} u(s)}_{H^{k+2}}
	\aleq e^{-c_{0}^{2} s}. 
\end{align*}
Thus, by standard ODE theory, \eqref{eq:caloric-ODE} may be solved and part~(1) follows. 

To prove part~(2), it suffices to establish the following additional (by no means optimal) bounds: For $n = 1, 2, \ldots$, $k = 0, 1, 2, \ldots$ and $s \geq 0$, 
\begin{equation} \label{eq:hmhf-t-reg}
\nrm{\rd_{t}^{n} u(s)}_{H^{k+4}} +
\nrm{\rd_{t}^{n} \rd_{s} u(s)}_{H^{k+2}} \aleq e^{-c_{0}^{2} s}, 
\end{equation}
where the implicit constant depends on $k, n$ and $\sum_{n'=0}^{n} \nrm{\rd_{t}^{n'} u_{0}}_{H^{k+4}}$. 

Bound \eqref{eq:hmhf-t-reg} is proved by induction on $n$. As $\rd_{t} u$ solves \eqref{eq:l-hmhf-ex} with $f = 0$, the base case $n = 1$ follows immediately from Propositions~\ref{prop:l-hmhf-ex}, \ref{prop:l-hmhf-gwp} and Corollary~\ref{cor:hmhf-gwp-higher}. For $n > 1$, $\rd_{t}^{n} u$ solves
\begin{align*}
& (D_{s} - D^{\ell} D_{\ell}) \rd_{t}^{n} u^{A} - \tensor{R}{_{CD}^{A}_{B}}(u) \rd_{\ell} u^{B} \rd_{t}^{n} u^{C} \rd^{\ell} u^D \\ 
& = - [\rd_{t}^{n-1}, D_{s} - D^{\ell} D_{\ell}] \rd_{t} u^{A} + [\rd_{t}^{n-1}, \tensor{R}{_{CD}^{A}_{B}}(u) \rd_{\ell} u^{B} \rd^{\ell} (\rd_{t}^{\ell} u)  \phi^{A}  u^{D}] (\rd_{t} u)^{C},
\end{align*}
which is \eqref{eq:l-hmhf-ex} with an $f^{A}$ that involves $u, \ldots, \rd_{t}^{\ell-1} u$. The desired bound again follows by applying Propositions~\ref{prop:l-hmhf-ex}--\ref{prop:l-hmhf-gwp} and Corollary~\ref{cor:hmhf-gwp-higher}, where $f$ is estimated by a suitable induction hypothesis on $u, \ldots, \rd_{t}^{\ell-1} u$. We leave the straightforward details to the reader. \qedhere
\end{proof}

We now employ the moving frame formalism with respect to the caloric gauge $e(x, s)$.
We define the \emph{heat tension field}
\begin{equation} 
 \psi_s = \Db^\ell \psi_\ell.
\end{equation}
We also introduce the decompositions
\begin{equation} \label{equ:linearization_via_caloric_gauge}
 \begin{aligned}
  A_\ell(s_0, t, x) &= - \int_{s_0}^\infty \partial_s A_\ell \, \ud s + A_\ell \vert_{s=\infty} =: \ringA_\ell + A^\infty_\ell, \\
  \psi_\ell(s_0, t, x) &= - \int_{s_0}^\infty \Db_\ell \psi_s \, \ud s + \psi_\ell \vert_{s=\infty} =: \ringpsi_\ell + \psi^\infty_\ell,
 \end{aligned}
\end{equation}
where we have the reconstruction formulas
\begin{align*}
 \ringA_\ell(s_0) &= - \int_{s_0}^\infty \partial_s A_\ell \, \ud s = - \int_{s_0}^\infty (\partial_s A_\ell - \partial_\ell A_s) \, \ud s = -  \tcv \int_{s_0}^\infty \Im (\psi_s \overline{\psi_\ell}) \, \ud s, \\
 \ringpsi_\ell(s_0) &= - \int_{s_0}^\infty \partial_s \psi_\ell \, \ud s = - \int_{s_0}^\infty \Db_\ell \psi_s \, \ud s = - \int_{s_0}^\infty \partial_\ell \psi_s \, \ud s - i \int_{s_0}^\infty A_\ell \psi_s \, \ud s \\
 &= - \int_{s_0}^\infty \partial_\ell \psi_s \, \ud s - i \int_{s_0}^\infty \biggl( \int_s^\infty (-\tcv) \Im(\psi_s \overline{\psi_\ell}) \, \ud s' + A_\ell^\infty \biggr) \psi_s \, \ud s.
\end{align*}

\subsection{Forward-in-$s$ bounds in caloric gauge} \label{subsec:forward-caloric}
In this subsection, we suppress the dependence of constants on $Q$, $\tilde{\Omg}$, the Coulomb frame $e^{\infty}$ for $Q^{\ast} T \tgmfd$ and $A^{\infty}$ introduced in Subsection~\ref{subsec:hm}.

The primary aim of this subsection is to translate bounds for the harmonic map heat flow $u$ in the extrinsic coordinates to those for the components $\psi_{s}, \psi_{x}$ in the caloric gauge constructed in Proposition~\ref{prop:caloric-gauge}. The main result is as follows: 

\begin{proposition} \label{prop:forward-caloric-psi}
Assume that $u_{0}$ obeys the smallness condition \eqref{eq:hmhf-wp-hyp} and let $0 < \dlt < \frac{1}{2}$. In the caloric gauge constructed in Proposition~\ref{prop:caloric-gauge}, the following bounds hold for $k = 0, 1, 2, \ldots$:
\begin{align} 
\nrm{m(s) s^{\frac{k+1}{2}} \nb^{(k)} \psi_{s}}_{L^{2}_{\ds} \cap L^{\infty}_{\ds}((0, \infty); L^{2})} 
	& \aleq \nrm{u_{0} - Q}_{H^{1+2\dlt}}, \label{eq:forward-caloric-psi-s} \\
\nrm{\ringPsi}_{L^{\infty}_{\ds}((0, \infty); L^{2})} + \nrm{m(s) s^{\frac{k+1}{2}} \nb^{(k+1)} \ringPsi}_{L^{2}_{\ds} \cap L^{\infty}_{\ds}((0, \infty); L^{2})} 
	& \aleq \nrm{u_{0} - Q}_{H^{1+2\dlt}}, \label{eq:forward-caloric-Psi}
\end{align}
where the implicit constants depend on $k$ and an upper bound for $\nrm{u_{0} - Q}_{H^{1+2\dlt}}$. Moreover,
\begin{align}
\nrm{m(s) s^{\frac{k+1}{2}} \nb^{(k)} \Omg \psi_{s}}_{L^{2}_{\ds} \cap L^{\infty}_{\ds}((0, \infty); L^{2})} & \aleq \nrm{u_{0} - Q}_{H^{1+2\dlt}} + \nrm{\calD_{\Omg} u_{0}}_{H^{1+2\dlt}},  \label{eq:forward-caloric-Omg-psi-s}\\
	\nrm{\nb \ringpsi_{\Omg}}_{L^{\infty}_{\ds}((0, \infty); L^{2})}
		+ \nrm{m(s) s^{\frac{k+1}{2}} \nb^{(k+2)} \ringpsi_{\Omg}}_{L^{2}_{\ds} \cap L^{\infty}_{\ds}((0, \infty); L^{2})}
	& \aleq \nrm{u_{0} - Q}_{H^{1+2\dlt}} + \nrm{\calD_{\Omg} u_{0}}_{H^{1+2\dlt}}, \label{eq:forward-caloric-Psi-Omg} \\
	\nrm{\calL_{\Omg} \ringPsi}_{L^{\infty}_{\ds}((0, \infty); L^{2})} + \nrm{m(s) s^{\frac{k+1}{2}} \nb^{(k+1)} \calL_{\Omg} \ringPsi}_{L^{2}_{\ds} \cap L^{\infty}_{\ds}((0, \infty); L^{2})}
	& \aleq \nrm{u_{0} - Q}_{H^{1+2\dlt}} + \nrm{\calD_{\Omg} u_{0}}_{H^{1+2\dlt}}, \label{eq:forward-caloric-Omg-Psi}
\end{align}
where the implicit constants depend on $k$ and upper bounds for $\nrm{u_{0} - Q}_{H^{1+2\dlt}}$ and $\nrm{\calD_{\Omg} u_{0}}_{H^{1+2\dlt}}$.
\end{proposition}

Proposition~\ref{prop:forward-caloric-psi} is proved by translating bounds for the harmonic map heat flow $u$ in the extrinsic formulation to those for $\psi_{s}, \psi_{x}$ in the caloric gauge. The following bounds for the connection $1$-form $A$ in the caloric gauge, whose proof is the second goal of this subsection, play a key role in going back and forth between the extrinsic formulation and the caloric gauge.

\begin{proposition} \label{prop:forward-caloric-A}
Assume that $u_{0}$ obeys the smallness condition \eqref{eq:hmhf-wp-hyp} and that $0 < \dlt < \frac{1}{4}$. In the caloric gauge constructed in Proposition~\ref{prop:caloric-gauge}, the following bound holds for $k = 0, 1, 2, \ldots$:
\begin{align}
%	\nrm{s^{\frac{k}{2}} \nb^{(k)} \ringA}_{L^{\infty}_{\ds}((0, \infty); L^{2})}
%	+ \nrm{s^{\frac{k+1}{2}} \nb^{(k)} \ringA}_{L^{\infty}_{\ds}((0, \infty); L^{\infty})} &\aleq_{k} \nrm{u_{0} - Q}_{H^{1}}, \label{eq:forward-caloric-A-1} \\
	\nrm{\ringA}_{L^{\infty}_{\ds}((0, \infty); L^{2})}
	+ \nrm{m(s) s^{\frac{k+1}{2}} \nb^{(k+1)} \ringA}_{L^{2}_{\ds} \cap L^{\infty}_{\ds}((0, \infty); L^{2})}
	&\aleq \nrm{u_{0} - Q}_{H^{1+2\dlt}}, \label{eq:forward-caloric-A}
\end{align}
where the implicit constant depends on $k$ and an upper bound for $\nrm{u_{0} - Q}_{H^{1+2\dlt}}$.  Moreover,
\begin{align}
	\nrm{\nb \ringA_{\Omg}}_{L^{\infty}_{\ds} ((0, \infty); L^{2})} 
	+ \nrm{m(s) s^{\frac{k+1}{2}} \nb^{(k+2)} \ringA_{\Omg}}_{L^{2}_{\ds} \cap L^{\infty}_{\ds} ((0, \infty); L^{2})} 
	& \aleq \nrm{u_{0} - Q}_{H^{1+2\dlt}} + \nrm{\calD_{\Omg} u}_{H^{1+2\dlt}}, \label{eq:forward-caloric-A-Omg} \\
	\nrm{\calL_{\Omg} \ringA}_{L^{\infty}_{\ds}((0, \infty); L^{2})}
	+ \nrm{m(s) s^{\frac{k+1}{2}} \nb^{(k+1)} \calL_{\Omg} \ringA}_{L^{2}_{\ds} \cap L^{\infty}_{\ds} ((0, \infty); L^{2})}
	& \aleq
		\nrm{u_{0} - Q}_{H^{1+2\dlt}} + \nrm{\calD_{\Omg} u}_{H^{1+2\dlt}}, \label{eq:forward-caloric-Omg-A}
\end{align}
where the implicit constants depend on $k$ and upper bounds for $\nrm{u_{0} - Q}_{H^{1+2\dlt}}$ and $\nrm{\calD_{\Omg} u_{0}}_{H^{1+2\dlt}}$.
\end{proposition}

Both Propositions~\ref{prop:forward-caloric-psi} and \ref{prop:forward-caloric-A} are based on long time bounds for the harmonic map heat flow in the extrinsic formulation, which follow from the results in Subsections~\ref{subsec:hmhf-lwp} and \ref{subsec:hmhf-stab}.
\begin{proposition} \label{prop:forward-extrinsic-cov}
Assume that $u_{0}$ obeys the smallness condition \eqref{eq:hmhf-wp-hyp} and that $0 < \dlt < \frac{1}{2}$. The following bounds hold $k = 0, 1, 2, \ldots$:
\begin{align} 
	\nrm{m(s) s^{\frac{k+1}{2}} D^{(k)} \rd_{s} u}_{L^{2}_{\ds} \cap L^{\infty}_{\ds} ((0, \infty); L^{2})} 
	& \aleq \nrm{u_{0} - Q}_{H^{1+2\dlt}}, \label{eq:forward-dsu} \\
	\nrm{\nb (u-Q)}_{L^{\infty}_{\ds} ((0, \infty); L^{2})} 
	+ \nrm{m(s) s^{\frac{k+1}{2}} D^{(k+1)} \nb (u-Q)}_{L^{2}_{\ds} \cap L^{\infty}_{\ds} ((0, \infty); L^{2})} 
	& \aleq \nrm{u_{0} - Q}_{H^{1+2\dlt}}. \label{eq:forward-du}
\end{align}
where the implicit constants depend on $k$ and an upper bound for $\nrm{u_{0} - Q}_{H^{1+2\dlt}}$. Moreover,
\begin{align} 
	\nrm{m(s) s^{\frac{k+1}{2}} D^{(k)} D_{s} \calD_{\Omg} u}_{L^{2}_{\ds} \cap L^{\infty}_{\ds} ((0, \infty); L^{2})} 
	& \aleq \nrm{\calD_{\Omg} u_{0}}_{H^{1+2\dlt}}, \label{eq:forward-ds-Omg-u} \\
	\nrm{D \calD_{\Omg} u}_{L^{2}_{\ds} \cap L^{\infty}_{\ds} ((0, \infty); L^{2})} 
	+ \nrm{m(s) s^{\frac{k+1}{2}} D^{(k+2)} \calD_{\Omg} u}_{L^{2}_{\ds} \cap L^{\infty}_{\ds} ((0, \infty); L^{2})} 
	& \aleq \nrm{\calD_{\Omg} u_{0}}_{H^{1+2\dlt}}, \label{eq:forward-Omg-u}
\end{align}
where the implicit constants depend on $k$ and an upper bound for $\nrm{u_{0} - Q}_{H^{1+2\dlt}}$.
\end{proposition}
\begin{remark} [Higher regularity] \label{rem:forward-high-reg}
With a straightforward modification of the argument given below, one may show that all bounds in Propositions~\ref{prop:forward-caloric-psi}, \ref{prop:forward-caloric-A} and \ref{prop:forward-extrinsic-cov} hold with $L^{2}$ and $H^{1+2\dlt}$ replaced by $H^{\sgm-1}$ and $H^{\sgm+2 \dlt}$, respectively, where $\sgm$ is any positive integer. This observation, along with Remark~\ref{rem:backward-high-reg} below, is used later to justify a continuous induction argument for the Schr\"odinger map evoluion.
\end{remark}

\begin{remark} 
In our proof below, we freely exploit the initial regularity $u_{0} - Q \in H^{1+2\dlt}$ with $\dlt > 0$ to simplify the argument. 
However, many of the the above estimates are true when $\dlt = 0$ as well. Such bounds may be proved by working with Bochner-Weitzenb\"ock type identities (see, for instance, \cite[Paper~IV,~Section~4]{Tao37}) for $s \aleq 1$ and exploiting the exponential $s$-decay in Theorem~\ref{thm:hmhf-gwp} and Proposition~\ref{prop:l-hmhf-gwp} for $s \ageq 1$.
\end{remark}

To prove the above results, we begin with a simple lemma for relating $\nb$ with $D$ or $\bfD$, whose obvious proof we omit:
\begin{lemma} \label{lem:d-switch}
Let $J$ be an interval and $B$ a $1$-form on $\bbH^{2} \times J$ satisfying the bound
\begin{equation*}
	\sum_{k'=0}^{k} \nrm{s^{\frac{k'+1}{2}} \nb^{(k')} B}_{L^{\infty}_{\ds} (J; L^{\infty})} \leq M.
\end{equation*}
Then for any function $f$ on $\bbH^{2} \times J$ and $1 \leq p, q \leq \infty$, we have
\begin{equation*}
	\sum_{k'=0}^{k} \nrm{s^{\frac{k'}{2}}\nb^{(k')} f}_{L^{q}_{\ds} (J; L^{p})} \simeq_{M} \sum_{k'=0}^{k} \nrm{s^{\frac{k'}{2}}(\nb + B)^{(k')} f}_{L^{q}_{\ds} (J; L^{p})}.
\end{equation*}
\end{lemma}

Now we establish Propositions~\ref{prop:forward-caloric-psi}, \ref{prop:forward-caloric-A} and \ref{prop:forward-extrinsic-cov} in the reverse order. In what follows, we freely use the embeddings $H^{\sgm} \hookrightarrow H^{\sgm'}$ for $\sgm' < \sgm$ and $H^{\sgm} \hookrightarrow L^{\infty}$ for $\sgm > 1$.
Moreover, we suppress the dependence of the implicit constants on $k$ and $\nrm{u_{0} - Q}_{H^{1+2\dlt}}$.
\begin{proof}[Proof of Proposition~\ref{prop:forward-extrinsic-cov}]
We proceed in two steps. 

\pfstep{Step~1: Non-covariant bounds}
The first step is to check, using the results in Subsections~\ref{subsec:hmhf-lwp} and \ref{subsec:hmhf-stab}, that all the bounds stated in Proposition~\ref{prop:forward-extrinsic-cov} hold when $D$ is replaced by $\nb$. 
\begin{itemize}
\item {\it  \eqref{eq:forward-du} with $D$ replaced by $\nb$.} We use \eqref{eq:hmhf-lwp-k} with $\sgm = 1 + 2 \dlt$ for $s \leq 1$, then \eqref{eq:hmhf-gwp} and \eqref{eq:hmhf-gwp-k} for $s \geq 1$.

\item {\it  \eqref{eq:forward-dsu} with $D$ replaced by $\nb$.} For the $L^{2}_{\ds} ([0, \infty); L^{2})$ norm on the LHS, we use \eqref{eq:hmhf-lwp-k} with $\sgm = 1 + 2 \dlt$ for $s \leq 1$, then \eqref{eq:hmhf-gwp} and \eqref{eq:hmhf-gwp-k} for $s \geq 1$, as before. To control the corresponding $L^{\infty}_{\ds} ([0, \infty); L^{2})$ norm, note that $\rd_{s} u$ obeys \eqref{eq:l-hmhf-ex} with $f = 0$, so that
\begin{equation*}
	\nrm{\rd_{s} u(s)}_{H^{k}} \aleq \begin{cases}
		\nrm{\rd_{s} u}_{L^{2}_{\ds}([\frac{s}{4}, \frac{s}{2}]; H^{k})} & \hbox{ when } s \leq 1, \\
		\nrm{\rd_{s} u}_{L^{2}_{\ds}([s-1, s-\frac{1}{2}]; H^{k})} & \hbox{ when } s \geq 1, 
\end{cases}
\end{equation*}
by the pigeonhole principle (to select a good $s'$ for which $\nrm{\rd_{s} u(s')}_{H^{k}}$ is controlled by the RHS), \eqref{eq:l-hmhf-ex-sgm}, \eqref{eq:hmhf-gwp} and Remark~\ref{rem:hmhf-small-data} (note that the length of the $s$-interval on the RHS is kept to be $O(1)$).

\item {\it \eqref{eq:forward-Omg-u} with $D$ replaced by $\nb$.} Note that $\calD_{\Omg} u$ obeys \eqref{eq:l-hmhf-ex} with $f = 0$. Thus, the desired estimate follows from \eqref{eq:l-hmhf-ex-sgm-k} with $\sgm = 1 + 2 \dlt$ for $s \leq 1$, then \eqref{eq:l-hmhf-gwp-1} and \eqref{eq:l-hmhf-gwp-k} for $s \geq 1$.

\item {\it \eqref{eq:forward-ds-Omg-u} with $D$ replaced by $\nb$.} Here, we do \emph{not} replace $D_{s}$ by $\rd_{s}$, i.e., we prove
\begin{equation*}
	\nrm{m(s) s^{\frac{k+1}{2}} \nb^{(k)} D_{s} \calD_{\Omg} u}_{L^{2}_{\ds} \cap L^{\infty}_{\ds}((0, \infty); L^{2})} \aleq \nrm{\calD_{\Omg} u_{0}}_{H^{1+2\dlt}}.
\end{equation*}
Recalling that $\calD_{\Omg} u$ obeys \eqref{eq:l-hmhf-ex} with $f = 0$, we see that it suffices to bound
\begin{equation*}
	\nrm{m(s) s^{\frac{k+1}{2}} \nb^{(k)} D^{\ell} D_{\ell} \calD_{\Omg} u}_{L^{2}_{\ds} \cap L^{\infty}_{\ds}((0, \infty); L^{2})},  \nrm{m(s) s^{\frac{k+1}{2}} \nb^{(k)} \tensor{R}{_{CD}^{A}_{B}}(u) \rd_{\ell} u^{B} \calD_{\Omg} u^{C} \rd^{\ell} u^{D}}_{L^{2}_{\ds} \cap L^{\infty}_{\ds}((0, \infty); L^{2})} 
\end{equation*}
by $\nrm{\calD_{\Omg} u_{0}}_{H^{1+2\dlt}}$ up to a constant depending on $k$ and $\nrm{u_{0}-Q}_{H^{1+2\dlt}}$.  Recalling that $D_{\ell} \phi^{A} = \rd_{\ell} \phi^{A} + S^{A}_{BC}(u) \rd_{j} u^{B} \phi^{C}$, the latter bounds follow from the previously proved bounds on $\calD_{\Omg} u$ and $\nb(u - Q)$, as well as Propositions~\ref{prop:frac-leib} and \ref{prop:frac-moser}; we omit the details.
\end{itemize}

\pfstep{Step~2: Conversion to covariant bounds}
To conclude the proof, we pass from the bounds involving $\nb$ proved in Step~1 to the bounds claimed in the proposition using Lemma~\ref{lem:d-switch}. More precisely, we write $D = \nb + B$, where
\begin{equation*}
	(B_{j})^{C}_{D} = S^{C}_{DE}(u) \rd_{j} u^{E}
	= \left( \tilde{S}^{C}_{D E}(u-Q; x) + S^{C}_{D E}(Q) \right) \left( \rd_{j} (u - Q)^{E} + \rd_{j} Q^{E} \right).
\end{equation*}
By Step~1, it is straightforward to show that
\begin{equation*}
	\nrm{s^{\frac{k+1}{2}} \nb^{(k)} B}_{L^{\infty}_{\ds} ((0, \infty); L^{\infty})} \aleq 1.
\end{equation*}
Then \eqref{eq:forward-dsu} and \eqref{eq:forward-ds-Omg-u} follow immediately from Step~1 and Lemma~\ref{lem:d-switch}. For the remaining estimates, we need to prove, for every $k \geq 0$,
\begin{align*}
	\nrm{m(s) s^{\frac{k+1}{2}} D \nb (u-Q)}_{L^{\infty}_{\ds} \cap L^{2}_{\ds}((0, \infty); H^{k})} 
	& \aleq \nrm{u_{0} - Q}_{H^{1+2\dlt}}, \\
	\nrm{m(s) s^{\frac{k+1}{2}} D^{(2)} \calD_{\Omg} u}_{L^{2}_{\ds} \cap L^{\infty}_{\ds} ((0, \infty); H^{k})} 
	& \aleq \nrm{\calD_{\Omg} u_{0}}_{H^{1+2\dlt}}.
\end{align*}
The last bound was already addressed in Step~1 (see the proof of \eqref{eq:forward-ds-Omg-u}); the first bound follows in a similar manner. \qedhere
\end{proof}

\begin{proof}[Proof of Proposition~\ref{prop:forward-caloric-A}]
\pfstep{Step~1: Covariant bounds for $\psi$}
First, we claim that the following bounds hold for $k = 0 , 1, 2, \ldots$:
\begin{align}
	\nrm{m(s) s^{\frac{k+1}{2}} \Db^{(k)} \psi_{s}}_{L^{2}_{\ds} \cap L^{\infty}_{\ds} ((0, \infty); L^{2})} & \aleq \nrm{u_{0} - Q}_{H^{1+2\dlt}}, \label{eq:forward-psi-s-cov} \\
	\nrm{m(s) s^{\frac{k+1}{2}} \Db^{(k)} \Psi}_{L^{\infty}_{\ds} ((0, \infty); L^{\infty})} & \aleq 1, \label{eq:forward-psi-cov} \\
	\nrm{m(s) s^{\frac{k+1}{2}} \Db^{(k)} \bfD_{\Omg} \psi_{s}}_{L^{2}_{\ds} \cap L^{\infty}_{\ds} ((0, \infty); L^{2})} & \aleq \nrm{u_{0} - Q}_{H^{1+2\dlt}} + \nrm{\calD_{\Omg} u_{0}}_{H^{1+2\dlt}}, \label{eq:forward-Omg-psi-s-cov} \\
	\nrm{\psi_{\Omg}}_{L^{\infty}_{\ds}((0, \infty); L^{\infty})} + \nrm{m(s) s^{\frac{k+1}{2}} \Db^{(k+1)} \psi_{\Omg}}_{L^{\infty}_{\ds} ((0, \infty); L^{\infty})} & \aleq_{\nrm{\calD_{\Omg} u_{0}}_{H^{1+2\dlt}}} 1. \label{eq:forward-psi-Omg-cov}
\end{align}
Indeed, \eqref{eq:forward-psi-s-cov} and \eqref{eq:forward-psi-cov} follow rather immediately from Proposition~\ref{prop:forward-extrinsic-cov} and the identities
\begin{equation*}
	e \Db^{(k)} \psi_{s} = D^{k} \rd_{s} u, \quad
	e \Db^{(k)} \psi_{j} = D^{k} \rd_{j} u.	
\end{equation*}
For \eqref{eq:forward-Omg-psi-s-cov} and \eqref{eq:forward-psi-Omg-cov}, we begin with the identities
\begin{equation*}
	e \Db^{(k)} \Db_{\Omg} \psi_{s} 
	= D^{(k)} D_{s} \calD_{\Omg} u - D^{(k)} D_{s} \tilde{\Omg}(u), \quad
	e \Db^{(k)} \psi_{\Omg} = D^{(k)} \calD_{\Omg} u - D^{(k)}\tilde{\Omg}(u).
\end{equation*}
For the first terms, we directly use Proposition~\ref{prop:forward-extrinsic-cov}. For the last terms, we use the Leibniz rule and Proposition~\ref{prop:forward-extrinsic-cov} to bound
\begin{align*}
	\nrm{m(s) s^{\frac{k+1}{2}} D^{(k)} D_{s} \tilde{\Omg}(u)}_{L^{2}_{\ds} \cap L^{\infty}_{\ds}((0, \infty); L^{2})}
	& \aleq \nrm{m(s) s^{\frac{k+1}{2}} D^{(k)} (\nb_{A} \tilde{\Omg}(u) \rd_{s} u^{A})}_{L^{2}_{\ds} \cap L^{\infty}_{\ds}((0, \infty); L^{2})}
	\aleq \nrm{u_{0} - Q}_{H^{1+2\dlt}}, \\
	\nrm{m(s) 2^{\frac{k}{2}} D^{(k)}\tilde{\Omg}(u)}_{L^{\infty}_{\ds}((0, \infty); L^{\infty})}
	& \aleq 1.
\end{align*}
Finally, the term $\nrm{\psi_{\Omg}}_{L^{\infty}_{\ds}((0, \infty); L^{\infty}}$ on the LHS of \eqref{eq:forward-psi-Omg-cov} can be handled directly using the identity $e \psi_{\Omg} = \calD_{\Omg} u - \tilde{\Omg}(u)$ and Proposition~\ref{prop:forward-extrinsic-cov}.

\pfstep{Step~2: Bounds for $\ringA$}
The starting point for the proofs of \eqref{eq:forward-caloric-A}, \eqref{eq:forward-caloric-A-Omg} and \eqref{eq:forward-caloric-Omg-A} are, respectively, the formulas
\begin{align*}
	\ringA_{j} (s) & = - \int_{s}^{\infty} \tcv \Im(\psi_{s} \overline{\psi_{j}})(s') \, \ud s',  \\
	\ringA_{\Omg}(s) &= - \int_{s}^{\infty} \tcv \Im(\psi_{s} \overline{\psi_{\Omg}})(s') \, \ud s', \\
	\calL_{\Omg} \ringA_{j}(s)
	& = - \int_{0}^{\infty} \tcv \Im( \Db_{\Omg} \psi_{s} \overline{\psi_{j}} + \psi_{s} \overline{\Db_{j} \psi_{\Omg}} )(s') \, \ud s'.
\end{align*}
The first two formulas are obvious from Subsection~\ref{subsec:caloric}. To verify the last formula, we start with
\begin{align*}
	\calL_{\Omg} \ringA_{j}(s)
	= - \int_{0}^{\infty} \tcv \Im( \Omg \psi_{s} \overline{\psi_{j}} + \psi_{s} \overline{\calL_{\Omg} \psi_{j}} )(s') \, \ud s'.
\end{align*}
Then since
\begin{equation*}
	\Omg \psi_{s} = \Db_{\Omg} \psi_{s} - i A_{\Omg} \psi_{s}, \quad
	\calL_{\Omg} \psi_{j} 
%	= \iota_{\rd_{j}} \iota_{\Omg} \ud \psi+ \rd_{j} \psi_{\Omg}
%	= - \iota_{\rd_{j}} \iota_{\Omg} (i A \wedge \psi) + \rd_{j} \psi_{\Omg}
%	= i A_{j} \psi(\Omg) - i A_{\Omg} \psi_{j}
	= \Db_{j} (\psi_{\Omg}) - i A_{\Omg} \psi_{j},
\end{equation*}
the desired identity follows.

By the above formulas and Step~1, the desired bounds for $\nrm{\ringA}_{L^{\infty}_{\ds} ((0, \infty); L^{2})}$, $\nrm{\nb \ringA_{\Omg}}_{L^{\infty}_{\ds}((0, \infty); L^{2})}$ and $\nrm{\calL_{\omg} \ringA}_{L^{\infty}_{\ds}((0, \infty); L^{2})}$ follow.  The remainder of \eqref{eq:forward-caloric-A}, \eqref{eq:forward-caloric-A-Omg} and \eqref{eq:forward-caloric-Omg-A} follows from the Leibniz rule 
\begin{equation*}
\rd_{j} \Im(\varphi_{1} \overline{\varphi_{2}}) = \Im(\Db_{j} \varphi_{1} \overline{\varphi_{2}}) + \Im(\varphi_{1} \overline{\Db_{j} \varphi_{2}}), 
\end{equation*}
as well as Schur's test (with respect to $L^{2}_{\ds}$ and $L^{\infty}_{\ds}$) and the covariant bounds from Step~1. \qedhere
\end{proof}

\begin{proof}[Proof of Proposition~\ref{prop:forward-caloric-psi}]
\pfstep{Step~1: Bounds for $\psi_{s}$}
By Proposition~\ref{prop:forward-caloric-A}, for any $k = 0, 1, 2, \ldots$, we have 
\begin{equation*}
	\nrm{s^{\frac{k+1}{2}} \nb^{(k)} A}_{L^{\infty}_{\ds}((0, \infty); L^{\infty})}
	\leq \nrm{s^{\frac{k+1}{2}} \nb^{(k)} \ringA}_{L^{\infty}_{\ds}((0, \infty); L^{\infty})} + \nrm{s^{\frac{k+1}{2}} \nb^{(k)} A^{\infty}}_{L^{\infty}_{\ds}((0, \infty); L^{\infty})} \aleq 1.
\end{equation*}
Thus, by Lemma~\ref{lem:d-switch}, the non-covariant bounds stated in Proposition~\ref{prop:forward-caloric-psi} follow from their covariant analogues. Indeed, \eqref{eq:forward-caloric-psi-s} immediately follows from \eqref{eq:forward-psi-s-cov} in the preceding proof. For \eqref{eq:forward-caloric-Omg-psi-s}, note that \eqref{eq:forward-Omg-psi-s-cov} and Lemma~\ref{lem:d-switch} imply
\begin{equation*}
	\nrm{m(s) s^{\frac{k+1}{2}} \nb^{(k)} \Db_{\Omg} \psi_{s}}_{L^{2}_{\ds} \cap L^{\infty}((0, \infty); L^{2})} \aleq \nrm{u_{0} - Q}_{H^{1+2\dlt}} + \nrm{\calD_{\Omg} u_{0}}_{H^{1+2\dlt}}.
\end{equation*}
Since $\Omg \psi_{s} = \Db_{\Omg} \psi_{s} - i A_{\Omg} \psi_{s}$, \eqref{eq:forward-caloric-Omg-psi-s} follows from the preceding bound combined with \eqref{eq:forward-caloric-psi-s} and \eqref{eq:forward-caloric-A-Omg}. 

\pfstep{Step~2: Remaining bounds}
Finally, \eqref{eq:forward-caloric-Psi}, \eqref{eq:forward-caloric-Psi-Omg} and \eqref{eq:forward-caloric-Omg-Psi} follow, respectively, from the formulas (see Subsection~\ref{subsec:caloric}):
\begin{align*}
	\ringpsi_{j}(s) & = - \int_{s}^{\infty} (\rd_{j} \psi_{s}  + i A_{j} \psi_{s}) (s') \, \ud s', \\
	\ringpsi_{\Omg}(s) &= - \int_{s}^{\infty} (\Omg \psi_{s} + i A_{\Omg} \psi_{s} ) (s') \, \ud s', \\
	\calL_{\Omg}  \ringpsi_{j}(s)
	& = -\int_{s}^{\infty} \left( \rd_{j} (\Omg \psi_{s}) + i A_{j} (\Omg \psi_{s}) + i (\calL_{\Omg} A)_{j} \psi_{s} \right) (s') \, \ud s'.
\end{align*}
Next, we bound $\ringPsi$ and $\ringpsi_{\Omg}$ using
\begin{equation*}
	\rd_{s} \psi_{j} = \rd_{j} \psi_{s} + i A_{j} \psi_{s}, \quad \rd_{s} \psi_{\Omg} = \Omg \psi_{s} + i A_{\Omg} \psi_{s}
\end{equation*}
and the bounds \eqref{eq:forward-caloric-psi-s}, \eqref{eq:forward-caloric-Omg-psi-s}, \eqref{eq:forward-caloric-A}, \eqref{eq:forward-caloric-A-Omg} and \eqref{eq:forward-caloric-Omg-A} for $\nrm{\psi_{s}}_{L^{2}}$, $\nrm{\Omg \psi_{s}}_{L^{2}}$, $\nrm{\ringA}_{L^{\infty}}$, $\nrm{\ringA_{\Omg}}_{L^{\infty}}$ and $\nrm{\calL_{\Omg} \ringA}_{L^{\infty}}$, respectively. We omit the straightforward details. \qedhere
\end{proof}

\subsection{Integration back to $s = 0$} \label{subsec:backward-caloric}
As before, we suppress the dependence of constants on $Q$, $\tilde{\Omg}$, the Coulomb frame $e^{\infty}$ for $Q^{\ast} T \tgmfd$ and $A^{\infty}$ introduced in Subsection~\ref{subsec:hm}.

In this subsection, we turn around the problem considered in Subsection~\ref{subsec:forward-caloric} and study the quantitative relationship between the initial map $u_{0}$ and the component $\psi_{s}$ in caloric gauge of the harmonic map heat flow development. The following result relates energy-type norms of these two objects under the assumption that the initial map were sufficiently close to $Q$.

\begin{proposition} \label{prop:backward-caloric}
Assume that $u_{0}$ obeys the smallness condition \eqref{eq:hmhf-wp-hyp} and that $0 < \dlt < \frac{1}{2}$. In the caloric gauge constructed in Proposition~\ref{prop:caloric-gauge}, we have
\begin{align}
	\nrm{u_{0} - Q}_{H^{1+2\dlt}} & \aleq \nrm{m(s) s^{\frac{1}{2}} \psi_{s}}_{L^{2}_{\ds}((0, \infty); L^{2})}, \label{eq:backward-caloric} \\
	\nrm{\calD_{\Omg} u_{0}}_{H^{1+2\dlt}} & \aleq \nrm{m(s) s^{\frac{1}{2}} \psi_{s}}_{L^{2}_{\ds}((0, \infty); L^{2})}
	+ \nrm{m(s) s^{\frac{1}{2}} \Omg \psi_{s}}_{L^{2}_{\ds}((0, \infty); L^{2})}. \label{eq:backward-caloric-Omg}
\end{align}
The implicit constant in \eqref{eq:backward-caloric-Omg} depends on an upper bound for $\nrm{u_{0}-Q}_{H^{1+2\dlt}} + \nrm{\calD_{\Omg} u_{0}}_{H^{1+2\dlt}}$.
\end{proposition}
In our application of Proposition~\ref{prop:backward-caloric}, there will be a bootstrap assumption that (in particular) ensures $\nrm{u_{0}-Q}_{H^{1+2\dlt}} + \nrm{\calD_{\Omg} u_{0}}_{H^{1+2\dlt}} < 1$. Hence the implicit constants in \eqref{eq:backward-caloric} and \eqref{eq:backward-caloric-Omg} will be bounded by an absolute constant.
\begin{remark} [Higher regularity] \label{rem:backward-high-reg}
With a straightforward modification of the argument given below, one may show that \eqref{eq:backward-caloric} and \eqref{eq:backward-caloric-Omg} hold with $L^{2}$ and $H^{1+2\dlt}$ replaced by $H^{\sgm-1}$ and $H^{\sgm+2 \dlt}$, respectively, where $\sgm$ is any positive integer. This observation, along with Remark~\ref{rem:forward-high-reg} above, is used later to justify a continuous induction argument for the Schr\"odinger map evoluion.
\end{remark}

Key to the proof of Proposition~\ref{prop:backward-caloric} is the following analogous result for \eqref{eq:hmhf} and \eqref{eq:l-hmhf-ex} in the extrinsic formulation:
\begin{lemma} \label{lem:backward-ex}
Assume that $u_{0}$ obeys the smallness condition \eqref{eq:hmhf-wp-hyp} and that $0 < \dlt < \frac{1}{2}$. Then the harmonic map heat flow $u$ with the initial data $u(s=0) = u_{0}$ satisfies
\begin{equation}\label{eq:backward-ex}
	\nrm{u_{0} - Q}_{H^{1+2\dlt}} \aleq \nrm{m(s) s^{\frac{1}{2}} \rd_{s} u}_{L^{2}_{\ds}((0, \infty); L^{2})}.
\end{equation}

Moreover, let $\phi$ be a solution to \eqref{eq:l-hmhf-ex} with $f = 0$ and $\phi(s=0) = \phi_{0}$. Then
\begin{equation} \label{eq:backward-ex-lin}
	\nrm{\phi_{0}}_{H^{1+2\dlt}} \aleq \nrm{m(s) s^{\frac{1}{2}} \rd_{s} \phi}_{L^{2}_{\ds}((0, \infty); L^{2})}.
\end{equation}
\end{lemma}
\begin{proof}
We have
\begin{align*}
	\nrm{u_{0} - Q}_{H^{1+2\dlt}}^{2}
	& \aleq \int_{0}^{\infty} \int_{0}^{\infty} (s (-\lap)^{\frac{1+ 2\dlt}{2}}\rd_{s} u(s), s'(-\lap)^{\frac{1+ 2\dlt}{2}} \rd_{s} u(s')) \, \frac{\ud s}{s}  \frac{\ud s'}{s'}\\
	& \aleq \int_{0}^{\infty} \int_{s}^{\infty} \frac{s^{\frac{1}{2}} m^{-1}(s)}{(s')^{\frac{1}{2}} m^{-1}(s')} (m(s) s^{\frac{1}{2}} \rd_{s} u(s), m^{-1}(s') (s')^{\frac{3}{2}} (-\lap)^{1+ 2\dlt} \rd_{s} u(s')) \, \frac{\ud s}{s}  \frac{\ud s'}{s'}\\
	& \aleq \nrm{m(s) s^{\frac{1}{2}} \rd_{s} u}_{L^{2}_{\ds}((0, \infty); L^{2})} \nrm{m^{-1}(s) s^{\frac{3}{2}} \rd_{s} u}_{L^{2}_{\ds}((0, \infty); H^{2+4\dlt})} 
\end{align*}
where we used Schur's test on the last line and symmetry in $s$ and $s'$ to pass to the second line. Next, for the second factor on the last line, we have
\begin{align*}
\nrm{m^{-1}(s) s^{\frac{3}{2}} \rd_{s} u}_{L^{2}_{\ds}((0, \infty); H^{2+4\dlt})}
& \aleq \nrm{s^{\frac{3}{2}+\dlt} \rd_{s} u}_{L^{2}_{\ds}([0, 1); H^{2+4\dlt})} + \nrm{s^{\frac{3}{2}} \rd_{s} u}_{L^{2}_{\ds}([1, \infty); H^{2+4\dlt})} \\
& \aleq \nrm{u_{0} - Q}_{H^{1+2\dlt}}
\end{align*}
where we used \eqref{eq:hmhf-lwp-k} for the first term and the combination of \eqref{eq:hmhf-gwp} and \eqref{eq:hmhf-gwp-k} for the second term.

For $\phi$, we proceed similarly as follows.
\begin{align*}
	\nrm{\phi_{0}}_{H^{1+ 2\dlt}}^{2}
	& \aleq \int_{0}^{\infty} \int_{0}^{\infty} (s (-\lap)^{\frac{1+ 2\dlt}{2}}\rd_{s} \phi(s), s'(-\lap)^{\frac{1+ 2\dlt}{2}} \rd_{s} \phi(s')) \, \frac{\ud s}{s}  \frac{\ud s'}{s'}\\
	& \aleq \int_{0}^{\infty} \int_{s}^{\infty} \frac{s^{\frac{1}{2}} m^{-1}(s)}{(s')^{\frac{1}{2}} m^{-1}(s')} (m(s) s^{\frac{1}{2}} \rd_{s} \phi(s), m^{-1}(s') (s')^{\frac{3}{2}} (-\lap)^{1+ 2\dlt} \rd_{s} \phi(s')) \, \frac{\ud s}{s}  \frac{\ud s'}{s'}\\
	& \aleq  \nrm{m(s) s^{\frac{1}{2}} \rd_{s} \phi}_{L^{2}_{\ds}((0, \infty); L^{2})} \nrm{m^{-1}(s) s^{\frac{3}{2}} \rd_{s} \phi}_{L^{2}_{\ds}((0, \infty); H^{2+4\dlt})}.
\end{align*}
Moreover,
\begin{align*}
\nrm{m^{-1}(s) s^{\frac{3}{2}} \rd_{s} \phi}_{L^{2}_{\ds}((0, \infty); H^{2+4\dlt})}
& \aleq \nrm{s^{\frac{3}{2}+\dlt} \rd_{s} \phi}_{L^{2}_{\ds}([0, 1); H^{2+4\dlt})} + \nrm{s^{\frac{3}{2}} \rd_{s} \phi}_{L^{2}_{\ds}([1, \infty); H^{2+4\dlt})} \\
& \aleq \nrm{\phi_{0}}_{H^{1+2\dlt}},
\end{align*}
where we used \eqref{eq:l-hmhf-ex-sgm-k} for the first term and the combination of \eqref{eq:l-hmhf-gwp-1} and \eqref{eq:l-hmhf-gwp-k} for the second term. \qedhere
\end{proof}

\begin{proof}[Proof of Proposition~\ref{prop:backward-caloric}]
To deduce the proposition from Lemma~\ref{lem:backward-ex}, we need to verify that
\begin{align*}
\nrm{m(s) s^{\frac{1}{2}} \rd_{s} u}_{L^{2}_{\ds}((0, \infty); L^{2})} 
	&\aleq \nrm{m(s) s^{\frac{1}{2}} \psi_{s}}_{L^{2}_{\ds}((0, \infty); L^{2})}, \\
\nrm{m(s) s^{\frac{1}{2}} \rd_{s} \phi}_{L^{2}_{\ds}((0, \infty); L^{2})} 
	& \aleq \nrm{m(s) s^{\frac{1}{2}} \psi_{s}}_{L^{2}_{\ds}((0, \infty); L^{2})}
	+ \nrm{m(s) s^{\frac{1}{2}} \Omg \psi_{s}}_{L^{2}_{\ds}((0, \infty); L^{2})},
\end{align*}
where the implicit constant in the second bound may depend on an upper bound for $\nrm{u_{0}-Q}_{H^{1+2\dlt}} + \nrm{\calD_{\Omg} u_{0}}_{H^{1+2\dlt}}$.

The first bound is obvious from the identity $\rd_{s} u = e \psi_{s}$. For the second bound, we use
\begin{equation*}
\rd_{s} \calD_{\Omg} u
= \rd_{s} (\Omg u) - \rd_{s} \tilde{\Omg}(u)
= e \rd_{s} \psi_{\Omg} - (\rd \tilde{\Omg})(u) e \psi_{s} 
= e \Omg \psi_{s} + e i A_{\Omg} \psi_{s}  - (\rd \tilde{\Omg})(u) e \psi_{s}
\end{equation*}
By boundedness of $\rd \tilde{\Omg}(u)$ and the estimate for $\nrm{A_{\Omg}}_{L^{\infty}}$ from \eqref{eq:forward-caloric-A-Omg}, the desired bound follows.  \qedhere
\end{proof}

\section{The Schr\"odinger maps evolution in the caloric gauge}
In this section, we begin the analysis of the Schr\"odinger maps equation in earnest. In Subsection~\ref{subsec:equ-mo}, we present the caloric gauge formulation of the Schr\"odinger maps. In Subsection~\ref{subsec:bootstrap}, we state the more precise version of the main result of the paper (Theorem~\ref{thm:main1}), and then reduce its proof to closing a set of bootstrap assumptions (Propositions~\ref{prop:core-bootstrap} and \ref{prop:prop-reg}) in the caloric gauge. Finally, in Subsection~\ref{subsec:part2-outline}, we outline the remainder of the paper, which is devoted to the proofs of Propositions~\ref{prop:core-bootstrap} and \ref{prop:prop-reg}, and state some conventions that will be effect in what follows.

\subsection{Equations of motion} \label{subsec:equ-mo}
In the caloric gauge formulation of the Schr\"odinger maps system, the Schr\"odinger map $u$ is analyzed through the \emph{heat tension field}
\EQ{ \label{eq:psi_s-def} 
\psi_s := \Db^\ell \psi_\ell .
}
In order to understand the evolution fo $\psi_{s}$ in $t$, we need to introduce and analyze the \emph{Schr\"odinger tension field}
\EQ{\label{eq:w-def} 
w := \psi_t - i \Db^\ell \psi_\ell .
}
The main equations of motion in the caloric gauge are as follows:
\begin{itemize}
\item Parabolic equation for $\psi_{s}$
\begin{equation} \label{equ:heat_equation_psi_s}
(\Db_{s} - \Db^{k} \Db_{k}) \psi_{s} + i \tcv \Im(\psi^{k} \overline{\psi_{s}}) \psi_{k} = 0.
\end{equation}

\item Parabolic equation for $w$
\EQ{ \label{equ:heat_equation_w}
(\Db_{s} - \Db^{k} \Db_{k}) w +i \tcv \Im(\psi^{k} \overline{w}) \psi_{k} 
& = i \tcv \psi^{k} \psi_{k} \overline{\psi_{s}}, \\
  w(0, t, x) &= 0 .
}

\item Schr\"odinger equation for $\psi_{s}$
 \begin{equation} \label{equ:schroedinger_equation_psi_s}
  (i \Db_t + \Db^k \Db_k) \psi_s = i \partial_s w + i \tcv \Im ( \psi^k \overline{\psi_s} ) \psi_k .
 \end{equation}
\end{itemize}
The overall scheme for analysis is to use \eqref{equ:schroedinger_equation_psi_s} to analyze the evolution in $t$ of $\psi_{s}$, while using \eqref{equ:heat_equation_w} to estimate $w$ in terms of $\psi_{s}$. Equation \eqref{equ:heat_equation_psi_s} is used to gain control of higher derivatives of $\psi_{s}$ at the price of adequate powers of $s$ through parabolic smoothing estimates. Moreover, \eqref{equ:heat_equation_psi_s} allows us to prove that $P_{\sgm} \psi_{s}(s)$ exhibits decay off of the diagonal regime $\set{\sgm \simeq s}$, which is useful since $P_{\sgm}$ arises in the definitions of $LE$ and $LE^{\ast}$. 

Equations~\eqref{equ:heat_equation_psi_s},~\eqref{equ:heat_equation_w}, and~\eqref{equ:schroedinger_equation_psi_s} are easily obtained from the definitions~\eqref{eq:psi_s-def} and~\eqref{eq:w-def} along with the curvature formula~\eqref{eq:comm-K}.  Indeed, to prove~\eqref{equ:schroedinger_equation_psi_s}
we have 
 \begin{align*}
  (i \Db_t + \Db^k \Db_k) \psi_s &= i \Db_s \psi_t + \Db^k \Db_s \psi_k \\
  &= i \Db_s \psi_t + \Db_s \Db^k \psi_k + [\Db^k, \Db_s] \psi_k \\
  &= i \Db_s (\psi_t - i \Db^k \psi_k) +  i \tcv \Im( \psi^k \overline{\psi_s} ) \psi_k \\
  &= i \partial_s w +  i \tcv \Im( \psi^k \overline{\psi_s} ) \psi_k.
 \end{align*}
 The heat equation~\eqref{equ:heat_equation_psi_s} follows from a similar computation. Lastly, to compute~\eqref{equ:heat_equation_w} we first observe that 
 \EQ{ \label{eq:w-def1} 
 w = \psi_t - i \Db^\ell \psi_\ell = \psi_t - i \psi_s.
 }
 We then compute, 
  \begin{align*}
  (\Db_s - \Db^k \Db_k) \psi_t &= \Db_s \psi_t - \Db^k \Db_t \psi_k \\
  &= \Db_t \psi_s - \Db_t \Db^k \psi_k - [\Db^k, \Db_t] \psi_k \\
  &= \Db_t \psi_s - \Db_t \psi_s + \tcv (-i) \Im (\psi^k \overline{\psi_t}) \psi_k \\
  &= -i \tcv \Im (\psi^k \overline{\psi_t}) \psi_k \\
  & = - i \tcv \Im(\psi^{k} \overline{w}) \psi_{k} - i \tcv \Im(\psi^{k} \overline{(i \psi_{s})}) \psi_{k}. \\
 \end{align*}
Using~\eqref{equ:heat_equation_psi_s} we see that 
\begin{align*}
 (\Db_s - \Db^k \Db_k) (-i \psi_s) &= - \tcv(-i) i \Im (\psi^k \overline{\psi_s}) \psi_k = - \tcv \Im (\psi^k \overline{\psi_s}) \psi_k.
\end{align*}
By \eqref{eq:w-def1}, $(\Db_{s} - \Db^{k} \Db_{k}) w$ equals the sum of the right-hand sides of the preceding two computations. Since
\begin{equation*}
	i \Im(\psi^{k} \overline{(i \psi_{s})}) \psi_{k} + \Im(\psi^{k} \overline{\psi_{s}}) \psi_{k}
	= \frac{1}{2i} \left(i \psi^{k} \overline{(i \psi_{s})} - i \overline{\psi^{k}} i \psi_{s} + \psi^{k} \overline{\psi_{s}} - \overline{\psi^{k}} \psi_{s} \right) \psi_{k},
	= - i \psi^{k} \overline{\psi_{s}} \psi_k
\end{equation*}
we arrive at \eqref{equ:heat_equation_w}.
Finally, the fact that $w(0, t, x) = 0$ follows from the fact that $u(0, t, x)$ solves~\eqref{equ:schroedinger_maps_equ}.  
 
\subsubsection{Refined structure of the dynamical equations} 
One key advantage of the caloric gauge formulation compared to other ways of analyzing \eqref{equ:schroedinger_maps_equ} (e.g., using Coulomb gauge) is that the variables $\psi_s$ and $w$ satisfy scalar dynamical equations, as is evident in \eqref{equ:heat_equation_psi_s}--\eqref{equ:schroedinger_equation_psi_s}; we refer to \cite{LOS5} for further discussion on this point in the related context of wave maps.
 A second key advantage of the caloric gauge formulation is that it yields a \emph{geometrically natural linearization} of the Schr\"odinger maps system~\eqref{equ:schroedinger_maps_equ}. Indeed, the harmonic map heat flow evolution of the Schr\"odinger map $u(t, x)$ at each fixed time $t$ converges to the unique harmonic map $Q$ that shares its boundary data at infinity -- this was shown in Section~\ref{sec:hmhf}. In the caloric gauge formulation this manifests at the level of $\psi_\ell$ and $A_\ell$ which naturally decompose as 
 \EQ{
 \psi_\ell(s) &= \ringpsi_\ell(s) + \psi_\ell^\infty, \\ 
 A_\ell(s) &= \ringA_\ell(s) + A_\ell^\infty,
 } 
 where $\psi^\infty$ and $A^\infty$ are the constant-in-$s$ gauge components of the harmonic map $Q$; see~\eqref{equ:linearization_via_caloric_gauge}. 
 In this vein we record the following refinements of the equations~\eqref{equ:heat_equation_psi_s},~\eqref{equ:heat_equation_w}, and~\eqref{equ:schroedinger_equation_psi_s}, which show how the linearized operator $H$ defined in~\eqref{eq:lin-Q} arises. First recall that 
\begin{equation} \label{eq:H-def} 
 H \psi_s = - (\nabla^k + i A^{\infty,k}) (\nabla_k + i A^\infty_k) \psi_s - \tcv |\psi^\infty_2|^2 \psi_s.
\end{equation}
 
 \begin{lemma}[Parabolic equation for $\psi_s$] \label{l:H-psi_s} 
 The equation~\eqref{equ:heat_equation_psi_s} can be written in the following form: 
\EQ{ \label{eq:heatH-psi_s} 
(\p_s  +  H )\psi_s 
& =  \tcv i\Im( \psi_s \ba{ \ringpsi^k}) \ringpsi_k + \tcv i \Im( \psi_s \ba{\ringpsi_k})\psi^{\infty, k} + \tcv i\Im( \psi_s \ba{  \psi^{\infty, k}}) \ringpsi_k \\
& \quad   + 2i \ringA^k \na_k \psi_s  + i (\na^k \ringA_k )\psi_s  - \ringA^k \ringA_k \psi_s  - 2 A^\infty_k \ringA^k \psi_s,
}
where $H$ is defined as in~\eqref{eq:H-def}. 
 \end{lemma} 
 
 \begin{lemma}[Parabolic equation for $w$] \label{l:H-w}
 The equation~\eqref{equ:heat_equation_w} can be written in the following form: 
 \EQ{ \label{eq:heatH-w} 
  (\partial_s + H) w &= 
  \tcv i\Im( w \ba{ \ringpsi^k}) \ringpsi_k + \tcv i \Im( w \ba{\ringpsi_k})\psi^{\infty, k} + \tcv i\Im( w \ba{  \psi^{\infty, k}}) \ringpsi_k \\
  & \quad  + 2 i \ringA^k \nabla_k w - 2 A^{\infty, k} \ringA_k w + i (\nabla_k \ringA^k) w - \ringA^k \ringA_k w \\
  &\quad + i \tcv \ringpsi^{k} \ringpsi_{k} \overline{\psi_{s}} + 2 i \tcv \psi^{\infty, k} \ringpsi_{k} \overline{\psi_{s}}.
 }
 where $H$ is defined as in~\eqref{eq:H-def}. 
 \end{lemma} 
 
 As discussed in Section~\ref{subsec:ideas}, note in Lemma~\ref{l:H-w} that there are no linear terms in $\psi_{t}$ and $\psi_{s}$.

\begin{lemma}[Schr\"odinger equation for $\psi_s$] \label{l:S-psi_s} 
The equation~\eqref{equ:schroedinger_equation_psi_s} can be written in the following form: 
\begin{equation} \label{eq:SH-psi_s}
 \begin{aligned}
  (i \partial_t - H) \psi_s &= i \partial_s w - 2 i \ringA^k \nabla_k \psi_s + 2 A^{\infty, k} \ringA_k \psi_s + \tcv i \Im ( \psi^{\infty,k} \overline{\psi_s} ) \ringpsi_k +\tcv i \Im ( \ringpsi^k \overline{\psi_s} ) \psi^\infty_k  \\
  &\phantom{=} - i (\nabla_k \ringA^k) \psi_s + \ringA^k \ringA_k \psi_s +\tcv i \Im ( \ringpsi^k \overline{\psi_s} ) \ringpsi_k + A_t \psi_s ,
 \end{aligned}
\end{equation}
where $H$ is defined as in~\eqref{eq:H-def}. 
\end{lemma} 

We prove Lemma~\ref{l:S-psi_s}, and note that the proof of Lemmas~\ref{l:H-psi_s},~\ref{l:H-w} are similar computations. 

\begin{proof}[Proof of Lemma~\ref{l:S-psi_s}] 
First note that 
\EQ{
i \Db_t \psi_s = i \p_t \psi_s - A_t \psi_s .
}
Next, writing $A = \ringA + A^{\infty}$ we have 
\EQ{
\Db^k \Db_k \psi_s  &= (\nabla^k + i ( \ringA^{k} + A^{\infty,k})) (\nabla_k + i ( \ringA^k + A^\infty_k) \psi_s \\
& = (\nabla^k + i A^{\infty,k}) (\nabla_k + i A^\infty_k) \psi_s +  2 i \ringA^k \nabla_k \psi_s - 2 A^{\infty, k} \ringA_k \psi_s + i (\nabla_k \ringA^k) \psi_s -  \ringA^k \ringA_k \psi_s .
}
Finally, writing  $\psi = \ringpsi + \psi^\infty$ we expand the term $i \tcv \Im ( \psi^k \overline{\psi_s} ) \psi_k$ that appears on the right-hand-side of~\eqref{eq:SH-psi_s} to extract all terms that are truly linear in $\psi_s$, i.e., that do not contain any copies of $\ringpsi$. As we observed in~\eqref{eq:CR},~\eqref{eq:lin-Q}, the favorable structure of the linear term is revealed by observing that the derivative components $\psi^\infty$ of the harmonic map $Q$ satisfy the Cauchy Riemann equations
\EQ{ \label{eq:CR1} 
\psi_1^\infty + i \psi^\infty_2 = 0 .
}
Using~\eqref{eq:CR1} and following the argument used to give~\eqref{eq:lin-Q} we see that 
\EQ{
i \tcv \Im ( \psi^k \overline{\psi_s} ) \psi_k &=   i \tcv \Im ( (\ringpsi^k + \psi^{k, \infty})\overline{ \psi_s} ) (\ringpsi_k   + \psi_k^\infty)  \\
& = - \tcv \abs{\psi^\infty}^2  \psi_s + \tcv i \Im ( \psi^{\infty,k} \overline{\psi_s} ) \ringpsi_k +\tcv i \Im ( \ringpsi^k \overline{\psi_s} ) \psi^\infty_k+\tcv i \Im ( \ringpsi^k \overline{\psi_s} ) \ringpsi_k .
}
This completes the proof. 
\end{proof}

\subsection{Reduction of the main theorem to the core bootstrap argument in the caloric gauge} \label{subsec:bootstrap}
We begin by stating a more precise version of Theorem~\ref{thm:main}.
\begin{theorem} [Refinement of the main theorem: asymptotic stability of $Q_\nu$] \label{thm:main1}
Let $\NN = \Hp^2$ or $\Sp^2$. If $\NN = \Hp^2$ consider any harmonic map $Q$ as in Proposition~\ref{prop:hm-classify}. If $\NN= \Sp^2$ consider any $Q$ as in Proposition~\ref{prop:hm-classify} that satisfies the strong linearized stability condition in Defintion~\ref{def:strong-stab}; see Proposition~\ref{prop:strong-stab}. Fix any bounded neighborhood $\ti \NN$ of $Q( \Hp^2)$ in $\NN$. 

There exists an $\eps_0>0$ with the following property.  Let 
%and consider the initial value problem for~\eqref{equ:schroedinger_maps_equ}  with 
$ u_{0}: \Hp^2 \to \NN$ with $u_0 (\Hp^2) \subset \ti \NN$ be smooth initial data for~\eqref{equ:schroedinger_maps_equ} such that 
\begin{equation}\label{eq:data} 
	\|u_{0} - Q\|_{H^{1+ 2 \de}} + \norm{\calD_{\Omg} u_{0}}_{H^{1+ 2\dlt}} \leq \eps_{0}
\end{equation}
for any small number $\de >0$. Then, there exists a unique global smooth solution $u: \R \times\Hp^2 \to \Hp^2$ to~\eqref{equ:schroedinger_maps_equ} with initial data $u(t=0) = u_{0}$ and satisfying  
\EQ{ \label{eq:global-small} 
	\sup_{t \in \R} \|u(t) - Q\|_{H^{1+ 2 \de}} + \norm{\calD_{\Omg} u(t)}_{H^{1+ 2\dlt}} \lesssim \eps_{0}.
}

In fact, $Q$ is \textbf{asymptotically stable} in the following sense. Let $u(t, x, s)$ denote the harmonic map heat flow resolution of $u(t, x)$ and let $e(t, x, s)$ denote the corresponding caloric gauge with limiting frame $e^{\infty}(t, x)$ as defined in Section~\ref{subsec:caloric}, and let $\psi_s(t, x, s)$ denote the \emph{heat tension field}. Then, 
\begin{itemize}
\item \emph{ \textbf{A priori control of dispersive norms}.} $\psi_s$ satisfies, 
\EQ{ \label{eq:S-norm-eps}
 \| \psi_s \|_{\calS(\R)}& \lesssim \nrm{m(s) s^{\frac{1}{2}} \langle \Omg \rangle \psi_{s} \vert_{t = 0}}_{(L^{\infty}_{\ds} \cap L^{2}_{\ds} )L^2_x}   \lesssim \|u_{0} - Q\|_{H^{1+ 2 \de}} + \norm{\calD_{\Omg} u_{0}}_{H^{1+ 2\dlt}} \lesssim  \eps_{0} .
} 
\item \emph{\textbf{Scattering}.} For every $s>0$, $\psi_s(t)$ scatters to a solution of the free Schr\"odinger equation as $t  \to \pm \infty$, i.e., there exist free Schr\"odinger waves $\phi_\pm(t, s)$, such that 
\EQ{ \label{eq:scattering} 
\| \psi_s(t)  - \phi_\pm(t, s) \|_{L^2_x} \to 0 \mas t \to \pm \infty .
}

\item \emph{\textbf{Uniform control of higher derivatives}.} 
Higher regularity of the initial data $\psi_s(s, 0, x)$ is preserved by the flow. In fact, we have the following uniform estimates, 
\EQ{ \label{eq:HS-norm}
\| H \psi_s \|_{\calS(\R)}  &\lesssim   \nrm{m(s) s^{\frac{1}{2}} \langle \Omg \rangle \psi_{s} \vert_{t = 0}}_{(L^{\infty}_{\ds} \cap L^{2}_{\ds} )H^2}   \lesssim \|u_{0} - Q\|_{H^{3+2\de}} + \norm{\calD_{\Omg} u_{0}}_{H^{3+ 2 \de}} , \\% \\
 \| H^2 \psi_s \|_{\calS(\R)} &\lesssim \nrm{m(s) s^{\frac{1}{2}} \langle \Omg \rangle \psi_{s} \vert_{t = 0}}_{(L^{\infty}_{\ds} \cap L^{2}_{\ds} )H^4}   \lesssim \|u_{0} - Q\|_{H^{5+2\de}} + \norm{\calD_{\Omg} u_{0}}_{H^{5+2\de}} .  %\label{eq:H2S-norm} 
}

\item \emph{\textbf{Qualitative pointwise convergence to $Q$}.} Finally, we have the following pointwise convergence to $Q$, 
\EQ{ \label{eq:pointwise-c}
\| u(t) -  Q \|_{L^\infty}  \to 0 \mas t \to \pm \infty .
}
\end{itemize}
\end{theorem}

Theorem~\ref{thm:main1} will be proved via a priori estimates, which we close with a bootstrap argument. In this section we briefly recall the standard argument used to reduce Theorem~\ref{thm:main1} to the following two propositions.  
\begin{proposition}[Core bootstrap argument in the caloric gauge] \label{prop:core-bootstrap}
There exists $\eps_0, C_1>0$ with the following property.  Let $I\subset \R$ be a time interval and assume that, 
\begin{equation} \label{eq:overall-bootstrap}
	\|u(t)- Q\|_{H^{1 +2 \de}}  + \norm{\calD_{\Omg} u(t)}_{H^{1+ 2\dlt}} \leq \eps_{0}^{1/2}, \quad \hbox{ for all } t \in I,
\end{equation}
and
\begin{equation} \label{eq:core-bootstrap}
	\nrm{\psi_{s}}_{\calS(I)} \leq 2 C_{1} \eps_{0} .
\end{equation}
Then, 
\begin{equation} \label{eq:core-bootstrap-imp}
	\nrm{\psi_{s}}_{\calS(I)} \leq C_{1} \eps_{0}.
\end{equation}
Moreover, 
\begin{equation} \label{eq:core-bootstrap-imp-N}
 \| m(s) s^{\frac{1}{2}} (i \rd_{t} - H) \psi_{s} \|_{L^\infty_{\ds} \cap L^2_{\ds}(LE^\ast + L^{\frac{4}{3}}_t L^{\frac{4}{3}}_x)(I)} 
 + \| m(s) s^{\frac{1}{2}} (i \rd_{t} - H) \Omg \psi_{s} \|_{L^\infty_{\ds} \cap L^2_{\ds}(LE^\ast + L^{\frac{4}{3}}_t L^{\frac{4}{3}}_x)(I)} \lesssim \eps_{0}^{\frac{3}{2}},
\end{equation}
where the implicit constant is independent of $\eps_{0}$ and $C_{1}$. 
\end{proposition}

\begin{proposition}[Auxiliary propagation of regularity in the caloric gauge] \label{prop:prop-reg}
There exists $\eps_1 >0$ with the following property. Let $I \subset \R$ be a time interval and suppose that  
\begin{equation} \label{eq:prop-reg-small}
	\nrm{\psi_{s}}_{\calS(I)} \leq \eps_{1},
\end{equation}
Then, 
\EQ{
	\nrm{H \psi_{s}}_{\calS(I)} &\aleq \nrm{m(s) s^{\frac{1}{2}} \langle \Omg \rangle \psi_{s} \vert_{t=0}}_{L^{\infty}_{\ds} \cap L^{2}_{\ds}([0, \infty); H^{2})} \lesssim  \| u_0 - Q \|_{H^{3 + 2\de}} + \| \calD_{\Om} u_0 \|_{H^{3 + 2\de}}, \\
		\nrm{H^2 \psi_{s}}_{\calS(I)} &\aleq \nrm{m(s) s^{\frac{1}{2}} \langle \Omg \rangle \psi_{s} \vert_{t=0}}_{L^{\infty}_{\ds} \cap L^{2}_{\ds}([0, \infty); H^{4})}  \lesssim   \| u_0 - Q \|_{H^{5 + 2\de}} + \| \calD_{\Om} u_0 \|_{H^{5 + 2\de}}.  \label{eq:H2S} 
}

\end{proposition}

The key ingredients in the reduction are the bounds in Proposition~\ref{prop:backward-caloric} in Section~\ref{sec:hmhf}, which allow us to transfer bounds for the heat flowed map $u(t, x, s)$ and its derivatives back to the Schr\"odinger map $u(t, x)$, and thus close the bootstrap. We remark that Proposition~\ref{prop:backward-caloric} is stated for the harmonic map heat flow evolution of the initial data $u_0$, but a straightforward modification extends this to the following statement for the $1$-parameter family of heat flows $u(t, x, s)$. 

\begin{proposition} \label{prop:backward-caloric-t}
Assume that $\eps_0$ in Proposition~\ref{prop:core-bootstrap} satisfies $\eps_0^{1/2}< \eps_*$ where $\eps_*$ is defined in Theorem~\ref{thm:hmhf-gwp} and that $0 < \dlt < \frac{1}{2}$. Then, 
\begin{align}
	 \sup_{t \in I} \|u(t)- Q\|_{H^{1+ 2 \de}}    & \aleq \nrm{m(s) s^{\frac{1}{2}}  \psi_{s}}_{L^{2}_{\ds}((0, \infty);L^\infty_t (I;  L^{2}_x))}, \label{eq:backward-caloric-t}  \\
	  \sup_{t \in I} \norm{\calD_{\Omg} u(t)}_{H^{1+ 2\dlt}} & \aleq \nrm{m(s) s^{\frac{1}{2}} \ang{\Om} \psi_{s}}_{L^{2}_{\ds}((0, \infty);L^\infty_t (I;  L^{2}_x))}, \label{eq:backward-caloric-t1} \\ 
	  \sup_{t \in I} \Big(\|u(t)- Q\|_{H^{3+ 2 \de}}  + \norm{\calD_{\Omg} u(t)}_{H^{3+ 2\dlt}} \Big)  & \aleq \nrm{m(s) s^{\frac{1}{2}} \ang{\Om} \psi_{s}}_{L^{2}_{\ds}((0, \infty);L^\infty_t (I;  H^2_x))}, \label{eq:backward-caloric-t3}  \\
	   \sup_{t \in I} \Big(\|u(t)- Q\|_{H^{5+ 2 \de}}  + \norm{\calD_{\Omg} u(t)}_{H^{5+ 2\dlt}} \Big)  & \aleq \nrm{m(s) s^{\frac{1}{2}} \ang{\Om} \psi_{s}}_{L^{2}_{\ds}((0, \infty);L^\infty_t (I;  H^4_x))}. \label{eq:backward-caloric-t5}  \\
	 %\\
	%\nrm{\calD_{\Omg} u_{0}}_{H^{1+2\dlt}} & \aleq \nrm{m(s) s^{\frac{1}{2}} \psi_{s}}_{L^{2}_{\ds}((0, \infty); L^{2})}
	%+ \nrm{m(s) s^{\frac{1}{2}} \Omg \psi_{s}}_{L^{2}_{\ds}((0, \infty); L^{2})}. \label{eq:backward-caloric-Omg-t}
\end{align}
\end{proposition}
We remark that the first two bounds above~\eqref{eq:backward-caloric-t} and~\eqref{eq:backward-caloric-t1} follow from the same proof as Proposition~\ref{prop:backward-caloric}. For the bounds~\eqref{eq:backward-caloric-t3} and~\eqref{eq:backward-caloric-t5} we use Remark~\ref{rem:backward-high-reg} with $\s = 3$ and $\s = 5$; see also Remark~\ref{rem:forward-high-reg}. 

\begin{proof}[Proof of Theorem~\ref{thm:main} and Theorem~\ref{thm:main1} assuming Propositions~\ref{prop:core-bootstrap} and \ref{prop:prop-reg} ] Let $u_0$ be smooth initial data for \eqref{equ:schroedinger_maps_equ} satisfying~\eqref{eq:data}. Since $u_0$ is smooth we may define $B_0<\infty$ by  
\EQ{
\|  u_0 - Q \|_{H^{5+ 2\de}} =: B_0 < \infty .
}
By Lemma~\ref{l:lwp} we can find a unique smooth solution $u(t)$ to~\eqref{equ:schroedinger_maps_equ} defined on a maximal forward time of existence $J = [0, T^*)$. Assume for contradiction that $T^*<\infty$. Then, by continuation criterion in Lemma~\ref{l:lwp} we must have 
\EQ{ \label{eq:bu} 
\lim_{t \to T_*}  \| u(t) - Q \|_{H^{5+ 2\de}} = \infty .
}
On the other hand, let $\eps_0>0$ be as in Proposition~\ref{prop:core-bootstrap}. By continuity of the flow, the assumption~\eqref{eq:data} yields a time $T_1  \le T^*$ for which both~\eqref{eq:overall-bootstrap} and~\eqref{eq:core-bootstrap} are satisfied.  Let $T_{1}^*$ be the maximal such time. The combination of  Proposition~\ref{prop:core-bootstrap} and Proposition~\ref{prop:backward-caloric-t} closes the bootstrap. Indeed, from~\eqref{eq:backward-caloric-t} and ~\eqref{eq:backward-caloric-t1} we have 
\EQ{
 \sup_{t \in I} \Big(\|u(t)- Q\|_{H^{1+ 2 \de}}  + \norm{\calD_{\Omg} u(t)}_{H^{1+ 2\dlt}} \Big)  \lesssim \| m(s) s^{\frac{1}{2}} \ang{\Om} \psi_s  \|_{L^{2}_{\ds}([0, \infty); L^{\infty}_t L^2_x)} \lesssim \| \psi_s \|_{\calS(I)}  \lesssim C_1 \eps_0 \le \frac{1}{2} \eps_0^{1/2} .
}
Thus, we must have $T_{1}^* = T^*$.  It then follows that the conclusions of Proposition~\ref{prop:prop-reg} hold with $I = [0, T^*)$. Using the elliptic regularity estimate
\EQ{
%\| v \|_{H^2} &\lesssim \| \De v \|_{L^2} \lesssim \| H v \|_{L^2} \\
\| v \|_{H^4} &\lesssim  \| H^2 v \|_{L^2} + \| v \|_{L^2} ,
}
we then deduce from the estimate~\eqref{eq:H2S} from Proposition~\ref{prop:prop-reg} and the estimate~\eqref{eq:backward-caloric-t5} that 
\EQ{
 \sup_{t \in [0, T^*)} \|u(t)- Q\|_{H^{5+ 2 \de}}  \lesssim B_0 ,
}
which contradicts~\eqref{eq:bu}. Hence $T^* = +\infty$. A similar argument establishes well-posedness backwards in infinite time. 
Thus~\eqref{eq:core-bootstrap-imp} holds with $I =\R$ which establishes~\eqref{eq:global-small} and~\eqref{eq:S-norm-eps}. Proposition~\ref{prop:prop-reg} then yields~\eqref{eq:HS-norm}. The scattering statement~\eqref{eq:scattering} is a standard consequence of \eqref{eq:core-bootstrap-imp-N} in Proposition~\ref{prop:core-bootstrap} with $I = \bbR$.

Finally, the proof of the pointwise estimate~\eqref{eq:pointwise-c} begins with the observation that 
\EQ{
\abs{u(t, x) - Q(x)} \le  \int_0^\infty \abs{\p_s ( u(t, x, s) - Q(x)) }\, \ud s  \le  \int_0^\infty \abs{\psi_s(t, x, s) }\, \ud s
}
and thus~\eqref{eq:pointwise-c} follows from the claim that 
\EQ{
\lim_{t \to \infty} \| u(t) - Q\|_{L^\infty_x} \le \lim_{t \to \infty} \int_0^\infty \| \psi_s(t, s) \|_{L^{\infty}_{ x}} \, \ud s  = 0.
}
The proof of the claimed limit requires two observations that we briefly sketch here, while referring the reader to the detailed identical argument given in~\cite[Proof of (4.8)]{LOS5}. First, we claim that the limit in $t$ can be passed through the $s$-integral. This follows from the dominated convergence theorem after showing that $\| \psi_s(s) \|_{L^\infty_{t, x}}$ is an integrable function. This latter claim can be deduced using the a priori estimates~\eqref{eq:HS-norm} (which break the scaling in $s$) to handle the small-$s$ regime, along with the decay estimates~\eqref{eq:forward-caloric-psi-s} to handle the large-$s$ regime. Second, we claim that for each fixed $s>0$, 
\EQ{
\| \psi_s(t, s) \|_{L^{\infty}_{ x}} \to 0 \mas t \to \infty. 
}
This can be proved using the global dispersive estimate~\eqref{eq:S-norm-eps} and in particular the scattering statement~\eqref{eq:scattering}.  Again, we refer the reader to~\cite[Proof of (4.8)]{LOS5} for the precise details.
\end{proof}

\subsection{Structure of the remainder of the paper and some conventions}
\label{subsec:part2-outline}
From now on, the goal is to establish Propositions~\ref{prop:core-bootstrap} and \ref{prop:prop-reg}. The remainder of the paper is organized as follows: 
\begin{itemize} 
\item In Section~\ref{s:heat-psi_s} we study the parabolic equation~\eqref{eq:heatH-psi_s} under the bootstrap assumptions~\eqref{eq:overall-bootstrap} and~\eqref{eq:core-bootstrap}. We show that $\psi_s$  satisfies several parabolic regularity type estimates that will be used in our analysis of the nonlinear Schr\"odinger equation~\eqref{eq:SH-psi_s} in Section~\ref{s:schrod}. 
\item In Section~\ref{s:w} we continue the preliminary analysis of the terms appearing on the right-hand side of~\eqref{eq:SH-psi_s}, proving several preliminary estimates on $\ringPsi, \ringA$. We then use the parabolic equation~\eqref{eq:heatH-w} for $w$ to establish estimates for $w, \p_s w$ and use these to prove auxiliary bounds on $\psi_t, A_t$. 
\item Finally, in Section~\ref{s:schrod} we complete the proofs of Propositions~\ref{prop:core-bootstrap} and \ref{prop:prop-reg} using the main linear estimate in Lemma~\ref{lem:main_linear_estimate} to study the Schr\"odinger equation~\eqref{eq:SH-psi_s}. 

\end{itemize} 

In the remainder of the paper, we work under the following conventions to ease the notation:
\begin{itemize}
\item The Gauss curvature $\tcv = \pm 1$ of the target manifold plays no role in the remainder of the paper. Hence, {\bf we fix the target $\NN = \Hp^2$ and $\tcv = -1$ in equations~\eqref{eq:heatH-psi_s},~\eqref{eq:heatH-w}, and~\eqref{eq:SH-psi_s}.} 
\item We always work on the time interval $I$. Hence, {\bf we omit $(I)$ in the notation for space-time norms,} e.g., $L^{4}_{t} L^{4}_{x} = L^{4}_{t} L^{4}_{x}(I)$.
\end{itemize}

\section{Parabolic estimates for $\psi_{s}$} \label{s:heat-psi_s} 

In this section we establish $L^p$ parabolic regularity-type estimates for $\psi_s$ under the  core bootstrap assumptions~\eqref{eq:core-bootstrap} by directly invoking properties of the heat flow $e^{ - s H}$ from Corollary~\ref{c:lin-para-Lp}. 

Recall that $\psi_s$ satisfies the nonlinear heat equation~\eqref{eq:heatH-psi_s}, which can be written in the form, 
\EQ{ \label{eq:psi_s-heat} 
\p_s \psi_s = -H \psi_s + \calN( \psi_s,  \Psi, A), 
}
where 
\EQ{ \label{eq:Hdef1} 
H \psi_s &= - ( \na^k + i A^{\infty, k})( \na_k + i A^\infty_k) \psi_s + \abs{ \psi_2^\infty}^2 \psi_s  \\
& =  - \De \psi_s  - 2 i A^{\infty, k} \na_k \psi_s - i ( \na^k A_k^{\infty}) \psi_s + ( \abs{A^\infty}^2 + \abs{ \psi_2^\infty}^2) \psi_s, 
}
and 
\EQ{ \label{eq:nonlin1} 
\NN(\psi_s, \Psi, A) %&=% \bfR^0( \psi_s,  \ringpsi^k) \ringpsi_k +  \bfR^0( \psi_s, \ringpsi_k)\psi^{\infty, k} + \bfR^0( \psi_s,  \psi^{\infty, k}) \ringpsi_k \\
%& \quad  + 2i \ringA^k \na_k \psi_s  + i (\na^k \ringA_k )\psi_s  - \ringA^k \ringA_k \psi_s  - 2 A^\infty_k \ringA^k \psi_s \\
& = -i\Im( \psi_s \ba{ \ringpsi^k}) \ringpsi_k - i \Im( \psi_s \ba{\ringpsi_k})\psi^{\infty, k} -i\Im( \psi_s \ba{  \psi^{\infty, k}}) \ringpsi_k \\
& \quad   + 2i \ringA^k \na_k \psi_s  + i (\na^k \ringA_k )\psi_s  - \ringA^k \ringA_k \psi_s  - 2 A^\infty_k \ringA^k \psi_s .
}
We also will require the heat equation satisfied by $\Om \psi_s$. Using Lemma~\ref{l:commLOm},~\eqref{eq:LOmA}, and~\eqref{eq:LOmpsi}  we have 
\EQ{\label{eq:Ompsi_s-heat} 
(\p_s + H) \Om \psi_s =  \Om \NN(\psi_s, \Psi, A) ,
}
where 
\EQ{ \label{eq:Omnonlin1} 
\Om \NN(\psi_s, \Psi, A)  &=  + i\Im (\calL_\Omega\ringpsi^k \overline{\psi_s}) \ringpsi_k+ i\Im (\ringpsi^k \overline{\Omega\psi_s}) \ringpsi_k+ i\Im (\ringpsi^k \overline{\psi_s}) \calL_\Omega\ringpsi_k \\
  &\quad  + i\Im (\calL_\Omega\ringpsi^k \overline{\psi_s}) \psi^\infty_k +i \Im (\ringpsi^k \overline{\Omega\psi_s}) \psi^\infty_k \\
  &\quad + i\Im (\psi^{\infty, k} \overline{\Omega\psi_s}) \ringpsi_k+ i\Im (\psi^{\infty, k} \overline{\psi_s}) \calL_\Omega\ringpsi_k\\
  &\quad + i\Im (i\psi^{\infty, k} \overline{\psi_s}) \ringpsi_k+ i\Im (\ringpsi^k \overline{\psi_s}) \psi^\infty_k \\
  &\quad + 2 i (\calL_\Omega\ringA^k) \nabla_k \psi_s+2 i \ringA^k \nabla_k \Omega \psi_s - 2 A^{\infty, k} (\calL_\Omega\ringA_k) \psi_s- 2 A^{\infty, k} \ringA_k \Omega \psi_s\\
  &\quad + i (\nabla_k \calL_\Omega\ringA^k) \psi_s + i (\nabla_k \ringA^k) \Omega \psi_s - 2(\calL_\Omega\ringA^k) \ringA_k \psi_s - \ringA^k \ringA_k \Omega \psi_s .
}

\subsection{Parabolic regularity theory for $\psi_s$}

The goal of this section is to prove the following parabolic regularity estimate for $\psi_s$. 
\begin{proposition}\label{p:par-reg-psi_s}  
Under the bootstrap hypothesis~\eqref{eq:overall-bootstrap} and~\eqref{eq:core-bootstrap} there holds
\begin{align} 
   \| \langle s \Delta \rangle^{\frac{5}{2}} \langle \Omega \rangle m(s) s^{\frac{1}{2}} \psi_s \|_{L^\infty_{\ds} (L^\infty_t L^2_x \cap L^4_t L^4_x \cap L^{\frac{8}{3}}_t L^{\frac{8}{3}}_x) } &\lesssim \|\psi_s\|_{\calS} , \label{equ:bound_psis_Linftys_Str} \\
    \| \langle s \Delta \rangle^{\frac{3}{2}} \langle \Omega \rangle  ( m(s) s^{\frac{1}{2}} \psi_s ) \|_{L^2_{\ds} L^\infty_t L^2_x} &\lesssim \|\psi_s\|_{\calS} , \label{equ:bound_psis_L2_Str} %\quad \Orange{\text{[for magnetic interaction term part]}}
   \end{align}
where the implicit constants above are independent of the constant $C_1$ from Proposition~\ref{prop:core-bootstrap}. 
%\EQ{
%\| s^{\frac{1}{2}} \psi_s \|_{L^\infty_{\ds} \cap L^2_{\ds} L^\infty_t L^2_x}  + \| s^{\frac{1}{2}} \na s^{\frac{1}{2}} \psi_s \|_{L^\infty_{\ds} \cap L^2_{\ds} L^\infty_t L^2_x}  + \| s \De s^{\frac{1}{2}} \psi_s \|_{L^\infty_{\ds} \cap L^2_{\ds} L^\infty_t L^2_x} \lesssim    \| \psi_s \|_{\calS}
%}
\end{proposition} 

The proof requires the following lemmas. The first is a direct consequence of Proposition~\ref{prop:forward-caloric-psi} and Proposition~\ref{prop:forward-caloric-A} from Section~\ref{sec:hmhf} and the second follows from the first and the Gagliardo-Nirenberg inequality. 

\begin{lemma} \label{l:ringPsi-L2}  Under the bootstrap assumptions~\eqref{eq:overall-bootstrap} we have 
\EQ{
 \label{equ:bound_ringPsi_Linftys_Str}
    \| \langle s \Delta \rangle^2 \langle \calL_\Omega \rangle \ringPsi \|_{L^\infty_{\ds} L^\infty_t L^2_x  } \lesssim \sqrt{\epsilon_0}.
}
\end{lemma} 

\begin{proof} This follows from an identical proof as the proof of Proposition~\ref{prop:forward-caloric-psi} where now we consider the one parameter (in $t\in I$) family of heat flows $u(t, x, s)$ and we use the bootstrap assumption~\eqref{eq:overall-bootstrap} to bound the right-hand sides of~\eqref{eq:forward-caloric-Psi} and~\eqref{eq:forward-caloric-Omg-Psi}. 
\end{proof} 

\begin{lemma} \label{l:APsi-infty} Under the bootstrap assumptions~\eqref{eq:overall-bootstrap}  we have 
\EQ{
  \| \langle s \Delta \rangle^{\frac{5}{2}} \langle \calL_\Omega \rangle s^{\frac{1}{2}} \ringA \|_{  L^\infty_{\ds} L^\infty_t L^\infty_x} \lesssim \sqrt{\eps_0}
  }
as well as, 
\EQ{
  \| \ang{s \De}^{2}\langle \calL_\Omega \rangle s^{\frac{1}{2}} \ringPsi \|_{ L^\infty_{\ds} L^\infty_t L^\infty_x} \lesssim \sqrt{\eps_0} %\\
 % \| \langle \calL_\Omega \rangle s \ringPsi \|_{L^\infty_{\ds} L^\infty_t L^\infty_x} \lesssim \eps_0
}
\end{lemma} 
\begin{proof} 
The proof is an immediate consequence of the Proposition~\ref{prop:forward-caloric-A}, Lemma~\ref{l:ringPsi-L2} and the Gagliardo-Nirenberg inequality, i.e., equation~\eqref{eq:GN-Linfty}. The argument is identical to the one presented in~\cite[Proof of Lemma 5.2]{LOS5}, and we refer the reader to~\cite{LOS5} for details. 
\end{proof}

\begin{proof}[Proof of Proposition~\ref{p:par-reg-psi_s}] 
The proof follows a nearly identical outline as~\cite[Proof of Proposition 5.1]{LOS5} using now Corollary~\ref{c:lin-para-Lp} as the main technical ingredient. As the arguments for each of the estimates in Proposition~\ref{p:par-reg-psi_s} are essentially the same, we just give the proof of just one of the estimates, namely, 
\EQ{ \label{eq:p-reg1d} 
\| s^{\frac{1}{2}}  (-\De)^{\frac{1}{2}} \ang{ \Om} m(s) s^{\frac{1}{2}} \psi_s \|_{L^\infty_{\ds} L^4_{t, x}} \lesssim \| \psi_s \|_{\calS} .
}
By the Duhamel formula and the nonlinear heat equation~\eqref{eq:psi_s-heat} we write 
\EQ{
\psi_s(s) = e^{-\frac{s}{2} H} \psi_s( s/2) + \int_{\frac{s}{2}}^s e^{-(s-s') H} \NN(\psi_s, \Psi, A)(s') \, \ud s' 
}
for each $s >0$. It follows that
\EQ{
s^{\frac{1}{2}}  (-\De)^{\frac{1}{2}}\ang{ \Om} m(s) s^{\frac{1}{2}} \psi_s(s)   &= s^{\frac{1}{2}}  (-\De)^{\frac{1}{2}}  e^{-\frac{s}{2} H} \ang{ \Om} m(s)s^{\frac{1}{2}} \psi_s( s/2) \\
&\quad + s m(s)\int_{\frac{s}{2}}^s (-\De)^{\frac{1}{2}}  e^{-(s-s') H}\ang{ \Om}  \NN(\psi_s, \Psi, A)(s') \, \ud s' .
}
Using the heat semigroup bound from Corollary~\ref{c:lin-para-Lp} with $ p =4$ we deduce that 
\EQ{
\| s^{\frac{1}{2}}  (-\De)^{\frac{1}{2}} \ang{ \Om} m(s)s^{\frac{1}{2}} \psi_s \|_{ L^4_{t, x}} &\lesssim \| (s/2)^{\frac{1}{2}} \ang{ \Om} m(s) \psi_s(s/2) \|_{ L^4_{t, x}} \\
&\quad + \int_{\frac{s}{2}}^s\frac{sm(s)}{(s-s')^{\frac{1}{2}}(s')^{\frac{1}{2}}m(s')} \| (s')^{\frac{3}{2}}m(s')\ang{ \Om}  \NN(\psi_s, \Psi, A)(s')  \|_{L^{4}_{t, x}} \, \frac{\ud s' }{s'} .
}
Via Schur's test we then arrive at 
\EQ{
\| s^{\frac{1}{2}}  (-\De)^{\frac{1}{2}}\ang{ \Om}  m(s)s^{\frac{1}{2}} \psi_s \|_{ L^\infty_{\ds} L^4_{t, x}} \lesssim \|  \psi_s \|_{ \calS }  +  \| m(s) s^{\frac{3}{2}} \ang{ \Om} \NN(\psi_s, \Psi, A)  \|_{L^{\infty}_{\ds} L^{4}_{t, x}} .
}
To prove~\eqref{eq:p-reg1d} it then suffices to show that 
\EQ{\label{eq:s32N4} 
\|m(s)  s^{\frac{3}{2}} \ang{ \Om} \NN(\psi_s, \Psi, A)  \|_{L^{\infty}_{\ds} L^{4}_{t, x}}  \lesssim \sqrt{\eps_0}  \| s^{\frac{1}{2}} \na \ang{ \Om} m(s) s^{\frac{1}{2}} \psi_s \|_{L^{\infty}_{\ds} L^{4}_{t, x}}  + \|  \psi_s \|_{ \calS } ,
}
since the first term on the right above can be absorbed into the left-hand side of \eqref{eq:p-reg1d} as long as $\eps_0$ is taken small enough. 

To prove~\eqref{eq:s32N4} we estimate each of the terms on the right-hand sides of~\eqref{eq:nonlin1} and~\eqref{eq:Omnonlin1}. The main ingredients in these estimates are Lemma~\ref{l:APsi-infty} and the $L^p$-Poincar\'e inequality~\eqref{eq:Pip}. We argue as follows. Consider the first term on the right-hand side of~\eqref{eq:nonlin1} along with the first term on the right-hand side of~\eqref{eq:Omnonlin1}. We have 
\EQ{
\|m(s)  s^{\frac{3}{2}} \Im( \psi_s \ba{ \ringpsi^k}) \ringpsi_k \|_{L^{\infty}_{\ds} L^{4}_{t, x}} \lesssim  \| s^{\frac{1}{2} }\ringPsi \|_{L^\infty_{\ds} L^{\infty}_{t, x}}^2 \|m(s) s^{\frac{1}{2}} \psi_s \|_{L^{\infty}_{\ds} L^{4}_{t, x}}  \lesssim  \eps_0 \| \psi_s \|_{\calS}  , \\
\|m(s)  s^{\frac{3}{2}} \Im( \LL_\Om \ringpsi^k\ba{\psi_s}  ) \ringpsi_k \|_{L^{\infty}_{\ds} L^{4}_{t, x}} \lesssim  \| s^{\frac{1}{2} }\LL_\Om \ringPsi \|_{L^\infty_{\ds} L^{\infty}_{t, x}} \| s^{\frac{1}{2} }\ringPsi \|_{L^\infty_{\ds} L^{\infty}_{t, x}} \|m(s) s^{\frac{1}{2}} \psi_s \|_{L^{\infty}_{\ds} L^{4}_{t, x}}  \lesssim  \eps_0 \| \psi_s \|_{\calS} ,
}
where in the last inequalities above we used Lemma~\ref{l:APsi-infty}. Next consider the second term on the right-hand side of~\eqref{eq:nonlin1} and the fourth term on the right-hand side of~\eqref{eq:Omnonlin1}. We have 
\EQ{
 \|m(s)  s^{\frac{3}{2}} \Im( \psi_s \ba{ \ringpsi^k}) \ringpsi_k^\infty \|_{L^{\infty}_{\ds} L^{4}_{t, x}} & \lesssim \| s^{\frac{1}{2} }\ringPsi \|_{L^\infty_{\ds} L^{\infty}_{t, x}} \| \Psi^{\infty} \|_{L^\infty} \|m(s) s \psi_s \|_{L^{\infty}_{\ds} L^{4}_{t, x}}  \\
 & \lesssim \sqrt{\eps_0} \| s^{\frac{1}{2}} (-\De)^{\frac{1}{2}}m(s)  s^{\frac{1}{2}}  \psi_s \|_{L^\infty_{\ds} L^{4}_{t, x}} ,  \\
  \|m(s)  s^{\frac{3}{2}} \Im(  \LL_\Om \ringpsi^k \ba{\psi_s}) \ringpsi_k^\infty \|_{L^{\infty}_{\ds} L^{4}_{t, x}} & \lesssim \| s^{\frac{1}{2} } \LL_\Om \ringPsi \|_{L^\infty_{\ds} L^{\infty}_{t, x}} \| \Psi^{\infty} \|_{L^\infty} \|m(s) s \psi_s \|_{L^{\infty}_{\ds} L^{4}_{t, x}}  \\
 & \lesssim \sqrt{\eps_0}  \| s^{\frac{1}{2}} (-\De)^{\frac{1}{2}}m(s)  s^{\frac{1}{2}}  \psi_s \|_{L^\infty_{\ds} L^{4}_{t, x}}   ,
 }
where in the last lines we again used Lemma~\ref{l:APsi-infty} and the $L^p$-Poincar\'e inequality to absorb the extra factor of $s^{\frac{1}{2}}$ on the term involving $\psi_s$.   

Next consider the fourth term on the right-hand side of~\eqref{eq:nonlin1} and the tenth term on the right-hand side of~\eqref{eq:Omnonlin1}. We have 
\EQ{
\|m(s)  s^{\frac{3}{2}}\ringA^k \na_k \psi_s \|_{L^{\infty}_{\ds} L^{4}_{t, x}} &\lesssim  \| s^{\frac{1}{2} }A \|_{L^\infty_{\ds} L^{\infty}_{t, x}}^2 \|m(s) s \na \psi_s \|_{L^{\infty}_{\ds} L^{4}_{t, x}}  \lesssim \sqrt{\eps_0}  \| s^{\frac{1}{2}} (-\De)^{\frac{1}{2}}m(s)  s^{\frac{1}{2}}  \psi_s \|_{L^\infty_{\ds} L^{4}_{t, x}}  , \\
\|m(s)  s^{\frac{3}{2}}( \LL_\Om \ringA^k) \na_k \psi_s  \|_{L^{\infty}_{\ds} L^{4}_{t, x}} &\lesssim   \| s^{\frac{1}{2}}\LL_\Om  A \|_{L^\infty_{\ds} L^{\infty}_{t, x}}^2 \|m(s) s \na \psi_s \|_{L^{\infty}_{\ds} L^{4}_{t, x}}  \lesssim \sqrt{\eps_0} \| s^{\frac{1}{2}} (-\De)^{\frac{1}{2}}m(s)  s^{\frac{1}{2}}  \psi_s \|_{L^\infty_{\ds} L^{4}_{t, x}} ,
}
again using Lemma~\ref{l:APsi-infty}. The remaining terms on the right-hand-sides of~\eqref{eq:nonlin1} and~\eqref{eq:Omnonlin1} can be handled in a similar manner to the previous three groupings of estimates. This completes the proof of~\eqref{eq:s32N4}. The remaining estimates in Proposition~\ref{p:par-reg-psi_s} are proved in an identical manner, after passing each additional derivative onto the nonlinearities~\eqref{eq:nonlin1} and~\eqref{eq:Omnonlin1}. 
\end{proof}

\subsection{Off-diagonal decay estimates } 
In later sections we will use the linear heat flow to localize $\psi_s(s)$ in frequency (measured in linear heat time $\sigma$) at each fixed nonlinear heat flow time $s$. In this section we show that these two heat flow frequency localizations are compatible in the sense made precise by the following lemmas. 
\begin{lemma}[Off-Diagonal decay] \label{l:offdiag} 
Let $(p, q) = (4, 4)$ or $(p, q) = (8/3, 8/ 3)$. Fix any $\s \le 1$. Then there exists $\be \in (0, 1)$ so that   
\EQ{
\| P_\s s^{\frac{1}{2}}\psi_s(s) \|_{L^p_t L^q_x} \lesssim \int_s^\infty  \left(\frac{s}{s'}\right)^{\frac{1}{2}}  \Big(\frac{ s'}{\s} \Big)^\be  \sum_{\ell=0}^2 \|(s')^{\frac{\ell}{2}} (-\De)^{\frac{\ell}{2}} (s')^{\frac{1}{2}}\psi_s(s') \|_{\Str_s} \,  \frac{\ud s'}{s'}   .
}
%\EQ{
%\| P_\s s^{\frac{1}{2}}\psi_s(s) \|_{L^p_t L^q_x} \lesssim \int_s^\infty  \left(\frac{s}{s'}\right)^{\frac{1}{2}}  \left(\frac{ s'}{\s} \right)^{\frac{3}{8}} \|  (s' )^{1-\be} (-\De)^{1-\be}  (s')^{\frac{1}{2}}\psi_s(s')) \|_{L^p_t L^q_x} \,  \frac{\ud s'}{s'}   +  \dots 
%}
\end{lemma} 
The above estimate becomes useful after integration in $\s \ge s$. 
\begin{corollary}  \label{c:offdiag} 
Let $(p, q) = (4, 4)$ or $(p, q) = (8/3, 8/ 3)$. Then,  %and let $\be>0$ be such that $\be + \de < \frac{1}{2}$ and $\be >\frac{1}{4}$. For example, we can take $\be = \frac{3}{8}$. Then,  
\EQ{
 \left\| \int_s^1  \| P_\s ( m(s) s^{\frac{1}{2}}  \psi_s(s)) \|_{L^p_t L^q_x}  \frac{ \ud \s}{\s} \right\|_{L^{2}_{\ds} \cap L^{\infty}_{\ds}} \lesssim  \sum_{\ell = 0}^2  \| s^{\frac{\ell}{2}}(-\De)^{\frac{\ell}{2}}m(s) s^{\frac{1}{2}}\psi_s \|_{L^{2}_{\ds} \cap L^{\infty}_{\ds} \Str_s} .%\Big( \frac{s}{\s}\Big)^{\beta} \sum_{n = 0}^2 \|(s^{\frac{1}{2}} \na)^n ( m(s) s^{\frac{1}{2}}  \psi_s) \|_{L^{\infty}_{\ds}L^p_t L^q_x}
}
\end{corollary}

We require the following lemma, which again follows from the regularity theory for the harmonic map heat flow in Section~\ref{sec:hmhf} and the bootstrap assumptions~\eqref{eq:overall-bootstrap}. 
\begin{lemma} \label{l:heatPsi} 
Under the bootstrap assumptions~\eqref{eq:overall-bootstrap} we have 
\EQ{
  \| s^{\frac{5}{4}} \na \psi_s \|_{L^\infty_{\ds} L^\infty_t L^4_x} + \| s^{\frac{11}{8}} \na \psi_s \|_{L^\infty_{\ds} L^\infty_t L^8_x}  + \| s^{\frac{23}{16}} \na \psi_s \|_{L^\infty_{\ds}L^\infty_t L^{16}_x}\lesssim  \sqrt{\eps_0}, 
}
\EQ{ \label{eq:heatPsi}
 \| s^{\frac{1}{4}} \ringPsi \|_{L^\infty_{\ds} L^\infty_t L^4_x} +  \| s^{\frac{3}{8}} \ringPsi \|_{L^\infty_{\ds} L^\infty_t L^8_x} +  \| s^{\frac{7}{16}} \ringPsi \|_{L^\infty_{\ds} L^\infty_t L^{16}_x} \lesssim \sqrt{\eps_0}. 
}
\end{lemma} 
\begin{proof} 
The proof is an immediate consequence of the Gagliardo-Nirenberg inequality (Lemma~\ref{lem:gagliardo_nirenberg}) and  Proposition~\ref{prop:forward-caloric-psi}. The argument is identical to the one presented in~\cite[Proof of Lemma 5.2]{LOS5}, to which we refer the reader for details.
\end{proof} 

\begin{proof}[Proof of Lemma~\ref{l:offdiag}]
The proof of Lemma~\ref{l:offdiag} is based on the fact that $\psi_s$ satisfies the nonlinear heat equation~\eqref{eq:psi_s-heat}. 
Using~\eqref{eq:psi_s-heat} we write 
\EQ{
 \psi_s(s) &= - \int_s^\infty  s'  \p_s \psi_s(s') \, \frac{\ud s'}{s'}  \\
 & = \int_s^\infty s' H \psi_s(s')\, \frac{\ud s'}{s'}  - \int_s^\infty s' \calN(\psi_s, \Psi, A)\, \frac{\ud s'}{s'} .
}
To clarify the exposition  we prove the lemma for $(p, q) = (4, 4)$ --  the argument is nearly identical for $(p, q) = (8/3, 8/3)$.  We fix  $\be = \frac{3}{8}$. 
\EQ{
\| P_\s (  s^{\frac{1}{2}}  \psi_s(s)) \|_{L^4_{t, x}} &\le s^{\frac{1}{2}} \int_s^\infty \| P_\s(s' H \psi_s(s')) \|_{L^4_{t, x}} \, \frac{\ud s'}{s'} \\
& \quad +     s^{\frac{1}{2}} \int_s^\infty  \|P_\s(s' \calN(\psi_s, \Psi, A)) \|_{L^4_{t, x}} \, \frac{\ud s'}{s'}  \\
& = I+ II .
}
Consider the first term $I$ on the right-hand side above.  Using the expression~\eqref{eq:Hdef1} we obtain the bound, 
\EQ{
 I  &=  s^{\frac{1}{2}}  \int_s^\infty  \| \s \De e^{\s \De} (s' H \psi_s(s')) \|_{L^4_{t, x}} \, \frac{\ud s'}{s'}   \\
& \lesssim s^{\frac{1}{2}}  \int_s^\infty  \| \s \De e^{\s \De} (s' (-\De) \psi_s(s')) \|_{L^4_{t, x}} \, \frac{\ud s'}{s'}   \\
& \quad + s^{\frac{1}{2}}  \int_s^\infty  \| \s \De e^{\s \De} (s' A^\infty_k \na^k \psi_s(s')) \|_{L^4_{t, x}} \, \frac{\ud s'}{s'}   \\
& \quad  + s^{\frac{1}{2}}  \int_s^\infty  \| \s \De e^{\s \De} (s' (\na^k A^\infty_k)  \psi_s(s')) \|_{L^4_{t, x}} \, \frac{\ud s'}{s'}   \\
& \quad + s^{\frac{1}{2}}  \int_s^\infty  \| \s \De e^{\s \De} (s' ( \abs{ A^\infty}^2 + \abs{ \psi_2^\infty}^2) \psi_s(s')) \|_{L^4_{t, x}} \, \frac{\ud s'}{s'}  \\
& = I_a + I_b + I_c + I_d.
}
We estimate $I_a$ as follows. Using $L^p$ regularity of the linear heat flow, we have 
\EQ{
I_a  &\lesssim s^{\frac{1}{2}}  \int_s^\infty  \Big(\frac{s'}{\s}\Big)^{\frac{3}{8}} \| (\s)^{1+ \frac{3}{8}} (-\De)^{1+ \frac{3}{8}}  e^{\s \De} (s')^{1-\frac{3}{8}} (-\De)^{1-\frac{3}{8}} \psi_s(s')) \|_{L^4_{t, x}} \, \frac{\ud s'}{s'}   \\
& \lesssim \int_s^\infty  \left(\frac{ s}{s'}\right)^{\frac{1}{2}}  \left(\frac{ s'}{\s} \right)^{\frac{3}{8}} \|  (s' )^{1-\frac{3}{8}} (-\De)^{1-\frac{3}{8}} (s')^{\frac{1}{2}}\psi_s(s')) \|_{L^4_{t, x}} \, \frac{\ud s'}{s'}  \\
& \lesssim \int_s^\infty  \left(\frac{ s}{s'}\right)^{\frac{1}{2}}  \left(\frac{ s'}{\s} \right)^{\frac{3}{8}}   \sum_{\ell = 0}^2 \|  (s')^{\ell/2} (-\De)^{\ell/2}  (s')^{\frac{1}{2}}\psi_s(s')) \|_{L^4_{t, x}} \,  \frac{\ud s'}{s'} ,
} 
where in the last line we used the $L^p$ interpolation inequality from Lemma~\ref{l:Lp-int}. 

Next for term $I_b$ we note that since $\s \le 1$ we have $1 \le \s^{-\frac{3}{8}}$.
\EQ{
I_b  &\lesssim s^{\frac{1}{2}} \int_s^\infty   \| \s \De e^{\s \De} (s' A^\infty_k \na^k \psi_s(s')) \|_{L^{4}_{t, x}} \,  \frac{\ud s'}{s'} \\
& \lesssim \int_s^\infty    \left(\frac{ s}{s'}\right)^{\frac{1}{2}} \s^{-\frac{3}{8}}   \| s' A^\infty_k \na^k (s')^{\frac{1}{2}}\psi_s(s') \|_{L^4_{t, x}} \, \frac{\ud s'}{s'} \\
& \lesssim \| A^\infty \|_{L^\infty_t L^8_x} \int_s^\infty  \left(\frac{ s}{s'}\right)^{\frac{1}{2}} \Big(\frac{s'}{\s}\Big)^{\frac{3}{8}} (s')^{1-\frac{3}{8}} \| \na( (s')^{\frac{1}{2}} \psi_s(s') )\|_{L^4_{t}L^8_x} \,   \frac{\ud s'}{s'} \\
&  \lesssim\| A^\infty \|_{L^\infty_t L^8_x} \int_s^\infty  \left(\frac{ s}{s'}\right)^{\frac{1}{2}} \Big(\frac{s'}{\s}\Big)^{\frac{3}{8}}   \|(s')^{1-\frac{3}{8}} (-\De)^{\frac{5}{8}} (s')^{\frac{1}{2}}\psi_s(s') \|_{L^4_{t}L^{4}_x} \,  \frac{\ud s'}{s'}  \\
& \lesssim \int_s^\infty  \left(\frac{ s}{s'}\right)^{\frac{1}{2}} \Big(\frac{s'}{\s}\Big)^{\frac{3}{8}}   \sum_{\ell = 0}^2 \|  (s')^{\ell/2} (-\De)^{\ell/2}  (s')^{\frac{1}{2}}\psi_s(s')) \|_{L^4_{t, x}} \,  \frac{\ud s'}{s'} ,
}
where in the third-to-last line we used Sobolev embedding.  

Next, consider term $I_c$. Arguing as above we have 
\EQ{
I_c & \lesssim\| \na A^\infty \|_{L^\infty_t L^8_x} \int_s^\infty  \left(\frac{ s}{s'}\right)^{\frac{1}{2}} \Big(\frac{s'}{\s}\Big)^{\frac{3}{8}}   \|(s')^{1-\frac{3}{8}} (-\De)^{\frac{1}{4}} (s')^{\frac{1}{2}}\psi_s(s') \|_{L^4_{t}L^{4}_x} \,  \frac{\ud s'}{s'} \\
& \lesssim\| \na A^\infty \|_{L^\infty_t L^8_x} \int_s^\infty  \left(\frac{ s}{s'}\right)^{\frac{1}{2}} \Big(\frac{s'}{\s}\Big)^{\frac{3}{8}}   \|(s')^{1-\frac{3}{8}} (-\De)^{\frac{5}{8}} (s')^{\frac{1}{2}}\psi_s(s') \|_{L^4_{t}L^{4}_x} \,  \frac{\ud s'}{s'} \\
& \lesssim \int_s^\infty  \left(\frac{ s}{s'}\right)^{\frac{1}{2}} \Big(\frac{s'}{\s}\Big)^{\frac{3}{8}}   \sum_{\ell = 0}^2 \|  (s')^{\ell/2} (-\De)^{\ell/2}  (s')^{\frac{1}{2}}\psi_s(s')) \|_{L^4_{t, x}} \,  \frac{\ud s'}{s'} ,
}
where in the second line we used the $L^p$ version of the Poincare inquality. Term $I_d$ is handled similarly. 

Next, consider the nonlinear term $II$. Using~\eqref{eq:nonlin1} we have %We estimate each of the terms on the right-hand side of~\eqref{eq:nonlin1}. 
\EQ{
 s^{\frac{1}{2}} \int_s^\infty  \|P_\s(s' \calN(\psi_s, \Psi, A)) \|_{L^4_{t, x}} \, \frac{\ud s'}{s'} & \lesssim s^{\frac{1}{2}} \int_s^\infty  \|P_\s(s'  \Im( \ba{\psi_s}  \ringpsi^k) \ringpsi_k ) \|_{L^4_{t, x}} \, \frac{\ud s'}{s'} \\
 &\quad + s^{\frac{1}{2}} \int_s^\infty  \|P_\s(s'  \Im( \ba{\psi_s} \ringpsi_k)\psi^{\infty, k}) \|_{L^4_{t, x}} \, \frac{\ud s'}{s'}\\
  &\quad + s^{\frac{1}{2}} \int_s^\infty  \|P_\s(s' \Im( \ba{\psi_s}  \psi^{\infty, k}) \ringpsi_k) \|_{L^4_{t, x}} \, \frac{\ud s'}{s'}\\
&\quad + s^{\frac{1}{2}} \int_s^\infty  \|P_\s(s'  \ringA^k \na_k \psi_s) \|_{L^4_{t, x}} \, \frac{\ud s'}{s'}\\
&\quad + s^{\frac{1}{2}} \int_s^\infty  \|P_\s(s' (\na^k \ringA_k )\psi_s ) \|_{L^4_{t, x}} \, \frac{\ud s'}{s'}\\
&\quad + s^{\frac{1}{2}} \int_s^\infty  \|P_\s(s'  \ringA^k \ringA_k \psi_s)\|_{L^4_{t, x}} \, \frac{\ud s'}{s'}\\
&\quad + s^{\frac{1}{2}} \int_s^\infty  \|P_\s(s' ( \ringA^k A^\infty_k  \psi_s ) \|_{L^4_{t, x}} \, \frac{\ud s'}{s'}\\
%&\quad + s^{\frac{1}{2}} \int_s^\infty  \|P_\s(s' (A^\infty_k \ringA^k \psi_s ) \|_{L^4_{t, x}} \, \frac{\ud s'}{s'}\\
& = II_a + II_b + \dots + II_g.
}
First, we treat $II_a$. Using  Sobolev embedding in the last line below we have 
\EQ{
II_a &= s^{\frac{1}{2}} \int_s^\infty  \| \s \De e^{\s \De}(s'  \Im( \ba{\psi_s}  \ringpsi^k) \ringpsi_k ) \|_{L^4_{t, x}} \, \frac{\ud s'}{s'}  \\
&  \lesssim s^{\frac{1}{2}} \int_s^\infty \s^{-\frac{3}{8}} \| \s^{1+\frac{3}{8}} (-\De)^{1+ \frac{3}{8}} e^{\s \De}(-\De)^{-\frac{3}{8}}(s'  \Im( \ba{\psi_s}  \ringpsi^k) \ringpsi_k ) \|_{L^4_{t, x}} \, \frac{\ud s'}{s'}  \\
& \lesssim s^{\frac{1}{2}} \int_s^\infty \s^{-\frac{3}{8}} \| (-\De)^{-\frac{3}{8}}(s'   \Im( \ba{\psi_s}  \ringpsi^k) \ringpsi_k ) \|_{L^4_{t, x}} \, \frac{\ud s'}{s'}  \\
& \lesssim s^{\frac{1}{2}} \int_s^\infty \Big( \frac{s'}{\s} \Big)^{\frac{3}{8}} \| (s')^{1-\frac{3}{8}}   \Im( \ba{\psi_s}  \ringpsi^k) \ringpsi_k ) \|_{L^4_{t} L^{\frac{8}{5}}_x } \, \frac{\ud s'}{s'} .
}
We can deduce from the last line above that 
\EQ{
II_a & \lesssim s^{\frac{1}{2}} \int_s^\infty \Big( \frac{s'}{\s} \Big)^{\frac{3}{8}}(s')^{1-\frac{3}{8}}  \| \ringPsi \|_{L^\infty_t L^4_x}^2 \|    \psi_s \|_{L^4_t L^8_x}    \, \frac{\ud s'}{s'} \\ 
 & \lesssim  \int_s^\infty \Big( \frac{s}{s'}\Big)^{\frac{1}{2}}\Big( \frac{s'}{\s} \Big)^{\frac{3}{8}}  \| (s')^{\frac{1}{4}}\ringPsi \|_{L^\infty_t L^4_x}^2 \|  (s')^{\frac{1}{8}} (-\De)^{\frac{1}{8}} (s')^{\frac{1}{2}}  \psi_s \|_{L^4_t L^4_x}    \, \frac{\ud s'}{s'} \\ 
 & \lesssim  \int_s^\infty \Big( \frac{s}{s'}\Big)^{\frac{1}{2}}\Big( \frac{s'}{\s} \Big)^{\frac{3}{8}}  \| (s')^{\frac{1}{4}}\ringPsi \|_{L^\infty_t L^4_x}^2 \|  (s')^{\frac{1}{8}} (-\De)^{\frac{1}{8}} (s')^{\frac{1}{2}}  \psi_s \|_{L^4_t L^4_x}    \, \frac{\ud s'}{s'} \\
 &\lesssim \eps^2 \int_s^\infty \Big( \frac{s}{s'}\Big)^{\frac{1}{2}}\Big( \frac{s'}{\s} \Big)^{\frac{3}{8}}  \sum_{\ell =0}^2 \|  (s')^{\frac{\ell}{2}} (-\De)^{\frac{\ell}{2}} (s')^{\frac{1}{2}}  \psi_s \|_{L^4_t L^4_x}    \, \frac{\ud s'}{s'} , 
}
where in the last line we used~\eqref{eq:heatPsi} from Lemma~\ref{l:heatPsi}. The remaining terms in the nonlinearity can be handled similarly. 
\end{proof} 

\begin{proof}[Proof of Corollary~\ref{c:offdiag}]
The proof is a consequence of Lemma~\ref{l:offdiag} and Schur's test. Indeed, set $\be = \frac{3}{8}$ in Lemma~\ref{l:offdiag} (in fact the proof is carried out for this particular value for $\beta$). Then we have 
\EQ{
 \int_s^1 \| P_\s ( m(s) &s^{\frac{1}{2}}  \psi_s(s)) \|_{L^p_t L^q_x} \, \frac{\ud \s}{\s}   \\
 & \lesssim  \int_s^1 m(s)  \int_s^\infty  \left(\frac{s}{s'}\right)^{\frac{1}{2}}  \Big(\frac{ s'}{\s} \Big)^{\frac{3}{8}}  \sum_{\ell=0}^2 \|(s')^{\frac{\ell}{2}} (-\De)^{\frac{\ell}{2}} (s')^{\frac{1}{2}}\psi_s(s') \|_{\Str_s} \,  \frac{\ud s'}{s'}   \frac{\ud \s}{\s}  \\
 & \lesssim  \int_s^\infty \frac{ m(s)}{m(s')} \left(\frac{s}{s'}\right)^{\frac{1}{2}} \sum_{\ell=0}^2 \|(s')^{\frac{\ell}{2}} (-\De)^{\frac{\ell}{2}} (s')^{\frac{1}{2}}\psi_s(s') \|_{\Str_s}  \int_s^1   \Big(\frac{ s'}{\s} \Big)^{\frac{3}{8}}    \frac{\ud \s}{\s} \,  \frac{\ud s'}{s'}  \\
 &  \lesssim  \int_s^\infty \frac{ m(s)}{m(s')} \left(\frac{s}{s'}\right)^{\frac{1}{2} - \frac{3}{8}} \sum_{\ell=0}^2 \|(s')^{\frac{\ell}{2}} (-\De)^{\frac{\ell}{2}} (s')^{\frac{1}{2}}\psi_s(s') \|_{\Str_s}   \,  \frac{\ud s'}{s'}.
}
To conclude we use the fact that 
\EQ{
K_{\frac{3}{8}}(s, s') = 1_{\{s' > s>0\}}(s, s')  \frac{m(s)}{m(s')} \left(\frac{ s}{s'}\right)^{\frac{1}{2}- \frac{3}{8}}
}
is a Schur kernel, and hence by Schur's test we obtain
\EQ{
 \bigg\|  \int_s^1 \| P_\s ( m(s) &s^{\frac{1}{2}}  \psi_s(s)) \|_{L^p_t L^q_x} \, \frac{\ud \s}{\s}  \bigg\|_{L^\infty_{\ds} \cap L^2_{\ds}} \lesssim  \sum_{\ell = 0}^2  \| s^{\frac{\ell}{2}}(-\De)^{\frac{\ell}{2}}m(s) s^{\frac{1}{2}}\psi_s \|_{L^{2}_{\ds} \cap L^{\infty}_{\ds} \Str_s} , 
}
as desired. 
To see that $K_{\be}(s, s')$ is a Schur kernel for any $\be < \frac{1}{2}$ we note that for $s \le 1$ we have 
\EQ{
m(s)s^{\frac{1}{2}- \be} \int_s^\infty (s')^{-\frac{1}{2} + \be} \frac{1}{m(s')} \, \frac{\ud s'}{s'}    \lesssim s^{\frac{1}{2} - \be - \de}  \int_s^\infty (s')^{-\frac{1}{2} + \be + \de}  \, \frac{\ud s'}{s'}   \lesssim 1, 
}
where above we recalled that $m(s) = \max(s^{-\de}, 1)$. 
If $s \ge 1$ we have 
\EQ{
m(s)s^{\frac{1}{2}- \be} \int_s^\infty (s')^{-\frac{1}{2} + \be} \frac{1}{m(s')} \, \frac{\ud s'}{s'}    \lesssim s^{\frac{1}{2} - \be}  \int_s^\infty (s')^{-\frac{1}{2} + \be}  \, \frac{\ud s'}{s'}   \lesssim 1.
}
Similarly, if $s' \le 1$,  
\EQ{
(s')^{-\frac{1}{2} + \be} m(s')^{-1} \int_0^{s'} s^{\frac{1}{2} - \be} m(s)  \, \ds  \lesssim (s')^{-\frac{1}{2} + \be + \de}  \int_0^{s'} s^{\frac{1}{2} - \be - \de}   \, \ds  \lesssim 1.
}
Finally, if $s'  \ge 1$ we have 
\EQ{
(s')^{-\frac{1}{2} + \be} m(s')^{-1} \int_0^{s'} s^{\frac{1}{2} - \be} m(s)  \, \ds  \lesssim  1.
}
This completes the proof. 
\end{proof}

\section{Preparations for the analysis of the Schr\"odinger equation for $\psi_s$}  \label{s:w} 

Working under the bootstrap assumptions \eqref{eq:overall-bootstrap} and \eqref{eq:core-bootstrap},  we prove a number of auxiliary estimates that will go into the analysis of the Schr\"odinger equation~\eqref{eq:SH-psi_s}  for $\psi_s$. In Subsection~\ref{sub:6basic} we prove additional estimates on $\ringA$ and $\ringPsi$ that are consequences of the previous section. Then in Subsection~\ref{sub:6w} we study the parabolic equation~\eqref{eq:heatH-w} for $w$, deducing bounds for $w$ and $\p_s w$. One of these will go directly in the analysis of the Schr\"odinger equation for $\psi_s$ and the other is used in Subsection~\ref{sub:6w} to prove auxiliary bounds on $A_t$ and $\psi_t$. Finally in Subsection~\ref{sub:6refined} we prove a few more involved estimates on $\ringA$.

\subsection{Basic estimates on $\ringA$ and $\ringPsi$}\label{sub:6basic}
Here we prove additional estimates on $\ringA$ and $\ringPsi$. First we introduce some notation. Recall that 
\EQ{
 \psi_\ell(s_0) &= - \int_{s_0}^\infty \bfD_\ell \psi_s (s)\, \ud s + \psi_\ell^\infty =: \ringpsi_\ell (s_{0}) + \psi_\ell^\infty , \\
A_\ell(s_0) &= - \int_{s_0}^\infty \Im (\psi_\ell \overline{\psi_s} ) (s) \, \ud s + A_\ell^\infty =:  \ringA_\ell (s_{0}) + A_\ell^\infty ,
}
where $\psi_\ell^\infty$ are the derivative components of the harmonic map $Q$ in the Coulomb gauge and $A_\ell^\infty$ is the associated connection form. We further decompose $\ringA_\ell$ into linear- and quadratic-in-$\ringpsi_\ell$ components, defining $\ringA_{L, \ell}, \ringA_{Q, \ell}$ by
\EQ{ \label{eq:ALAQ-def} 
 \ringA_\ell(s_0) &= - \int_{s_0}^\infty \Im (\psi_\ell \overline{\psi_s} ) \, \ud s \\
 &= - \int_{s_0}^\infty \Im (\psi_\ell^\infty \overline{\psi_s} ) \, \ud s - \int_{s_0}^\infty \Im ( \ringpsi_\ell \overline{\psi_s} ) \, \ud s \\
 &=: \ringA_{L,\ell}(s_0) + \ringA_{Q,\ell}(s_0) .
}
In addition, it follows from~\eqref{eq:LOmpsi}  that 
\begin{align*}
\begin{split}
\calL_\Omega \AL^k(s)=-\int_s^\infty \Im(\psi^{\infty,k}\overline{\Omega\psi_s}) \, \ud s'-\int_s^\infty \Im(i\psi^{\infty,k}\overline{\psi_s}) \, \ud s' ,
\end{split}
\end{align*}
and
\begin{align*}
\begin{split}
\calL_\Omega\AQ^k(s) = -\int_s^\infty \Im(\calL_\Omega\ringpsi^{k}\overline{\psi_s}) \, \ud s'-\int_s^\infty \Im(\ringpsi^{k}\overline{\Omega\psi_s}) \, \ud s' .
\end{split}
\end{align*}

\begin{lemma} \label{lem:bounds_ringA_Linfty}
 Under the bootstrap assumptions~\eqref{eq:overall-bootstrap} and~\eqref{eq:core-bootstrap}, we have that
 \begin{align} 
  \| \langle s \Delta \rangle^{\frac{1}{2}} \langle \calL_\Omega \rangle s^{\frac{1}{2}} \AL \|_{L^\infty_{\ds} L^\infty_t L^\infty_x} &\lesssim \|\psi_s\|_{\calS} ,\label{eq:bound_ringAL_Linftyall_Str} \\
  \| \langle s \Delta \rangle^{\frac{1}{2}} \langle \calL_\Omega \rangle s^{\frac{1}{2}} \AQ \|_{L^\infty_{\ds} L^\infty_t L^\infty_x} &\lesssim \sqrt{\epsilon_0} \|\psi_s\|_{\calS} .\label{eq:bound_ringAQ_Linftyall_Str}
 \end{align} 
\end{lemma}
\begin{proof}
 We begin with the proof of~\eqref{eq:bound_ringAL_Linftyall_Str}, where we first consider the base case without a covariant derivative or a Lie derivative. Using the Gagliardo-Nirenberg inequality, Poincar\'e's inequality, and~\eqref{equ:bound_psis_Linftys_Str}, we find for any $s > 0$ that
 \begin{align*}
  \| \ringA_L(s) \|_{L^\infty_t L^\infty_x} &\lesssim \int_s^\infty \| \Psi^\infty \|_{L^\infty_x} \| \psi_s \|_{L^\infty_t L^\infty_x} \, \ud s' \\
  &\lesssim \int_s^\infty \| \jap{s' \Delta} (s')^{\frac{1}{2}} \psi_s \|_{L^\infty_t L^2_x} \, \frac{\ud s'}{s'} \\
  &\lesssim \| \jap{s\Delta}^{\frac{3}{2}} s^{\frac{1}{2}} \psi_s \|_{L^\infty_{\ds} L^\infty_t L^2_x} \int_s^\infty (s')^{-\frac{3}{2}} \, \ud s' \\
  &\lesssim s^{-\frac{1}{2}} \|\psi_s\|_{\calS},
 \end{align*}
 which yields the bound
 \[
  \| s^{\frac{1}{2}} \ringA_L \|_{L^\infty_{\ds} L^\infty_t L^\infty_x} \lesssim \|\psi_s\|_{\calS}.
 \]
 The more general estimate~\eqref{eq:bound_ringAL_Linftyall_Str} follows in a similar manner. Next, we turn to the proof of~\eqref{eq:bound_ringAQ_Linftyall_Str}. Applying the Gagliardo-Nirenberg inequality and invoking the bounds~\eqref{equ:bound_ringPsi_Linftys_Str} as well as~\eqref{equ:bound_psis_Linftys_Str}, we obtain that
 \begin{align*}
  \| \ringA_Q(s) \|_{L^\infty_t L^\infty_x} &\lesssim \int_s^\infty \|\ringPsi\|_{L^\infty_t L^\infty_x} \|\psi_s\|_{L^\infty_t L^\infty_x} \, \ud s' \\
  &\lesssim \|\jap{s\Delta} \ringPsi \|_{L^\infty_{\ds} L^\infty_t L^2_x} \|\jap{s\Delta} s^{\frac{1}{2}} \psi_s \|_{L^\infty_{\ds} L^\infty_t L^2_x} \int_s^\infty (s')^{-\frac{3}{2}} \, \ud s' \\
  &\lesssim s^{-\frac{1}{2}} \sqrt{\epsilon_0} \|\psi_s\|_{\calS},
 \end{align*}
 which implies the estimate
 \[
  \| s^{\frac{1}{2}} \ringA_Q \|_{L^\infty_{\ds} L^\infty_t L^\infty_x} \lesssim \sqrt{\epsilon_0} \|\psi_s\|_{\calS}.
 \]
 Analogously, one derives the more general bound~\eqref{eq:bound_ringAQ_Linftyall_Str}.
\end{proof}

The next lemma concerns $\ringPsi$.

\begin{lemma} \label{lem:bounds_ringPsi}
 Assuming \eqref{eq:overall-bootstrap} and \eqref{eq:core-bootstrap}, there holds
 \begin{align}
  \| \langle \calL_\Omega \rangle s^{\frac{1}{2}} \ringPsi \|_{L^\infty_{\ds} L^\infty_t L^\infty_x} &\lesssim \|\psi_s\|_{\calS} ,  \label{eq:bound_ringPsi_Linftyall_Str2} \\
  \| \langle \calL_\Omega \rangle s \ringPsi \|_{L^\infty_{\ds} L^\infty_t L^\infty_x} &\lesssim \|\psi_s\|_{\calS}  , \label{eq:bound_ringPsi_Linftyall_Str1} \\
  \| \langle s \Delta \rangle^2 \langle \calL_\Omega \rangle \ringPsi \|_{L^\infty_{\ds} (L^4_t L^4_x \cap L^{\frac{8}{3}}_t L^{\frac{8}{3}}_x) } &\lesssim \|\psi_s\|_{\calS} .\label{equ:bound_ringPsi_Linftys_Str_2} 
 \end{align}
\end{lemma}
\begin{proof}
 First, we recall that
 \[
  \ringPsi(s) = - \int_s^\infty \bfD \psi_s \, \ud s' = - \int_s^\infty ( \nabla \psi_s + i \bfA^\infty \psi_s + i \ringA \psi_s ) \, \ud s'. 
 \]
 Using the Gagliardo-Nirenberg inequality, Poincar\'e's inequality, as well as the bounds~\eqref{equ:bound_psis_Linftys_Str} and \eqref{eq:bound_ringAL_Linftyall_Str}--\eqref{eq:bound_ringAQ_Linftyall_Str}, we conclude that
 \begin{align*}
  \|\ringPsi(s)\|_{L^\infty_t L^\infty_x} &\lesssim \int_s^\infty \| (s')^{\frac{1}{2}} \nabla (s')^{\frac{1}{2}} \psi_s \|_{L^\infty_t L^\infty_x} \frac{\ud s'}{s'} + \| \bfA^\infty \|_{L^\infty_x} \int_s^\infty (s')^{-\frac{1}{2}} \| (s')^{\frac{1}{2}} \psi_s \|_{L^\infty_t L^\infty_x} \, \ud s'  \\
  &\quad \quad \quad + \| s^{\frac{1}{2}} \ringA \|_{L^\infty_{\ds} L^\infty_t L^\infty_x} \int_s^\infty \| (s')^{\frac{1}{2}} \psi_s \|_{L^\infty_t L^\infty_x} \frac{\ud s'}{s'} \\
  &\lesssim \Bigl( (1 + \|\bfA^\infty\|_{L^\infty_x}) \| \jap{s\Delta}^{\frac{3}{2}} s^{\frac{1}{2}} \psi_s \|_{L^\infty_{\ds} L^\infty_t L^2_x} + \| s^{\frac{1}{2}} \ringA \|_{L^\infty_{\ds} L^\infty_t L^\infty_x} \| \jap{s\Delta} s^{\frac{1}{2}} \psi_s \|_{L^\infty_t L^2_x} \Bigr) \int_s^\infty (s')^{-\frac{3}{2}} \, \ud s' \\
  &\lesssim s^{-\frac{1}{2}} \bigl( (1 + \|\bfA^\infty\|_{L^\infty_x}) \|\psi_s\|_{\calS} + \|\psi_s\|_{\calS}^2 \bigr).
 \end{align*}
 Proceeding analogously to estimate $\|\calL_\Omega \ringPsi\|_{L^\infty_t L^\infty_x}$, we obtain the bound~\eqref{eq:bound_ringPsi_Linftyall_Str2}. The second bound~\eqref{eq:bound_ringPsi_Linftyall_Str1} follows in a similar manner upon applying Poincar\'e's inequality once more in order to gain an additional factor $(s')^{-\frac{1}{2}}$ of decay.

 Finally, we turn to the proof of~\eqref{equ:bound_ringPsi_Linftys_Str_2}. We only consider the base case without covariant derivatives or a Lie derivative, since the other cases can be treated similarly. Then we have 
 \begin{align*}
  \| \ringPsi \|_{L^\infty_{\ds} L^4_t L^4_x \cap L^{\frac{8}{3}}_t L^{\frac{8}{3}}_x} &\lesssim \int_0^\infty \| (s')^{\frac{1}{2}} \nabla (s')^{\frac{1}{2}} \psi_s \|_{L^4_t L^4_x \cap L^{\frac{8}{3}}_t L^{\frac{8}{3}}_x} \frac{\ud s'}{s'} \\
  &\quad \quad + \|\bfA^\infty\|_{L^\infty_x} \int_0^\infty \| (s')^{\frac{1}{2}} \psi_s \|_{L^4_t L^4_x \cap L^{\frac{8}{3}}_t L^{\frac{8}{3}}_x} (s')^{-\frac{1}{2}} \, \ud s' \\
  &\quad \quad + \int_0^\infty \| (s')^{\frac{1}{2}} \ringA\|_{L^\infty_t L^\infty_x} \| (s')^{\frac{1}{2}} \psi_s\|_{L^4_t L^4_x \cap L^{\frac{8}{3}}_t L^{\frac{8}{3}}_x} \, \frac{\ud s'}{s'}. 
 \end{align*}
 In order to estimate the first term on the right-hand side, we use the extra regularity\footnote{By ``extra regularity," we are referring to the hypothesis that the data $u_0$ is close to $Q$ in $H^{1+ 2\de}$ rather than just $H^1$, which is the energy topology. At a technical level this manifests in our inclusion of the $s$-weight $m(s)$ in the norm $\calS$, which here allows us to easily integrate in the small $s$ regime. In the sequel we will often loosely use the phrase ``extra regularity" as justification for this same type of argument.} and Poincar\'e's inequality together with the bound~\eqref{equ:bound_psis_Linftys_Str} to find that
 \begin{align*}
  \int_0^\infty \| (s')^{\frac{1}{2}} \nabla (s')^{\frac{1}{2}} \psi_s \|_{L^4_t L^4_x \cap L^{\frac{8}{3}}_t L^{\frac{8}{3}}_x} \frac{\ud s'}{s'} &\lesssim \| \jap{s\Delta}^{\frac{1}{2}} m(s) s^{\frac{1}{2}} \psi_s \|_{L^\infty_{\ds} L^4_t L^4_x \cap L^{\frac{8}{3}}_t L^{\frac{8}{3}}_x} \int_0^1 (s')^{-1+\delta} \, \ud s' \\
  &\quad + \| \jap{s\Delta} s^{\frac{1}{2}} \psi_s \|_{L^\infty_{\ds} L^4_t L^4_x \cap L^{\frac{8}{3}}_t L^{\frac{8}{3}}_x} \int_1^\infty (s')^{-\frac{3}{2}} \, \ud s' \\
  &\lesssim \|\psi_s\|_{\calS}.
 \end{align*}
 The second-term on the right-hand side can be handled analogously by applying Poincar\'e's inequality twice to get enough decay in $s'$ for heat times $s' \geq 1$. For the third term on the right-hand side we can proceed in the same manner as for the first term, only that we additionally invoke the bounds~\eqref{eq:bound_ringAL_Linftyall_Str}--\eqref{eq:bound_ringAQ_Linftyall_Str}.
\end{proof}

We collect a few more estimates on $\ringA$ in the following lemma.

\begin{lemma} \label{lem:bounds_ringA_Str}
 Assuming \eqref{eq:overall-bootstrap} and \eqref{eq:core-bootstrap}, we have that
 \begin{align}
  \| \langle \calL_\Omega \rangle \ringA_L \|_{L^\infty_{\ds} L^4_t L^4_x \cap L^{\frac{8}{3}}_t L^{\frac{8}{3}}_x} &\lesssim \|\psi_s\|_{\calS(I)} ,  \label{eq:bound_ringAL_Linftys_Str} \\
  \| \nabla \langle \calL_\Omega \rangle \ringA_L \|_{L^\infty_{\ds} L^4_t L^4_x \cap L^{\frac{8}{3}}_t L^{\frac{8}{3}}_x} &\lesssim \|\psi_s\|_{\calS(I)} ,  \label{eq:bound_nabla_ringAL_Linftys_Str} \\
  \| \langle \calL_\Omega \rangle \ringA_Q \|_{L^\infty_{\ds} L^4_t L^4_x \cap L^{\frac{8}{3}}_t L^{\frac{8}{3}}_x} &\lesssim \|\psi_s\|_{\calS(I)}^2 ,   \label{eq:bound_ringAQ_Linftys_Str} \\
  \| \nabla \langle \calL_\Omega \rangle \ringA_Q \|_{L^\infty_{\ds} L^2_t L^2_x} &\lesssim \|\psi_s\|_{\calS(I)}^2 .   \label{eq:bound_nabla_ringAQ_Linftys_L2tx} 
 \end{align}
\end{lemma}
\begin{proof}
 In what follows we only prove the asserted estimates \eqref{eq:bound_ringAL_Linftys_Str}--\eqref{eq:bound_nabla_ringAQ_Linftys_L2tx} without a Lie derivative. The more general versions follow in a similar manner. We begin with~\eqref{eq:bound_ringAL_Linftys_Str}. Using Poincar\'e's inequality and~\eqref{equ:bound_psis_Linftys_Str}, we obtain from the definition of $\ringA_L$ that
 \begin{align*}
  \| \ringA_L \|_{L^\infty_{\ds} L^4_t L^4_x \cap L^{\frac{8}{3}}_t L^{\frac{8}{3}}_x} &\lesssim \int_0^\infty \|\Psi^\infty\|_{L^\infty_x} \| \psi_s \|_{L^4_t L^4_x \cap L^{\frac{8}{3}}_t L^{\frac{8}{3}}_x} \, \ud s' \\
  &\lesssim \| s^{\frac{1}{2}} \psi_s \|_{L^\infty_{\ds} L^4_t L^4_x \cap L^{\frac{8}{3}}_t L^{\frac{8}{3}}_x} \int_0^1 (s')^{-\frac{1}{2}} \, \ud s' + \| \jap{s\Delta} s^{\frac{1}{2}} \psi_s \|_{L^\infty_{\ds} L^4_t L^4_x \cap L^{\frac{8}{3}}_t L^{\frac{8}{3}}_x} \int_1^\infty (s')^{-\frac{3}{2}} \, \ud s' \\
  &\lesssim \|\psi_s\|_{\calS}.
 \end{align*}
 For the proof of~\eqref{eq:bound_nabla_ringAL_Linftys_Str} we note that 
 \[
  \nabla \ringA_L(s) = - \int_s^\infty \Im \, ( \nabla \Psi^\infty \overline{\psi_s} ) \, \ud s' - \int_s^\infty \Im \, ( \Psi^\infty \overline{\nabla \psi_s}) \, \ud s'. 
 \]
 Hence, using Poincar\'e's inequality, extra regularity and the bound~\eqref{equ:bound_psis_Linftys_Str}, we find that 
 \begin{align*}
  \| \nabla \ringA_L \|_{L^\infty_{\ds} L^4_t L^4_x \cap L^{\frac{8}{3}}_t L^{\frac{8}{3}}_x} &\lesssim \int_0^\infty \| \nabla \Psi^\infty \|_{L^\infty_x} \| (s')^{\frac{1}{2}} \psi_s \|_{L^4_t L^4_x \cap L^{\frac{8}{3}}_t L^{\frac{8}{3}}_x} (s')^{-\frac{1}{2}} \, \ud s' \\
  &\quad \quad + \int_0^\infty \| \Psi^\infty \|_{L^\infty_x} \| (s')^{\frac{1}{2}} \nabla (s')^{\frac{1}{2}} \psi_s \|_{L^4_t L^4_x \cap L^{\frac{8}{3}}_t L^{\frac{8}{3}}_x} \, \frac{\ud s'}{s'} \\
  &\lesssim \| \nabla \Psi^\infty \|_{L^\infty_x} \| s^{\frac{1}{2}} \psi_s \|_{L^\infty_{\ds} L^4_t L^4_x \cap L^{\frac{8}{3}}_t L^{\frac{8}{3}}_x} \int_0^1 (s')^{-\frac{1}{2}} \, \ud s' \\
  &\quad \quad + \| \nabla \Psi^\infty \|_{L^\infty_x} \| \jap{s\Delta} s^{\frac{1}{2}} \psi_s \|_{L^\infty_{\ds} L^4_t L^4_x \cap L^{\frac{8}{3}}_t L^{\frac{8}{3}}_x} \int_1^\infty (s')^{-\frac{3}{2}} \, \ud s' \\
  &\quad \quad + \| \Psi^\infty \|_{L^\infty_x} \| \jap{s\Delta}^{\frac{1}{2}} m(s) s^{\frac{1}{2}} \psi_s \|_{L^\infty_{\ds} L^4_t L^4_x \cap L^{\frac{8}{3}}_t L^{\frac{8}{3}}_x} \int_0^1 (s')^{-1+\delta} \, \ud s' \\
  &\quad \quad + \| \Psi^\infty \|_{L^\infty_x} \| \jap{s\Delta} s^{\frac{1}{2}} \psi_s \|_{L^\infty_{\ds} L^4_t L^4_x \cap L^{\frac{8}{3}}_t L^{\frac{8}{3}}_x} \int_1^\infty (s')^{-\frac{3}{2}} \, \ud s' \\
  &\lesssim \|\psi_s\|_{\calS},
 \end{align*}
 which furnishes the base case for~\eqref{eq:bound_nabla_ringAL_Linftys_Str}. It remains to prove the estimates~\eqref{eq:bound_ringAQ_Linftys_Str}--\eqref{eq:bound_nabla_ringAQ_Linftys_L2tx}. Here we recall that
 \[
  \ringA_Q(s) = - \int_s^\infty \Im \, ( \ringPsi \overline{\psi_s} ) \, \ud s'.
 \]
 From~\eqref{equ:bound_psis_Linftys_Str}, \eqref{eq:bound_ringPsi_Linftyall_Str2}, and an application of Poincar\'e's inequality combined with extra regularity, we infer that
 \begin{align*}
  \| \ringA_Q \|_{L^\infty_{\ds} L^4_t L^4_x \cap L^{\frac{8}{3}}_t L^{\frac{8}{3}}_x} &\lesssim \int_0^\infty \| \ringPsi \|_{L^\infty_t L^\infty_x} \| \psi_s \|_{L^4_t L^4_x \cap L^{\frac{8}{3}}_t L^{\frac{8}{3}}_x} \, \ud s' \\
  &\lesssim \| s^{\frac{1}{2}} \ringPsi \|_{L^\infty_{\ds} L^\infty_t L^\infty_x} \int_0^\infty \| (s')^{\frac{1}{2}} \psi_s \|_{L^4_t L^4_x \cap L^{\frac{8}{3}}_t L^{\frac{8}{3}}_x} \, \frac{\ud s'}{s'} \\ 
  &\lesssim \| s^{\frac{1}{2}} \ringPsi \|_{L^\infty_{\ds} L^\infty_t L^\infty_x} \| m(s) s^{\frac{1}{2}} \psi_s \|_{L^\infty_{\ds} L^4_t L^4_x \cap L^{\frac{8}{3}}_t L^{\frac{8}{3}}_x} \int_0^1 (s')^{-1+\delta} \, \ud s' \\ 
  &\quad \quad + \| s^{\frac{1}{2}} \ringPsi \|_{L^\infty_{\ds} L^\infty_t L^\infty_x} \| \jap{s\Delta}^{\frac{1}{2}} s^{\frac{1}{2}} \psi_s \|_{L^\infty_{\ds} L^4_t L^4_x \cap L^{\frac{8}{3}}_t L^{\frac{8}{3}}_x} \int_0^1 (s')^{-\frac{3}{2}} \, \ud s' \\ 
  &\lesssim \|\psi_s\|_{\calS}^2.
 \end{align*}
 Finally, for the proof of~\eqref{eq:bound_nabla_ringAQ_Linftys_L2tx} we invoke the estimates~\eqref{equ:bound_psis_Linftys_Str}, \eqref{equ:bound_ringPsi_Linftys_Str_2} and argue as usual via Poincar\'e's inequality together with extra regularity to conclude that
 \begin{align*}
  \| \nabla \ringA_Q \|_{L^\infty_{\ds} L^2_t L^2_x} &\lesssim \int_0^\infty \bigl( \| \nabla \ringPsi \|_{L^4_t L^4_x} \| \psi_s \|_{L^4_t L^4_x} + \| \ringPsi \|_{L^4_t L^4_x} \| \nabla \psi_s \|_{L^4_t L^4_x} \bigr) \, \ud s' \\
  &\lesssim \| \jap{s\Delta}^{\frac{1}{2}} \ringPsi \|_{L^\infty_{\ds} L^4_t L^4_x} \int_0^\infty \| \jap{s\Delta}^{\frac{1}{2}} s^{\frac{1}{2}} \psi_s \|_{L^4_t L^4_x} \, \frac{\ud s'}{s'} \\
  &\lesssim \| \jap{s\Delta}^{\frac{1}{2}} \ringPsi \|_{L^\infty_{\ds} L^4_t L^4_x} \| \jap{s\Delta} m(s) s^{\frac{1}{2}} \psi_s \|_{L^\infty_{\ds} L^4_t L^4_x} \\
  &\lesssim \|\psi_s\|_{\calS}^2. \qedhere
 \end{align*}
\end{proof}

\subsection{Estimates on $w$, $A_t$, and $\psi_t$}\label{sub:6w} 

In this subsection we start by proving two estimates on $w$, one in the dual Strichartz norm $L_t^{\frac{4}{3}}L_x^{\frac{4}{3}}$ and one in the Strichartz norm  $L_t^4L_x^4$. The latter is then used to prove additional estimates on $A_t$ and $\psi_t$.

\begin{lemma} \label{lem:bounds_on_w}
 Let $w = w(t, x, s)$ denote the Schr\"odinger tension field as defined in \eqref{eq:w-def}. Then, assuming \eqref{eq:overall-bootstrap} and \eqref{eq:core-bootstrap}, we have that 
 \begin{equation} \label{equ:nonlinear_estimates_partials_w}
  \| m(s) s^{\frac{1}{2}}  \langle \Omega \rangle \partial_s w \|_{L^\infty_{\ds} \cap L^2_{\ds} L^{\frac{4}{3}}_t L^{\frac{4}{3}}_x(I)} \lesssim \|\psi_s\|_{\calS(I)}^2 + \sqrt{\epsilon_0} \|\psi_s\|_{\calS(I)} ,
 \end{equation}
 and
 \begin{equation} \label{equ:nonlinear_estimates_w_for_psit}
  \| s^{\frac{1}{2}} \jap{\Omega} w \|_{L^\infty_{\ds} L^4_t L^4_x(I)} \lesssim \|\psi_s\|_{\calS(I)}^2 + \sqrt{\epsilon_0} \|\psi_s\|_{\calS(I)}.
 \end{equation}
\end{lemma}
\begin{proof}
We begin by recalling that the Schr\"odinger tension field $w$ obeys the parabolic equation~\eqref{eq:heatH-w}, which we record again here for convenience. 
\begin{equation} \label{equ:heat_w}
\begin{aligned}
  (\partial_s + H) w &= 
   i\Im( \ringpsi^k \ba{w}) \ringpsi_k + i \Im( \ringpsi_k \ba{w})\psi^{\infty, k} + i\Im( \psi^{\infty, k} \ba{w}) \ringpsi_k \\
 & \quad  + 2 i \ringA^k \nabla_k w - 2 A^{\infty, k} \ringA_k w + i (\nabla_k \ringA^k) w - \ringA^k \ringA_k w \\
  &\quad - i \ringpsi^{k} \ringpsi_{k} \overline{\psi_{s}} - 2 i \psi^{\infty, k} \ringpsi_{k} \overline{\psi_{s}}.
\end{aligned}
\end{equation}
Applying $\Omega$ to~\eqref{equ:heat_w} yields the following heat equation for $\Omega w$
\begin{equation} \label{equ:heat_Omega_w}
\begin{aligned}
(\partial_s+H)\Omega w &=
i\Im (\calL_\Omega\ringpsi^k \overline{w}) \ringpsi_k+ i\Im (\ringpsi^k \overline{\Omega w}) \ringpsi_k+ i\Im (\ringpsi^k \overline{w}) \calL_\Omega\ringpsi_k \\
  &\quad  + i\Im (\calL_\Omega\ringpsi^k \overline{w}) \psi^\infty_k +i \Im (\ringpsi^k \overline{\Omega w}) \psi^\infty_k \\
  &\quad + i\Im (\psi^{\infty, k} \overline{\Omega w}) \ringpsi_k+ i\Im (\psi^{\infty, k} \overline{w}) \calL_\Omega\ringpsi_k\\
  &\quad + i\Im (i\psi^{\infty, k} \overline{w}) \ringpsi_k+ i\Im (\ringpsi^k \overline{w}) \psi^\infty_k \\
&\quad+ 2 i \calL_\Omega\ringA^k \nabla_k w+2 i \ringA^k \nabla_k \Omega w - 2 A^{\infty, k} \calL_\Omega\ringA_k w- 2 A^{\infty, k} \ringA_k \Omega w\\
  &\quad + i (\nabla_k \calL_\Omega\ringA^k) w + i (\nabla_k \ringA^k) \Omega w - 2\calL_\Omega\ringA^k \ringA_k w - \ringA^k \ringA_k \Omega w \\
   &\quad - 2 i \calL_\Omega \ringpsi^{k} \ringpsi_{k} \ba{\psi_{s}}
   - i \ringpsi^{k} \ringpsi_{k} \ba{\Omega \psi_{s}} \\ 
 &\quad  + 2 \psi^{\infty, k} \ringpsi_{k} \ba{\psi_{s}}
   - 2 i \psi^{\infty, k} \calL_\Omega \ringpsi_{k} \ba{\psi_{s}}
      - 2 i \psi^{\infty, k} \ringpsi_{k} \ba{\Omega \psi_{s}}.
\end{aligned}
\end{equation}
In what follows the right-hand sides of~\eqref{equ:heat_w} and~\eqref{equ:heat_Omega_w} are denoted by $G$, respectively by $\Omega G$.

\medskip 

We now start with the proof of~\eqref{equ:nonlinear_estimates_partials_w_reduction}. Since $\partial_s w = - H w + G$, it holds that
\begin{align*}
 \| m(s) s^{\frac{1}{2}} \jap{\Omega} \partial_s w \|_{L^\infty_{\ds} \cap L^2_{\ds} L^{\frac{4}{3}}_t L^{\frac{4}{3}}_x} \lesssim \| m(s) s^{\frac{1}{2}} \jap{\Omega} H w \|_{L^\infty_{\ds} \cap L^2_{\ds} L^{\frac{4}{3}}_t L^{\frac{4}{3}}_x} + \| m(s) s^{\frac{1}{2}} \jap{\Omega} G \|_{L^\infty_{\ds} \cap L^2_{\ds} L^{\frac{4}{3}}_t L^{\frac{4}{3}}_x}.
\end{align*}
We will prove that
\begin{equation} \label{equ:bound_H_w}
 \| m(s) s^{\frac{1}{2}} \jap{\Omega} H w \|_{L^\infty_{\ds} \cap L^2_{\ds} L^{\frac{4}{3}}_t L^{\frac{4}{3}}_x} \lesssim \| m(s) s^{\frac{1}{2}} \langle \Omega \rangle G \|_{L^\infty_{\ds} \cap L^2_{\ds} L^{\frac{4}{3}}_t L^{\frac{4}{3}}_x} + \| s^{\frac{1}{2}} \nabla m(s) s^{\frac{1}{2}} \langle \Omega \rangle G \|_{L^\infty_{\ds} \cap L^2_{\ds} L^{\frac{4}{3}}_t L^{\frac{4}{3}}_x}.
\end{equation}
Then the estimate \eqref{equ:nonlinear_estimates_partials_w} would follow upon establishing that
\begin{equation} \label{equ:nonlinear_estimates_partials_w_todo1}
 \| m(s) s^{\frac{1}{2}} \langle \Omega \rangle G \|_{L^\infty_{\ds} \cap L^2_{\ds} L^{\frac{4}{3}}_t L^{\frac{4}{3}}_x} + \| s^{\frac{1}{2}} \nabla m(s) s^{\frac{1}{2}} \langle \Omega \rangle G \|_{L^\infty_{\ds} \cap L^2_{\ds} L^{\frac{4}{3}}_t L^{\frac{4}{3}}_x} \lesssim \|\psi_s\|_{\calS(I)}^2 + \sqrt{\epsilon_0} \|\psi_s\|_{\calS(I)}.
\end{equation}
However, in order to have more flexibility in estimating the nonlinear terms in $\jap{\Omega} G$, we will show, concurrently with \eqref{equ:bound_H_w}, that
\begin{equation} \label{equ:bound_Delta_w}
 \begin{aligned}
  &\sum_{j=0}^2 \| m(s) s^{-\frac{1}{2}}  (-s\Delta)^{\frac{j}{2}} \langle \Omega \rangle w \|_{L^\infty_{\ds} \cap L^2_{\ds} L^{\frac{4}{3}}_t L^{\frac{4}{3}}_x} \\
  &\quad \quad \lesssim \| m(s) s^{\frac{1}{2}} \langle \Omega \rangle G \|_{L^\infty_{\ds} \cap L^2_{\ds} L^{\frac{4}{3}}_t L^{\frac{4}{3}}_x} + \| s^{\frac{1}{2}} \nabla m(s) s^{\frac{1}{2}} \langle \Omega \rangle G \|_{L^\infty_{\ds} \cap L^2_{\ds} L^{\frac{4}{3}}_t L^{\frac{4}{3}}_x}.
 \end{aligned}
\end{equation}
Then, provided $\| \psi_s \|_{\calS}$ is sufficiently small, in order to conclude the estimate~\eqref{equ:nonlinear_estimates_partials_w}, instead of verifying~\eqref{equ:nonlinear_estimates_partials_w_todo1}, it suffices to prove that
\begin{equation} \label{equ:nonlinear_estimates_partials_w_reduction}
 \begin{aligned}
  &\| m(s) s^{\frac{1}{2}} \langle \Omega \rangle G \|_{L^\infty_{\ds} \cap L^2_{\ds} L^{\frac{4}{3}}_t L^{\frac{4}{3}}_x} + \| s^{\frac{1}{2}} \nabla m(s) s^{\frac{1}{2}} \langle \Omega \rangle G \|_{L^\infty_{\ds} \cap L^2_{\ds} L^{\frac{4}{3}}_t L^{\frac{4}{3}}_x} \\
  &\quad \lesssim (\|\psi_s\|_{\calS} + \sqrt{\epsilon_0}) \Bigl( \|\psi_s\|_{\calS} + \sum_{j=0}^2 \| m(s) s^{-\frac{1}{2}}  (-s\Delta)^{\frac{j}{2}} \langle \Omega \rangle w \|_{L^\infty_{\ds} \cap L^2_{\ds} L^{\frac{4}{3}}_t L^{\frac{4}{3}}_x} \Bigr). 
 \end{aligned}
\end{equation}

Thus, before establishing the nonlinear estimates~\eqref{equ:nonlinear_estimates_partials_w_reduction}, we first have to prove~\eqref{equ:bound_H_w} and~\eqref{equ:bound_Delta_w}. To this end we note that since $w(0) = 0$ by definition of the Schr\"odinger tension field, the heat equation~\eqref{equ:heat_w} for~$w$ can be written in Duhamel form as
\begin{equation*}
 w(s) = \int_0^s e^{-(s-s') H} G(s') \, \ud s'.
\end{equation*}
We begin with the term with $j=0$ in \eqref{equ:bound_Delta_w}. We have by Schur's test and Corollary~\ref{c:lin-para-Lp} that
\begin{equation} \label{equ:bound_w}
 \begin{aligned}
  \| m(s) s^{-\frac{1}{2}} \langle \Omega \rangle w \|_{L^\infty_{\ds} \cap L^2_{\ds} L^{\frac{4}{3}}_t L^{\frac{4}{3}}_x} &\lesssim \biggl\| \int_0^s m(s) s^{-\frac{1}{2}} \| e^{-(s-s')H} \jap{\Omega} G(s') \|_{L^{\frac{4}{3}}_t L^{\frac{4}{3}}_x} \, \ud s' \biggr\|_{L^\infty_{\ds} \cap L^2_{\ds}} \\
  &\lesssim \biggl\| \int_0^s \frac{m(s)}{m(s')} \Bigl( \frac{s'}{s} \Bigr)^{\frac{1}{2}} \| m(s') (s')^{\frac{1}{2}} \jap{\Omega} G(s') \|_{L^{\frac{4}{3}}_t L^{\frac{4}{3}}_x} \frac{\ud s'}{s'} \biggr\|_{L^\infty_{\ds} \cap L^2_{\ds}} \\
  &\lesssim \| m(s) s^{\frac{1}{2}} \langle \Omega \rangle G \|_{L^\infty_{\ds} \cap L^2_{\ds} L^{\frac{4}{3}}_t L^{\frac{4}{3}}_x},
 \end{aligned}
\end{equation}
where we used that $\frac{m(s)}{m(s')} ( \frac{s'}{s} )^{\frac{1}{2}} \chi_{\{0 \leq s' \leq s\}}$ is a Schur's kernel. Next, we turn to \eqref{equ:bound_H_w}. We find that
\begin{equation}
 \begin{aligned}
  \| m(s) s^{\frac{1}{2}} H \jap{\Omega} w \|_{L^\infty_{\ds} \cap L^2_{\ds} L^{\frac{4}{3}}_t L^{\frac{4}{3}}_x} &\lesssim \biggl\| \int_0^{\frac{s}{2}} m(s) s^{\frac{1}{2}} \| H e^{-(s-s')H} \jap{\Omega} G(s') \|_{L^{\frac{4}{3}}_t L^{\frac{4}{3}}_x} \, \ud s' \biggr\|_{L^\infty_{\ds} \cap L^2_{\ds}} \\
  &\quad \quad + \biggl\| \int_{\frac{s}{2}}^s m(s) s^{\frac{1}{2}} \| e^{-(s-s')H} H \jap{\Omega} G(s') \|_{L^{\frac{4}{3}}_t L^{\frac{4}{3}}_x} \, \ud s' \biggr\|_{L^\infty_{\ds} \cap L^2_{\ds}},
 \end{aligned}
\end{equation}
where we used the fact that $H$ and $e^{-(s-s') H}$ commute for the last term. To proceed, note that $H = -\Delta + H_{\lot}$ with 
\[
 H_{\lot} = - 2 i A^{\infty, k} \nabla_k - i (\nabla_k A^{\infty, k}) + A_k^\infty A^{\infty, k} + |\psi_2^\infty|^2.
\]
The contribution of $-\Delta$ to the first term on the right-hand side can be bounded by
\begin{align*}
 &\biggl\| \int_0^{\frac{s}{2}} \frac{m(s)}{m(s')} \frac{ s^{\frac{1}{2}} (s')^{\frac{1}{2}} }{s-s'} \| (s-s') \Delta e^{-(s-s')H} m(s') (s')^{\frac{1}{2}} \jap{\Omega} G(s') \|_{L^{\frac{4}{3}}_t L^{\frac{4}{3}}_x} \frac{\ud s'}{s'} \biggr\|_{L^\infty_{\ds} \cap L^2_{\ds}} \\
 &\lesssim \biggl\| \int_0^{\frac{s}{2}} \frac{m(s)}{m(s')} \Bigl( \frac{s'}{s} \Bigr)^{\frac{1}{2}} \| m(s') (s')^{\frac{1}{2}} \jap{\Omega} G(s') \|_{L^{\frac{4}{3}}_t L^{\frac{4}{3}}_x} \frac{\ud s'}{s'} \biggr\|_{L^\infty_{\ds} \cap L^2_{\ds}} \\
 &\lesssim \| m(s) s^{\frac{1}{2}} \jap{\Omega} G \|_{L^\infty_{\ds} \cap L^2_{\ds} L^{\frac{4}{3}}_t L^{\frac{4}{3}}_x},
\end{align*}
where we used~Corollary~\ref{c:lin-para-Lp} and the fact that $\frac{m(s)}{m(s')} ( \frac{s'}{s} )^{\frac{1}{2}} \chi_{\{0 \leq s' \leq \frac{s}{2}\}}$ is a Schur's kernel. The contribution of $H_{\lot}$ can be treated in a similar manner, by boundedness of the coefficients of $H_{\lot}$ and invoking Poincar\'e's inequality on occasion to manufacture a good Schur's kernel.
To treat the second term on the right-hand side, we again split $H = - \Delta + H_{\lot}$. The contribution of $-\Delta$ can be bounded by
\begin{align*}
 &\biggl\| \int_{\frac{s}{2}}^s \frac{m(s)}{m(s')} \frac{s^{\frac{1}{2}}}{(s-s')^{\frac{1}{2}}} \| e^{-(s-s')H} (s-s')^{\frac{1}{2}} \nabla^k (s')^{\frac{1}{2}} \nabla_k m(s') (s')^{\frac{1}{2}} \jap{\Omega} G(s') \|_{L^{\frac{4}{3}}_t L^{\frac{4}{3}}_x} \frac{\ud s'}{s'} \biggr\|_{L^\infty_{\ds} \cap L^2_{\ds}} \\
 &\lesssim \biggl\| \int_{\frac{s}{2}}^s \frac{m(s)}{m(s')} \frac{s^{\frac{1}{2}}}{(s-s')^{\frac{1}{2}}} \| (s')^{\frac{1}{2}} \nabla m(s') (s')^{\frac{1}{2}} \jap{\Omega} G(s') \|_{L^{\frac{4}{3}}_t L^{\frac{4}{3}}_x} \frac{\ud s'}{s'} \biggr\|_{L^\infty_{\ds} \cap L^2_{\ds}} \\
 &\lesssim \| s^{\frac{1}{2}} \nabla m(s) s^{\frac{1}{2}} \jap{\Omega} G \|_{L^\infty_{\ds} \cap L^2_{\ds} L^{\frac{4}{3}}_t L^{\frac{4}{3}}_x}.
\end{align*}
where in the first line above we have applied the following straightforward modification of the scalar estimates in Corollary~\ref{c:lin-para-Lp},  
\EQ{
\| e^{-sH} s^{\frac{1}{2}} \nabla^k \bfxi_k \|_{L^p_x} \lesssim \|\bfxi\|_{L^p_x}
}
applied now to a tensor $\bfxi$. The contribution of $H_{\lot}$ is handled similarly, using Poincar\'e's inequality in addition.
Combining the two previous bounds, \eqref{equ:bound_H_w} follows. Moreover, arguing similarly as in the preceding proof, one can show that
\begin{equation} \label{equ:bound_Hlot}
 \begin{aligned}
  &\| m(s) s^{\frac{1}{2}} H_{\lot} \jap{\Omega} w \|_{L^\infty_{\ds} \cap L^2_{\ds} L^{\frac{4}{3}}_t L^{\frac{4}{3}}_x} \\
  &\quad \quad \lesssim \| m(s) s^{\frac{1}{2}} \langle \Omega \rangle G \|_{L^\infty_{\ds} \cap L^2_{\ds} L^{\frac{4}{3}}_t L^{\frac{4}{3}}_x} + \| s^{\frac{1}{2}} \nabla m(s) s^{\frac{1}{2}} \langle \Omega \rangle G \|_{L^\infty_{\ds} \cap L^2_{\ds} L^{\frac{4}{3}}_t L^{\frac{4}{3}}_x}.
 \end{aligned}
\end{equation}
From \eqref{equ:bound_H_w} and \eqref{equ:bound_Hlot}, it follows that
\begin{equation} \label{equ:bound_Delta_w_second}
 \begin{aligned}
  &\| m(s) s^{\frac{1}{2}} \Delta \jap{\Omega} w \|_{L^\infty_{\ds} \cap L^2_{\ds} L^{\frac{4}{3}}_t L^{\frac{4}{3}}_x} \\
  &\quad \quad \lesssim \| m(s) s^{\frac{1}{2}} \langle \Omega \rangle G \|_{L^\infty_{\ds} \cap L^2_{\ds} L^{\frac{4}{3}}_t L^{\frac{4}{3}}_x} + \| s^{\frac{1}{2}} \nabla m(s) s^{\frac{1}{2}} \langle \Omega \rangle G \|_{L^\infty_{\ds} \cap L^2_{\ds} L^{\frac{4}{3}}_t L^{\frac{4}{3}}_x}. 
 \end{aligned}
\end{equation}
Interpolating between~\eqref{equ:bound_w} and \eqref{equ:bound_Delta_w_second}, we obtain the entire estimate~\eqref{equ:bound_Delta_w}. 

It remains to prove~\eqref{equ:nonlinear_estimates_partials_w_reduction}. To this end we group the nonlinear terms in $\jap{\Omega} G$ (that is, the nonlinear terms on the right-hand sides of the equation~\eqref{equ:heat_w} for $w$ and of~\eqref{equ:heat_Omega_w} for $\Omega w$) into the following types: 
\begin{enumerate}[(I)]
\item Involving $w$:
\begin{enumerate}[(i)]
\item $\calL_\Omega\ringA^k\nabla_kw$
\item $\ringA^k\nabla_k\jap{\Omega}w$
\item $(\nabla_k\calL_\Omega\ringA^k)w$
\item $(\nabla_k\ringA^k)\jap{\Omega}w$
\item $A^\infty_k\calL_\Omega\ringA^k w$
\item $A^\infty_k\ringA^k\jap{\Omega}w$
\item $\Im(\psi^{\infty, k} \ba{w}) \calL_\Omega \ringpsi_{k}$ and permutations
\item $\Im(\psi^{\infty, k} \ba{\jap{\Omega} w}) \ringpsi_{k}$ and permutations
\item $\calL_\Omega\ringA_k\ringA^kw$
\item $\ringA_k\ringA^k\jap{\Omega}w$
\item $\Im(\calL_\Omega \ringpsi^{k} \ba{w}) \ringpsi_{k}$ and permutations
\item $\Im(\ringpsi^{k} \ba{\jap{\Omega} w}) \ringpsi_{k}$
\end{enumerate}
\item Cubic and higher:
\begin{enumerate}[(i)]
\item $\calL_\Omega \ringpsi^{k} \ringpsi_{k} \ba{\psi_{s}}$
\item $\ringpsi^{k} \ringpsi_{k} \ba{\jap{\Omega} \psi_{s}}$
\end{enumerate}
\item Quadratic:
\begin{enumerate}[(i)]
\item $\psi^{\infty, k} \calL_\Omega \ringpsi_{k} \ba{\psi_{s}}$
\item $\psi^{\infty, k} \ringpsi_{k} \ba{\jap{\Omega} \psi_{s}}$
\end{enumerate}
\end{enumerate}
We now verify~\eqref{equ:nonlinear_estimates_partials_w_reduction} for each term in the preceding list. For the term (I-i) we use the bounds~\eqref{eq:bound_ringAL_Linftyall_Str}--\eqref{eq:bound_ringAQ_Linftyall_Str} to conclude that
\begin{align*}
 \| m(s) s^{\frac{1}{2}} \calL_\Omega \ringA^k \nabla_k w \|_{L^\infty_{\ds} \cap L^2_{\ds} L^{\frac{4}{3}}_t L^{\frac{4}{3}}_x} &\lesssim \| s^{\frac{1}{2}} \calL_\Omega \ringA \|_{L^\infty_{\ds} L^\infty_t L^\infty_x} \| m(s) \nabla w \|_{L^\infty_{\ds} \cap L^2_{\ds} L^{\frac{4}{3}}_t L^{\frac{4}{3}}_x} \\
 &\lesssim \|\psi_s\|_{\calS} \| m(s) \nabla w \|_{L^\infty_{\ds} \cap L^2_{\ds} L^{\frac{4}{3}}_t L^{\frac{4}{3}}_x}
\end{align*}
and similarly that
\begin{align*}
 \| s^{\frac{1}{2}} \nabla ( m(s) s^{\frac{1}{2}} \calL_\Omega \ringA^k \nabla_k w ) \|_{L^\infty_{\ds} \cap L^2_{\ds} L^{\frac{4}{3}}_t L^{\frac{4}{3}}_x} &\lesssim \| s^{\frac{1}{2}} \nabla s^{\frac{1}{2}} \calL_\Omega \ringA \|_{L^\infty_{\ds} L^\infty_t L^\infty_x} \| m(s) \nabla w \|_{L^\infty_{\ds} \cap L^2_{\ds} L^{\frac{4}{3}}_t L^{\frac{4}{3}}_x} \\
 &\quad \quad + \| \calL_\Omega \ringA \|_{L^\infty_{\ds} L^\infty_t L^\infty_x} \| m(s) s^{\frac{1}{2}} \Delta w \|_{L^\infty_{\ds} \cap L^2_{\ds} L^{\frac{4}{3}}_t L^{\frac{4}{3}}_x} \\
 &\lesssim \|\psi_s\|_{\calS} \sum_{j=0}^2 \| m(s) s^{-\frac{1}{2}} (-s \Delta)^{\frac{j}{2}} w \|_{L^\infty_{\ds} \cap L^2_{\ds} L^{\frac{4}{3}}_t L^{\frac{4}{3}}_x}.
\end{align*}
The terms (I-ii)--(I-iv) can be treated similarly to the term (I-i). For the term (I-v) we invoke the bounds~\eqref{eq:bound_ringAL_Linftyall_Str}--\eqref{eq:bound_ringAQ_Linftyall_Str} and also apply Poincar\'e's inequality to obtain the desired estimate
\begin{align*}
 \| m(s) s^{\frac{1}{2}} A^\infty_k \calL_\Omega \ringA^k w \|_{L^\infty_{\ds} \cap L^2_{\ds} L^{\frac{4}{3}}_t L^{\frac{4}{3}}_x} &\lesssim \| A^\infty \|_{L^\infty_x} \| s^{\frac{1}{2}} \calL_\Omega \ringA \|_{L^\infty_{\ds} L^\infty_t L^\infty_x} \| m(s) w \|_{L^\infty_{\ds} \cap L^2_{\ds} L^{\frac{4}{3}}_t L^{\frac{4}{3}}_x} \\
 &\lesssim \| \psi_s \|_{\calS} \| m(s) \nabla w \|_{L^\infty_{\ds} \cap L^2_{\ds} L^{\frac{4}{3}}_t L^{\frac{4}{3}}_x}.
\end{align*}
Then one can bound $\| s^{\frac{1}{2}} \nb m(s) s^{\frac{1}{2}} A^\infty_k \calL_\Omega \ringA^k w \|_{L^\infty_{\ds} \cap L^2_{\ds} L^{\frac{4}{3}}_t L^{\frac{4}{3}}_x}$ similarly. The terms (I-vi)--(I-viii) can be dealt with in a similar manner as the term (I-v). Finally, for the term (I-ix) we have by the estimates~\eqref{eq:bound_ringAL_Linftyall_Str}--\eqref{eq:bound_ringAQ_Linftyall_Str} that
\begin{align*}
 \| m(s) s^{\frac{1}{2}} \calL_\Omega \ringA_k \ringA^k w \|_{L^\infty_{\ds} \cap L^2_{\ds} L^{\frac{4}{3}}_t L^{\frac{4}{3}}_x} &\lesssim \| s^{\frac{1}{2}} \calL_\Omega \ringA \|_{L^\infty_{\ds} L^\infty_t L^\infty_x} \| s^{\frac{1}{2}} \ringA \|_{L^\infty_{\ds} L^\infty_t L^\infty_x} \| m(s) s^{-\frac{1}{2}} w \|_{L^\infty_{\ds} \cap L^2_{\ds} L^{\frac{4}{3}}_t L^{\frac{4}{3}}_x} \\
 &\lesssim \|\psi_s\|_{\calS}^2 \| m(s) s^{-\frac{1}{2}} w \|_{L^\infty_{\ds} \cap L^2_{\ds} L^{\frac{4}{3}}_t L^{\frac{4}{3}}_x}
\end{align*}
and the estimate for $\| s^{\frac{1}{2}} \nabla m(s) s^{\frac{1}{2}} \calL_\Omega \ringA_k \ringA^k w \|_{L^\infty_{\ds} \cap L^2_{\ds} L^{\frac{4}{3}}_t L^{\frac{4}{3}}_x}$ proceeds similarly. The terms (I-x)--(I-xii) are estimated analogously to the term (I-ix).

Next, we turn to the treatment of the terms of type (II). By the bounds \eqref{eq:bound_ringPsi_Linftyall_Str2} and \eqref{equ:bound_ringPsi_Linftys_Str_2}, we have for the term (II-i) that
\begin{align*}
 \| m(s) s^{\frac{1}{2}} |\calL_\Omega \ringPsi| | \ringPsi| |\psi_{s}| \|_{L^\infty_{\ds} \cap L^2_{\ds} L^{\frac{4}{3}}_t L^{\frac{4}{3}}_x} 
 \lesssim \| \calL_\Omega \ringPsi \|_{L^\infty_{\ds} L^4_t L^4_x} \| \ringPsi \|_{L^\infty_{\ds} L^4_t L^4_x} \| m(s) s^{\frac{1}{2}} \psi_s \|_{L^\infty_{\ds} \cap L^2_{\ds} L^4_t L^4_x} 
 \lesssim \|\psi_s\|_{\calS}^3,
\end{align*}
which is of the desired form, and the estimate for $\| s^{\frac{1}{2}} \nabla m(s) s^{\frac{1}{2}} \text{(II-i)}\|_{L^\infty_{\ds} \cap L^2_{\ds} L^{\frac{4}{3}}_t L^{\frac{4}{3}}_x}$ is analogous. The term (II-ii) can be treated similarly to the term (II-i).

Finally, we estimate the quadratic terms of type (III). For the term~(III-i) we use the bounds \eqref{equ:bound_psis_Linftys_Str} and \eqref{equ:bound_ringPsi_Linftys_Str_2} to obtain that
\begin{align*}
 \| m(s) s^{\frac{1}{2}} |\Psi^\infty| |\calL_\Omega \ringPsi| |\psi_{s}| \|_{L^\infty_{\ds} \cap L^2_{\ds} L^{\frac{4}{3}}_t L^{\frac{4}{3}}_x} 
& \lesssim \|\Psi^\infty\|_{L^\infty_x} \| \calL_\Omega \ringPsi \|_{L^\infty_{\ds} L^{\frac{8}{3}}_t L^{\frac{8}{3}}_x} \| m(s) s^{\frac{1}{2}} \psi_s \|_{L^\infty_{\ds} \cap L^2_{\ds} L^{\frac{8}{3}}_t L^{\frac{8}{3}}_x} \\
 &\lesssim \|\psi_s\|_{\calS}^2,
\end{align*}
as desired, and the estimate for $\| s^{\frac{1}{2}} \nabla m(s) s^{\frac{1}{2}} \text{(III-i)}\|_{L^\infty_{\ds} \cap L^2_{\ds} L^{\frac{4}{3}}_t L^{\frac{4}{3}}_x}$ proceeds in the same manner. Then the term (III-ii) can be handled similarly. This finishes the proof of the estimate~\eqref{equ:nonlinear_estimates_partials_w}.

\medskip 

For the proof of~\eqref{equ:nonlinear_estimates_w_for_psit} we note that by arguing exactly as in the proof of~\eqref{equ:nonlinear_estimates_partials_w}, provided $\|\psi_s\|_{\calS}$ is sufficiently small, it suffices to show that
\begin{equation} \label{equ:nonlinear_estimates_w_reduction}
 \begin{aligned}
  &\| s^{\frac{1}{2}} \langle \Omega \rangle G \|_{L^\infty_{\ds} \cap L^2_{\ds} L^4_t L^4_x} + \| s^{\frac{1}{2}} \nabla s^{\frac{1}{2}} \langle \Omega \rangle G \|_{L^\infty_{\ds} \cap L^2_{\ds} L^4_t L^4_x} \\
  &\quad \lesssim (\|\psi_s\|_{\calS} + \sqrt{\epsilon_0}) \Bigl( \|\psi_s\|_{\calS} + \sum_{j=0}^2 \| s^{-\frac{1}{2}}  (-s\Delta)^{\frac{j}{2}} \langle \Omega \rangle w \|_{L^\infty_{\ds} \cap L^2_{\ds} L^4_t L^4_x} \Bigr). 
 \end{aligned}
\end{equation}
Then it remains to verify~\eqref{equ:nonlinear_estimates_w_reduction} for all nonlinear terms of type (I)--(III). We do this term by term, without $\jap{s\Delta}^{\frac{1}{2}}$ on $G$, the general case being similar. 

We begin with the quadratic terms (III). Here (III-ii) is bounded by 
\begin{align*}
 \|\Psi^\infty\|_{L_x^\infty} \|s\ringPsi\|_{\Ls^\infty L_t^\infty L_x^\infty} \|\jap{\Omega}ms^{\frac{1}{2}}\psi_s\|_{\LLs L_t^4L_x^4}\lesssim \sqrt{\epsilon_0} \|\psi_s\|_\calS,
\end{align*}
where we use \eqref{equ:bound_ringPsi_Linftys_Str}, Gagliardo-Nirenberg, and Poincar\'e to bound $\|s\ringPsi\|_{\Ls L_t^\infty L_x^\infty}$. (III-i) is treated similarly. 

Next we turn to the cubic and higher terms (II). These are treated in the same manner as the quadratic terms (III) except that we do not need to use Poincar\'e. For instance, for (II-ii) we get
\begin{align*}
 \|s^{\frac{1}{2}}\ringPsi\|_{\Ls^\infty L_t^\infty L_x^\infty}^2 \|\jap{\Omega}ms^{\frac{1}{2}}\psi_s\|_{\LLs L_t^4L_x^4}\lesssim \epsilon_0 \|\psi_s\|_\calS,
\end{align*}
by \eqref{equ:bound_ringPsi_Linftys_Str} and Gagliardo-Nirenberg.

Finally, the treatment of the terms (I-i)--(I-xii) involving $w$ is identical to their treatment in the proof of~\eqref{equ:nonlinear_estimates_partials_w_reduction} only that $L_t^{\frac{4}{3}}L_x^{\frac{4}{3}}$ is now replaced by $L_t^4L_x^4$. Indeed, here we always place the factor involving $w$ in $\LLs L_t^4L_x^4$ and the terms involving $A$ or $\Psi$ in $\Ls^\infty L_t^\infty L_x^\infty$. The resulting contribution is then bounded by
%%%%%%%
%%%%%%%
\begin{align*}
\begin{split}
(\sqrt{\epsilon_0}+\|\psi_s\|_\calS)\|\jap{s\Delta}\jap{\Omega}ms^{\frac{1}{2}}w\|_{\LLs L_t^4 L_x^4}
\end{split}
\end{align*}
as desired.
\end{proof}

As a corollary of \eqref{equ:nonlinear_estimates_w_for_psit} we prove the following bounds on $A_t$ and $\psi_t$.

\begin{lemma} \label{lem:bounds_Psi_A_t}
 Assuming \eqref{eq:overall-bootstrap} and \eqref{eq:core-bootstrap}, we have that 
 \begin{align} 
  \| \langle \Omega \rangle s^{\frac{1}{2}} \psi_t \|_{L^\infty_{\ds} L^4_t L^4_x} &\lesssim \|\psi_s\|_{\calS} ,  \label{equ:bound_psit_Linftys_L4tx} \\
  \| \langle \calL_\Omega \rangle A_t \|_{L^\infty_{\ds} L^2_t L^2_x} &\lesssim \|\psi_s\|_{\calS}^2 .  \label{eq:bound_At_Linftys_L2tx}
 \end{align}
\end{lemma}  
\begin{proof}
 The first bound~\eqref{equ:bound_psit_Linftys_L4tx} follows easily by writing $\psi_t = w + i \psi_s$ and using the estimate~\eqref{equ:nonlinear_estimates_w_for_psit}, 
 \begin{align*}
  \| \langle \Omega \rangle s^{\frac{1}{2}} \psi_t \|_{L^\infty_{\ds} L^4_t L^4_x} &\lesssim \| \langle \Omega \rangle s^{\frac{1}{2}} w \|_{L^\infty_{\ds} L^4_t L^4_x} + \| \langle \Omega \rangle s^{\frac{1}{2}} \psi_s \|_{L^\infty_{\ds} L^4_t L^4_x} \lesssim \|\psi_s\|_{\calS}.
 \end{align*}
 Then invoking the estimates \eqref{equ:bound_psis_Linftys_Str}, \eqref{equ:bound_psit_Linftys_L4tx}, and arguing as usual via Poincar\'e's inequality and extra regularity, we obtain from
 \[
  A_t(s) = - \int_s^\infty \Im \, ( \psi_t \overline{\psi_s} ) \, \ud s'
 \]
 that
 \begin{align*}
  \|A_t\|_{L^\infty_{\ds} L^2_t L^2_x} &\lesssim \int_0^\infty \|\psi_t\|_{L^4_t L^4_x} \|\psi_s\|_{L^4_t L^4_x} \, \ud s' \\
  &\lesssim \| s^{\frac{1}{2}} \psi_t \|_{L^\infty_{\ds} L^4_t L^4_x} \| m(s) s^{\frac{1}{2}} \psi_s \|_{L^\infty_{\ds} L^4_t L^4_x} \int_0^1 (s')^{-1+\delta} \, \ud s' \\
  &\quad \quad + \| s^{\frac{1}{2}} \psi_t \|_{L^\infty_{\ds} L^4_t L^4_x} \| \jap{s\Delta}^{\frac{1}{2}} s^{\frac{1}{2}} \psi_s \|_{L^\infty_{\ds} L^4_t L^4_x} \int_1^\infty (s')^{-\frac{3}{2}} \, \ud s' \\
  &\lesssim \|\psi_s\|_{\calS}^2.
 \end{align*}
 In an analogous manner, one can show that $\| \calL_\Omega A_t \|_{L^\infty_{\ds} L^2_t L^2_x} \lesssim \|\psi_s\|_{\calS}^2$, which finishes the proof of~\eqref{eq:bound_At_Linftys_L2tx}.
\end{proof}

%\newpage 

\subsection{Refined estimates on $\ringA$}\label{sub:6refined}

In this final subsection we prove a number of more refined estimates on $\ringA$ which will be used in our delicate treatment of the magnetic interaction term 
\EQ{
- 2 i \ringA^k \nabla_k \psi_s 
} 
that appears on the right-hand side of the Schr\"odinger equation~\eqref{eq:SH-psi_s} for $\psi_s$. To access the local smoothing norms $LE, LE^*$ needed to treat this term in Section~\ref{s:schrod}, we require an auxiliary frequency localization of $\ringA$ using the Littlewood Paley projections $P_\s$ defined in~\eqref{eq:LP-proj}. We write
\begin{equation} \label{eq:lowA}
 \begin{aligned}
 P_{\geq \sigma} \ringA_\ell(s) &:= \Im \int_s^\infty \biggl( \int_{s'}^\infty \nabla_\ell (P_{\geq \sigma} \psi_s)(s'') \, \ud s'' \biggr) \overline{P_{\geq \sigma} \psi_s}(s') \, \ud s' \\
 &\qquad + \Im \int_s^\infty \biggl( \int_{s'}^\infty i A_\ell(s'') (P_{\geq \sigma} \psi_s)(s'') \, \ud s'' \biggr) \overline{P_{\geq \sigma} \psi_s}(s') \, \ud s' \\
 &\qquad - \Im \int_s^\infty \psi_\ell^\infty \overline{P_{\geq \sigma} \psi_s}(s') \, \ud s' .
 \end{aligned}
\end{equation}
We also define
\begin{align}\label{eq:OmegalowA}
\begin{split}
 P_{\geq \sigma}\calL_\Omega \ringA_\ell(s) :=\calL_\Omega P_{\geq\sigma}\ringA.
 \end{split}
\end{align}

%%%%%%%%%%%%%
\begin{lemma}\label{lem:Mestimates}
Under the bootstrap assumptions \eqref{eq:overall-bootstrap} and \eqref{eq:core-bootstrap}, for any $\sigma'\leq 1$, $P_{\geq \sigma'}\ringbfA$ satisfies the following estimates, where $p\in (1,2)$ and $\alpha\equiv\alpha(p):=\frac{2-p}{2p}$, and the implicit constants are independent of $s>0$:
%%%%%%%%%%%
%%%%%%%%%%%
\begin{align}
&M_1:=(\sigma')^{\frac{1}{2}}\|P_{\geq \sigma'}\ringbfA\|_{L_{t}^\infty L_x^\infty(\Ann_{\leq -k_{\sigma'}})}\lesssim \|\psi_s\|_\calS , \label{eq:M1}\\
&M_2:=\sum_{-k_{\sigma'}\leq \ell\leq 0}2^{\ell}\|P_{\geq \sigma'}\ringbfA\|_{L_{t}^\infty L_x^\infty(\Ann_\ell)}\lesssim \|\psi_s\|_\calS , \label{eq:M2}\\
&M_3:=\|r^4P_{\geq\sigma'}\ringbfA\|_{L_{t}^\infty L_x^\infty(\Ann_{\geq0})}\lesssim \|\psi_s\|_\calS , \label{eq:M3}\\
&\tilM_1:=(\sigma')^{\frac{1}{2}}\|\nabla P_{\geq \sigma'}\ringbfA\|_{L_{t}^\infty L_x^\infty(\Ann_{\leq -k_{\sigma'}})}\lesssim (\sigma')^{-\frac{1}{2}} \|\psi_s\|_\calS , \label{eq:tilM1}\\
&\tilM_2:=\sum_{-k_{\sigma'}\leq \ell\leq 0}2^{\ell}\|\nabla P_{\geq \sigma'}\ringbfA\|_{L_{t}^\infty L_x^\infty(\Ann_\ell)}\lesssim (\sigma')^{-\frac{1}{2}}  \|\psi_s\|_\calS , \label{eq:tilM2}\\
&\tilM_3:=\|r^4\nabla P_{\geq\sigma'}\ringbfA\|_{L_{t}^\infty L_x^\infty(\Ann_{\geq0})}\lesssim (\sigma')^{-\frac{1}{2}}  \|\psi_s\|_\calS , \label{eq:tilM3}\\
&M_4:=\|P_{\geq \sigma'}\ringbfA\|_{L_t^\infty L_x^{\frac{2p}{2-p}}}\lesssim (\sigma')^{-\frac{1}{2}+\alpha}\|\psi_s\|_\calS , \label{eq:M4}\\
&M_5:=\|\sinh^{\frac{1}{2}}rP_{\geq\sigma'}\ringbfA\|_{L_t^\infty L_x^{\frac{2p}{2-p}}}\lesssim (\sigma')^{-\frac{1}{4}+\alpha}\|\psi_s\|_\calS .  \label{eq:M5}
\end{align}
\end{lemma}
%%%%%%%%%%%%%
%%%%%%%%%%%%%

%%%%%%%%%%%%
%%%%%%%%%%%%
\begin{proof}
We divide the proof into three steps corresponding to the three terms in the expansion of $P_{\geq\sigma'}\ringA$ in \eqref{eq:lowA}. Along the way we will use the estimates \eqref{equ:bound_ringPsi_Linftys_Str}, \eqref{equ:bound_psis_L2_Str}, and \eqref{equ:bound_psis_Linftys_Str} many times without explicit reference.

{\bf{Step 1.}} We consider the contribution of 
%%%%%%%
%%%%%%%
\begin{align*}
\begin{split}
P_{\geq\sigma'}\ringbfB:=\Im\int_s^\infty\int_{s'}^\infty\nabla P_{\geq \sigma'}\psi_s(s'')\ud s'' \overline{P_{\geq \sigma'}\psi_s}(s')\ud s',
\end{split}
\end{align*}
and denote the corresponding terms in the left-hand side of \eqref{eq:M1}--\eqref{eq:M5} by $M_1^\ringbfB,\dots, M_5^\ringbfB$. Here the contributions of $\tilM_1^\ringbfB,\tilM_2^\ringbfB,\tilM_3^\ringbfB$ to \eqref{eq:tilM1}, \eqref{eq:tilM2}, \eqref{eq:tilM3} are similar to those of  $M_1^\ringbfB,M_2^\ringbfB,M_3^\ringbfB$ to \eqref{eq:M1}, \eqref{eq:M2}, \eqref{eq:M3} and are left out. First
%%%%%%%
%%%%%%%
\begin{align*}
\begin{split}
M_1^\ringbfB&\lesssim \int_s^\infty \int_{s'}^\infty\|\nabla P_{\geq\sigma'}\psi_s(s'')\|_{L_{t}^\infty L_x^\infty}\ud s'' \|\jap{\sigma'\Delta}P_{\geq\sigma'}\psi_s(s')\|_{L_t^\infty L_x^2}\ud s'\\
&\lesssim \int_s^\infty\int_{s'}^\infty (\frac{s'}{s''})^{\frac{1}{2}}\|\jap{s''\Delta}^{\frac{3}{2}}(s'')^{\frac{1}{2}}\psi_s(s'')\|_{L^\infty_tL_x^2}\dspp \|(s')^{\frac{1}{2}}\psi_s(s')\|_{L_t^\infty L_x^2}\dsp\\
&\lesssim \Big\| \int_{s'}^\infty (\frac{s'}{s''})^{\frac{1}{2}}\|\jap{s''\Delta}^{\frac{3}{2}}(s'')^{\frac{1}{2}}\psi_s(s'')\|_{L^\infty_tL_x^2}\dspp\Big\|_{L^2_\dsp}\|(s')^{\frac{1}{2}}\psi_s(s')\|_{L^2_\dsp L_t^\infty L_x^2}\\
&\lesssim \|\psi_s\|_\calS^2
\end{split}
\end{align*}
by Schur's test, proving \eqref{eq:M1}. For $M_2^\ringbfB$ we further divide
%%%%%%%
%%%%%%%
\begin{align*}
\begin{split}
M_2^\ringbfB&\lesssim \int_s^{\sigma'} \sum_{-k_{\sigma'}\leq \ell\leq 0}2^\ell \Big\|\int_{s'}^\infty \nabla P_{\geq\sigma'}\psi_s(s'') P_{\geq\sigma'}\psi_s(s')\ud s''\Big\|_{L_{t}^\infty L_x^\infty(\Ann_\ell)}\ud s'\\
&\quad+\int_{\sigma'}^1\sum_{-k_{\sigma'}\leq\ell\leq -k_{s'}}2^\ell \Big\|\int_{s'}^\infty \nabla P_{\geq\sigma'}\psi_s(s'') P_{\geq\sigma'}\psi_s(s')\ud s''\Big\|_{L_{t}^\infty L_x^\infty(\Ann_\ell)}\ud s'\\
&\quad + \int_{\sigma'}^1 \sum_{-k_{s'}\leq\ell\leq0}2^{\ell}\Big\|\int_{s'}^\infty \nabla P_{\geq\sigma'}\psi_s(s'') P_{\geq\sigma'}\psi_s(s')\ud s''\Big\|_{L_{t}^\infty L_x^\infty(\Ann_\ell)}\ud s'\\
&\quad +\int_1^\infty \sum_{-k_{\sigma'}\leq\ell\leq0}2^\ell \Big\|\int_{s'}^\infty \nabla P_{\geq\sigma'}\psi_s(s'') P_{\geq\sigma'}\psi_s(s')\ud s''\Big\|_{L_{t}^\infty L_x^\infty(\Ann_\ell)}\ud s'\\
&=:M_{2,1}^\ringbfB+M_{2,2}^\ringbfB+M_{2,3}^\ringbfB+M_{2,4}^\ringbfB.
\end{split}
\end{align*}
Using the radial Sobolev estimate and the fact that $(\sigma')^{-\delta}\leq (s')^{-\delta}$ for $s'\leq \sigma'$
%%%%%%%
%%%%%%%
\begin{align*}
\begin{split}
M^\ringbfB_{2,1}&\lesssim \int_s^{\sigma'}\sum_{-k_{\sigma'}\leq\ell\leq 0}\int_{s'}^\infty \|\sinh^{\frac{1}{2}}r\nabla P_{\geq\sigma'}\psi_s(s'')\|_{L_{t}^\infty L_x^\infty}\ud s''\|\sinh^{\frac{1}{2}}r P_{\geq \sigma'}\psi_s(s')\|_{L_{t}^\infty L_x^\infty}\ud s'\\
&\lesssim \int_s^{\sigma'} \sum_{-k_{\sigma'}\leq \ell\leq0}\int_{s'}^\infty (\frac{s'}{s''})^{\frac{1}{4}}\|\jap{s''\Delta}\jap{\Omega}(s'')^{\frac{1}{2}}\psi_s(s'')\|_{L_t^\infty L_x^2}\dspp \|\jap{s'\Delta}^{\frac{1}{2}}\jap{\Omega}(s')^{\frac{1}{2}}\psi_s(s')\|_{L_t^\infty L_x^2}\dsp\\
&\lesssim \int_s^{\sigma'} \int_{s'}^\infty (\frac{s'}{s''})^{\frac{1}{4}}\|\jap{s''\Delta}\jap{\Omega}(s'')^{\frac{1}{2}}\psi_s(s'')\|_{L_t^\infty L_x^2}\dspp \|\jap{s'\Delta}^{\frac{1}{2}}\jap{\Omega}m(s')(s')^{\frac{1}{2}}\psi_s(s')\|_{L_t^\infty L_x^2}\dsp\\
&\lesssim \|\psi_s\|_\calS^2
\end{split}
\end{align*}
by Cauchy-Schwarz and Schur's test. Next, by Gagliardo-Nirenberg and adding up the spatial weights to gain a factor of $(s')^{\frac{1}{2}}$
%%%%%%%
%%%%%%%
\begin{align*}
\begin{split}
M^{\ringbfB}_{2,2}&\lesssim \int_{\sigma'}^1(s')^{\frac{1}{2}}\int_{s'}^\infty \|\nabla P_{\geq\sigma'}\psi_s(s'')\|_{L_{t}^\infty L_x^\infty}\ud s'' \|P_{\geq \sigma'}\psi_s(s')\|_{L_{t}^\infty L_x^\infty}\ud s'\\
&\lesssim \int_{\sigma'}^1\int_{s'}^\infty (\frac{s'}{s''})^{\frac{1}{2}}\|\jap{s''\Delta}^{\frac{3}{2}}(s'')^{\frac{1}{2}}\psi_s(s'')\|_{L_t^\infty L_x^2}\dspp \|\jap{s'\Delta}(s')^{\frac{1}{2}}\psi_s(s')\|_{L_t^\infty L_x^2}\dsp\\
&\lesssim \|\psi_s\|_\calS^2.
\end{split}
\end{align*}
Similarly, using radial Sobolev,
%%%%%%%
%%%%%%%
\begin{align*}
\begin{split}
M^\ringbfB_{2,3}&\lesssim \int_{\sigma'}^1(s')^{-\delta}\int_{s'}^\infty\|\sinh^{\frac{1}{2}}r|\nabla|P_{\geq\sigma'}\psi_s(s'')\|_{L_{t}^\infty L_x^\infty}\ud s'' \|\sinh^{\frac{1}{2}}r P_{\geq\sigma'}\psi_s(s')\|_{L_{t}^\infty L_x^\infty}\ud s'\\
&\lesssim \int_{\sigma'}^1\int_{s'}^\infty(\frac{s'}{s''})^{\frac{1}{4}}\|\jap{s''\Delta}\jap{\Omega}(s'')^{\frac{1}{2}}\psi_s(s'')\|_{L_t^\infty L_x^2}\dspp \|\jap{s'\Delta}^{\frac{1}{2}}\jap{\Omega}m(s')(s')^{\frac{1}{2}}\psi_s(s')\|_{L_t^\infty L_x^2}\dsp\\
&\lesssim \|\psi_s\|_\calS^2.
\end{split}
\end{align*}
For $M^\ringbfB_{2,4}$ we add up the spatial weights and, since $s'\geq1$, freely insert a factor of $(s')^{\frac{1}{2}}$ to get
%%%%%%%
%%%%%%%
\begin{align*}
\begin{split}
M^\ringbfB_{2,4}&\lesssim \int_1^\infty(s')^{\frac{1}{2}}\int_{s'}^\infty \|\nabla P_{\geq\sigma'}\psi_s(s'')\|_{L_{t,x}^\infty}\ud s'' \|P_{\geq\sigma'}\psi_s(s')\|_{L_{t,x}^\infty}\ud s'\\&\lesssim\int_1^\infty\int_{s'}^\infty (\frac{s'}{s''})^{\frac{1}{2}}\|\jap{s''\Delta}^{\frac{3}{2}}(s'')^{\frac{1}{2}}\psi_s(s'')\|_{L_t^\infty L_x^2}\dspp \|\jap{s'\Delta}(s')^{\frac{1}{2}}\psi_s(s')\|_{_t^\infty L_x^2}\dsp\\
&\lesssim \|\psi_s\|_\calS^2.
\end{split}
\end{align*}
This completes the proof of \eqref{eq:M2}. For $M_3^\ringbfB$ we use radial Sobolev to get
%%%%%%%
%%%%%%%
\begin{align*}
\begin{split}
M^\ringbfB_3&\lesssim \int_s^\infty \int_{s'}^\infty \|\sinh^{\frac{1}{2}}r\nabla P_{\geq\sigma'}\psi_s(s'')\|_{L_{t,x}^\infty}\ud s'' \|\sinh^{\frac{1}{2}}r P_{\geq\sigma'}\psi_s(s')\|_{L_{t,x}^\infty}\ud s''\\
&\lesssim \int_s^\infty \int_{s'}^\infty (\frac{s'}{s''})^{\frac{1}{4}} \|\jap{s''\Delta}\jap{\Omega}(s'')^{\frac{1}{2}}\psi_s(s'')\|_{L_t^\infty L_x^2}\dspp \|\jap{s'\Delta}^{\frac{1}{2}}\jap{\Omega}(s')^{\frac{1}{2}}\psi_s(s')\|_{L_t^\infty L_x^2}\dsp\\
&\lesssim \|\psi_s\|_\calS^2.
\end{split}
\end{align*}
For $M^\ringbfB_4$ note that by Gagliardo-Nirenberg, and with $\alpha\equiv\alpha(p):=\frac{2-p}{2p}$
%%%%%%%
%%%%%%%
\begin{align*}
\begin{split}
\|P_{\geq\sigma'}\ringbfB\|_{L_t^\infty L_{x}^{\frac{2p}{2-p}}}\lesssim \|P_{\geq\sigma'}\ringbfB\|_{L_t^\infty L_x^2}^{2\alpha} \|\nabla P_{\geq\sigma'}\ringbfB\|_{L_t^\infty L_x^2}^{1-2\alpha}.
\end{split}
\end{align*}
Now
%%%%%%%
%%%%%%%
\begin{align*}
\begin{split}
\|P_{\geq\sigma'}\ringbfB\|_{L_t^\infty L_x^2}&\lesssim \int_s^\infty \int_{s'}^\infty \|\nabla P_{\geq \sigma'}\psi_s(s'')\|_{L_{t,x}^\infty}\ud s'' \|P_{\geq \sigma'}\psi_s(s')\|_{L_t^\infty L_x^2}\ud s'\\
&\lesssim \int_s^\infty \int_{s'}^\infty (\frac{s'}{s''})^{\frac{1}{2}}\|\jap{s''\Delta}^{\frac{3}{2}}(s'')^{\frac{1}{2}}\psi_s(s'')\|_{L_t^\infty L_x^2}\dspp \|(s')^{\frac{1}{2}}\psi_s(s')\|_{L_t^\infty L_x^2}\dsp\\
&\lesssim \|\psi_s\|_\calS^2.
\end{split}
\end{align*}
Similarly, distributing the derivative on the factors of $P_{\geq \sigma'}\ringbfB$ and absorbing by $P_{\geq\sigma'}$ we get
%%%%%%%
%%%%%%%
\begin{align*}
\begin{split}
\|\nabla P_{\geq\sigma'}\ringbfB\|_{L_t^\infty L_x^2}\lesssim (\sigma')^{-\frac{1}{2}}\|\psi_s\|_\calS^2,
\end{split}
\end{align*}
which together with the previous two estimates gives \eqref{eq:M4}. Finally by the same Gagliardo-Nirenberg inequality combined with radial Sobolev 
%%%%%%%
%%%%%%%
\begin{align*}
\begin{split}
M^\ringbfB_5&\lesssim \int_s^\infty \int_{s'}^\infty \|\sinh^{\frac{1}{2}}r \nabla P_{\geq \sigma'}\psi_s(s'')\|_{L_{t}^\infty L_x^\infty}\ud s'' \|P_{\geq\sigma'}\psi_s(s')\|^{2\alpha}_{L_t^\infty L_x^2}\|\nabla P_{\geq\sigma'}\psi_s(s')\|_{L_t^\infty L_x^2}^{1-2\alpha}\ud s'\\
&\lesssim (\sigma')^{-\frac{1}{4}+\alpha} \int_s^\infty \int_{s'}^\infty (\frac{s'}{s''})^{\frac{1}{4}}\|\jap{s''\Delta}\jap{\Omega}(s'')^{\frac{1}{2}}\psi_s(s'')\|_{L_t^\infty L_x^2}\dspp \|\jap{s'\Delta}^{\frac{1}{4(1-2\alpha)}}\jap{\Omega}(s')^{\frac{1}{2}}\psi_s(s')\|_{L_t^\infty L_x^2}\dsp\\
&\lesssim (\sigma')^{-\frac{1}{4}+\alpha} \|\psi_s\|_\calS^2.
\end{split}
\end{align*}

{\bf{Step 2.}} We consider the contribution of 
%%%%%%%
%%%%%%%
\begin{align*}
\begin{split}
P_{\geq\sigma'}\ringbfC:=\Im \int_s^\infty \biggl( \int_{s'}^\infty A(s'') (P_{\geq \sigma'} \psi_s)(s'') \, \ud s'' \biggr) \, \overline{ P_{\geq \sigma'} \psi_s}(s') \, \ud s',
\end{split}
\end{align*}
and denote the corresponding terms in the left-hand side of \eqref{eq:M1}--\eqref{eq:M5} by $M_1^\ringbfC,\dots, M_5^\ringbfC$. We also decompose $A$ in the usual way as
%%%%%%%
%%%%%%%
\begin{align*}
\begin{split}
A = A^\infty+\ringbfA.
\end{split}
\end{align*}
Except for  $\tilM_1^\ringbfC,\tilM_2^\ringbfC,\tilM_3^\ringbfC$, the contribution of $\ringbfA$ can always be bounded in exactly the same way as in Step~1 by placing $\ringbfA$ in $L_{t}^\infty L_x^\infty$ and using that thanks to the estimates~\eqref{eq:bound_ringAL_Linftyall_Str}--\eqref{eq:bound_ringAQ_Linftyall_Str} we have
%%%%%%%
%%%%%%%
\begin{align*}
\begin{split}
\|\ringbfA(s'')\|_{L_t^\infty L_x^\infty}\lesssim (s'')^{-\frac{1}{2}}\|\psi_s\|_\calS.
\end{split}
\end{align*}
Therefore, in the rest of the argument we consider only the contribution of $A^\infty$ as well as $\tilM_1^\ringbfC,\tilM_2^\ringbfC,\tilM_3^\ringbfC$. Let us start by discussing the contribution of $\ringbfA$ to $\tilM_1^\ringbfC,\tilM_2^\ringbfC,\tilM_3^\ringbfC$. Using the estimates~\eqref{eq:bound_ringAL_Linftyall_Str}--\eqref{eq:bound_ringAQ_Linftyall_Str} we have
%%%%%%%
%%%%%%%
\begin{align*}
\begin{split}
\|\nabla\ringbfA(s'')\|_{L_{t}^\infty L_x^\infty}\lesssim (s'')^{-1}\|\psi_s\|_\calS.
\end{split}
\end{align*}
It follows that the contribution of $\ringbfA$ to $\tilM_1^\ringbfC$ is bounded by
%%%%%%%
%%%%%%%
\begin{align*}
\begin{split}
&\int_s^\infty\int_{s'}^\infty (s'')^{-1}\|P_{\geq\sigma'}\psi_s(s'')\|_{L_{t,x}^\infty}\ud s'' \|(\sigma')^{\frac{1}{2}}P_{\geq\sigma'}\psi_s(s')\|_{L_{t,x}^\infty}\ud s'\\
&\lesssim (\sigma')^{-\frac{1}{2}}\int_s^\infty \int_{s'}^\infty (\frac{s'}{s''})^{\frac{1}{2}}\|\jap{\sigma'\Delta}P_{\geq\sigma'}(s'')^{\frac{1}{2}}\psi_s(s'')\|_{L_t^\infty L_x^2}\dspp \|\jap{\sigma'\Delta}P_{\geq\sigma'}(s')^{\frac{1}{2}}\psi_s(s')\|_{L_t^\infty L_x^2}\dsp\\
&\lesssim (\sigma')^{-\frac{1}{2}}\int_s^\infty \int_{s'}^\infty (\frac{s'}{s''})^{\frac{1}{2}}\|(s'')^{\frac{1}{2}}\psi_s(s'')\|_{L_t^\infty L_x^2}\dspp \|(s')^{\frac{1}{2}}\psi_s(s')\|_{L_t^\infty L_x^2}\dsp\\
&\lesssim (\sigma')^{-\frac{1}{2}}\|\psi_s\|_\calS^2.
\end{split}
\end{align*}
Similarly, the contribution of $\ringbfA$ to $\tilM_{2,1}^\ringbfC$ is bounded by
%%%%%%%
%%%%%%%
\begin{align*}
\begin{split}
&\int_s^{\sigma'}\int_{s'}^\infty (s'')^{-1}\|\sinh^{\frac{1}{2}}r P_{\geq\sigma'}\psi_s(s'')\|_{L_{t}^\infty L_x^\infty}\ud s'' \|\sinh^{\frac{1}{2}}rP_{\geq \sigma'}m(s')\psi_s(s')\|_{L_{t}^\infty L_x^\infty}\ud s'\\
&\lesssim\int_s^{\sigma'}\int_{s'}^\infty (\frac{s'}{s''})^{\frac{1}{2}}(\sigma')^{-\frac{1}{4}}\|\jap{\sigma'\Delta}^{\frac{1}{2}}(s'')^{\frac{1}{2}}P_{\geq\sigma'}\psi_s(s'')\|_{L_{t}^\infty L_x^2}\dspp (\sigma')^{-\frac{1}{4}} \|\jap{\sigma'\Delta}P_{\geq \sigma'}m(s')(s')^{\frac{1}{2}}\psi_s(s')\|_{L_{t}^\infty L_x^2}\dsp\\
&\lesssim(\sigma')^{-\frac{1}{2}}\int_s^{\sigma'}\int_{s'}^\infty (\frac{s'}{s''})^{\frac{1}{2}}\|(s'')^{\frac{1}{2}}P_{\geq\sigma'}\psi_s(s'')\|_{L_{t}^\infty L_x^2}\dspp  \|m(s')(s')^{\frac{1}{2}}\psi_s(s')\|_{L_{t}^\infty L_x^2}\dsp\\
&\lesssim(\sigma')^{-\frac{1}{2}}\|\psi_s\|_\calS^2.
\end{split}
\end{align*}
The contribution of $\ringbfA$ to $\tilM_{2,2}^\ringbfC, \tilM_{2,3}^\ringbfC, \tilM_{2,4}^\ringC$, and $\tilM_{3}^\ringbfC$ can be treated in similar ways.

Finally, using the bound $\|A^\infty\|_{L_x^\infty}\lesssim1$, the contribution of $A^\infty$ is bounded in exactly the same way as in Step 1, where we simply use Poincar\'e to introduce an extra derivative. 

{\bf{Step 3.}} The contribution of
%%%%%%%
%%%%%%%
\begin{align*}
\begin{split}
\Im \int_s^\infty \Psi^\infty \overline{P_{\geq\sigma'}\psi_s(s')}\ud s'
\end{split}
\end{align*}
is treated using similar arguments and we will be brief. For $M_1$ we bound this contribution by
%%%%%%%
%%%%%%%
\begin{align*}
\begin{split}
\int_s^\infty \|\psi_s\|_{L_{t}^\infty L_x^2}\ud s'\lesssim \int_s^1\|(s')^{\frac{1}{2}}\psi_s(s')\|_{L_t L_x^2}(s')^{-\frac{1}{2}}\ud s'+\int_1^\infty \|(s')^{\frac{1}{2}}\psi_s(s')\|_{L_t^\infty L_x^2}(s')^{-\frac{1}{2}}\ud s'.
\end{split}
\end{align*}
The second integral is bounded using Poincar\'e and \eqref{equ:bound_psis_Linftys_Str} by
%%%%%%%
%%%%%%%
\begin{align*}
\begin{split}
\int_1^\infty \|(s'\Delta)(s')^{\frac{1}{2}}\psi_s(s')\|_{L_t^\infty L_x^2}(s')^{-\frac{1}{2}}\dsp\lesssim \|(s\Delta)s^{\frac{1}{2}}\psi_s\|_{\Ls^\infty L_t^\infty L_x^2}\lesssim \|\psi_s\|_\calS.
\end{split}
\end{align*}
The first integral is simply bounded by
%%%%%%%
%%%%%%%
\begin{align*}
\begin{split}
\|s^{\frac{1}{2}}\psi_s\|_{\Ls^\infty L_t^\infty L_x^2}\lesssim \|\psi_s\|_\calS.
\end{split}
\end{align*}
For $M_2$ the corresponding contribution is bounded by
%%%%%%%
%%%%%%%
\begin{align*}
\begin{split}
\int_s^\infty\|\sinh^{\frac{1}{2}}r \psi_s\|_{L_{t}^\infty L_x^\infty}\ud s'\lesssim \int_s^\infty \|\jap{s'\Delta}^{\frac{1}{2}}(s')^{\frac{1}{2}}\psi_s(s')\|_{L_t^\infty L_x^2}(s')^{-\frac{3}{4}}\ud s'\lesssim \|\psi_s\|_\calS,
\end{split}
\end{align*}
where we have used the bound $r\leq \sinh r$, radial Sobolev, Poincar\'e,  \eqref{equ:bound_psis_Linftys_Str}, and the bound
%%%%%%%
%%%%%%%
\begin{align*}
\begin{split}
\sum_{-k_{\sigma'} \leq \ell \leq 0} \|r^{\frac{1}{2}}\Psi^\infty\|_{L_x^\infty(\Ann_\ell)}\lesssim  \|\Psi^\infty\|_{L_x^\infty}\lesssim1.   
\end{split}
\end{align*}
which follows from~\eqref{eq:cf-coeff}. 
Using $\|r^4\Psi^\infty\|_{L_x^\infty}\lesssim 1$ (which is again from~\eqref{eq:cf-coeff}) the contribution to $M_3$ is bounded similarly by
%%%%%%%
%%%%%%%
\begin{align*}
\begin{split}
\|\Psi^\infty\|_{L_x^\infty}\int_s^\infty \|\sinh^{\frac{1}{2}}r\psi_s(s')\|_{L_{t}^\infty L_x^\infty}\ud s'\lesssim \|\psi_s\|_\calS.
\end{split}
\end{align*}
Similarly, by Poincar\'e, the contribution of $\nabla \Psi^\infty$ to $\tilM_1$ is bounded by
%%%%%%%
%%%%%%%
\begin{align*}
\begin{split}
\int_s^\infty \|\jap{\sigma'\Delta}P_{\geq\sigma'}\psi_s(s')\|_{L_{t}^\infty L_x^2}\ud s'& \lesssim (\sigma')^{-\frac{1}{2}}\int_s^\infty \|\jap{\sigma'\Delta}^{\frac{3}{2}}P_{\geq\sigma'}\psi_s(s')\|_{L_{t}^\infty L_x^2}\ud s'\\
&\lesssim (\sigma')^{-\frac{1}{2}}\int_s^\infty \|\psi_s(s')\|_{L_{t}^\infty L_x^2}\ud s'\lesssim (\sigma')^{-\frac{1}{2}}\|\psi_s\|_\calS,
\end{split}
\end{align*}
where we have used Gagliardo-Nirenberg and Poincar\'e. The contribution of $\nabla \Psi^\infty$ to $\tilM_2$ is bonded by
%%%%%%%
%%%%%%%
\begin{align*}
\begin{split}
\int_s^\infty\|P_{\geq\sigma'}\psi_s(s')\|_{L_{t}^\infty L_x^\infty}\lesssim (\sigma')^{-\frac{1}{2}}\int_{s}^\infty \|\jap{\sigma'\Delta}P_{\geq\sigma'}\psi_s(s')\|_{L_t^\infty L_x^2}\ud s'\lesssim (\sigma')^{-\frac{1}{2}}\|\psi_s\|_\calS,
\end{split}
\end{align*}
where we have used
%%%%%%%
%%%%%%%
\begin{align*}
\begin{split}
\sum_{-k_{\sigma'}\leq \ell \leq 0} \|r\nabla \Psi^\infty\|_{L_x^\infty(\Ann_\ell)}\lesssim \|\nabla \Psi^\infty\|_{L_x^\infty}\lesssim 1.
\end{split}
\end{align*}
Since $\|r^4\nabla\Psi^\infty\|_{L_x^\infty}\lesssim 1$, the contribution of $\nabla \Psi^\infty$ to $\tilM_3$ is also bounded by
%%%%%%%
%%%%%%%
\begin{align*}
\begin{split}
\int_s^\infty\|P_{\geq\sigma'}\psi_s(s')\|_{L_{t}^\infty L_x^\infty}\ud s'\lesssim (\sigma')^{-\frac{1}{2}}\|\psi_s\|_\calS.
\end{split}
\end{align*}
By Gagliardo-Nirenberg the contribution of $\Psi^\infty$ to $M_4$ is bounded by
%%%%%%%
%%%%%%%
\begin{align*}
\begin{split}
\int_s^\infty \|P_{\geq \sigma'}\psi_s(s')\|_{L_t^\infty L_x^2}^{2\alpha}\|\nabla P_{\geq\sigma'}\psi_s(s') \|_{L_t^\infty L_x^2}^{1-2\alpha}\ud s'\lesssim (\sigma')^{-\frac{1}{2}+\alpha}\int_s^{\infty} \|\psi_s(s')\|_{L_t^\infty L_x^2}\ud s'\lesssim  (\sigma')^{-\frac{1}{2}+\alpha}\|\psi_s\|_\calS.
\end{split}
\end{align*}
Finally, since $\|\sinh^{\frac{1}{2}}r\Psi^\infty\|_{L_x^\infty}\lesssim 1$, using Gagliard-Nirenberg the contribution of $\Psi^\infty$ to $M_5$ is bounded by
%%%%%%%
%%%%%%%
\begin{align*}
\begin{split}
&\int_s^\infty \|P_{\geq\sigma}\psi_s(s')\|_{L_t^\infty L_x^2}^{2\alpha}\|\jap{\sigma'\Delta}^{\frac{\frac{1}{4}-\alpha}{1-2\alpha}}P_{\geq\sigma'}\jap{s'\Delta}^{\frac{1}{4(1-2\alpha)}}\psi_s(s')\|_{L_t^\infty L_x^2}^{1-2\alpha}(\sigma')^{-\frac{1}{4}+\alpha}(s')^{-\frac{1}{4}}\ud s'\\
&\lesssim (\sigma')^{-\frac{1}{4}+\alpha}\int_s^\infty \|\jap{s'\Delta}^{\frac{1}{4(1-2\alpha)}}\psi_s(s')\|_{L_t^\infty L_x^2}(s')^{-\frac{1}{4}}\ud s'\lesssim (\sigma')^{-\frac{1}{4}+\alpha}\|\psi_s\|_\calS. \qedhere
\end{split}
\end{align*}
\end{proof}
%%%%%%%%%%%%
%%%%%%%%%%%%

The next lemma provides bounds to deal with the contribution of $P_{\geq\sigma'} \calL_\Omega \ringA$ to the high frequency estimate in treating the magnetic interaction term.
%%%%%%%%%%%%%
%%%%%%%%%%%%%
\begin{lemma}\label{lem:MOmegaestimates}
Under the bootstrap assumptions \eqref{eq:overall-bootstrap} and \eqref{eq:core-bootstrap}, for any $\sigma'\leq 1$, $P_{\geq \sigma'}\ringA$ satisfies the following estimates, where $p\in (1,2)$ and $\alpha\equiv\alpha(p):=\frac{2-p}{2p}$, and the implicit constants are independent of $s>0$:
%%%%%%%%%%%
%%%%%%%%%%%
\begin{align}
&M\Omega_1:= (\sigma')^{\frac{1}{2}}\|P_{\geq \sigma'} \calL_\Omega \ringbfA\|_{L_{t}^\infty L_x^\infty(\Ann_{\leq -k_{\sigma'}})}\lesssim \|\psi_s\|_\calS , \label{eq:MOmega1}\\
&M\Omega_2:= \sum_{-k_{\sigma'}\leq \ell\leq 0}2^{\ell}\|P_{\geq \sigma'} \calL_\Omega \ringbfA\|_{L_{t}^\infty L_r^\infty L_\theta^2(\Ann_\ell)}\lesssim \|\psi_s\|_\calS , \label{eq:MOmega2}\\
&M\Omega_3:= \|r^4 P_{\geq\sigma'} \calL_\Omega \ringbfA\|_{L_{t}^\infty L_r^\infty L_\theta^2(\Ann_{\geq0})}\lesssim \|\psi_s\|_\calS , \label{eq:MOmega3}\\
&\tilM\Omega_1:=(\sigma')^{\frac{1}{2}}\|\nabla P_{\geq \sigma'} \calL_\Omega \ringbfA\|_{L_{t}^\infty L_x^\infty(\Ann_{\leq -k_{\sigma'}})}\lesssim (\sigma')^{-\frac{1}{2}} \|\psi_s\|_\calS , \label{eq:tilMOmega1}\\
&\tilM\Omega_2:=\sum_{-k_{\sigma'}\leq \ell\leq 0}2^{\ell}\|\nabla P_{\geq \sigma'} \calL_\Omega \ringbfA\|_{L_{t}^\infty L_r^\infty L_\theta^2(\Ann_\ell)}\lesssim (\sigma')^{-\frac{1}{2}}  \|\psi_s\|_\calS , \label{eq:tilMOmega2}\\
&\tilM\Omega_3:=\|r^4\nabla P_{\geq\sigma'} \calL_\Omega \ringbfA\|_{L_{t}^\infty L_r^\infty L_\theta^2(\Ann_{\geq0})}\lesssim (\sigma')^{-\frac{1}{2}}  \|\psi_s\|_\calS , \label{eq:tilMOmega3}\\
&M\Omega_4:=\|P_{\geq \sigma'} \calL_\Omega \ringbfA\|_{L_t^\infty L_x^{\frac{2p}{2-p}}}\lesssim (\sigma')^{-\frac{1}{2}+\alpha}\|\psi_s\|_\calS , \label{eq:MOmega4}\\
&M\Omega_5:=\|\sinh^{\frac{1}{2}}rP_{\geq\sigma'} \calL_\Omega \ringbfA\|_{L_t^\infty L_x^{\frac{2p}{2-p}}}\lesssim (\sigma')^{-\frac{1}{4}+\alpha}\|\psi_s\|_\calS . \label{eq:MOmega5}
\end{align}
\end{lemma}
%%%%%%%%%%%%%
%%%%%%%%%%%%%

\begin{proof}
The proof is very similar to that of Lemma~\ref{lem:Mestimates} so we only show special cases which demonstrate the strategy. In view of the definition~\eqref{eq:OmegalowA} we decompose $P_{\geq\sigma'}\ringA$ as before and apply $\calL_\Omega$.

{\bf{Step 1.}} We consider the contributions of 
%%%%%%%
%%%%%%%
\begin{align*}
 I &:=\Im \int_s^\infty \int_{s'}^\infty \nabla P_{\geq\sigma'}\psi_s(s'')\ud s'' P_{\geq\sigma'}\Omega\psi_s(s')\ud s',\\
 II &:=\Im \int_s^\infty \int_{s'}^\infty \nabla P_{\geq\sigma'}\Omega\psi_s(s'')\ud s'' P_{\geq\sigma'}\psi_s(s')\ud s'.
\end{align*}
We carry out three examples. Decomposing
%%%%%%%
%%%%%%%
\begin{align*}
\begin{split}
M\Omega_2=M\Omega_{2,1}+\dots+M\Omega_{2,4}
\end{split}
\end{align*}
as in $M_2$ in the proof of Lemma~\ref{lem:Mestimates}, the contribution of $I$ to $M\Omega_{2,1}$ is bounded by
%%%%%%%
%%%%%%%
\begin{align*}
\begin{split}
&\int_s^{\sigma'}\int_{s'}^\infty \|\sinh^{\frac{1}{2}}r \nabla P_{\geq\sigma'}\psi_s(s'')\|_{L_{t}^\infty L_x^\infty}\ud s'' \|\sinh^{\frac{1}{2}}rP_{\geq\sigma'}m(s')\Omega\psi_s(s')\|_{L_{t}^\infty L_r^\infty L_\theta^2}\ud s'\\
&\lesssim \int_s^{\sigma'}\int_{s'}^\infty (\frac{s'}{s''})^{\frac{1}{4}}\|\jap{s''\Delta}(s'')^{\frac{1}{2}}\jap{\Omega}\psi_s(s')\|_{L_t^\infty L_x^2}\dspp \|\jap{s'\Delta}^{\frac{1}{2}}m(s')(s')^{\frac{1}{2}}\Omega \psi_s(s')\|_{L_t^\infty L_x^2}\dsp\\
&\lesssim \|\psi_s\|_\calS^2.
\end{split}
\end{align*}
Here for the $L_\theta^2$ term we have applied radial Sobolev embedding to a radial function, which is why we do not pick up an extra factor of $\Omega$. Similarly, again applying radial Sobolev embedding to a radial function, the contribution of $I$ to $M\Omega_{2,2}$ is bounded as 
%%%%%%%
%%%%%%%
\begin{align*}
\begin{split}
&\int_{\sigma'}^1 \sum_{- k_{\sgm'} \leq \ell \leq - k_{s'}} \int_{s'}^\infty  2^{\frac{\ell}{2}} \|\nabla P_{\geq\sigma'}\psi_s(s'')\|_{L_{t}^\infty L_x^\infty}\ud s'' \| \sinh^{\frac{1}{2}} r P_{\geq\sigma'}\Omega\psi_s(s')\|_{L_{t}^\infty L_r^\infty L_\theta^2}\ud s'\\
&\lesssim \int_{\sigma'}^1\int_{s'}^\infty(s')^{\frac{1}{4}}\|\nabla P_{\geq\sigma'}\psi_s(s'')\|_{L_{t}^\infty L_x^\infty}\ud s'' \|\sinh^{\frac{1}{2}} r P_{\geq\sigma'}\Omega\psi_s(s')\|_{L_{t}^\infty L_r^\infty L_\theta^2}\ud s'\\
&\lesssim \int_{\sigma'}^1\int_{s'}^\infty (\frac{s'}{s''})^{\frac{1}{2}}\|\jap{s''\Delta}^{\frac{3}{2}}(s'')^{\frac{1}{2}}\psi_s(s'')\|_{L_t^\infty L_x^2}\dspp\|\jap{s'\Delta}^{\frac{1}{2}} (s')^{\frac{1}{2}}\Omega \psi_s(s')\|_{L_t^\infty L_x^2}\dsp\\
&\lesssim \|\psi_s\|_\calS^2.
\end{split}
\end{align*}

Finally to bound the contribution of $II$ to $M\Omega_5$ we have to argue a bit differently from before and bound this by
%%%%%%%
%%%%%%%
\begin{align*}
\begin{split}
&\int_s^\infty \int_{s'}^\infty \|\nabla P_{\geq\sigma'}\Omega\psi_s(s'')\|_{L_t^\infty L_x^2}^{2\alpha}\|\Delta P_{\geq\sigma'}\Omega\psi_s(s'')\|_{L_t^\infty L_x^2}^{1-2\alpha}\ud s'' \|\sinh^{\frac{1}{2}}r P_{\geq\sigma'}\psi_s(s')\|_{L_{t}^\infty L_x^\infty}\ud s'\\
&\lesssim (\sigma')^{-\frac{1}{4}+\alpha} \int_s^\infty \int_{s'}^\infty (\frac{s'}{s''})^{\frac{1}{4}}\|\jap{s''\Delta}^{\frac{3-4\alpha}{4-8\alpha}}(s'')^{\frac{1}{2}}\Omega\psi_s(s'')\|_{L_t^\infty L_x^2}\dspp \|\jap{s'\Delta}^{\frac{1}{2}}(s')^{\frac{1}{2}}\jap{\Omega}\psi_s(s')\|_{L_t^\infty L_x^2}\dsp\\
&\lesssim (\sigma')^{-\frac{1}{4}+\alpha}\|\psi_s\|_\calS^2.
\end{split}
\end{align*}

{\bf{Steps 2, 3.}} It remains to consider the contribution of $\calL_\Omega$ applied to the last two terms in the expansion \eqref{eq:lowA} of $P_{\geq\sigma'}\ringA$. But these can be treated using similar arguments as above in the same way as in the proof of Lemma~\ref{lem:Mestimates}.
\end{proof}
%%%%%%%%%%%%%
%%%%%%%%%%%%%

Finally, in the next lemma we collect the estimates needed in treating parts of the magnetic interaction term with Strichartz estimates.
%%%%%%%%%%%%%
%%%%%%%%%%%%%
\begin{lemma}\label{lem:noMestimates}
Under the bootstrap assumptions \eqref{eq:overall-bootstrap} and \eqref{eq:core-bootstrap}, for any $\sigma'\leq 1$, $P_{\geq \sigma'}\ringA$ satisfies the following estimates:
%%%%%%%
%%%%%%%
\begin{align}
 \|\nabla_k (P_{\geq\sigma'}\jap{\calL_\Omega}\ringA^k)(s)\|_{\Ls^\infty(L_{t}^2L_x^2+L_t^{\frac{8}{3}}L_x^{\frac{8}{3}})} &\lesssim \|\psi_s\|_\calS ,\label{eq:divPsigmaAStrich}\\
 \|P_{\leq\sigma'}\jap{\calL_\Omega}\ringbfA\|_{\Ls^\infty(L_{t}^2L_x^2+L_t^{\frac{8}{3}}L_x^{\frac{8}{3}})} &\lesssim (\sigma')^{\frac{1}{2}}\|\psi_s\|_{\calS} , \label{eq:PlesssigmaAStrich}\\
 \|\jap{\calL_\Omega}\ringbfA\|_{\Ls^\infty L_t^{\frac{8}{3}}L_x^{\frac{8}{3}}} &\lesssim \|\psi_s\|_\calS . \label{eq:ringAL83}
\end{align}
\end{lemma}
%%%%%%%%%%%%%
%%%%%%%%%%%%%

\begin{proof}
For simplicity of notation we give the proof for $\ringbfA$ only, but the treatment of $\calL_\Omega\ringbfA$ is identical because we never use the radial Sobolev estimate. Let us start with \eqref{eq:divPsigmaAStrich} and consider the following terms separately:
\begin{align}
\begin{split}
 I &:= \Im \int_s^\infty \biggl( \int_{s'}^\infty \Delta (P_{\geq \sigma'} \psi_s)(s'') \, \ud s'' \biggr) \, \overline{ P_{\geq \sigma'} \psi_s}(s') \, \ud s' , \\
 II &:= \Im \int_s^\infty \biggl( \int_{s'}^\infty \nabla_\ell (P_{\geq \sigma'} \psi_s)(s'') \, \ud s'' \biggr) \, \nabla^\ell\overline{ P_{\geq \sigma'} \psi_s}(s') \, \ud s' , \\
 III &:=   \Im \,i \int_s^\infty \biggl( \int_{s'}^\infty (\nabla^\ell A_\ell(s'')) (P_{\geq \sigma'} \psi_s)(s'') \, \ud s'' \biggr) \, \overline{ P_{\geq \sigma'} \psi_s}(s') \, \ud s' , \\
 IV &:=  \Im\, i \int_s^\infty \biggl( \int_{s'}^\infty  A_\ell(s'') \nabla^\ell(P_{\geq \sigma'} \psi_s)(s'') \, \ud s'' \biggr) \, \overline{ P_{\geq \sigma'} \psi_s}(s') \, \ud s' , \\
 V &:=  \Im \,i \int_s^\infty \biggl( \int_{s'}^\infty  A_\ell(s'') (P_{\geq \sigma'} \psi_s)(s'') \, \ud s'' \biggr) \,\nabla^\ell \overline{ P_{\geq \sigma'} \psi_s}(s') \, \ud s' , \\
 VI &:=\Im \int_s^\infty (\nabla^\ell\psi_\ell^\infty) \, \overline{ P_{\geq \sigma'} \psi_s}(s') \, \ud s', \\
 VII &:=\Im \int_s^\infty \psi_\ell^\infty \, \nabla^\ell\overline{ P_{\geq \sigma'} \psi_s}(s') \, \ud s' .
\end{split}
\end{align}
We will bound $I$--$V$ in $L_{t}^2L^2_x$ and $VI$ and $VII$ in $L_{t}^{\frac{8}{3}} L_{x}^{\frac{8}{3}}$. First
%%%%%%%
%%%%%%%
\begin{align*}
\begin{split}
\|I\|_{L_{t}^{2} L_x^2}\lesssim \int_s^\infty \int_{s'}^\infty (\frac{s'}{s''})^{\frac{1}{2}}\|s''\Delta (s'')^{\frac{1}{2}}\psi_s(s'')\|_{L_{t}^4 L_x^4}\dspp\|(s')^{\frac{1}{2}}\psi_s(s')\|_{L_{t}^4 L_x^4}\dsp\lesssim \|\psi_s\|_\calS^2
\end{split}
\end{align*}
by Schur's test. For $II$ we have
%%%%%%%
%%%%%%%
\begin{align*}
\begin{split}
\|II\|_{L_{t}^{2} L_x^2}\lesssim \int_s^\infty \int_{s'}^\infty \|(s'')^{\frac{1}{2}}\nabla(s'')^{\frac{1}{2}}\psi_s(s'')\|_{L_t^4 L_x^4}\dspp \|(s')^{\frac{1}{2}}\nabla(s')^{\frac{1}{2}}\psi_s(s')\|_{L_t^4 L_x^4}\dsp\lesssim \|\psi_s\|_\calS^2.
\end{split}
\end{align*}
Here we have used use Poincar\'e and extra regularity assumption to estimate
%%%%%%%
%%%%%%%
\begin{align*}
\begin{split}
\int_0^\infty \|s^{\frac{1}{2}}\nabla s^{\frac{1}{2}}\psi_s(s)\|_{L_t^4 L_x^4}\ds&\lesssim \int_0^1 s^{\delta}\|s^{\frac{1}{2}}\nabla s^{\frac{1}{2}}m(s)\psi_s(s)\|_{L_t^4 L_x^4}\ds \\
&\quad+\int_1^\infty s^{-\frac{1}{2}}\|s\Delta s^{\frac{1}{2}}\psi_s(s)\|_{L_t^4 L_x^4}\ds\\
& \lesssim \|\jap{s\Delta}s^{\frac{1}{2}}m(s)s^{\frac{1}{2}}\psi_s\|_{\Ls^\infty L_t^4 L_x^4}\lesssim \|\psi_s\|_\calS.
\end{split}
\end{align*}
For $\|III\|_{L_t^2L_x^2}$ the contribution of $A^\infty$ is bounded by
%%%%%%%
%%%%%%%
\begin{align*}
\begin{split}
\|\nabla A^\infty\|_{L_{x}^\infty}\int_s^\infty \int_{s'}^\infty \|(s'')^{\frac{1}{2}}\psi_s(s'')\|_{L_{t}^4L_x^4}(s'')^{-\frac{1}{2}}\ud s'' \|\psi_s(s')\|_{L_{t}^4L_x^4}\ud s'\lesssim \|\psi_s\|_\calS^2
\end{split}
\end{align*}
by Poincar\'e and as for $I$. Here we used $\|\nabla A^\infty\|_{L_x^\infty}\lesssim1$, which follows from~\eqref{eq:cf-coeff}. 
Next, using the estimates~\eqref{eq:bound_ringAL_Linftyall_Str}--\eqref{eq:bound_ringAQ_Linftyall_Str} and Schur's test, the contribution of $\ringbfA$ to $\|III\|_{L_{t}^2L_x^2}$ is bounded by
%%%%%%%
%%%%%%%
\begin{align*}
\begin{split}
\int_s^\infty \int_{s'}^{\infty} (\frac{s'}{s''})^{\frac{1}{2}}\|(s'')^{\frac{1}{2}}\psi_s(s'')\|_{L_{t}^4L_x^4}\dspp \|(s')^{\frac{1}{2}}\psi_s(s')\|_{L_{t}^4L_x^4}\dsp\lesssim \|\psi_s\|_\calS^2.
\end{split}
\end{align*}
For $\|IV\|_{L_t^2 L_x^2}$ we argue similarly. First, for $A^\infty$ we have
%%%%%%%
%%%%%%%
\begin{align*}
\begin{split}
\int_s^\infty \int_{s'}^\infty \|(s'')^{\frac{1}{2}}\nabla (s'')^{\frac{1}{2}}\psi_s(s'')\|_{L_{t}^4L_x^4}\dspp \|(s')^{\frac{1}{2}}\psi_s(s')\|_{L_{t}^4L_x^4}(s')^{-\frac{1}{2}}\ud s'\lesssim \|\psi_s\|_\calS^2
\end{split}
\end{align*}
by Poincar\'e and as for $I$. For $\ringbfA$, using \eqref{eq:bound_ringAL_Linftyall_Str}--\eqref{eq:bound_ringAQ_Linftyall_Str}, the corresponding contribution is bounded by
%%%%%%%
%%%%%%%
\begin{align*}
\begin{split}
\int_s^\infty \int_{s'}^\infty (\frac{s'}{s''})^{\frac{1}{2}}\|(s'')^{\frac{1}{2}}\nabla (s'')^{\frac{1}{2}}\psi_s(s'')\|_{L_{t}^4L_x^4}\dspp \|(s')^{\frac{1}{2}}\psi_s(s')\|_{L_t^4L_{x}^4}\dsp\lesssim \|\psi_s\|_\calS^2.
\end{split}
\end{align*}
For $V$ the contribution of $A^\infty$ is bounded by
%%%%%%%
%%%%%%%
\begin{align*}
\begin{split}
\int_s^\infty \int_{s'}^\infty \|(s'')^{\frac{1}{2}}\psi_s(s'')\|_{L_{t}^4L_x^4}(s'')^{-\frac{1}{2}}\ud s'' \|(s')^{\frac{1}{2}}\nabla (s')^{\frac{1}{2}}\psi_s(s')\|_{L_{t}^4L_x^4}\dsp\lesssim \|\psi_s\|_\calS^2
\end{split}
\end{align*}
by Poincar\'e and extra regularity as for $II$. Similarly the contribution of $\ringbfA$ is bounded by
%%%%%%%
%%%%%%%
\begin{align*}
\begin{split}
\int_s^\infty \int_{s'}^\infty \|(s'')^{\frac{1}{2}}\psi_s(s'')\|_{L_{t}^4L_x^4}\dspp \|(s')^{\frac{1}{2}}\nabla (s')^{\frac{1}{2}}\psi_s(s')\|_{L_{t}^4L_x^4}\lesssim \|\psi_s\|_\calS^2
\end{split}
\end{align*}
by Poincar\'e and extra regularity. Next,
%%%%%%%
%%%%%%%
\begin{align*}
\begin{split}
\|VI\|_{L_{t}^{\frac{8}{3}}L_{x}^{\frac{8}{3}}}\lesssim \int_s^\infty \|(s')^{\frac{1}{2}}\psi_s(s')\|_{L_t^{\frac{8}{3}}L_{x}^{\frac{8}{3}}}(s')^{-\frac{1}{2}}\ud s'\lesssim \|\psi_s\|_\calS
\end{split}
\end{align*}
by Poincar\'e, and where we used $\|\nabla\Psi^\infty\|_{L_x^\infty}\lesssim1$. Similarly
%%%%%%%
%%%%%%%
\begin{align*}
\begin{split}
\|VII\|_{L_t^{\frac{8}{3}}L_{x}^{\frac{8}{3}}}\lesssim \int_s^\infty \|(s')^{\frac{1}{2}}\nabla (s')^{\frac{1}{2}}\psi_s(s')\|_{L_t^{\frac{8}{3}}L_{x}^{\frac{8}{3}}}\dsp\lesssim \|\psi_s\|_\calS
\end{split}
\end{align*}
by Poincar\'e and extra regularity, and where we used $\|\Psi^\infty\|_{L_x^\infty}\lesssim1$. This completes the proof of \eqref{eq:divPsigmaAStrich}.

Next we turn to \eqref{eq:PlesssigmaAStrich} and consider the following terms separately:
%%%%%%%
%%%%%%%
\begin{align*}
\begin{split}
I &:=\Im\int_s^\infty \int_{s'}^\infty \nabla P_{\leq\sigma'}\psi_s(s'')\ud s'' \overline{P_{\geq \sigma'}\psi_s(s')}\ud s' , \\
II &:=\Im\int_s^\infty \int_{s'}^\infty \nabla P_{\geq\sigma'}\psi_s(s'')\ud s'' \overline{P_{\leq \sigma'}\psi_s(s')}\ud s' , \\
III &:=\Im\int_s^\infty \int_{s'}^\infty \nabla P_{\leq\sigma'}\psi_s(s'')\ud s'' \overline{P_{\leq \sigma'}\psi_s(s')}\ud s' , \\
IV &:=\Im \,i\int_s^\infty \int_{s'}^\infty \ringbfA(s'') P_{\leq\sigma'}\psi_s(s'')\ud s'' \overline{P_{\geq \sigma'}\psi_s(s')}\ud s' , \\
V &:=\Im \,i\int_s^\infty \int_{s'}^\infty \ringbfA(s'') P_{\geq\sigma'}\psi_s(s'')\ud s'' \overline{P_{\leq \sigma'}\psi_s(s')}\ud s' , \\
VI &:=\Im \,i\int_s^\infty \int_{s'}^\infty \ringbfA(s'') P_{\leq\sigma'}\psi_s(s'')\ud s'' \overline{P_{\leq \sigma'}\psi_s(s')}\ud s' , \\
VII &:=\Im\int_s^\infty \Psi^\infty \overline{P_{\leq\sigma'}\psi_s(s')}\ud s' .
\end{split}
\end{align*}
We bound $I$--$VI$ in $L_{t}^2L_x^2$ and $VII$ in $L_{t}^{\frac{8}{3}}L_{x}^{\frac{8}{3}}$. For $I$ first note that by writing $P_\sigma = -\sigma\Delta e^{\sigma\Delta}$ we get
%%%%%%%
%%%%%%%
\begin{align*}
\begin{split}
\|\nabla P_{\leq\sigma'}\psi_s(s'')\|_{L_{t}^4L_{x}^4}\leq \int_0^{\sigma'}\|\nabla P_{\sigma}\psi_s(s'')\|_{L_{t}^4L_{x}^4}\dsigma\lesssim (\sigma')^{\frac{1}{2}}\|\Delta\psi_s(s'')\|_{L_{t}^4L_{x}^4}.
\end{split}
\end{align*}
It follows that
%%%%%%%
%%%%%%%
\begin{align*}
\begin{split}
\|I\|_{L_{t}^2L_x^2}\lesssim (\sigma')^{\frac{1}{2}}\int_s^\infty \int_{s'}^\infty (\frac{s'}{s''})^{\frac{1}{2}}\|s''\Delta (s'')^{\frac{1}{2}}\psi_s(s'')\|_{L_{t}^4L_{x}^4}\dspp \|(s')^{\frac{1}{2}}\psi_s(s')\|_{L_{t}^4L_{x}^4}\dsp\lesssim (\sigma')^{\frac{1}{2}}\|\psi_s\|_\calS^2.
\end{split}
\end{align*}
Next, by a similar argument as above

%%%%%%%
%%%%%%%
\begin{align*}
\begin{split}
\|P_{\leq\sigma'}\psi_s(s')\|_{L_{t}^4L_{x}^4}\leq \int_0^{\sigma'}\|P_{\sigma}\psi_s(s')\|_{L_{t}^4L_{x}^4}\dsigma\lesssim (\sigma')^{\frac{1}{2}}\|\nabla\psi_s(s')\|_{L_{t}^4L_{x}^4}.
\end{split}
\end{align*}
It follows that, using Poincar\'e and extra regularity, 
%%%%%%%
%%%%%%%
\begin{align*}
\begin{split}
\|II\|_{L_{t}^2L_{x}^2}\lesssim (\sigma')^{\frac{1}{2}}\int_s^\infty \int_{s'}^\infty \|(s'')^{\frac{1}{2}}\nabla(s'')^{\frac{1}{2}}\psi_s(s'')\|_{L_{t}^4L_{x}^4}\dspp \|(s')^{\frac{1}{2}}\nabla(s')^{\frac{1}{2}}\psi_s(s')\|_{L_{t}^4L_{x}^4}\dsp\lesssim (\sigma')^{\frac{1}{2}}\|\psi_s\|_\calS^2.
\end{split}
\end{align*}
For $III$ splitting $P_\sigma$ as above we get
%%%%%%%
%%%%%%%
\begin{align*}
\begin{split}
\|\nabla P_{\leq \sigma'}\psi_s(s'')\|_{L_{t}^4L_{x}^4}\lesssim (\sigma')^{\frac{1}{4}}\|(-\Delta)^{\frac{3}{4}}\psi_s(s'')\|_{L_{t}^4L_{x}^4},\qquad \|P_{\leq\sigma'}\psi_s(s')\|_{L_{t}^4L_{x}^4}\lesssim (\sigma')^{\frac{1}{4}}\|(-\Delta)^{\frac{1}{4}}\psi_s(s')\|_{L_{t}^4L_{x}^4},
\end{split}
\end{align*}
so
%%%%%%%
%%%%%%%
\begin{align*}
\begin{split}
\|III\|_{L_{t}^2L_{x}^2}&\lesssim (\sigma')^{\frac{1}{2}}\int_s^\infty \int_{s'}^\infty(\frac{s'}{s''})^{\frac{1}{4}}\|(-s''\Delta)^{\frac{3}{4}}(s'')^{\frac{1}{2}}\psi_s(s'')\|_{L_{t}^4L_{x}^4}\dsp\|(-s'\Delta)^{\frac{1}{4}}(s')^{\frac{1}{2}}\psi_s(s')\|_{L_{t}^4L_{x}^4}\dsp\\
&\lesssim (\sigma')^{\frac{1}{2}}\|\psi_s\|_\calS^2.
\end{split}
\end{align*}
For $\|IV\|_{L_t^2L_x^2}$, $\|V\|_{L_t^2L_x^2}$, and $\|VI\|_{L_t^2L_x^2}$ note that in view of the estimates~\eqref{eq:bound_ringAL_Linftyall_Str}--\eqref{eq:bound_ringAQ_Linftyall_Str} these terms are bounded similarly to $\|I\|_{L_t^2L_x^2}$, $\|II\|_{L_t^2L_x^2}$, and $\|III\|_{L_t^2L_x^2}$ respectively. For $VII$ we again split $P_\sigma$ as above to conclude that
%%%%%%%
%%%%%%%
\begin{align*}
\begin{split}
\|P_{\leq\sigma'}\psi_s(s')\|_{L_{t}^{\frac{8}{3}}L_{x}^{\frac{8}{3}}}\lesssim (\sigma')^{\frac{1}{2}}\|\nabla\psi_s(s')\|_{L_{t}^{\frac{8}{3}}L_{x}^{\frac{8}{3}}},
\end{split}
\end{align*}
so that
%%%%%%%
%%%%%%%
\begin{align*}
\begin{split}
\|VII\|_{L_{t}^{\frac{8}{3}}L_{x}^{\frac{8}{3}}}\lesssim (\sigma')^{\frac{1}{2}}\|\Psi^\infty\|_{L_x^\infty}\int_s^\infty \|(s')^{\frac{1}{2}}\nabla (s')^{\frac{1}{2}}\psi_s(s')\|_{L_{t}^{\frac{8}{3}}L_{x}^{\frac{8}{3}}}\dsp\lesssim (\sigma')^{\frac{1}{2}}\|\psi_s\|_\calS
\end{split}
\end{align*}
using Poincar\'e and extra regularity. This completes the proof of \eqref{eq:PlesssigmaAStrich}. 

Finally we turn to \eqref{eq:ringAL83}. For $\bfAQ$ using Gagliardo-Nirenberg, \eqref{eq:bound_ringPsi_Linftyall_Str2}, extra regularity, and Poincar\'e
%%%%%%%
%%%%%%%
\begin{align*}
\begin{split}
\|\bfAQ\|_{L^\infty_\ds L^{\frac{8}{3}}_{t}  L^{\frac{8}{3}}_{x} }\leq \int_0^\infty\|s^{\frac{1}{2}}\ringPsi\|_{L_{t}^\infty L_x^\infty}\|s^{\frac{1}{2}}\psi_s\|_{L_{t}^{\frac{8}{3}} L_{x}^{\frac{8}{3}}}\ds\lesssim \|\psi_s\|_\calS^2.
\end{split}
\end{align*}
Similarly, using Poincar\'e,
%%%%%%%
%%%%%%%
\begin{align*}
\begin{split}
\|\bfAL\|_{L^\infty_\ds L_{t}^{\frac{8}{3}}L_{x}^{\frac{8}{3}}}&\leq \|\Psi^\infty\|_{L^\infty_x}\|s^{\frac{1}{2}}\psi_s\|_{L^\infty_\ds  L_{t}^{\frac{8}{3}}L_{x}^{\frac{8}{3}}}\int_0^1s^{-\frac{1}{2}}\ud s\\
&\quad+\|\Psi^\infty\|_{L^\infty_x}\|(s^{\frac{1}{2}}\nabla)^2s^{\frac{1}{2}}\psi_s\|_{L^\infty_\ds\cap L^2_\ds  L_{t}^{\frac{8}{3}}L_{x}^{\frac{8}{3}}}\int_1^\infty s^{-\frac{3}{2}}\ud s\\
& \lesssim \|\psi_s\|_\calS.
\end{split}
\end{align*}
The treatment of $\calL_\Omega \ringbfA$ is identical as we never used the radial Sobolev estimate.
\end{proof}

\section{Analysis of the Schr\"odinger equation for $\psi_{s}$} \label{s:schrod} 

The purpose of this section is to prove Proposition~\ref{prop:core-bootstrap}, that is, to establish the a priori bound~\eqref{eq:core-bootstrap-imp} on the dispersive norm $\calS$ of the heat tension field $\psi_s$, as well as the nonlinearity bound~\eqref{eq:core-bootstrap-imp-N}. Recall from~\eqref{eq:SH-psi_s} that the heat tension field $\psi_s$ satisfies the Schr\"odinger equation  
\begin{equation} \label{equ:schroed_psi_s}
 \begin{aligned}
  (i \partial_t - H) \psi_s &= - 2 i \ringA^k \nabla_k \psi_s + i \partial_s w  + 2 A^{\infty, k} \ringA_k \psi_s - i \Im ( \psi^{\infty,k} \overline{\psi_s} ) \ringpsi_k - i \Im ( \ringpsi^k \overline{\psi_s} ) \psi^\infty_k \\
  &\quad \quad - i (\nabla_k \ringA^k) \psi_s + \ringA^k \ringA_k \psi_s - i \Im ( \ringpsi^k \overline{\psi_s} ) \ringpsi_k  + A_t \psi_s.
 \end{aligned}
\end{equation}
Since $A^\infty$ and $|\psi_2^\infty|^2$ are both independent of $\theta$, the angular vector field $\Omega$ commutes with $H$. Correspondingly, applying $\Omega$ to~\eqref{equ:schroed_psi_s} yields the following equation for $\Omega \psi_s$
\begin{equation} \label{equ:schroed_Omega_psi_s}
 \begin{aligned}
  (i\partial_t-H)\Omega\psi_s &= - 2 i \calL_\Omega\ringA^k \nabla_k \psi_s - 2 i \ringA^k \nabla_k\Omega \psi_s + i \partial_s\Omega w +  2 A^{\infty, k} \calL_\Omega\ringA_k \psi_s+2 A^{\infty, k} \ringA_k \Omega\psi_s\\
  &\quad - i \Im ( \psi^{\infty,k} \overline{\Omega\psi_s} ) \ringpsi_k- i \Im ( \psi^{\infty,k} \overline{\psi_s} ) \calL_\Omega\ringpsi_k - i \Im ( \ringpsi^k \overline{\Omega\psi_s} ) \psi^\infty_k- i \Im ( \calL_\Omega\ringpsi^k \overline{\psi_s} ) \psi^\infty_k  \\
  &\quad  - i (\nabla_k \calL_\Omega\ringA^k) \psi_s - i (\nabla_k \ringA^k) \Omega\psi_s+ 2\calL_\Omega\ringA^k \ringA_k \psi_s+ \ringA^k \ringA_k \Omega\psi_s\\
  &\quad - i \Im ( \calL_\Omega\ringpsi^k \overline{\psi_s} ) \ringpsi_k- i \Im ( \ringpsi^k \overline{\Omega\psi_s} ) \ringpsi_k - i \Im ( \ringpsi^k \overline{\psi_s} ) \calL_\Omega\ringpsi_k + \calL_\Omega A_t \psi_s+ A_t \Omega\psi_s \\
  &\quad - i \Im ( i\psi^{\infty,k} \overline{\psi_s} ) \ringpsi_k + \Im ( \ringpsi^k \overline{\psi_s} ) \psi^\infty_k.
 \end{aligned}
\end{equation}
We denote by $\calN := (i \partial_t + H) \psi_s$ the right-hand side of~\eqref{equ:schroed_psi_s} and by $\calN^\Omega := (i \partial_t + H) \Omega \psi_s$ the right-hand side of~\eqref{equ:schroed_Omega_psi_s}, and group them into the following four types:
\begin{itemize}
 \item Magnetic interaction terms:
  \begin{enumerate}[(i)]
   \item $\ringA^k \nabla_k\jap{\Omega}\psi_s$
   \item $\calL_\Omega \ringA^k \nabla_k\psi_s$   
  \end{enumerate}
 
 \item $w$ term:
 \begin{enumerate}[(i)]
  \item $\jap{\Omega} \partial_s w$
 \end{enumerate}
 
 \item Quadratic power-type terms:
 \begin{enumerate}[(i)]
  \item $\Im(\psi^\infty_k\overline{\jap{\Omega}\psi_s})\ringpsi^k$ and permutations
  \item $\Im(\psi^\infty_k\overline{\psi_s})\calL_\Omega\ringpsi^k$ and permutations 
  \item $A^\infty_k\ringA^k\jap{\Omega}\psi_s$ 
  \item $A^\infty_k\calL_\Omega\ringA^k \psi_s$
  \item $(\nabla_k\AL^k)\jap{\Omega}\psi_s$  
  \item $(\nabla_k\calL_\Omega\AL^k)\psi_s$
 \end{enumerate}
 
 \item Cubic and higher power-type terms:
 \begin{enumerate}[(i)]
 
  \item $\Im(\ringpsi^k\overline{\jap{\Omega}\psi_s})\ringpsi_k$
  \item $\Im(\calL_\Omega\ringpsi^k\overline{\psi_s})\ringpsi_k$ and permutations  
  \item $\ringA_k\ringA^k\jap{\Omega}\psi_s$  
  \item $\calL_\Omega \ringA_k\ringA^k \psi_s$ 
  \item $A_t\jap{\Omega}\psi_s$
  \item $\calL_\Omega A_t\psi_s=\Omega A_t\psi_s$  
  \item $(\nabla_k\AQ^k)\jap{\Omega}\psi_s$  
  \item $(\nabla_k\calL_\Omega\AQ^k)\psi_s$
 \end{enumerate}

\end{itemize}

In what follows, in view of~\eqref{eq:core-bootstrap}, {\bf we may assume throughout that $\|\psi_s\|_{\calS(I)} < 1$}. It allows us to bound, on occasion, higher powers $\|\psi_s\|_{\calS(I)}^k$ of the dispersive norm of $\psi_s$ just by $\|\psi_s\|_{\calS(I)}$ for any integer $k \geq 1$, which simplifies the exposition. The proof of Proposition~\ref{prop:core-bootstrap} reduces to the estimates in the following four lemmas. The first two of these concern the magnetic interaction term: 

%%%%%%%%%%%%%%%
%%%%%%%%%%%%%%%
\begin{lemma}[Nonlinear estimates: magnetic interaction term (i)]\label{lem:magnetic1}
 Assuming \eqref{eq:overall-bootstrap} and \eqref{eq:core-bootstrap}, we have that 
\begin{align}
  \| \ringA^k \nabla_k \jap{\Omega} ( m(s) s^{\frac{1}{2}} \psi_s ) \|_{L^\infty_{\ds} \cap L^2_{\ds}(LE^\ast + L^{\frac{4}{3}}_t L^{\frac{4}{3}}_x)} \lesssim \|\psi_s\|_{\calS}^2.
 \end{align}
\end{lemma}
%%%%%%%%%%%%%%%
%%%%%%%%%%%%%%%
\begin{lemma}[Nonlinear estimates: magnetic interaction term (ii)]\label{lem:magnetic2}
 Assuming \eqref{eq:overall-bootstrap} and \eqref{eq:core-bootstrap}, we have that
\begin{align}
  \| \calL_\Omega \ringA^k \nabla_k ( m(s) s^{\frac{1}{2}} \psi_s ) \|_{L^\infty_{\ds} \cap L^2_{\ds}(LE^\ast + L^{\frac{4}{3}}_t L^{\frac{4}{3}}_x)} \lesssim \|\psi_s\|_{\calS}^2.
 \end{align}
\end{lemma}
%%%%%%%%%%%%%%%
%%%%%%%%%%%%%%%
The next two lemmas provide estimates on the remaining terms in the nonlinearity, except for $\partial_sw$ which was already treated in Lemma~\ref{lem:bounds_on_w}. For the cubic and higher order terms we have:
%%%%%%%%%%%%%%%
%%%%%%%%%%%%%%%
\begin{lemma}(Nonlinear estimates: cubic and higher terms)\label{lem:cubic_and_higher}
 Assuming \eqref{eq:overall-bootstrap} and \eqref{eq:core-bootstrap}, we have that 
 \begin{align}
  \| | \ringPsi|^2 \jap{\Omega} (m(s) s^{\frac{1}{2}} \psi_s) \|_{L^\infty_{\ds} \cap L^2_{\ds}L^{\frac{4}{3}}_t L^{\frac{4}{3}}_x} &\lesssim \|\psi_s\|_{\calS}^2 \label{equ:nonlinear_estimates_cubic1} \\
  \| |\calL_\Omega \ringPsi| |\ringPsi| (m(s) s^{\frac{1}{2}} \psi_s) \|_{L^\infty_{\ds} \cap L^2_{\ds}L^{\frac{4}{3}}_t L^{\frac{4}{3}}_x} &\lesssim \|\psi_s\|_{\calS}^2 \label{equ:nonlinear_estimates_cubic2} \\  
  \| \ringA_k \ringA^k \jap{\Omega} (m(s) s^{\frac{1}{2}} \psi_s) \|_{L^\infty_{\ds} \cap L^2_{\ds}L^{\frac{4}{3}}_t L^{\frac{4}{3}}_x} &\lesssim \|\psi_s\|_{\calS}^2 \label{equ:nonlinear_estimates_cubic3} \\       
  \| \calL_\Omega \ringA_k \ringA^k (m(s) s^{\frac{1}{2}} \psi_s) \|_{L^\infty_{\ds} \cap L^2_{\ds}L^{\frac{4}{3}}_t L^{\frac{4}{3}}_x} &\lesssim \|\psi_s\|_{\calS}^2 \label{equ:nonlinear_estimates_cubic4} \\       
  \| A_t \jap{\Omega} (m(s) s^{\frac{1}{2}} \psi_s) \|_{L^\infty_{\ds} \cap L^2_{\ds}L^{\frac{4}{3}}_t L^{\frac{4}{3}}_x} &\lesssim \|\psi_s\|_{\calS}^2 \label{equ:nonlinear_estimates_cubic5} \\   
  \| \calL_\Omega A_t (m(s) s^{\frac{1}{2}} \psi_s) \|_{L^\infty_{\ds} \cap L^2_{\ds}L^{\frac{4}{3}}_t L^{\frac{4}{3}}_x} &\lesssim \|\psi_s\|_{\calS}^2 \label{equ:nonlinear_estimates_cubic6} \\     
  \| (\nabla_k \AQ^k) \jap{\Omega} (m(s) s^{\frac{1}{2}} \psi_s) \|_{L^\infty_{\ds} \cap L^2_{\ds}L^{\frac{4}{3}}_t L^{\frac{4}{3}}_x} &\lesssim \|\psi_s\|_{\calS}^2 \label{equ:nonlinear_estimates_cubic7} \\     
  \| (\nabla_k \calL_\Omega \AQ^k) (m(s) s^{\frac{1}{2}} \psi_s) \|_{L^\infty_{\ds} \cap L^2_{\ds}L^{\frac{4}{3}}_t L^{\frac{4}{3}}_x} &\lesssim \|\psi_s\|_{\calS}^2 \label{equ:nonlinear_estimates_cubic8} 
 \end{align}
\end{lemma}
%%%%%%%%%%%%%%%
%%%%%%%%%%%%%%%
Finally, the quadratic terms are treated in the following lemma.
%%%%%%%%%%%%%%%
%%%%%%%%%%%%%%%
\begin{lemma}(Nonlinear estimates: quadratic terms)\label{lem:quadratic_terms}
 Assuming \eqref{eq:overall-bootstrap} and \eqref{eq:core-bootstrap}, we have that
 \begin{align}
  \| |\Psi^\infty| |\ringPsi| \jap{\Omega} (m(s) s^{\frac{1}{2}} \psi_s) \|_{L^\infty_{\ds} \cap L^2_{\ds} L^{\frac{4}{3}}_t L^{\frac{4}{3}}_x} &\lesssim \|\psi_s\|_{\calS}^2 \label{equ:nonlinear_estimates_quadratic1} \\
  \| |\Psi^\infty| |\calL_\Omega \ringPsi| (m(s) s^{\frac{1}{2}} \psi_s) \|_{L^\infty_{\ds} \cap L^2_{\ds} L^{\frac{4}{3}}_t L^{\frac{4}{3}}_x} &\lesssim \|\psi_s\|_{\calS}^2 \label{equ:nonlinear_estimates_quadratic2} \\     
  \| A^\infty_k \ringA^k\jap {\Omega} (m(s) s^{\frac{1}{2}} \psi_s) \|_{L^\infty_{\ds} \cap L^2_{\ds}L^{\frac{4}{3}}_t L^{\frac{4}{3}}_x} &\lesssim \|\psi_s\|_{\calS}^2 \label{equ:nonlinear_estimates_quadratic3} \\
  \| A^\infty_k \calL_\Omega \ringA^k (m(s) s^{\frac{1}{2}} \psi_s) \|_{L^\infty_{\ds} \cap L^2_{\ds}L^{\frac{4}{3}}_t L^{\frac{4}{3}}_x} &\lesssim \|\psi_s\|_{\calS}^2 \label{equ:nonlinear_estimates_quadratic4} \\  
  \| (\nabla_k\AL^k) \jap{\Omega} (m(s) s^{\frac{1}{2}} \psi_s) \|_{L^\infty_{\ds} \cap L^2_{\ds}L^{\frac{4}{3}}_t L^{\frac{4}{3}}_x} &\lesssim \|\psi_s\|_{\calS}^2 \label{equ:nonlinear_estimates_quadratic5} \\
  \| (\nabla_k\calL_\Omega\AL^k) (m(s) s^{\frac{1}{2}} \psi_s) \|_{L^\infty_{\ds} \cap L^2_{\ds}L^{\frac{4}{3}}_t L^{\frac{4}{3}}_x} &\lesssim \|\psi_s\|_{\calS}^2 \label{equ:nonlinear_estimates_quadratic6}
 \end{align}
\end{lemma}
%%%%%%%%%%%%%%%
%%%%%%%%%%%%%%%
Lemmas~\ref{lem:magnetic1},~\ref{lem:magnetic2},~\ref{lem:cubic_and_higher}, and~\ref{lem:quadratic_terms} are proved in Subsections~\ref{subsec:7magnetic} and~\ref{subsec:7non-magnetic} below. Here we simply observe how Proposition~\ref{prop:core-bootstrap} follows easily from these lemmas along with Lemma~\ref{lem:bounds_on_w}.

%%%%%%%%%%%%%%%
%%%%%%%%%%%%%%%
\begin{proof}[Proof of Proposition~\ref{prop:core-bootstrap}]
Under the bootstrap assumptions~\eqref{eq:overall-bootstrap} and \eqref{eq:core-bootstrap}, the bound \eqref{eq:core-bootstrap-imp} follows by a standard continuity argument from the estimate
\begin{equation} \label{equ:nonlinear_estimates_put_together}
\begin{aligned}
 \|\psi_s\|_{\calS(I)} & \lesssim \| \langle \Omega \rangle (m(s) s^{\frac{1}{2}} \psi_s) |_{t=0} \|_{L^\infty_{\ds} \cap L^2_{\ds} L^2_x} + \|\psi_s\|_{\calS(I)}^2 + \sqrt{ \epsilon_0} \|\psi_s\|_{\calS(I)}.
\end{aligned}\end{equation}
Multiplying \eqref{equ:schroed_psi_s}, respectively \eqref{equ:schroed_Omega_psi_s}, by $m(s) s^{\frac{1}{2}}$, invoking the main linear estimate from Lemma~\ref{lem:main_linear_estimate}, and summing up the two resulting estimates, we deduce that for all $s > 0$,
\begin{equation*}
 \|\psi_s(s)\|_{\calS_s(I)} \lesssim \| \langle \Omega \rangle (m(s) s^{\frac{1}{2}} \psi_s) |_{t=0} \|_{L^2_x} + \| m(s) s^{\frac{1}{2}} \calN \|_{LE^\ast(I) + L^{\frac{4}{3}}_t L^{\frac{4}{3}}_x(I)} + \| m(s) s^{\frac{1}{2}} \calN^\Omega \|_{LE^\ast(I) + L^{\frac{4}{3}}_t L^{\frac{4}{3}}_x(I)}.
\end{equation*}
Taking the $L^\infty_{\ds} \cap L^2_{\ds}(\bbR^+; \calS_s(I))$ norm of both sides, we find that
\begin{equation*}
 \begin{aligned}
  \|\psi_s\|_{\calS(I)} &\lesssim \| \langle \Omega \rangle (m(s) s^{\frac{1}{2}} \psi_s) |_{t=0} \|_{L^\infty_{\ds} \cap L^2_{\ds} L^2_x} \\
  &\qquad + \| m(s) s^{\frac{1}{2}} \calN \|_{L^\infty_{\ds} \cap L^2_{\ds}(LE^\ast + L^{\frac{4}{3}}_t L^{\frac{4}{3}}_x)(I)} + \| m(s) s^{\frac{1}{2}} \calN^\Omega \|_{L^\infty_{\ds} \cap L^2_{\ds}(LE^\ast + L^{\frac{4}{3}}_t L^{\frac{4}{3}}_x)(I)}.
 \end{aligned}
\end{equation*}
Thus, the proof of~\eqref{equ:nonlinear_estimates_put_together} reduces to establishing, under the assumptions~\eqref{eq:overall-bootstrap} and~\eqref{eq:core-bootstrap}, that
\begin{equation} \label{equ:nonlinear_estimates}
 \| m(s) s^{\frac{1}{2}} \calN \|_{L^\infty_{\ds} \cap L^2_{\ds}(LE^\ast + L^{\frac{4}{3}}_t L^{\frac{4}{3}}_x)(I)} + \| m(s) s^{\frac{1}{2}} \calN^\Omega \|_{L^\infty_{\ds} \cap L^2_{\ds}(LE^\ast + L^{\frac{4}{3}}_t L^{\frac{4}{3}}_x)(I)} \lesssim \|\psi_s\|_{\calS(I)}^2 + \sqrt{\epsilon_0} \|\psi_s\|_{\calS(I)}.
\end{equation}
This is accomplished in Lemmas~\ref{lem:bounds_on_w},~\ref{lem:magnetic1},~\ref{lem:magnetic2},~\ref{lem:cubic_and_higher}, and~\ref{lem:quadratic_terms}. Finally, \eqref{eq:core-bootstrap-imp-N} follows from \eqref{equ:nonlinear_estimates} and \eqref{eq:core-bootstrap-imp}.
\end{proof}
%%%%%%%%%%%%%%%
%%%%%%%%%%%%%%%

\begin{proof}[Proof of Proposition~\ref{prop:prop-reg}]
Once we have established the estimate 
\EQ{ \label{eq:psi_s-S-norm} 
\| \psi_s \|_{\calS(I)} \lesssim \eps_0
}
by proving Proposition~\ref{prop:core-bootstrap}, the proof of Proposition~\ref{prop:prop-reg} follows from an identical argument after applying $H$ and $H^2$ to both sides of the equation~\eqref{equ:schroed_psi_s} and using the bound~\eqref{eq:psi_s-S-norm} along with Remark~\ref{rem:forward-high-reg}, where the latter translates the higher regularity of the initial data to the $s$-evolution of $\psi_s\rest_{t=0}$. Crucially, $H$ commutes with the left-hand side of~\eqref{equ:schroed_psi_s}. This is standard and we omit the details. 
\end{proof}

\subsection{Estimates for the magnetic interaction terms}\label{subsec:7magnetic}

In this section we treat the magnetic interaction terms. Specifically, our goal is to prove Lemmas~\ref{lem:magnetic1} and~\ref{lem:magnetic2}.

In the proofs of these lemmas we will use the notation $P_{\geq \sigma} \ringA$ and $P_{\geq \sigma} \calL_\Omega \ringA$ introduced in~\eqref{eq:lowA} and~\eqref{eq:OmegalowA}. For the convenience of the reader we recall that
\begin{align*}
\begin{split}
P_{\geq\sigma}\ringbfA&:=\Im \int_s^\infty \Big(\int_{s'}^\infty\nabla P_{\geq_\sigma}\psi_s(s'')\ud s''\Big)\overline{P_{\geq\sigma}\psi_s(s')} \, \ud s'\\
&\quad+\Im \int_s^\infty \Big(\int_{s'}^\infty \bfA(s'') P_{\geq_\sigma}\psi_s(s'')\ud s''\Big)\overline{P_{\geq\sigma}\psi_s(s')}\, \ud s'\\
&\quad-\Im\int_s^\infty \Psi^\infty\overline{P_{\geq \sigma}\psi_s(s')}\, \ud s',
\end{split}
\end{align*}
and
%%%%%%%
%%%%%%%
\begin{align*}
\begin{split}
P_{\leq\sigma}\ringbfA:=\ringbfA-P_{\geq\sigma}\ringbfA.
\end{split}
\end{align*}
\begin{proof}[Proof of Lemma~\ref{lem:magnetic1}]
Using a linear heat flow Littlewood-Paley decomposition and moving around the derivative, we decompose the magnetic interaction term on the right-hand side into
\begin{align}
 \ringA^k \nabla_k (\jap{\Omega} m(s) s^{\frac{1}{2}} \psi_s ) &= \nabla_k \int_0^1 ( P_{\geq \sigma'} \ringA )^k P_{\sigma'} ( \jap{\Omega}m(s) s^{\frac{1}{2}} \psi_s ) \, \frac{\ud \sigma'}{\sigma'} - \int_0^1 \nabla_k ( P_{\geq \sigma'} \ringA)^k P_{\sigma'} (\jap{\Omega} m(s) s^{\frac{1}{2}} \psi_s ) \, \frac{\ud \sigma'}{\sigma'}\nonumber \\
 &\quad \quad + \int_0^1 (P_{\leq \sigma'} \ringA)^k \nabla_k P_{\sigma'} (\jap{\Omega} m(s) s^{\frac{1}{2}} \psi_s ) \, \frac{\ud \sigma'}{\sigma'} + \ringA^k \nabla_k P_{\geq 1} (\jap{\Omega} m(s) s^{\frac{1}{2}} \psi_s ) \nonumber\\
 &=: I + II + III + IV. \label{eq:maglemma1maindecomp}
\end{align}
We will estimate the  term $I$ in the dual local smoothing space $LE^\ast$, and all other terms in the dual Strichartz space $L^{\frac{4}{3}}_t L^{\frac{4}{3}}_x$. These estimates are carried out separately for each term in the proceeding steps. For simplicity of notation we will write $m(s) s^{\frac{1}{2}} \psi_s$ instead of $\jap{\Omega} m(s) s^{\frac{1}{2}} \psi_s$, but since we never apply angular derivatives to this term, the same exact argument works for $\jap{\Omega} m(s) s^{\frac{1}{2}} \psi_s$.
%%%%%%%%%%%%%
\subsection*{Step 1: Term $I$}
%%%%%%%%%%%%%
Here our goal is to prove the estimate 
\[
 \| I \|_{LE^\ast} \lesssim \| \psi_s \|_{\calS}^2,
\]
which boils down to establishing the following two estimates
\begin{align}
 \biggl\| \sigma^{\frac{1}{4}} \Bigl\| P_\sigma \nabla_k \int_0^1 ( P_{\geq \sigma'} \ringA^k ) P_{\sigma'} ( m(s) s^{\frac{1}{2}} \psi_s ) \, \frac{\ud \sigma'}{\sigma'}  \Bigr\|_{LE_\sigma^\ast} \biggr\|_{L^\infty_{\frac{\ud s}{s}} \cap L^2_{\frac{\ud s}{s}} L^2_{\frac{\ud \sigma}{\sigma}} (0, \frac{1}{2})} &\lesssim \| \psi_s \|_{\calS}^2, \label{equ:high_freq_termI} \\
 \biggl\| \Bigl\| P_{\geq \sigma} \nabla_k \int_0^1 ( P_{\geq \sigma'} \ringA^k ) P_{\sigma'} ( m(s) s^{\frac{1}{2}} \psi_s ) \, \frac{\ud \sigma'}{\sigma'}  \Bigr\|_{LE_\low^\ast} \biggr\|_{L^\infty_{\frac{\ud s}{s}} \cap L^2_{\frac{\ud s}{s}} L^2_{\frac{\ud \sigma}{\sigma}} (\frac{1}{8}, 4)} &\lesssim \| \psi_s \|_{\calS}^2. \label{equ:low_freq_termI}
\end{align}

\subsubsection*{Step 1a: The high-frequency estimate~\eqref{equ:high_freq_termI}}

We begin with the much more difficult high-frequency estimate~\eqref{equ:high_freq_termI}. To this end we further decompose into 
\begin{align}\label{eq:IHF12}
\begin{split}
 \sigma^{\frac{1}{4}} P_\sigma \nabla_k \int_0^1 ( P_{\geq \sigma'} \ringA^k ) P_{\sigma'} ( m(s) s^{\frac{1}{2}} \psi_s ) \, \frac{\ud \sigma'}{\sigma'} &= \sigma^{\frac{1}{4}} P_\sigma \nabla_k \int_0^\sigma ( P_{\geq \sigma'} \ringA^k ) P_{\sigma'} ( m(s) s^{\frac{1}{2}} \psi_s ) \, \frac{\ud \sigma'}{\sigma'} \\
 &\quad + \sigma^{\frac{1}{4}} P_\sigma \nabla_k \int_\sigma^1 ( P_{\geq \sigma'} \ringA^k ) P_{\sigma'} ( m(s) s^{\frac{1}{2}} \psi_s ) \, \frac{\ud \sigma'}{\sigma'} \\
 &=: I_1 + I_2.
 \end{split}
\end{align}
Starting with $I_2$, we would like to show that 
\begin{equation} \label{equ:IHF2_bound}
 \Bigl\| \bigl\| I_{2} \bigr\|_{LE_\sigma^\ast} \Bigr\|_{L^\infty_{\frac{\ud s}{s}} \cap L^2_{\frac{\ud s}{s}} L^2_{\frac{\ud \sigma}{\sigma}} (0, \frac{1}{2})} \lesssim \| \psi_s \|_{\calS}^2. 
\end{equation}
To this end it is more favorable to let the derivative fall back inside the integral and consider 
\begin{align}\label{eq:IHF2decomp}
  I_{2} = \sigma^{\frac{1}{4}} P_\sigma \int_\sigma^1 \nabla_k ( P_{\geq \sigma'} \ringA^k ) P_{\sigma'} ( m(s) s^{\frac{1}{2}} \psi_s ) \, \frac{\ud \sigma'}{\sigma'} + \sigma^{\frac{1}{4}} P_\sigma \int_\sigma^1 ( P_{\geq \sigma'} \ringA^k ) \nabla_k P_{\sigma'} ( m(s) s^{\frac{1}{2}} \psi_s ) \, \frac{\ud \sigma'}{\sigma'}.
\end{align}
For the second term on the right-hand side above we drop the projection $P_\sigma$ in $LE_{\sigma}^\ast$, use that $\|F\|_{LE_\sigma^\ast} \lesssim \|F\|_{LE_{\sigma'}^\ast}$ since $\sigma' \geq \sigma$, apply H\"older's inequality, and drop $\nabla$ in $LE_{\sigma'}$, to bound this term as
\begin{align*}
 &\Big\| \sigma^{\frac{1}{4}} P_\sigma \int_\sigma^1 ( P_{\geq \sigma'} \ringA^k ) \nabla_k P_{\sigma'} ( m(s) s^{\frac{1}{2}} \psi_s ) \, \frac{\ud \sigma'}{\sigma'} \Big\|_{LE_\sigma^\ast} \\
 &\lesssim \int_\sigma^1 \sigma^{\frac{1}{4}} \bigl\| ( P_{\geq \sigma'} \ringA ^k) \nabla_k P_{\sigma'} ( m(s) s^{\frac{1}{2}} \psi_s ) \bigr\|_{LE_{\sigma'}^\ast} \, \frac{\ud \sigma'}{\sigma'} \\
 &\lesssim \int_\sigma^1 \Bigl( \frac{\sigma}{\sigma'} \Bigr)^{\frac{1}{4}} \bigl( M_1 + M_2 + M_3 \bigr) (\sigma')^{-\frac{1}{4}} \| P_{ \frac{\sigma'}{2} } ( m(s) s^{\frac{1}{2}} \psi_s ) \bigr\|_{LE_{\sigma'}} \, \frac{\ud \sigma'}{\sigma'},
\end{align*}
where
\begin{align*}
 M_1 &:= (\sigma')^{\frac{1}{2}} \| P_{\geq \sigma'} \ringbfA(s) \|_{L^\infty_t L^\infty_x (A_{\leq -k_{\sigma'}})} \\
 M_2 &:= \sum_{-k_{\sigma'} \leq \ell < 0} 2^\ell \| P_{\geq \sigma'} \ringbfA(s) \|_{L^\infty_t L^\infty_x(A_\ell)} \\
 M_3 &:= \| r^4 P_{\geq \sigma'} \ringbfA(s) \|_{L^\infty_t L^\infty_x(A_{\geq 0})}.
\end{align*}
By Schur's test the desired bound~\eqref{equ:IHF2_bound} follows if we can establish that 
\[
 M_j \lesssim \|\psi_s\|_{\calS} \quad \text{ for } j = 1, 2, 3,
\]
but these estimates are proved in \eqref{eq:M1}--\eqref{eq:M3} in Lemma~\ref{lem:Mestimates}. The first term on the right-hand side of \eqref{eq:IHF2decomp} can be treated similarly where we instead use the estimates \eqref{eq:tilM1}--\eqref{eq:tilM3} in Lemma~\ref{lem:Mestimates}. This finishes the treatment of the term $I_{2}$ in \eqref{eq:IHF12}. 

Next we turn to the term $I_1$ in \eqref{eq:IHF12} where we would like to prove that
\begin{equation} \label{equ:IHF1_bound}
 \Bigl\| \bigl\| I_{1} \bigr\|_{LE_\sigma^\ast} \Bigr\|_{L^\infty_{\frac{\ud s}{s}} \cap L^2_{\frac{\ud s}{s}} L^2_{\frac{\ud \sigma}{\sigma}} (0, \frac{1}{2})} \lesssim \| \psi_s \|_{\calS}^2. 
\end{equation}
Using Lemma~\ref{l:bern1}, for any fixed $p\in(1,2)$, and noting that $-k_\sigma \geq -k_{\sigma'}$ for $\sigma' \leq \sigma$, we obtain that 
\begin{equation} \label{equ:bound_IHF1}
 \begin{aligned}
  \| I_{1} \|_{LE_\sigma^\ast} &\lesssim \int_0^\sigma \sigma^{-\alpha} \bigl\| ( P_{\geq \sigma'} \ringbfA ) P_{\sigma'} ( m(s) s^{\frac{1}{2}} \psi_s ) \bigr\|_{L^2_t L^p_x(A_{\leq -k_{\sigma'}})} \, \frac{\ud \sigma'}{\sigma'} \\
  &\quad + \int_0^\sigma \sigma^{-\alpha} \sum_{-k_{\sigma'} \leq \ell < -k_\sigma} \bigl\| ( P_{\geq \sigma'} \ringbfA ) P_{\sigma'} ( m(s) s^{\frac{1}{2}} \psi_s ) \bigr\|_{L^2_t L^p_x(A_{\ell})} \, \frac{\ud \sigma'}{\sigma'} \\
  &\quad + \int_0^\sigma \Bigl( \frac{\sigma'}{\sigma} \Bigr)^{\frac{1}{4}} (L_2 + L_3) (\sigma')^{-\frac{1}{4}} \| P_{\sigma'} ( m(s) s^{\frac{1}{2}} \psi_s ) \|_{LE_{\sigma'}} \, \frac{\ud \sigma'}{\sigma'},
 \end{aligned}
\end{equation}
where 
\begin{align*}
 L_2 &:= \sum_{-k_\sigma \leq \ell < 0} 2^\ell \| P_{\geq \sigma'} \ringbfA(s) \|_{L^\infty_t L^\infty_x(A_\ell)} \\
 L_3 &:= \| r^4 P_{\geq \sigma'} \ringbfA(s) \|_{L^\infty_t L^\infty_x(A_{\geq 0})}
\end{align*}
By Schur's test the contributions of $L_2$ and $L_3$ to~\eqref{equ:IHF1_bound} are under control if we can show that 
\begin{equation*}
 L_j \lesssim \|\psi_s\|_{\calS}^2 \quad \text{ for } j = 2, 3.
\end{equation*}
But note that $L_3$ is exactly the same as $M_3$ above, while the condition $\sigma'\leq\sigma$ implies $L_2\leq M_2$ so these estimates also follow from  \eqref{eq:M2},\eqref{eq:M3} in Lemma~\ref{lem:Mestimates}. It remains to estimate the first two terms on the right-hand side of~\eqref{equ:bound_IHF1}. For the first term on the right-hand side, by \eqref{eq:M4} in Lemma~\ref{lem:Mestimates}, and with $\alpha = \frac{2-p}{2p}$,
\[
 \| P_{\geq \sigma'} \ringbfA(s) \|_{L^\infty_t L^{\frac{2p}{2-p}}_x(\bbH^2)} \lesssim (\sigma')^{-\frac{1}{2}+\alpha} \| \psi_s \|_{\calS}.
\]
It follows that
\begin{align*}
 &\int_0^\sigma \sigma^{-\alpha} \bigl\| ( P_{\geq \sigma'} \ringbfA ) P_{\sigma'} ( m(s) s^{\frac{1}{2}} \psi_s ) \bigr\|_{L^2_t L^p_x(A_{\leq -k_{\sigma'}})} \, \frac{\ud \sigma'}{\sigma'} \\
 &\lesssim \int_0^\sigma \sigma^{-\alpha} \| P_{\geq \sigma'} \ringbfA \|_{L^\infty_t L^{\frac{2p}{2-p}}_x(\bbH^2)} \| P_{\sigma'} ( m(s) s^{\frac{1}{2}} \psi_s ) \|_{L^2_t L^2_x(A_{\leq -k_{\sigma'}})} \, \frac{\ud \sigma'}{\sigma'} \\
 &\lesssim \|\psi_s\|_{\calS} \int_0^\sigma \Bigl( \frac{\sigma'}{\sigma} \Bigr)^\alpha (\sigma')^{-\frac{1}{2}} \| P_{\sigma'} ( m(s) s^{\frac{1}{2}} \psi_s ) \|_{L^2_t L^2_x(A_{\leq -k_{\sigma'}})} \, \frac{\ud \sigma'}{\sigma'}.
\end{align*}
By Schur's test, the last line is bounded in $L^2_\dsigma(0,\frac{1}{2})$ by
\[
 \|\psi_s\|_{\calS}^2 \, \| m(s) s^{\frac{1}{2}} \psi_s(s) \|_{LE},
\]
which is of the desired form. For the second term on the right-hand side of~\eqref{equ:bound_IHF1}, by \eqref{eq:M5} in Lemma~\ref{lem:Mestimates}
\[
 \| r^{\frac{1}{2}} P_{\geq \sigma'} \ringbfA \|_{L^\infty_t L^{\frac{2p}{2-p}}_x(A_\ell)} \lesssim \| \sinh^{\frac{1}{2}}(r) P_{\geq \sigma'} \ringA \|_{L^\infty_t L^{\frac{2p}{2-p}}_x(\bbH^2)} \lesssim (\sigma')^{-\frac{1}{4}+\alpha} \|\psi_s\|_{\calS}^2.
\]
It follows that
\begin{align*}
 &\int_0^\sigma \sigma^{-\alpha} \sum_{-k_{\sigma'} \leq \ell < k_\sigma} \bigl\| ( P_{\geq \sigma'} \ringbfA ) P_{\sigma'} ( m(s) s^{\frac{1}{2}} \psi_s ) \bigr\|_{L^2_t L^p_x(A_{\ell})} \, \frac{\ud \sigma'}{\sigma'} \\    
 &\lesssim \int_0^\sigma \sigma^{-\alpha} \sum_{-k_{\sigma'} \leq \ell < k_\sigma} \bigl\| r^{\frac{1}{2}} P_{\geq \sigma'} \ringbfA \bigr\|_{L^\infty_t L^{\frac{2p}{2-p}}_x(\bbH^2)} \| r^{-\frac{1}{2}} P_{\sigma'} ( m(s) s^{\frac{1}{2}} \psi_s ) \bigr\|_{L^2_t L^2_x(A_{\ell})} \, \frac{\ud \sigma'}{\sigma'} \\    
 &\lesssim \|\psi_s\|_{\calS} \int_0^\sigma \sum_{-k_{\sigma'} \leq \ell < k_\sigma} \Bigl( \frac{\sigma'}{\sigma} \Bigr)^\alpha (\sigma')^{-\frac{1}{4}}  \| r^{-\frac{1}{2}} P_{\sigma'} ( m(s) s^{\frac{1}{2}} \psi_s ) \bigr\|_{L^2_t L^2_x(A_{\ell})} \, \frac{\ud \sigma'}{\sigma'} \\    
 &\lesssim \|\psi_s\|_{\calS} \int_0^\sigma \Bigl( \frac{\sigma'}{\sigma} \Bigr)^{-\varepsilon} \Bigl( \frac{\sigma'}{\sigma} \Bigr)^\alpha (\sigma')^{-\frac{1}{4}} \sup_{-k_{\sigma'} \leq \ell < 0}  \| r^{-\frac{1}{2}} P_{\sigma'} ( m(s) s^{\frac{1}{2}} \psi_s ) \bigr\|_{L^2_t L^2_x(A_{\ell})} \, \frac{\ud \sigma'}{\sigma'} \\
 &\lesssim \|\psi_s\|_{\calS} \int_0^\sigma \Bigl( \frac{\sigma'}{\sigma} \Bigr)^{\alpha - \varepsilon} (\sigma')^{-\frac{1}{4}} \| P_{\sigma'} ( m(s) s^{\frac{1}{2}} \psi_s ) \bigr\|_{LE_{\sigma'}} \, \frac{\ud \sigma'}{\sigma'}.
\end{align*}
By Schur's test, the last line is bounded in $L^2_\dsigma(0,\frac{1}{2})$ by
\[
 \|\psi_s\|_{\calS}^2 \, \| m(s) s^{\frac{1}{2}} \psi_s(s) \|_{LE},
\]
as desired. This completes the treatment of the term $I_{1}$ and the proof of \eqref{equ:high_freq_termI}.
%%%%%%%%%%%
\subsubsection*{Step 1b: The low-frequency estimate~\eqref{equ:low_freq_termI}}
%%%%%%%%%%%
Using Lemma~\ref{l:bern2}, for any fixed $p\in(1,2)$
\begin{equation}\label{equ:termIA_estimate}
 \begin{aligned}
 &\biggl( \int_{\frac{1}{8}}^4 \, \Bigl\| P_{\geq \sigma} \nabla_k \int_0^1 (P_{\geq \sigma'} \ringA^k) P_{\sigma'} ( m(s) s^{\frac{1}{2}} \psi_s ) \, \frac{\ud \sigma'}{\sigma'} \Bigr\|_{LE_{\low}^\ast}^2 \, \frac{\ud \sigma}{\sigma} \biggr)^{\frac{1}{2}} \\
 &\lesssim \biggl( \int_{\frac{1}{8}}^4 \sigma^{-2+2\alpha} \frac{\ud \sigma}{\sigma} \biggr)^{\frac{1}{2}} \int_0^1 \bigl\| (P_{\geq \sigma'} \ringbfA) P_{\sigma'} ( m(s) s^{\frac{1}{2}} \psi_s ) \bigr\|_{L^2_t L^p_x(A_{\leq 0})} \frac{\ud \sigma'}{\sigma'} \\
 &\quad + \biggl( \int_{\frac{1}{8}}^4 \sigma^{-1} \frac{\ud \sigma}{\sigma} \biggr)^{\frac{1}{2}} \int_0^1 \bigl\| r^2 (P_{\geq \sigma'} \ringbfA) P_{\sigma'} (m(s) s^{\frac{1}{2}} \psi_s) \bigr\|_{L^2_t L^2_x(A_{\geq 0})} \frac{\ud \sigma'}{\sigma'} \\
 &\lesssim \int_0^1 \bigl\| (P_{\geq \sigma'} \ringbfA) P_{\sigma'} (m(s) s^{\frac{1}{2}} \psi_s) \bigr\|_{L^2_t L^p_x(A_{\leq 0})} \frac{\ud \sigma'}{\sigma'} \\
 &\quad + \int_0^1 \bigl\| r^2 (P_{\geq \sigma'} \ringbfA) P_{\sigma'} (m(s) s^{\frac{1}{2}} \psi_s) \bigr\|_{L^2_t L^2_x(A_{\geq 0})} \frac{\ud \sigma'}{\sigma'}.
 \end{aligned}
\end{equation}
We estimate the first term on the right-hand side as
\begin{align*}
 \int_0^1 \bigl\| (P_{\geq \sigma'} \ringbfA) P_{\sigma'} (m(s) s^{\frac{1}{2}} \psi_s) \bigr\|_{L^2_t L^p_x(A_{\leq 0})} \frac{\ud \sigma'}{\sigma'} &\lesssim \int_0^1 \bigl\| P_{\geq \sigma'} \ringA \bigr\|_{L^\infty_t L^{\frac{2p}{2-p}}_x(\bbH^2)} \| P_{\sigma'} ( m(s) s^{\frac{1}{2}} \psi_s ) \|_{L^2_t L^{2}_x(A_{\leq 0})} \, \frac{\ud \sigma'}{\sigma'} \\
 &\lesssim \int_0^1 \bigl\| P_{\geq \sigma'} \ringbfA \bigr\|_{L^\infty_t L^{\frac{2p}{2-p}}_x(\bbH^2)} \| P_{\sigma'} (m(s) s^{\frac{1}{2}} \psi_s) \|_{LE_{\sigma'}} \, \frac{\ud \sigma'}{\sigma'} \\
 &\simeq \int_0^1 (\sigma')^{\frac{1}{4}} \bigl\| P_{\geq \sigma'} \ringbfA \bigr\|_{L^\infty_t L^{\frac{2p}{2-p}}_x(\bbH^2)} (\sigma')^{-\frac{1}{4}} \| P_{\sigma'} (m(s) s^{\frac{1}{2}} \psi_s) \|_{LE_{\sigma'}} \, \frac{\ud \sigma'}{\sigma'}.
\end{align*}
Here we have used the estimate (valid for $\sigma'\leq1$)
%%%%%%%
%%%%%%%
\begin{align*}
\begin{split}
\|f\|_{L_t^2L_x^2(A_{\leq0})}\lesssim (\sigma')^{-\frac{1}{4}}\|f\|_{L_t^2L_x^2(A_{\leq-k_{\sigma'}})} + \biggl( \sum_{\ell=-k_{\sigma'}}^02^{\frac{\ell}{2}} \biggr) \sup_{-k_{\sigma'}\leq \ell\leq0}\|r^{-\frac{1}{2}}f\|_{L_t^2L_x^2(A_\ell)}\lesssim \|f\|_{LE_{\sigma'}}.
\end{split}
\end{align*}
By \eqref{eq:M4} in Lemma~\ref{lem:Mestimates}
\begin{align*}
 \bigl\| P_{\geq \sigma'} \ringbfA(s) \bigr\|_{L^\infty_t L^{\frac{2p}{2-p}}_x(\bbH^2)} \lesssim (\sigma')^{-\frac{p-1}{p}} \| \psi_s \|_{\calS}.
\end{align*}
Correspondingly, choosing $ p \in(1, \frac{4}{3})$, by Cauchy-Schwarz 
\begin{align*}
 &\int_0^1 \bigl\| (P_{\geq \sigma'} \ringbfA) P_{\sigma'} (m(s) s^{\frac{1}{2}} \psi_s) \bigr\|_{L^2_t L^p_x(A_{\leq 0})} \frac{\ud \sigma'}{\sigma'} \\
 &\lesssim \int_0^1 (\sigma')^{\frac{1}{4}} (\sigma')^{-\frac{p-1}{p}} \| \psi_s \|_{\calS} (\sigma')^{-\frac{1}{4}} \| P_{\sigma'} ( m(s) s^{\frac{1}{2}} \psi_s ) \|_{LE_{\sigma'}} \, \frac{\ud \sigma'}{\sigma'} \\
 &\lesssim \| \psi_s \|_{\calS} \biggl( \int_0^1 (\sigma')^{\frac{2}{p} - \frac{3}{2}} \, \frac{\ud \sigma'}{\sigma'} \biggr)^{\frac{1}{2}} \biggl( \int_0^1 (\sigma')^{-\frac{1}{2}} \| P_{\sigma'} (m(s) s^{\frac{1}{2}} \psi_s) \|_{LE_{\sigma'}}^2 \, \frac{\ud \sigma'}{\sigma'} \biggr)^{\frac{1}{2}} \\
 &\lesssim \| \psi_s \|_{\calS} \, \| m(s) s^{\frac{1}{2}} \psi_s(s) \|_{LE}, 
\end{align*}
which is of the desired form.

For the second term on the right-hand side of~\eqref{equ:termIA_estimate} we use \eqref{eq:M3} in Lemma~\ref{lem:Mestimates} and Cauchy-Schwarz to estimate
\begin{align*}
 &\int_0^1 \bigl\| r^2 (P_{\geq \sigma'} \ringbfA) P_{\sigma'} (m(s) s^{\frac{1}{2}} \psi_s) \bigr\|_{L^2_t L^2_x(A_{\geq 0})} \frac{\ud \sigma'}{\sigma'} \\
 &\lesssim \int_0^1 \bigl\| r^4 P_{\geq \sigma'} \ringbfA \bigr\|_{L^\infty_t L^\infty_x(A_{\geq 0})} \bigl\| r^{-2} P_{\sigma'} (m(s) s^{\frac{1}{2}} \psi_s) \bigr\|_{L^2_t L^2_x(A_{\geq 0})} \frac{\ud \sigma'}{\sigma'} \\
 &\lesssim \int_0^1 \bigl\| r^4 P_{\geq \sigma'} \ringbfA \bigr\|_{L^\infty_t L^\infty_x(A_{\geq 0})} \bigl\| P_{\sigma'} (m(s) s^{\frac{1}{2}} \psi_s) \bigr\|_{LE_{\sigma'}} \frac{\ud \sigma'}{\sigma'} \\
 &\lesssim \int_0^1 (\sigma')^{\frac{1}{4}} \| \psi_s \|_{\calS} (\sigma')^{-\frac{1}{4}} \bigl\| P_{\sigma'} (m(s) s^{\frac{1}{2}} \psi_s) \bigr\|_{LE_{\sigma'}} \frac{\ud \sigma'}{\sigma'} \\
 &\lesssim \| \psi_s \|_{\calS} \biggl( \int_0^1 (\sigma')^{-\frac{1}{2}} \bigl\| P_{\sigma'} (m(s) s^{\frac{1}{2}} \psi_s) \bigr\|_{LE_{\sigma'}}^2 \frac{\ud \sigma'}{\sigma'} \biggr)^{\frac{1}{2}} \\
 &\lesssim \| \psi_s \|_{\calS} \bigl\| m(s) s^{\frac{1}{2}} \psi_s \bigr\|_{LE}.
\end{align*}
This is again of the desired form and finishes the proof of the low-frequency estimate~\eqref{equ:low_freq_termI}.
%%%%%%%%%%%%%
%%%%%%%%%%%%%
\subsection*{Step 2: Term $II$}
%%%%%%%%%%%%%
%%%%%%%%%%%%%
Going back to \eqref{eq:maglemma1maindecomp} we intend to bound the term $II$ in $\LLs L^{\frac{4}{3}}_t L^{\frac{4}{3}}_x$. First, 
\begin{align}\label{eq:lemmag1step2temp1}
\begin{split}
 \| II \|_{L^{\frac{4}{3}}_t L^{\frac{4}{3}}_x} &\lesssim \int_0^1 \| \nabla_k ( P_{\geq \sigma'} \ringA^k )(s) \|_{L^2_t L^2_x+L_{t}^{\frac{8}{3}}L_x^{\frac{8}{3}}} \| P_{\sigma'} ( m(s) s^{\frac{1}{2}} \psi_s ) \|_{L^4_t L^4_x\cap L_{t}^{\frac{8}{3}}L_x^{\frac{8}{3}}} \, \frac{\ud \sigma'}{\sigma'}\\
 &\lesssim \|\psi_s\|_\calS\int_0^1 \| P_{\sigma'} ( m(s) s^{\frac{1}{2}} \psi_s ) \|_{L^4_t L^4_x\cap L_{t}^{\frac{8}{3}}L_x^{\frac{8}{3}}} \, \frac{\ud \sigma'}{\sigma'},
\end{split}
\end{align}
where we have used \eqref{eq:divPsigmaAStrich} in Lemma~\ref{lem:noMestimates}. 
To control the right-hand side we argue as follows. First, 
\begin{align}
\begin{split}\label{eq:offdiagapplication1}
 &\int_0^1 \| P_{\sigma'} ( m(s) s^{\frac{1}{2}} \psi_s ) \|_{L^4_t L^4_x\cap L_t^{\frac{8}{3}}L_x^{\frac{8}{3}}} \, \frac{\ud \sigma'}{\sigma'} \\
 &= \int_0^{s} \| P_{\sigma'} ( m(s) s^{\frac{1}{2}} \psi_s ) \|_{L^4_t L^4_x\cap L_t^{\frac{8}{3}}L_x^{\frac{8}{3}}} \, \frac{\ud \sigma'}{\sigma'} + \int_s^1 \| P_{\sigma'} \bigl( m(s) s^{\frac{1}{2}} \psi_s \bigr) \|_{L^4_t L^4_x\cap L_t^{\frac{8}{3}}L_x^{\frac{8}{3}}} \, \frac{\ud \sigma'}{\sigma'}.
\end{split}
\end{align}
The first term on the right is estimated via, 
\EQ{
\int_0^{s} \| P_{\sigma'} ( m(s) s^{\frac{1}{2}} \psi_s ) \|_{L^4_t L^4_x\cap L_t^{\frac{8}{3}}L_x^{\frac{8}{3}}} \, \frac{\ud \sigma'}{\sigma'} & \lesssim  \int_0^s \frac{\sigma'}{s} \| s \Delta e^{\sigma' \Delta} ( m(s) s^{\frac{1}{2}} \psi_s ) \|_{L^4_t L^4_x\cap L_t^{\frac{8}{3}}L_x^{\frac{8}{3}}} \, \frac{\ud \sigma'}{\sigma'}  \\
 &\lesssim \| s \Delta ( m(s) s^{\frac{1}{2}} \psi_s ) \|_{L^4_t L^4_x\cap L_t^{\frac{8}{3}}L_x^{\frac{8}{3}}}. 
 }
To estimate the second term on the right-hand-side of~\eqref{eq:offdiagapplication1} we use the off-diagonal decay estimates established in Corollary~\ref{c:offdiag} for $P_{\sigma'} ( m(s) s^{\frac{1}{2}} \psi_s )$ in $L^4_t L^4_x\cap L_t^{\frac{8}{3}}L_x^{\frac{8}{3}}$. Indeed, taking $\LLs$ of the second term, Corollary~\ref{c:offdiag} yields, 
\EQ{
\Big \|\int_s^1 \| P_{\sigma'} \bigl( m(s) s^{\frac{1}{2}} \psi_s \bigr) \|_{L^4_t L^4_x\cap L_t^{\frac{8}{3}}L_x^{\frac{8}{3}}} \Big\|_{\LLs} & \lesssim  \sum_{\ell = 0}^2  \| s^{\frac{\ell}{2}}(-\De)^{\frac{\ell}{2}}m(s) s^{\frac{1}{2}}\psi_s \|_{L^{2}_{\ds} \cap L^{\infty}_{\ds} \Str_s}  \lesssim  \|\psi_s\|_\calS.
}
Together with \eqref{eq:lemmag1step2temp1} this shows that
%%%%%%%
%%%%%%%
\begin{align*}
\begin{split}
 \| II \|_{\LLs L^{\frac{4}{3}}_t L^{\frac{4}{3}}_x}\lesssim \|\psi_s\|_\calS^2.
\end{split}
\end{align*}
%%%%%%%%%%%%%
%%%%%%%%%%%%%
\subsection*{Step 3: Term $III$}
%%%%%%%%%%%%%
%%%%%%%%%%%%%
We estimate $III$ in \eqref{eq:maglemma1maindecomp} in $L_t^{\frac{4}{3}}L_x^{\frac{4}{3}}$ as
%%%%%%%
%%%%%%%
\begin{align*}
\begin{split}
&\Big\|\int_0^{1}(P_{\leq \sigma'}\ringA(s)^k)\nabla_k P_{\sigma'}m(s)s^{\frac{1}{2}}\psi_s(s)\dsigmap\Big\|_{ L_{t}^{\frac{4}{3}}L_x^{\frac{4}{3}}}\\
&\leq \int_0^{1}\|P_{\leq\sigma'}\ringbfA(s)\|_{ L^2_{t}L_x^2+L_{t}^{\frac{8}{3}}L_x^{\frac{8}{3}} }\|\nabla P_{\sigma'}m(s)s^{\frac{1}{2}}\psi_s(s)\|_{ L^4_{t}L_x^4\cap L_{t}^{\frac{8}{3}}L_x^{\frac{8}{3}}}\dsigmap\\
&\lesssim \int_0^{1}\|P_{\leq\sigma'}\ringbfA(s)\|_{ L^2_{t}L_x^2+L_{t}^{\frac{8}{3}}L_x^{\frac{8}{3}}} (\sigma')^{-\frac{1}{2}}\| P_{\frac{\sigma'}{2}}m(s)s^{\frac{1}{2}}\psi_s(s)\|_{ L^4_{t}L_x^4\cap L_{t}^{\frac{8}{3}}L_x^{\frac{8}{3}}}\dsigmap.
\end{split}
\end{align*}
Applying \eqref{eq:PlesssigmaAStrich} in Lemma~\ref{lem:noMestimates} it follows that
%%%%%%%
%%%%%%%
\begin{align*}
\begin{split}
&\Big\|\int_0^{1}(P_{\leq \sigma'}\ringA(s)^k)\nabla_k P_{\sigma'}m(s)s^{\frac{1}{2}}\psi_s(s)\dsigmap\Big\|_{  L_{t}^{\frac{4}{3}} L_x^{\frac{4}{3}}}\lesssim \|\psi_s\|_{\calS}\int_0^{\frac{1}{2}}\|P_{\sigma'}m(s)s^{\frac{1}{2}}\psi_s(s)\|_{ L^4_{t}L_x^4\cap L_{t}^{\frac{8}{3}}L_x^{\frac{8}{3}}}\,\dsigmap.
\end{split}
\end{align*}
Now we can bound the last integral using the off-diagonal decay estimates from Corollary~\ref{c:offdiag} as in \eqref{eq:offdiagapplication1} in Step 2 above to conclude that
%%%%%%%
%%%%%%%
\begin{align*}
\begin{split}
&\|III\|_{ L^\infty_\ds \cap L^2_\ds L_{t,x}^{\frac{4}{3}}}\lesssim \|\psi_s\|_{\calS}\|\jap{s\Delta}m(s)s^{\frac{1}{2}}\psi_s(s)\|_{L^\infty_\ds \cap L^2_\ds (L^4_{t}L_x^4\cap L_{t}^{\frac{8}{3}}L_x^{\frac{8}{3}})}\lesssim \|\psi_s\|_{\calS}^2.
\end{split}
\end{align*}
%%%%%%%%%%%%%
%%%%%%%%%%%%%
\subsection*{Step 4: Term $IV$}
%%%%%%%%%%%%%
%%%%%%%%%%%%%
We place $IV$ in \eqref{eq:maglemma1maindecomp} in $L^{\frac{4}{3}}_t L^{\frac{4}{3}}_x$ and bound by
\begin{align*}
 \| \ringA^k \nabla_k P_{\geq 1} ( m(s) s^{\frac{1}{2}} \psi_s ) \|_{L^{\frac{4}{3}}_t L^{\frac{4}{3}}_x} &\lesssim \| \ringA \|_{L^{\frac{8}{3}}_t L^{\frac{8}{3}}_x} \| m(s) s^{\frac{1}{2}} \psi_s \|_{L^{\frac{8}{3}}_t L^{\frac{8}{3}}_x} \lesssim \|\psi_s\|_{\calS} \| m(s) s^{\frac{1}{2}} \psi_s \|_{L^{\frac{8}{3}}_t L^{\frac{8}{3}}_x},
\end{align*}
where we have used estimate~\eqref{eq:ringAL83} from Lemma~\ref{lem:noMestimates}. It follows that 
%%%%%%%
%%%%%%%
\begin{align*}
\begin{split}
\|IV\|_{\LLs L_t^{\frac{4}{3}}L_x^{\frac{4}{3}}}\lesssim \|\psi_s\|_\calS^2.
\end{split}
\end{align*}
\end{proof}
%%%%%%%%%%%%%%%
%%%%%%%%%%%%%%%

The proof of Lemma~\ref{lem:magnetic2} is similar to that of Lemma~\ref{lem:magnetic1}. The main difference is that in our applications of the radial Sobolev estimate from Lemma~\ref{lem:radial_sobolev} we need to treat the radial and angular directions differently to make sure that each term is differentiated with respect to $\Omega$ at most once.

\begin{proof}[Proof of Lemma~\ref{lem:magnetic2}]
Starting as in the proof of Lemma~\ref{lem:magnetic1} we decompose $\calL_\Omega \ringA^k \nabla_k ( m(s) s^{\frac{1}{2}} \psi_s )$ as
\begin{align*}
 \calL_\Omega \ringA^k \nabla_k ( m(s) s^{\frac{1}{2}} \psi_s ) &= \nabla_k \int_0^1 ( P_{\geq \sigma'} \calL_\Omega \ringA ^k) P_{\sigma'} ( m(s) s^{\frac{1}{2}} \psi_s ) \, \frac{\ud \sigma'}{\sigma'} + \int_0^1 \nabla_k ( P_{\geq \sigma'} \calL_\Omega \ringA^k) P_{\sigma'} ( m(s) s^{\frac{1}{2}} \psi_s ) \, \frac{\ud \sigma'}{\sigma'} \\
 &\quad \quad + \int_0^1 (P_{\leq \sigma'} \calL_\Omega \ringA^k) \nabla_k P_{\sigma'} ( m(s) s^{\frac{1}{2}} \psi_s ) \, \frac{\ud \sigma'}{\sigma'} + \calL_\Omega \ringA^k \nabla_k P_{\geq 1} ( m(s) s^{\frac{1}{2}} \psi_s ) \\
 &=: I\Omega + II\Omega + III\Omega + IV\Omega.
\end{align*}
The terms $II\Omega$--$IV\Omega$ can be estimated using just Strichartz estimates in the exact same manner as the terms~$II$--$IV$ in the poof of Lemma~\ref{lem:magnetic1}. The reason is that we never used the radial Sobolev estimates from Lemma~\ref{lem:radial_sobolev} in bounding these terms, hence the presence of $\calL_\Omega$ does not affect the proofs. To estimate the term $I\Omega$ we proceed as in the treatment of term $I$ in the proof of Lemma~\ref{lem:magnetic1}, keeping in mind to place the term which carries $\calL_\Omega$ in $L^2_\theta$.
The desired estimate is
\[
 \| I\Omega \|_{\LLs LE^\ast} \lesssim \| \psi_s \|_{\calS}^2,
\]
which boils down to establishing the following two estimates:
\begin{align}
 \biggl\| \sigma^{\frac{1}{4}} \Bigl\| P_\sigma \nabla_k \int_0^1 ( P_{\geq \sigma'} \calL_\Omega \ringA ^k) P_{\sigma'} ( m(s) s^{\frac{1}{2}} \psi_s ) \, \frac{\ud \sigma'}{\sigma'}  \Bigr\|_{LE_\sigma^\ast} \biggr\|_{L^\infty_{\frac{\ud s}{s}} \cap L^2_{\frac{\ud s}{s}} L^2_{\frac{\ud \sigma}{\sigma}} (0, \frac{1}{2})} &\lesssim \| \psi_s \|_{\calS}^2 .\label{equ:high_freq_termI_Omega} \\
 \biggl\| \Bigl\| P_{\geq \sigma} \nabla_k \int_0^1 ( P_{\geq \sigma'} \calL_\Omega \ringA ^k) P_{\sigma'} ( m(s) s^{\frac{1}{2}} \psi_s ) \, \frac{\ud \sigma'}{\sigma'}  \Bigr\|_{LE_\low^\ast} \biggr\|_{L^\infty_{\frac{\ud s}{s}} \cap L^2_{\frac{\ud s}{s}} L^2_{\frac{\ud \sigma}{\sigma}} (\frac{1}{8}, 4)} &\lesssim \| \psi_s \|_{\calS}^2 .\label{equ:low_freq_termI_Omega}
\end{align}

We begin with the more difficult high-frequency estimate~\eqref{equ:high_freq_termI_Omega} and further decompose into 
\begin{align}\label{eq:IOmegaHF12}
\begin{split}
 \sigma^{\frac{1}{4}} P_\sigma \nabla_k \int_0^1 ( P_{\geq \sigma'} \calL_\Omega \ringA ^k) P_{\sigma'} ( m(s) s^{\frac{1}{2}} \psi_s ) \, \frac{\ud \sigma'}{\sigma'} &= \sigma^{\frac{1}{4}} P_\sigma \nabla_k \int_0^\sigma ( P_{\geq \sigma'} \calL_\Omega \ringA ^k) P_{\sigma'} ( m(s) s^{\frac{1}{2}} \psi_s ) \, \frac{\ud \sigma'}{\sigma'} \\
 &\quad + \sigma^{\frac{1}{4}} P_\sigma \nabla_k \int_\sigma^1 ( P_{\geq \sigma'} \calL_\Omega \ringA ^k) P_{\sigma'} ( m(s) s^{\frac{1}{2}} \psi_s ) \, \frac{\ud \sigma'}{\sigma'} \\
 &=: I\Omega_{1} + I\Omega_{2}.
 \end{split}
\end{align}
Starting with the second term we want to show that 
\begin{equation} \label{equ:IOmegaHF2_bound}
 \Bigl\| \bigl\| I\Omega_{2} \bigr\|_{LE_\sigma^\ast} \Bigr\|_{L^\infty_{\frac{\ud s}{s}} \cap L^2_{\frac{\ud s}{s}} L^2_{\frac{\ud \sigma}{\sigma}} (0, \frac{1}{2})} \lesssim \| \psi_s \|_{\calS}^2, 
\end{equation}
and for this it is more favorable to let the derivative fall back inside the integral and consider 
\begin{align}\label{eq:IOmegaHF2decomp}
  I\Omega_{2} = \sigma^{\frac{1}{4}} P_\sigma \int_\sigma^1 \nabla_k ( P_{\geq \sigma'} \calL_\Omega \ringA ^k) P_{\sigma'} ( m(s) s^{\frac{1}{2}} \psi_s ) \, \frac{\ud \sigma'}{\sigma'}  + \sigma^{\frac{1}{4}} P_\sigma \int_\sigma^1 ( P_{\geq \sigma'} \calL_\Omega \ringA ^k) \nabla_k P_{\sigma'} ( m(s) s^{\frac{1}{2}} \psi_s ) \, \frac{\ud \sigma'}{\sigma'} .
\end{align}
For the second term on the right-hand side we drop the projection $P_\sigma$ in $LE_{\sigma}^\ast$ and use that $\|F\|_{LE_\sigma^\ast} \lesssim \|F\|_{LE_{\sigma'}^\ast}$ (since $\sigma' \geq \sigma$) to write
\begin{align*}
\Big\| \sigma^{\frac{1}{4}} P_\sigma \int_\sigma^1 ( P_{\geq \sigma'} \calL_\Omega \ringA ^k) \nabla_k P_{\sigma'} ( m(s) s^{\frac{1}{2}} \psi_s ) \, \frac{\ud \sigma'}{\sigma'} \Big\|_{LE_\sigma^\ast} &\lesssim \int_\sigma^1 \sigma^{\frac{1}{4}} \bigl\| ( P_{\geq \sigma'} \calL_\Omega \ringA^k) \nabla_k P_{\sigma'} ( m(s) s^{\frac{1}{2}} \psi_s ) \bigr\|_{LE_{\sigma'}^\ast} \, \frac{\ud \sigma'}{\sigma'}.
\end{align*}
First we have to estimate the integrand a bit more carefully as 
\begin{align*}
 &\bigl\| ( P_{\geq \sigma'} \calL_\Omega \ringA ^k) \nabla_k P_{\sigma'} ( m(s) s^{\frac{1}{2}} \psi_s ) \bigr\|_{LE_{\sigma'}^\ast} \\
 &= (\sigma')^{\frac{1}{4}} \bigl\| ( P_{\geq \sigma'} \calL_\Omega \ringA^k) \nabla_k P_{\sigma'} ( m(s) s^{\frac{1}{2}} \psi_s ) \bigr\|_{L^2_t L^2_x(A_{\leq -k_{\sigma'}})} \\
 &\quad + \sum_{-k_{\sigma'} \leq \ell < 0} 2^{\frac{1}{2} \ell} \bigl\| ( P_{\geq \sigma'} \calL_\Omega \ringA^k) \nabla_k P_{\sigma'} ( m(s) s^{\frac{1}{2}} \psi_s ) \bigr\|_{L^2_t L^2_x(A_\ell)} \\
 &\quad + \bigl\| r^2 ( P_{\geq \sigma'} \calL_\Omega \ringA^k) \nabla_k P_{\sigma'} ( m(s) s^{\frac{1}{2}} \psi_s ) \bigr\|_{L^2_t L^2_x(A_{\geq 0})} \\
 &\leq (\sigma')^{\frac{1}{2}} \bigl\| P_{\geq \sigma'} \calL_\Omega \ringbfA \bigr\|_{L^\infty_t L^\infty_x(A_{\leq -k_{\sigma'}})} (\sigma')^{-\frac{1}{4}} \bigl\| \nabla P_{\sigma'} ( m(s) s^{\frac{1}{2}} \psi_s ) \bigr\|_{L^2_t L^2_x(A_{\leq -k_{\sigma'}})} \\
 &\quad + \sum_{-k_{\sigma'} \leq \ell < 0} 2^\ell \bigl\| P_{\geq \sigma'} \calL_\Omega \ringbfA \bigr\|_{L^\infty_t  L^\infty_r L^2_\theta (A_\ell)} 2^{-\frac{1}{2} \ell} \bigl\| \sinh^{\frac{1}{2}}(r) \nabla P_{\sigma'} ( m(s) s^{\frac{1}{2}} \psi_s ) \bigr\|_{L^2_t  L^2_r L^\infty_\theta (A_\ell)} \\
 &\quad + \bigl\| r^4 P_{\geq \sigma'} \calL_\Omega \ringbfA \bigr\|_{L^\infty_t  L^\infty_r L^2_\theta (A_{\geq 0})} \bigl\| \sinh^{\frac{1}{2}}(r) \, r^{-2} \nabla P_{\sigma'} ( m(s) s^{\frac{1}{2}} \psi_s ) \bigr\|_{L^2_t  L^2_r L^\infty_\theta (A_{\geq 0})} .
\end{align*}
Applying Sobolev embedding on the sphere $\bbS^1$ we further bound this by
\begin{align*}
 &\lesssim (\sigma')^{\frac{1}{2}} \bigl\| P_{\geq \sigma'} \calL_\Omega \ringbfA \bigr\|_{L^\infty_t L^\infty_x(A_{\leq -k_{\sigma'}})} (\sigma')^{-\frac{1}{4}} \bigl\| \nabla P_{\sigma'} ( m(s) s^{\frac{1}{2}} \psi_s ) \bigr\|_{L^2_t L^2_x(A_{\leq -k_{\sigma'}})} \\
 &\quad + \biggl( \sum_{-k_{\sigma'} \leq \ell < 0} 2^\ell \bigl\| P_{\geq \sigma'} \calL_\Omega \ringbfA \bigr\|_{L^\infty_t L^\infty_r L^2_\theta (A_\ell)} \biggr) \sup_{-k_{\sigma'} \leq \ell < 0}  2^{-\frac{1}{2} \ell} \bigl\| \sinh^{\frac{1}{2}}(r) \nabla P_{\sigma'} (\langle \Omega \rangle  m(s) s^{\frac{1}{2}}   \psi_s ) \bigr\|_{L^2_t L^2_r  L^2_\theta (A_\ell)} \\
 &\quad + \bigl\| r^4 P_{\geq \sigma'} \calL_\Omega \ringbfA \bigr\|_{L^\infty_t L^\infty_r L^2_\theta (A_{\geq 0})} \bigl\| \sinh^{\frac{1}{2}}(r) \, r^{-2} \nabla P_{\sigma'} ( \langle \Omega \rangle m(s) s^{\frac{1}{2}}    \psi_s ) \bigr\|_{L^2_t L^2_r  L^2_\theta (A_{\geq 0})} \\
 &\lesssim \biggl( (\sigma')^{\frac{1}{2}} \bigl\| P_{\geq \sigma'} \calL_\Omega \ringbfA \bigr\|_{L^\infty_t L^\infty_x(\bbH^2)} + \sum_{-k_{\sigma'} \leq \ell < 0} 2^\ell \bigl\| P_{\geq \sigma'} \calL_\Omega \ringbfA \bigr\|_{L^\infty_t L^\infty_r L^2_\theta (A_\ell)} \\
 &\qquad \qquad \qquad \qquad \qquad \qquad + \bigl\| r^4 P_{\geq \sigma'} \calL_\Omega \ringbfA \bigr\|_{L^\infty_t L^\infty_r L^2_\theta (A_{\geq 0})} \biggr) \bigl\| \nabla P_{\sigma'} (\langle \Omega \rangle m(s) s^{\frac{1}{2}}  \psi_s ) \bigr\|_{ LE_{\sigma'} } \\
 &\lesssim \bigl( M\Omega_1 + M\Omega_2 + M\Omega_3 \bigr) (\sigma')^{-\frac{1}{2}} \bigl\| P_{\frac{\sigma'}{2}} (\langle \Omega \rangle m(s) s^{\frac{1}{2}}  \psi_s ) \bigr\|_{ LE_{\sigma'} } ,
\end{align*}
where
%%%%%%%
%%%%%%%
\begin{align*}
\begin{split}
&M\Omega_1:=(\sigma')^{\frac{1}{2}} \bigl\| P_{\geq \sigma'} \calL_\Omega \ringbfA \bigr\|_{L^\infty_t L^\infty_x(\bbH^2)},\\
&M\Omega_2:=\sum_{-k_{\sigma'} \leq \ell < 0} 2^\ell \bigl\| P_{\geq \sigma'} \calL_\Omega \ringbfA \bigr\|_{L^\infty_t L^\infty_r L^2_\theta (A_\ell)} ,\\
&M\Omega_3:=\bigl\| r^4 P_{\geq \sigma'} \calL_\Omega \ringbfA \bigr\|_{L^\infty_t L^\infty_r L^2_\theta (A_{\geq 0})}.
\end{split}
\end{align*}
If we can show that
\[
 M\Omega_j \lesssim \|\psi_s\|_{\calS} \, \text{ for } j = 1, 2, 3,
\]
the desired bound~\eqref{equ:IOmegaHF2_bound} follows by Schur's test from
\begin{align*}
 &\int_\sigma^1 \sigma^{\frac{1}{4}} \bigl\| ( P_{\geq \sigma'} \calL_\Omega \ringA^k) \nabla_k P_{\sigma'} ( m(s) s^{\frac{1}{2}} \psi_s ) \bigr\|_{LE_{\sigma'}^\ast} \, \frac{\ud \sigma'}{\sigma'} \\
 &\lesssim \int_\sigma^1 \Bigl( \frac{\sigma}{\sigma'} \Bigr)^{\frac{1}{4}} \bigl( M\Omega_1 + M\Omega_2 + M\Omega_3 \bigr) (\sigma')^{-\frac{1}{4}} \bigl\| P_{\frac{\sigma'}{2}} ( m(s) s^{\frac{1}{2}} \langle \Omega \rangle \psi_s ) \bigr\|_{ LE_{\sigma'} }\dsigmap. 
\end{align*}
Now the desired bounds on $M\Omega_j$, $j = 1, 2, 3$, are precisely the contents of \eqref{eq:MOmega1}, \eqref{eq:MOmega2}, \eqref{eq:MOmega3} in Lemma~\ref{lem:MOmegaestimates}. The treatment of the first term on the right-hand side of \eqref{eq:IOmegaHF2decomp} is similar where we instead use estimates \eqref{eq:tilMOmega1}, \eqref{eq:tilMOmega2}, \eqref{eq:tilMOmega3} in Lemma~\ref{lem:MOmegaestimates}. This completes the proof of \eqref{equ:IOmegaHF2_bound}.

Turning to the first term on the right-hand side of \eqref{eq:IOmegaHF12} we would like to prove
\begin{equation} \label{equ:IOmegaHF1_bound}
 \Bigl\| \bigl\| I\Omega_{1} \bigr\|_{LE_\sigma^\ast} \Bigr\|_{L^\infty_{\frac{\ud s}{s}} \cap L^2_{\frac{\ud s}{s}} L^2_{\frac{\ud \sigma}{\sigma}} (0, \frac{1}{2})} \lesssim \| \psi_s \|_{\calS}^2. 
\end{equation}
To this end we use Lemma~\ref{l:bern1} with $p\in(1,2)$ fixed, and upon noting that $-k_\sigma \geq -k_{\sigma'}$ for $0 < \sigma' \leq \sigma$, write 
\begin{align*}
 \bigl\| I\Omega_{1} \bigr\|_{LE_\sigma^\ast} &\lesssim \int_0^\sigma \sigma^{-\alpha} \bigl\| ( P_{\geq \sigma'} \calL_\Omega \ringbfA ) P_{\sigma'} ( m(s) s^{\frac{1}{2}} \psi_s ) \bigr\|_{L^2_t L^p_x(A_{\leq -k_{\sigma'}})} \, \frac{\ud \sigma'}{\sigma'} \\
  &\quad + \int_0^\sigma \sigma^{-\alpha} \sum_{-k_{\sigma'} \leq \ell < -k_\sigma} \bigl\| ( P_{\geq \sigma'} \calL_\Omega \ringbfA ) P_{\sigma'} ( m(s) s^{\frac{1}{2}} \psi_s ) \bigr\|_{L^2_t L^p_x(A_{\ell})} \, \frac{\ud \sigma'}{\sigma'} \\
  &\quad + \int_0^\sigma \sigma^{-\frac{1}{4}} \sum_{-k_\sigma \leq \ell < 0} 2^{\frac{1}{2} \ell} \bigl\| ( P_{\geq \sigma'} \calL_\Omega \ringbfA ) P_{\sigma'} ( m(s) s^{\frac{1}{2}} \psi_s ) \bigr\|_{L^2_t L^2_x(A_{\ell})} \, \frac{\ud \sigma'}{\sigma'} \\
  &\quad + \int_0^\sigma \sigma^{-\frac{1}{4}} \bigl\| r^2 ( P_{\geq \sigma'} \calL_\Omega \ringbfA ) P_{\sigma'} ( m(s) s^{\frac{1}{2}} \psi_s ) \bigr\|_{L^2_t L^2_x(A_{\geq 0})} \, \frac{\ud \sigma'}{\sigma'} \\
 &\lesssim \int_0^\sigma  \sigma^{-\alpha} \bigl\| P_{\geq \sigma'} \calL_\Omega \ringbfA \bigr\|_{L^\infty_t L^{\frac{2p}{2-p}}_x(\bbH^2)} \bigl\| P_{\sigma'} ( m(s) s^{\frac{1}{2}} \psi_s ) \bigr\|_{L^2_t L^2_x(A_{\leq -k_{\sigma'}})} \, \frac{\ud \sigma'}{\sigma'} \\
 &\quad + \int_0^\sigma \sigma^{-\alpha} \sum_{-k_{\sigma'} \leq \ell < -k_\sigma} 2^{\frac{1}{2} \ell} \bigl\| P_{\geq \sigma'} \calL_\Omega \ringbfA \bigr\|_{L^\infty_t L^{\frac{2p}{2-p}}_x(A_\ell)} 2^{-\frac{1}{2} \ell} \bigl\| P_{\sigma'} ( m(s) s^{\frac{1}{2}} \psi_s ) \bigr\|_{L^2_t L^2_x(A_{\ell})} \, \frac{\ud \sigma'}{\sigma'} \\
 &\quad + \int_0^\sigma \sigma^{-\frac{1}{4}} \sum_{-k_\sigma \leq \ell < 0} 2^\ell \bigl\| P_{\geq \sigma'} \calL_\Omega \ringbfA \bigr\|_{L^\infty_t  L^\infty_r L^2_\theta(A_\ell)} 2^{-\frac{1}{2} \ell} \bigl\| \sinh^{\frac{1}{2}}(r) P_{\sigma'} ( m(s) s^{\frac{1}{2}} \psi_s ) \bigr\|_{L^2_t  L^2_r L^\infty_\theta  (A_{\ell})} \, \frac{\ud \sigma'}{\sigma'} \\
 &\quad + \int_0^\sigma \sigma^{-\frac{1}{4}} \bigl\| r^4 ( P_{\geq \sigma'} \calL_\Omega \ringbfA ) \bigr\|_{L^\infty_t  L^\infty_r L^2_\theta (A_{\geq 0})} \bigl\| \sinh^{\frac{1}{2}}(r) \, r^{-2} P_{\sigma'} ( m(s) s^{\frac{1}{2}} \psi_s ) \bigr\|_{L^2_t  L^2_r L^\infty_\theta  (A_{\geq 0})} \, \frac{\ud \sigma'}{\sigma'}.
\end{align*}
Applying $L^\infty_\theta$-Sobolev embedding on $\bbS^1$ to bound
%%%%%%%
%%%%%%%
\begin{align*}
\begin{split}
 \bigl\| \sinh^{\frac{1}{2}}(r) P_{\sigma'} ( m(s) s^{\frac{1}{2}} \psi_s ) \bigr\|_{L^2_t  L^2_r L^\infty_\theta  (A_{\ell})} &\lesssim \bigl\| P_{\sigma'} ( \jap{\Omega}m(s) s^{\frac{1}{2}} \psi_s ) \bigr\|_{L^2_t L^{2}_{x} (A_{\ell})},\\
 \bigl\| \sinh^{\frac{1}{2}}(r) \, r^{-2} P_{\sigma'} ( m(s) s^{\frac{1}{2}} \psi_s ) \bigr\|_{L^2_t  L^2_r L^\infty_\theta  (A_{\geq 0})} &\lesssim \bigl\| r^{-2} P_{\sigma'} ( \jap{\Omega}m(s) s^{\frac{1}{2}} \psi_s ) \bigr\|_{L^2_t L^{2}_{x} (A_{\geq 0})},
\end{split}
\end{align*}
the desired estimate~\eqref{equ:IOmegaHF1_bound} then follows by Schur's test (as for the term $I_{1}$ in the proof of Lemma~\ref{lem:magnetic1}) using the bounds \eqref{eq:MOmega2}, \eqref{eq:MOmega3}, \eqref{eq:MOmega4}, and \eqref{eq:MOmega5} from Lemma~\ref{lem:MOmegaestimates}. This completes the proof of \eqref{equ:IOmegaHF1_bound}, and hence of the high-frequency estimate \eqref{equ:high_freq_termI_Omega}.

The proof of the low-frequency estimate \eqref{equ:low_freq_termI_Omega} is based on that of \eqref{equ:low_freq_termI} in the proof of Lemma~\ref{lem:magnetic1}, using similar modifications as above to separate the angular and radial directions, and where in the process we use the estimates \eqref{eq:MOmega3} and \eqref{eq:MOmega4} in Lemma~\ref{lem:MOmegaestimates}. We omit the details.
\end{proof}

\subsection{Analysis of the quadratic, cubic, and higher terms} \label{subsec:7non-magnetic}
In this subsection we complete our analysis by proving Lemmas~\ref{lem:cubic_and_higher} and~\ref{lem:quadratic_terms}.
\begin{proof}[Proof of Lemma~\ref{lem:cubic_and_higher}]
 The proofs of the asserted estimates are all straightforward consequences of H\"older's inequality and the previously established bounds on $\ringA$, $A_t$, and $\ringPsi$ in Lemma~\ref{lem:bounds_ringPsi}, Lemma~\ref{lem:bounds_ringA_Str}, and Lemma~\ref{lem:bounds_Psi_A_t}. We begin with the proof of~\eqref{equ:nonlinear_estimates_cubic1}. From~\eqref{equ:bound_ringPsi_Linftys_Str_2} we obtain 
 \begin{align*}
  \| | \ringPsi|^2 \jap{\Omega} (m(s) s^{\frac{1}{2}} \psi_s) \|_{L^\infty_{\ds} \cap L^2_{\ds}L^{\frac{4}{3}}_t L^{\frac{4}{3}}_x} \lesssim \| \ringPsi \|_{L^\infty_\ds L^4_t L^4_x}^2 \| \langle \Omega \rangle (m(s) s^{\frac{1}{2}} \psi_s)\|_{L^\infty_\ds \cap L^2_\ds L^4_t L^4_x} \lesssim \|\psi_s\|_{\calS}^3.
 \end{align*}
 Analogously, using~\eqref{equ:bound_ringPsi_Linftys_Str_2}, the estimate~\eqref{equ:nonlinear_estimates_cubic2} follows from
 \begin{align*}
  \| |\calL_\Omega \ringPsi| |\ringPsi| (m(s) s^{\frac{1}{2}} \psi_s) \|_{L^\infty_{\ds} \cap L^2_{\ds}L^{\frac{4}{3}}_t L^{\frac{4}{3}}_x} \lesssim \| \calL_\Omega \ringPsi \|_{L^\infty_\ds L^4_t L^4_x} \| \ringPsi \|_{L^\infty_\ds L^4_t L^4_x} \| m(s) s^{\frac{1}{2}} \psi_s \|_{L^\infty_\ds \cap L^2_\ds L^4_t L^4_x} \lesssim \|\psi_s\|_{\calS}^3.
 \end{align*}
 Next, invoking the bounds~\eqref{eq:bound_ringAL_Linftys_Str} and~\eqref{eq:bound_ringAQ_Linftys_Str}, we infer the estimate~\eqref{equ:nonlinear_estimates_cubic3} from
 \begin{align*}
  \| \ringA_k \ringA^k \jap{\Omega} (m(s) s^{\frac{1}{2}} \psi_s) \|_{L^\infty_{\ds} \cap L^2_{\ds}L^{\frac{4}{3}}_t L^{\frac{4}{3}}_x} \lesssim \| \ringA \|_{L^\infty_{\ds} L^4_t L^4_x}^2 \| \jap{\Omega} (m(s) s^{\frac{1}{2}} \psi_s) \|_{L^\infty_{\ds} \cap L^2_{\ds} L^4_t L^4_x} \lesssim \|\psi_s\|_{\calS}^3
 \end{align*}
 and the estimate~\eqref{equ:nonlinear_estimates_cubic4} follows similarly from
 \begin{align*}
  \| \calL_\Omega \ringA_k \ringA^k (m(s) s^{\frac{1}{2}} \psi_s) \|_{L^\infty_{\ds} \cap L^2_{\ds}L^{\frac{4}{3}}_t L^{\frac{4}{3}}_x} &\lesssim \| \calL_\Omega \ringA \|_{L^\infty_{\ds} L^4_t L^4_x} \| \ringA \|_{L^\infty_{\ds} L^4_t L^4_x} \| \jap{\Omega} (m(s) s^{\frac{1}{2}} \psi_s) \|_{L^\infty_{\ds} \cap L^2_{\ds} L^4_t L^4_x} \\ &\lesssim \|\psi_s\|_{\calS}^3.
 \end{align*}
 Moreover, using~\eqref{eq:bound_At_Linftys_L2tx}, we conclude that
 \begin{align*}
  \| A_t \jap{\Omega} (m(s) s^{\frac{1}{2}} \psi_s) \|_{L^\infty_{\ds} \cap L^2_{\ds} L^{\frac{4}{3}}_t L^{\frac{4}{3}}_x} &\lesssim \| A_t \|_{L^\infty_{\ds} L^2_t L^2_x} \| \jap{\Omega} (m(s) s^{\frac{1}{2}} \psi_s) \|_{L^\infty_{\ds} \cap L^2_{\ds} L^4_t L^4_x} \lesssim \|\psi_s\|_{\calS}^3,
 \end{align*}
 which proves~\eqref{equ:nonlinear_estimates_cubic5}. In a similar manner, we can use~\eqref{eq:bound_At_Linftys_L2tx} to deduce~\eqref{equ:nonlinear_estimates_cubic6}. Finally, the estimate~\eqref{equ:nonlinear_estimates_cubic7} is a consequence of~\eqref{eq:bound_nabla_ringAQ_Linftys_L2tx} and 
 \begin{align*}
  \| (\nabla_k \AQ^k) \jap{\Omega} (m(s) s^{\frac{1}{2}} \psi_s) \|_{L^\infty_{\ds} \cap L^2_{\ds}L^{\frac{4}{3}}_t L^{\frac{4}{3}}_x} \lesssim \| \nabla \ringA_Q \|_{L^\infty_{\ds} L^2_t L^2_x} \| \jap{\Omega} (m(s) s^{\frac{1}{2}} \psi_s) \|_{L^\infty_{\ds} \cap L^2_{\ds}L^4_t L^4_x} \lesssim \|\psi_s\|_{\calS}^3.
 \end{align*}
 The last estimate~\eqref{equ:nonlinear_estimates_cubic8} follows analogously.
\end{proof}
%%%%%%%%%%%%%
%%%%%%%%%%%%%
Finally, we prove Lemma~\ref{lem:quadratic_terms}.

%%%%%%%%%%%%%
%%%%%%%%%%%%%
\begin{proof}[Proof of Lemma~\ref{lem:quadratic_terms}]
 The asserted estimates are again straightforward consequences of H\"older's inequality and the previously established bounds on $\ringA$ and $\ringPsi$. In order to prove~\eqref{equ:nonlinear_estimates_quadratic1}, we use~\eqref{equ:bound_ringPsi_Linftys_Str_2} to conclude that
 \begin{align*}
  \| |\Psi^\infty| |\ringPsi| \jap{\Omega} (m(s) s^{\frac{1}{2}} \psi_s) \|_{L^\infty_{\ds} \cap L^2_{\ds} L^{\frac{4}{3}}_t L^{\frac{4}{3}}_x} \lesssim \|\Psi^\infty\|_{L^\infty_x} \| \ringPsi \|_{L^\infty_{\ds} L^{\frac{8}{3}}_t L^{\frac{8}{3}}_x} \| \jap{\Omega} (m(s) s^{\frac{1}{2}} \psi_s) \|_{L^\infty_{\ds} \cap L^2_{\ds} L^{\frac{8}{3}}_t L^{\frac{8}{3}}_x} \lesssim \|\psi_s\|_{\calS}^2.
 \end{align*}
 The estimate~\eqref{equ:nonlinear_estimates_quadratic2} follows analogously from~\eqref{equ:bound_ringPsi_Linftys_Str_2}. Next, we invoke~\eqref{eq:bound_ringAL_Linftys_Str} and \eqref{eq:bound_ringAQ_Linftys_Str} to deduce 
 \begin{align*}
  \| A^\infty_k \ringA^k\jap {\Omega} (m(s) s^{\frac{1}{2}} \psi_s) \|_{L^\infty_{\ds} \cap L^2_{\ds}L^{\frac{4}{3}}_t L^{\frac{4}{3}}_x} &\lesssim \| \bfA^\infty \|_{L^\infty_x} \| \ringA \|_{L^\infty_{\ds} L^{\frac{8}{3}}_t L^{\frac{8}{3}}_x} \| \jap{\Omega} (m(s) s^{\frac{1}{2}} \psi_s) \|_{L^\infty_{\ds} \cap L^2_{\ds} L^{\frac{8}{3}}_t L^{\frac{8}{3}}_x} \lesssim \|\psi_s\|_{\calS}^2,
 \end{align*}
 which establishes~\eqref{equ:nonlinear_estimates_quadratic3}. The estimate~\eqref{equ:nonlinear_estimates_quadratic4} follows similarly. Finally, we use~\eqref{eq:bound_nabla_ringAL_Linftys_Str} to deduce~\eqref{equ:nonlinear_estimates_quadratic5} from
 \begin{align*}
  \| (\nabla_k\AL^k) \jap{\Omega} (m(s) s^{\frac{1}{2}} \psi_s) \|_{L^\infty_{\ds} \cap L^2_{\ds}L^{\frac{4}{3}}_t L^{\frac{4}{3}}_x} \lesssim \| \nabla \ringA_L \|_{L^\infty_{\ds} L^{\frac{8}{3}}_t L^{\frac{8}{3}}_x} \| \jap{\Omega} (m(s) s^{\frac{1}{2}} \psi_s) \|_{L^\infty_{\ds} \cap L^2_{\ds}L^{\frac{8}{3}}_t L^{\frac{8}{3}}_x} \lesssim \|\psi_s\|_{\calS}^2.
 \end{align*}
 Then \eqref{equ:nonlinear_estimates_quadratic6} can be proven analogously.
\end{proof}

\bibliographystyle{plain}
\bibliography{researchbib}

\bigskip

\centerline{\scshape Andrew Lawrie}
\smallskip
{\footnotesize
 \centerline{Department of Mathematics, Massachusetts Institute of Technology}
\centerline{77 Massachusetts Ave, 2-267, Cambridge, MA 02139, U.S.A.}
\centerline{\email{ alawrie@mit.edu}}
} 

\medskip

\centerline{\scshape Jonas L\"uhrmann}
\smallskip
{\footnotesize
 \centerline{Department of Mathematics, Texas A\&M University}
 \centerline{Blocker 218B, College Station, TX 77843-3368, U.S.A.}
 \centerline{\email{luhrmann@math.tamu.edu}}
}

\medskip

\centerline{\scshape Sung-Jin Oh}
\smallskip
{\footnotesize
 \centerline{Department of Mathematics, UC Berkeley}
\centerline{Evans Hall 970, Berkeley, CA 94720-3840, U.S.A.}
\centerline{\email{sjoh@math.berkley.edu}}
} 

\medskip

\centerline{\scshape Sohrab Shahshahani}
\medskip
{\footnotesize
% please put the address of the first author
 \centerline{Department of Mathematics, University of Massachusetts, Amherst}
\centerline{710 N. Pleasant Street,
Amherst, MA 01003-9305, U.S.A.}
\centerline{\email{sohrab@math.umass.edu}}
}

\end{document}